\documentclass[12pt,a4paper,reqno]{amsart}

\usepackage[applemac]{inputenc}
\usepackage{amsmath}
\usepackage{amsthm}
\usepackage{amscd}
\usepackage{amsfonts}
\usepackage{amssymb}
\usepackage{mathrsfs}
\usepackage{geometry}
\usepackage{enumerate}
\usepackage{xcolor}

\newgeometry{top=2.2cm,bottom=1.78cm,left=2cm,right=2cm}

\newtheorem{thm}{Theorem}[section]
\newtheorem{cor}[thm]{Corollary}
\newtheorem{lem}[thm]{Lemma}
\newtheorem{prop}[thm]{Proposition}

\numberwithin{equation}{thm}

\theoremstyle{definition}
\newtheorem{defn}[thm]{Definition}
\newtheorem{rem}[thm]{Remark}
\newtheorem{rems}[thm]{Remarks}
\newtheorem{conv}[thm]{Convention}
\newtheorem{notn}[thm]{Notation}
\newtheorem{exam}[thm]{Example}

\newcommand{\D}{\mathcal{D}}
\newcommand{\fsG}{\mathscr{G}}
\newcommand{\LS}{\mathcal{L}}
\newcommand{\US}{\mathfrak{A}}

\newcommand{\fcM}{\mathcal{M}}
\newcommand{\fcB}{\mathcal{B}}
\newcommand{\fcP}{\mathcal{P}}
\newcommand{\fcY}{\mathcal{Y}}

\newcommand{\frb}{\mathfrak{b}}
\newcommand{\frp}{\mathfrak{p}}
\newcommand{\frm}{\mathfrak{m}}

\def\f#1{\mathfrak{#1}}
\newcommand{\fS}{\f{S}}
\newcommand{\C}{\mathcal{C}}
\def\MN(#1){ M_{#1}(\mathbb{N})}
\def\MNR(#1,#2){ M_{#1}(\mathbb{N})_{#2}}
\def\MNS(#1){ M_{#1}(\mathbb{N})^{\pm}}

\newcommand{\ZZ}{\mathbb{Z}}
\newcommand{\QQ}{\mathbb{Q}}
\newcommand{\ZG}{{{\mathbb{Z}}}_2}
\newcommand{\NN}{\mathbb{N}}

\def\MZ(#1){ M_{#1}(\ZG)}
\def\NZ(#1){ (\NN|\ZG)^{#1}}
\def\NZST(#1,#2,#3){ (\NN|\ZG)^{#1}_{#2|#3}}
\def\NZS(#1,#2){ {\NN}^{#1}_{#2}}
\def\MNZ(#1,#2){ M_{#1}(\NN | \ZG)_{#2}}
\def\MNZN(#1){ M_{#1}(\NN | \ZG)}
\def\CMNZ(#1,#2,#3){\Lambda(#1,#2|#3)}
\def\CMN(#1,#2){\Lambda(#1,#2)}
\def\CMNP(#1,#2){\Lambda_{#1,#2}}
\def\MNZNS(#1){ \MNZN(#1)^{\pm}}
\def\SE#1{{#1}^{\bar{0}}}
\def\SO#1{{#1}^{\bar{1}}}
\def\SEE#1{{#1}^{\bar{0}}}
\def\SOE#1{{#1}^{\bar{1}}}
\def\SUP#1{\SE{#1}|\SO{#1}}
\def\SS(#1,#2){{#1}^{\ol{#2}}}
\def\SSE(#1,#2){{#1}^{\ol{#2}}}

\def\ESE(#1, #2, #3){ \SEE{#1}_{{#2},{#3}} }
\def\ESO(#1, #2, #3){ \SOE{#1}_{{#2},{#3}} }
\def\bs#1{\boldsymbol{#1}}

\def\Qqs(#1,#2){ \mathcal{Q}(#1,#2) }

\def\lcase#1{\MakeLowercase{#1}}

\newcommand{\SerA}{\f{H}^c_r}
\newcommand{\HCR}{\mathcal{H}^c_{r, R}}
\newcommand{\Heck}{\mathcal{H}_{r,R}}

\newcommand{\QqnrR}{\mathcal{Q}_{\lcase{q}}(\lcase{n},\lcase{r}; R)}

\newcommand{\SQqnrR}{{\widetilde{\mathcal{Q}}_{\lcase{q}}(\lcase{n},\lcase{r}; R)}}
\newcommand{\Qqnr}{\mathcal{Q}_{\lcase{q}}{(\lcase{n},\lcase{r})}}
\newcommand{\SQvnR}{\widetilde{\boldsymbol{\mathcal{Q}}}_{\lcase{v}}{(\lcase{n})}}
\newcommand{\SQvnrR}{\widetilde{\boldsymbol{\mathcal{Q}}}_{\lcase{v}}(\lcase{n},\lcase{r})}
\newcommand{\bsSQvnr}{\widetilde{\boldsymbol{\mathcal{Q}}}_{\lcase{v}}(\lcase{n},\lcase{r})}

\newcommand{\USnv}{{\US_{v}(n)}}
\def\ol#1{\overline{#1}}

\newcommand{\ep}{\epsilon}
\newcommand{\Qv}{\mathbb Q({v})}

\def\qn{\mathfrak{\lcase{q}}_n}

\def\Uqn{\bsU(\mathfrak{\lcase{q}}_n)}

\def\Qvs(#1){\mathcal{Q}_{\lcase{v}}{(\lcase{#1})}}
\def\Uvqn{{\boldsymbol U}_{\!{v}}(\mathfrak{\lcase{q}}_{n})}

\def\Uvq(#1){U_{\lcase{v}}(\mathfrak{\lcase{q}}_{\lcase{#1}})}

\def\USN(#1){{\US[#1]}_{v}}
\newcommand{\Qnr}{\mathcal{\MakeUppercase{Q}}{(\lcase{n},\lcase{r})}}

\def\SABJR(#1,#2,#3,#4){({#1}|{#2})[\bs{#3}, #4]}
\def\SABJRS(#1,#2,#3,#4){({#1}|{#2})[#3, #4]}
\def\SABJS(#1,#2,#3){({#1}|{#2})[#3]}
\def\SAJRS(#1,#2,#3){{#1}[#2, #3]}
\def\SAJS(#1,#2){{#1}[#2]}
\def\SABJ(#1,#2,#3){({#1}|{#2})[\bs{#3}]}
\def\SAJR(#1,#2,#3){{#1}[\bs{#2}, #3]}
\def\SAJ(#1,#2){{#1}[\bs{#2}]}
\def\ABJR(#1,#2,#3,#4){({#1}|{#2})(\bs{#3}, #4)}
\def\ABJRS(#1,#2,#3,#4){({#1}|{#2})(#3, #4)}
\def\ABJS(#1,#2,#3){({#1}|{#2})(#3)}
\def\AJRS(#1,#2,#3){{#1}(#2, #3)}
\def\AJS(#1,#2){{#1}(#2)}
\def\ABJ(#1,#2,#3){({#1}|{#2})(\bs{#3})}
\def\AJR(#1,#2,#3){{#1}(\bs{#2}, #3)}
\def\AJ(#1,#2){{#1}(\bs{#2})}
\def\ABSUM#1{\widehat{#1}}
\def\snorm#1{|{#1}|}
\def\STDUE(#1,#2){({#1}+E_{{#2},{#2}+1}-E_{{#2}+1,{#2}+1}|0)}
\def\STDUO(#1,#2){({#1}-E_{{#2}+1,{#2}+1}|E_{{#2},{#2}+1})}
\def\STDLE(#1,#2){({#1}-E_{{#2},{#2}}+E_{{#2}+1,{#2}}|0)}
\def\STDLO(#1,#2){({#1}-E_{{#2},{#2}}|E_{{#2}+1,{#2}})}
\def\STDDE(#1){(D_{#1}|0)}
\def\STDDO(#1,#2){({#1}-E_{{#2},{#2}}|E_{{#2},{#2}})}

\newcommand{\tspan}{\mathrm{span}}

\newcommand{\End}{\mathrm{End}}

\newcommand{\ro}{\mathrm{ro}}
\newcommand{\co}{\mathrm{co}}
\newcommand{\AP}{{A^+_{h,k}}}
\newcommand{\TDAP}{T_{d_{A^+_{h,k}}}}
\newcommand{\AM}{{A^-_{h,k}}}
\newcommand{\TDAM}{T_{d_{A^-_{h,k}}}}

\def\STEPX#1#2{{[\![{#1}]\!]}_{#2}}

\def\STEP#1{ {[\![{#1}]\!]}_{{q}} }
\def\STEPP#1{{[\![{#1}]\!]}_{{q}^2}}
\def\STEPPD#1{{[\![{#1}]\!]}_{{q},{q}^2}}
\def\STEPPDR#1{{[\![{#1}]\!]}_{{q}^2,{q}}}

\def\VSTEP#1{ {[\![{#1}]\!]}_{{v}^2} }
\def\VSTEPP#1{{[\![{#1}]\!]}_{{v}^4}}
\def\VSTEPPD#1{ \STEPX{#1}{{v}^2,{v}^4} }
\def\VSTEPPDR#1{{[\![{#1}]\!]}_{{v}^4,{v}^2}}
\newcommand{\where}{\ \bs{|} \ }

\def\Hom{\mathrm{Hom}}

\newcommand{\spaceintv}{}
\def\intd(#1,#2,#3){\left[\begin{matrix}{#1};{#2}\\{#3}\end{matrix}\right]}
\def\intds(#1,#2){\left[\begin{matrix}{#1}\\{#2}\end{matrix}\right]}
\def\intdss(#1,#2){\intd({#1},{0},{#2})}
\def\diag{{\rm{diag}}}
\def\ker{\rm{ker}}
\def\xn{\bs{\xi}_n}
\def\parity#1{p({#1})}

\def\TAIJ(#1,#2){ T^\lhd_{({#1}, {#2})} }
\def\TDIJ(#1,#2){ T^\rhd_{({#1}, {#2})} }
\def\rmText#1{}
\def\rmForm#1{}

\def\AK(#1,#2){ {\overleftarrow{\bf r}}^{#2}_{#1} }
\def\BK(#1,#2){ {\overrightarrow{\bf r}}^{#2}_{#1} }
\newcommand{\YK}{{\widetilde a}^{\bar1}_{h-1,k}}
\newcommand{\ft}{{\boldsymbol{\mathrm{t}}}}

\newcommand{\cbefore}{{\overleftarrow c}_{\!\!A\;\;}^{h,k}}
\newcommand{\cafter}{{\overrightarrow c}_{\!\!A}^{h+1,k}}
\newcommand{\ckbefore}{{\overleftarrow c}_{\!\!A\;\;}^{h,k}}
\newcommand{\ckafter}{{\overrightarrow c}_{\!\!A}^{h,k}}
\newcommand{\cmiddle}{c_{A;h,k}^{h+1,k}}

\def\fkf{{\mathfrak f}}
\def\bsU{\boldsymbol U}

\iffalse
\def\RZ{\hlt{{\mathsf{E}}}}
\def\RP{\hlt{{\mathsf{P}}}}
\def\genE{\hlt{\mathsf{E}}}
\def\genF{\hlt{\mathsf{F}}}
\def\genK{\hlt{\mathsf{K}}}
\else
\def\RZ{{{\mathsf{E}}}}
\def\RP{{{\mathsf{P}}}}
\def\genE{{\mathsf{E}}}
\def\genF{{\mathsf{F}}}
\def\genK{{\mathsf{K}}}
\fi

\allowdisplaybreaks[1]

\title[Approaching quantum queer supergroups  using finite dimensional algebras]
{Approaching quantum queer supergroups using\\ finite dimensional superalgebras \\
\tiny{Preliminary version}
}
\author{Jie Du, Haixia Gu, Zhenhua Li and Jinkui Wan}

\address{Jie Du, School of Mathematics, University of New South Wales, UNSW Sydney 2052, Australia}
\email{j.du@unsw.edu.au}
\address{Haixia Gu, School of Science, Huzhou University, Huzhou 313000, China}
\email{ghx@zjhu.edu.cn}
\address{Zhenhua Li, School of Mathematical Sciences, Xiamen University, Xiamen 361005, China}
\email{zhen-hua.li@qq.com}
\address{Jinkui Wan,  Department of Mathematics, Beijing Institute of Technology, Beijing 100081, China}
\email{ wjk302@gmail.com}
\keywords{quantum  queer  supergroup, Hecke-Clifford superalgebra,
	quantum queer Schur superalgebra, realization.
}
\subjclass[2020]{17B37, 17A70, 20G42, 20C08}
\begin{document}
\maketitle

\begin{abstract}The idea of using a sequence of finite dimensional algebras to approach a quantum linear group (i.e., a quantum
$\mathfrak{gl}_n$)  was first introduced by Beilinson--Lusztig--MacPherson \cite{BLM}. 
In their work, 
the algebras are convolution algebras of some finite partial flag varieties
whose certain structure constants relative to the orbital basis satisfy a stabilization property. 
This property leads to the definition of an infinite dimensional idempotented algebra. 
Finally, taking a limit process yields a new realization for the quantum $\mathfrak{gl}_n$. 
Since then, this work has been modified \cite{DF2} 
and generalized to quantum affine $\mathfrak{gl}_n$ (see \cite{GV,L} for the geometric approach 
and \cite{DDF,DF} for the algebraic approach and a new realization) 
and quantum super  $\mathfrak{gl}_{m|n}$ \cite{DG}, 
and, more recently,
to convolution algebras arising from type $B/C$ geometry 
and $i$-quantum groups $\boldsymbol U^\jmath$ and $\boldsymbol U^\imath$; 
see \cite{BKLW,DWu1, DWu2}. 
This paper extends the algebraic approach 
to the quantum queer supergroup $\Uvqn$ via finite dimensional queer $q$-Schur superalgebras.
\end{abstract}

\maketitle

\tableofcontents

\section{Introduction}\label{sec_introduction}

Since the introduction of quantum groups as quantized enveloping algebras
 of semisimple Lie algebras in late eighties of last century, their realizations and related
applications have achieved outstanding progress. First, C. M. Ringel developed a realization 
for the positive part of such a quantum group, using Ringel--Hall algebras 
associated with the representation category of quivers. 
Then G. Lusztig's geometric approach to quantum groups and canonical bases advanced the theory to a new level. 
On the other hand, almost at the same time, 
Beilinson, Lusztig and MacPherson (see \cite{BLM}) initiated a new method 
of using a sequence of geometrically defined convolution algebras associated 
with some (finite) partial flag varieties (i.e., $q$-Schur algebras) to approach quantum linear groups. 
This approach results in a new realization for the entire quantum group $\bsU_{\!{v}}(\mathfrak{gl}_n)$ 
via its regular representation defined explicitly through 
some multiplication formulas of generators on a basis for $\bsU_{\!{v}}(\mathfrak{gl}_n)$. 
The convolution algebra approach was soon generalized to quantum affine $\mathfrak{gl}_n$ in \cite{GV, L}, 
which motivated an algebraic approach to their new realizations developed in \cite{DDF, DF}.

Recently, motivated from the work by Bao and Wang \cite{BW}, 
in which $i$-quantum groups (or more precisely, quantum symmetric pairs) 
play a key role to give a reformulation of Kazhdan-Lusztig theory that provides a
perfect solution to the problem of character formulas of $\mathfrak{osp}$ Lie superalgebras,
this geometric approach has also been generalized to convolution algebras 
arising from type B/C geometry, modified $i$-quantum groups, and their canonical bases.  
Building on \cite{BKLW} and the new Schur--Weyl duality in \cite{BW}, 
new realizations for the $i$-quantum groups $\boldsymbol U^\jmath$ and $\boldsymbol U^\imath$ are obtained in \cite{DWu1,DWu2}.

Though the geometric (or convolution algebra) approach seems not suitable
 for the realization problem for quantum supergroups, 
 the idea of using finite dimensional superalgebras to approach quantum supergroups 
continues to shed lights on the development of the algebraic approach.
The quantum supergroups we are interested in are quantized enveloping algebra of 
the Lie superalgebras arising from two types  
of finite dimensional simple associative superalgebras over the complex number field. 
These Lie superalgebras are called
 the general linear Lie superalgebra {$\mathfrak{gl}_{m|n}$} and the queer Lie superalgebra {$\qn$}.

 Using the theory of $q$-Schur superalgebras associated to Hecke algebras developed in \cite{DR}, the first two authors discovered a BLM type new realization for quantum linear supergroup $\bsU_{\!{v}}(\mathfrak{gl}_{m|n})$; see \cite{DG}. This algebraic method has further been refined in \cite{DGZ} to a three-step approach:
 \begin{enumerate}
 \item[(1)] developing some commutation formulas in Hecke algebras,
 \item[(2)] deriving some short (element) multiplication formulas, which gives the regular representation of $q$-Schur superalgebras, and some long (element) multiplication formulas, on which the Drinfeld--Jimbo type defining relations are built, and finally,
 \item[(3)] using a triangular relation between a monomial basis 
 and the BLM type long element basis to establish a new realization for $\bsU_{\!{v}}(\mathfrak{gl}_{m|n})$ 
 through its regular representation arising from the long multiplication formulas.
 \end{enumerate}  Note that the convolution algebra approach combines (1) and (2), 
 obtaining the short multiplication formulas via convolution products. Note also that the work \cite{DG}, 
 especially the partial integral Schur--Weyl duality, played a key role in the investigation  \cite{DLZ} of polynomial super representation theory of $\bsU_{\!{v}}(\mathfrak{gl}_{m|n})$ in positive characteristic, 
 which generalises the classical theory developed in \cite{BKu}.

The representaiton theory of the algebraic Lie supergroups of type $Q$ and their associated Lie superalgebras, called
the queer Lie superalgebra {$\qn$} here, 
has been studied  extensively at both the ordinary and modular levels; see \cite{BK1, BK2} and the references therein.
The queer Lie superalgebra {$\qn$}, on the one hand, 
behaves in many aspects as the Lie algebra $\mathfrak{gl}_n$ or the Lie superalgebra $\mathfrak{gl}_{m|n}$. 
But, on the other hand, 
it differs drastically from the basic classical Lie superalgebras in terms of the Cartan subalgebra,
 which is not purely
even, and the invariant bilinear form, 
which does not exists for {$\qn$}.
Thus, we do have a beautiful analog of the Schur-Weyl duality discovered by Sergeev  \cite{Ser}, 
often referred as Sergeev duality, and
the notion of queer Schur superalgebra can be introduced. 
On the other hand, an unusual highest weight theory needs to be developed in \cite{BK1, BK2} by Brundan and Kleshchev 
in order to solve the classification problem of finite dimensional irreducible polynomial supermodules. 
They also determine the irreducible projective representations of
the symmetric group $\fS_r$.

At the quantum level, Olshanski \cite{Ol} constructed a quantum deformation $\Uvqn$ of the universal enveloping algebra
$U(\qn)$ and established a quantum analog of the Sergeev duality in the generic case. 
Then, in \cite{GJKK, GJKKK}, highest weight modules and crystal bases in the quantum characteristic 0 case are  investigated. Also,
a quantum analog of Schur superalgebras, called  queer {$q$}-Schur superalgebra (or quantum queer Schur superalgebra),
 was  obtained by the first and last authors (see \cite{DW1} 
 and independently in \cite{BGJKW} using walled Brauer-Clifford superalgebras),
 and presentations for both queer Schur and $q$-Schur superalgebras were given therein.
Later, an integral theory for the queer $q$-Schur superalegbra was developed in \cite{DW2}.
These developments indicate that
a BLM type realization for $\Uvqn$ using queer $q$-Schur superalgebras should exist and the polynomial representation theory in positive quantum characteristic (i.e., in the root of unity case), 
as a quantum analog of \cite{BK1,BK2}, would also be developed.  
Naturally, one expects that the three-step approach above 
should be generalized from $\bsU_{\!{v}}(\mathfrak{gl}_{m|n})$ to  $\Uvqn$.

This project was initiated almost ten years ago.  The main obstacles occurred in deriving short multiplication formulas for odd generators. After some unsuccessful early attempts, we decided to tackle the classical (i.e., ${v}=1$) case first.
In \cite{GLL}, a new realization of the enveloping superalgebra {$\Uqn$} is obtained 
and the three-step approach was carried out. In particular, 
key commutation formulas in Sergeev algebras emerge and both short and long multiplication formulas 
are derived in queer Schur superalgebras {$\Qnr$} with structure constants independent of {$r \ge 0$}. 
This work provides a road map for the realization in the quantum case.

However, unlike the classical case in \cite{GLL}, we still could not write down explicitly all  multiplication formulas
 involving odd generators  because of the complexity caused by the Clifford part. This
 shows a drastic difference from the cases of $\mathfrak{gl}_n$
  and $\mathfrak{gl}_{m|n}$, as well as the classical case.
 To fill this gap, we introduce the SDP condition
 (see Definition \ref{defn:SDP}) and derive some multiplication formulas 
 under the SDP condition.
 \footnote{These formulas are raw formulas which are not normalized. See Remark \ref{raw}(2).} 
 Fortunately, by an inductive approach (cf. Remark \ref{induction_N}), 
 this is enough to realize the generators and relations for the quantum queer supergroup $\bsU_v(\mathfrak{q}_n)$ in queer $q$-Schur superalgebras. 
 In this way, we successfully completed Step (1) 
 by discovering some key commutation formulas in Hecke--Clifford algebras required in Step (2).
Note that the combinatorics developed in \cite{DGZ} and \cite{DW1,DW2} played decisive roles.
 Finally, by using the PBW basis theorem established in \cite{DW1},
  we obtain
a realization of the supergroup $\Uvqn$, completing Step (3).

It is worthwhile to point out that the first author together with Y. Lin and Z. Zhou has also given in \cite{DLZ} a realization of {$\Uvqn$} via ${v}$-differential operators. 
This approach aims at constructing directly the regular module of $\Uvqn$. It is not clear how this new construction relates to the  Schur--Weyl--Olshanski duality.

We expect in forthcoming papers to apply 
the new realization via queer $q$-Schur algebras to the study of the integral Schur--Weyl--Olshanski duality
 and to develop the modular representation theory of quantum queer supergroups at roots of unity, generalizing the work \cite{BK1, BK2} to the quantum case.

This  paper is organized as follows.
In Section \ref{sec_basicformulas}, 
we deduce some useful commutation formulas in the Hecke-Clifford superalgebra and introduce the important SDP condition as certain commutation formulas between certain elements 
in the Clifford subalgebra and elements associated with shortest double coset representatives 
in the Hecke subalgbera. We also verify the condition for two cases.
In Section \ref{sec_qqschur}, we first follow \cite{DW2, DW1} 
to introduce queer $q$-Schur superalgebra $\QqnrR$. 
We then introduce a superalgebra $\SQqnrR$ as a twisted version of $\QqnrR$ 
and as the homomorphic image of the quantum queer supergroup $\Uvqn$ as well. 
In the next two sections, 
we derive the multiplication formulas of the defining basis by certain homogeneous generators for the superalgebra
$\QqnrR$, 
where the even case is done in Section \ref{even case}, 
while the odd case is done in Section \ref{odd case}.
 Similar to BLM's original method, we introduce
in Section \ref{sec_spanningsets} some long elements which span {$\SQqnrR$} uniformly on $r$ and derive their multiplication formulas by the images of the generators of $\Uvqn$. 
Note that the ``structure constants'' in these multiplication formulas are independent of $r$. We then check all the defining relations of $\Uvqn$ in $\SQvnrR$ in Section \ref{defining relations}.
 In this way, we obtain superalgebra homomorphisms from $\Uvqn$ to $\SQvnrR$,
 which induce a homomorphism $\bs{\xi}_n$ from $\Uvqn$ to the direct product $\SQvnR$ of  $\SQvnrR$ (Theorem \ref{qqschur_reltion}).
Finally, in Section \ref{sec_generators},
we determine the image {$\USnv$} of the homomorphism $\bs{\xi}_n$ and prove that $\bs{\xi}_n$ is injective. 
Thus, the new superalgebra {$\USnv$} is a new realization of {$\Uvqn$} (Theorem \ref{map_iso}).

\begin{notn}\label{sec_notations}
For any $m\in {\ZZ}_+$ and indeterminates {$y, z$}, set
\begin{align}\label{stepd}
	{\STEPX{m}{y}} = 1 + {y} + \cdots + {y}^{m-1}
	 = \frac{{y}^{m} - 1}{{y} - 1},\qquad {\STEPX{m}{y, z}} ={\STEPX{m}{y}} - {\STEPX{m}{z}}.
\end{align}

For {$i,j \in \ZZ$} with {$i \le j$}, let {$[i, j] = \{ i, i+1, \cdots, j \}$}.
Throughout this paper,
$R$ is a commutative ring of characteristic not equal to 2,
$\QQ$ is the field of rational numbers,
$n$ and $r$ are integers and {$n \ge 2$}, {$r \ge 1$}.
Let {${q}, {v}$} be indeterminates,  {${q} = {v}^2$}.
\end{notn}

\spaceintv
\section{Preliminaries in the Hecke-Clifford superalgebra }\label{sec_basicformulas}

We shall establish a number of relations and commutation formulas in $\HCR$
which will be useful to compute the multiplication formulas
in the queer $q$-Schur superalgebra
(i.e., the quantum queer Schur superalgebra).

The Clifford superalgebra {$\C_r$} is an associative superalgebra over {$R$} defined by odd generators {$c_1, \cdots , c_r$} and relations: for $1\leq i,j\leq r$,
\begin{equation}\label{cliff}
	c_i^2 = -1,  \quad c_i c_j = - c_j c_i \; (i \ne j).
\end{equation}
Let {$\fS_{r}$} be the symmetric group on {$r$} letters,
with generators  {$ s_{i}=(i, i+1)$} for all {$1 \le i < r$}.
The Hecke  algebra  {$\Heck$} associated with {$\fS_{r}$} and $q\in R^\times$
is an algebra  over {${R}$}  generated by
{$T_{i} = T_{s_{i}}$} for all  {$1 \le i < r$},
subject to the relations: for $1 \le i, j \le r-1$,
\begin{equation}\label{Hecke}
	(T_i - {q} )(T_i + 1) = 0, \quad T_i T_j = T_j T_i\,
	 (|i-j| >1),\quad
	T_i T_{i+1} T_i = T_{i+1} T_i T_{i+1}\; (i\neq r-1).
\end{equation}
The {\it Hecke-Clifford superalgebra}  {$\HCR$}  is a  superalgebra  over {${R}$}  generated by even generators
{$T_1, \cdots, T_{r-1}$} and odd generators {$c_1, \cdots, c_r$},
subject to the relations \eqref{cliff}-\eqref{Hecke} and the extra relations: for $1\leq i\leq r-1, 1\leq j\leq r,$
\begin{equation}\label{Hecke-Cliff}
T_i c_j = c_j T_i\,(j\neq i, i+1), \quad
T_i c_i = c_{i+1} T_i, \quad T_i c_{i+1} = c_i T_i - ({q}-1)(c_i - c_{i+1}).
\end{equation}
The following relations can be deduced directly from \eqref{cliff}-\eqref{Hecke-Cliff} and will be frequently used in the sequel: for any {$1\leq k\leq r-1$} and $\alpha\in\Lambda(n,r)$,
\begin{equation}\label{Tick-inv}
\begin{aligned}
(1)\;\;&{q} T_{k}^{-1} = T_{k} -  ({q} - 1),
\qquad
&(2)\;\;   x_{\alpha} T_{k}^{-1} =  {q}^{-1} x_{\alpha} \quad  \mbox{ if }  s_{k} \in \fS_{\alpha},\\
(3)\;\;&T_{k} c_{k+1}
= {q} c_{k} T_{k}^{-1} + ({q} - 1) c_{k+1},
\qquad
&(4)\;\; c_{k}T_{k}
= {q} T_{k}^{-1} c_{k+1}  + ({q} - 1) c_{k}.
\end{aligned}
\end{equation}

For {$w \in \fS_r$} and $\alpha\in {\ZZ}^r_2$, let
$$T_w = T_{i_1} \cdots T_{i_k},\qquad c^\alpha=c_1^{\alpha_1}c_2^{\alpha_2}\cdots c_r^{\alpha_r}.$$
Here $w= s_{i_1} \cdots s_{i_k}$  is any reduced expression of $w$. 
Then, by \cite[Lemma 2.2]{DW1}, both
$\{T_wc^\alpha\mid w\in\fS_r,\alpha\in {\ZZ}^r_2\}$ and $\{c^\alpha T_w\mid w\in\fS_r,\alpha\in {\ZZ}^r_2\}$ form $R$-bases for $\HCR$.

Define the set of  compositions of $r$ in $n$ parts:
\begin{equation}\label{Lanr}
\CMN(n,r)=\{\lambda=(\lambda_1,\lambda_2,\ldots,\lambda_n)\in \NN^n \where  |\lambda|:=\sum_i\lambda_i=r \}.
\end{equation}
Associated with {${\lambda} = ({\lambda}_1, \cdots , {\lambda}_{n}) \in \CMN(n,r)$}, define its {\it partial sum sequence} $ \widetilde{\lambda}_1,\ldots,\widetilde{\lambda}_n=r$, where
\begin{equation}\label{latilde}
\widetilde{\lambda}_i = \sum_{k=1}^i \lambda_{k} \;\;(1\leq i\leq n), \text{ and set }\widetilde{\lambda}_0=0.
\end{equation}
Denote by {$\fS_{\lambda}$} the standard Young subgroup of {$\fS_r$} corresponding to {$\lambda$},
and let
\begin{align}\label{xlambda}
	x_{\lambda} = x_{\fS_{\lambda}} = \sum_{w \in \fS_{\lambda}} T_w.
\end{align}
It is known that for any {$T_{i}$} satisfying {$  s_i \in \fS_{\lambda} $}, we have
\begin{align}\label{Tixlambda}
	T_i x_{\lambda} =x_{\lambda}  T_i = {q} x_{\lambda}.
\end{align}

Denote by $\MN(n)$ the set of $n\times n$-matrices with entries being non-negative integers.
For any {$A=(a_{i,j}) \in \MN(n)$}, let
\begin{align}\label{nuA}
	\nu_{A} := (a_{1,1}, \cdots, a_{n,1}, a_{1,2}, \cdots, a_{n,2}, \cdots,  a_{1,n}, \cdots, a_{n,n})=(\nu_1,\nu_2,\ldots,\nu_{n^2}).
\end{align}
If $\nu_k=a_{i,j}$ for some $1\leq i,j\leq n $, then we write the partial sum $\widetilde{\nu}_k$ as
\begin{align}\label{mtildehk}
\widetilde{\nu}_k=\widetilde{a}_{i,j} := \sum_{p=1}^{j-1}\sum_{u=1}^{n} a_{u,p} + \sum_{u=1}^{i} a_{u,j}.
\end{align}
For convenience of later use (see \eqref{map d_A} and Definition \ref{defn:SDP}), we also set
\begin{equation}
 \widetilde a_{0,1}=0\text{ and }
\widetilde{a}_{0,j} =\widetilde{a}_{n,j-1}\; (j\geq2).
\end{equation}

For {$ \lambda,\mu \in \CMN(n,r)$},
let {$\D_{\lambda}$} be the shortest representatives of right cosets of {$\fS_{\lambda}$} in {$\fS$},
let {$\D_{\lambda, \mu} $} denote the shortest representatives of the {$\fS_{\lambda}$}-{$\fS_{\mu}$} double cosets of {$\fS_r$}.

Let {$  \MNR(n,r) = \{ A=(a_{i,j}) \in \MN(n) \where  r=|A| \} $}, where {$|A|:=\sum_{i,j} a_{i,j} $}.
According to \cite[Section 4.2]{DDPW}, there is a bijection
\begin{equation}\label{jmath}
\aligned
	\mathfrak{j}: \MNR(n,r) &\to \{ (\lambda,  d, \mu) \where \lambda,\mu \in \CMN(n,r), d \in \D_{\lambda, \mu} \} \\
	A & \mapsto (\ro(A), d_A, \co(A)).
	\endaligned
\end{equation}
Here, for $A=(a_{i,j})$, if $\lambda=\ro(A)$ and $\mu =\co(A)$, then $\lambda_i=\sum_{j=1}^na_{i,j}$ and
$\mu_j=\sum_{i=1}^na_{i,j}$. Call $\ro(A)$ (resp., $\co(A)$) the {\it row sum} (resp., {\it column sum}) vector of $A$.
Moreover, the shortest double coset representative $d_A$ has a reduced expression as follows according to \cite[Lemma 3.2]{DGZ}:
\begin{equation}\label{d_A}
\begin{aligned}
	d_A =
		(w_{2,1}w_{3,1}\cdots w_{n,1})
		(w_{2,2}w_{3,2}\cdots w_{n,2})
		\cdots
		(w_{2,n-1}w_{3,n-1}\cdots w_{n,n-1}) ,
\end{aligned}
\end{equation}
where $w_{i,j}$ is defined by
\begin{equation}\label{wij}
\begin{aligned}
w_{i,j}=\left\{
\begin{array}{ll}
1,\text{ if }a_{i,j}=0, \text{ or } a_{i,j}>0 \text{ and }\sigma_{i-1,j} = \widetilde{a}_{i-1,j},\\
(s_{\sigma_{i-1,j}} s_{\sigma_{i-1,j} - 1} \cdots s_{\widetilde{a}_{i-1,j} +1})
		\cdots
		(s_{\sigma_{i-1,j}+a_{i,j}-1} s_{\sigma_{i-1,j}+a_{i,j}-2 } \cdots s_{\widetilde{a}_{i,j}}), \text{ otherwise}.
\end{array}
\right.
\end{aligned}
\end{equation}
Here, for $\lambda=\ro(A)$ and $\mu=\co(A)$,
\begin{align}\label{sigmaij}
{\sigma}_{i,j}
	:=\sum_{p=1}^{j-1}  \sum_{u=1}^{n}a_{u,p}
		+ \sum_{u \le i, p \ge j}  a_{u,p} = \widetilde\mu_{j-1}+\sum_{u \le i, p \ge j}  a_{u,p}
	= \widetilde{\lambda}_{i}	 + \sum_{u > i, p < j}  a_{u,p}.
\end{align}
Observe from \eqref{mtildehk} that the following useful fact holds when computing $d_A$:
\begin{align}\label{sigma-a}
\sigma_{i,j}-\widetilde{a}_{i,j}=\sum_{u\leq i, p\geq j+1}a_{u,p}.
\end{align}

\begin{rem}\label{rem:dA=1}
By definition, 
 {$d_A=1$}
$\iff$  $\sigma_{i-1,j}-\widetilde{a}_{i-1,j}=0$ for all $a_{i,j}>0$
 with $2 \le i \le  n,  1 \le j \le n-1$  
 $\iff$ 
 $ \sum_{u \le i-1 , p\geq j+1 }a_{u,p} = 0$ for all $a_{i, j}> 0$   
 with $2 \le i \le  n,  1 \le j \le n-1$.
\end{rem}

For any {$A \in \MN(n)$} and {$1 \le h \le n-1, 1 \le k \le n$}, denote
\begin{align}\label{Ahk}
	& A^+_{h,k} = A + E_{h,k} - E_{h+1, k} \in \MN(n), \qquad
	  A^-_{h,k} = A - E_{h,k} + E_{h+1, k} \in \MN(n),
\end{align}
where {$E_{i,j}$} denotes the matrix with entry $1$ at {$(i,j)$} and {$0$} at other entries.

Given {$A =(a_{i,j}) \in \MN(n)$}, define the partial row sums:
\begin{equation}\label{prsum}
\aligned
\AK(h,k) =\AK(h,k)(A),\text{ where }&\AK(h,1)=0, \;\; \AK(h,k) = \sum^{k-1}_{u=1}a_{h, u}\;(2\leq k\leq n),\;\AK(h,n+1)=\ro(A)_h;\\
\BK(h,k) =\BK(h,k)(A),\text{ where }&\BK(h,0)=\ro(A)_h,\;\; \BK(h,k) = \sum^{n}_{j=k+1}a_{h, j}\;(1\leq k<n),\;\; \BK(h,n)=0.
\endaligned
\end{equation}

\begin{lem}\label{prop_pjshift}
	Let {$A =(a_{i,j})= \MN(n)$},
	{$\lambda = \ro(A)$}. Keep the above notations.

{\rm(1)}
 For each $1\leq h\leq n-1$ and $1\leq k\leq n$ such that $a_{h+1,k}>0$, we have
\begin{align*}
&\sum_{j=\AK(h+1,k)}^{\AK(h+1,k+1) - 1}
		T_{\widetilde{\lambda}_{h}+1} T_{\widetilde{\lambda}_{h}+2} \cdots T_{\widetilde{\lambda}_{h}+j}
		{T_{d_A}}
=
		T_{\widetilde{\lambda}_{h}} T_{\widetilde{\lambda}_{h}-1}\cdots T_{\widetilde{\lambda}_{h}-\BK(h,k)+1}
		{{\TDAP}}
		(\sum_{j=\AK(h+1,k)}^{\AK(h+1,k+1) - 1}
		T_{\widetilde{a}_{h,k}+1} \cdots T_{\widetilde{a}_{h,k}+p_j}).
\end{align*}

{\rm(2)}
For each $1\leq h\leq n$ and $1\leq k\leq n$ such that $a_{h,k}>0$, we have
\begin{align*}
&\sum_{j=\BK(h,k)}^{\BK(h,k-1) - 1}
	T_{\widetilde{\lambda}_{h}-1} \cdots T_{\widetilde{\lambda}_{h}-j }
	T_{d_{A}}
	= T_{\widetilde{\lambda}_{h}} T_{\widetilde{\lambda}_{h}+1}\cdots T_{\widetilde{\lambda}_{h}+\AK(h+1,k)-1}  T_{d_{A^-_{i,k}}}
	(\sum_{j=\BK(h,k)}^{\BK(h,k-1) - 1}
		T_{\widetilde{a}_{h,k}-1} \cdots T_{\widetilde{a}_{h,k}-q_j}).
\end{align*}
Here $p_j=j-\AK(h+1,k)$ and $q_j = j - \BK(h,k)$. (Note that, in case {$j=0$}, we set
	{$T_{\widetilde{\lambda}_{h}+1} \cdots T_{\widetilde{\lambda}_{h}+j }=1$} and
	{$T_{\widetilde{\lambda}_{h}-1} \cdots T_{\widetilde{\lambda}_{h}-j } = 1$}).
\end{lem}
\begin{proof}
We only prove {\rm(1)}, and the proof of {\rm(2)} is analogous.
For each {$j \in  [\AK(h+1,k), \AK(h+1,k+1) - 1]$} and  {$p_j=j-\AK(h+1,k)$},
according to \cite[Proposition 3.6{\rm(1)}]{DGZ},
we have
\begin{align*}
&s_{\widetilde{\lambda}_{h}+1} s_{\widetilde{\lambda}_{h}+2} \cdots s_{\widetilde{\lambda}_{h}+  j} d_{A}
=
		s_{\widetilde{\lambda}_{h}} s_{\widetilde{\lambda}_{h}-1}\cdots
		s_{\widetilde{\lambda}_{h} - \BK(h,k) +1}
		{d_{A^+_{h,k}}}
		 s_{\widetilde{a}_{h,k}+1} \cdots s_{\widetilde{a}_{h,k}+p_j}.
\end{align*}
Length additivity of both sides gives
\begin{align*}
&T_{\widetilde{\lambda}_{h}+1} T_{\widetilde{\lambda}_{h}+2} \cdots T_{\widetilde{\lambda}_{h}+  j} T_{d_{A}}
=
		T_{\widetilde{\lambda}_{h}} T_{\widetilde{\lambda}_{h}-1}\cdots
		T_{\widetilde{\lambda}_{h} - \BK(h,k) +1}
		T_{d_{A^+_{h,k}}}
		 T_{\widetilde{a}_{h,k}+1} \cdots T_{\widetilde{a}_{h,k}+p_j}.
\end{align*}
Hence, (1) follows.
\end{proof}

For any {$1 \le i \le j \le r-1$},
define the following elements in {$\HCR$}:
\begin{equation}\label{Tij}
\begin{aligned}
\TAIJ(i,j) &=  1 + T_{i} + T_{i} T_{i+1} + \cdots +  T_{i} T_{i+1}  \cdots  T_{j}, \;\quad\TAIJ(i,i-1)=1;\\
\TDIJ(j,i) &=  1 + T_{j} + T_{j} T_{j-1} + \cdots +  T_{j} T_{j-1}  \cdots  T_{i}, \;\quad\TDIJ(i-1,i)=1.
\end{aligned}
\end{equation}
The  notation  {$\TAIJ(i,j)$} will be  frequently used later on in this paper.
\begin{lem}\label{Tnsum}
	For  {$ h \in [1,n-1] , k \in [1,n] $}, {$A=(a_{i,j}) \in \MN(n)$} and
	 $\mu = \co{(A)}$,
	we have
\begin{align*}
{\rm(1)} \quad &\TAIJ( {\widetilde{a}_{h,k}} + 1,  { {\widetilde{a}_{h,k}} + a_{h+1,k}-1})
		\sum_{\sigma\in\mathcal{D}_{\nu_A}\cap {\fS}_\mu} T_{\sigma}
		= \TDIJ( {{\widetilde{a}_{h,k}}}, {{\widetilde{a}_{h,k}} - a_{h,k} + 1})
		\sum_{\sigma\in\mathcal{D}_{\nu_{A^+_{h,k}}}\cap {\fS}_\mu} T_{\sigma}, \text{ if }\; a_{h+1,k} \ge 1;\\
{\rm(2)} \quad 	&\TDIJ( {{\widetilde{a}_{h,k}} - 1} ,   {{\widetilde{a}_{h,k}} - a_{h,k} + 1})
		\sum_{\sigma\in\mathcal{D}_{\nu_A}\cap {\fS}_\mu} T_{\sigma}
		 = \TAIJ( {{\widetilde{a}_{h,k}}} , {{\widetilde{a}_{h,k}} + a_{h+1,k} - 1})
	\sum_{\sigma\in\mathcal{D}_{\nu_{A^-_{h,k}}}\cap {\fS}_\mu} T_{\sigma}, \text{ if }\;a_{h,k}\geq1.
\end{align*}
\end{lem}
\begin{proof}
By \cite[Proposition 3.4]{GLL},
we have
\begin{align*}
	& ( 1 + s_{{\widetilde{a}_{h,k}} + 1} + s_{{\widetilde{a}_{h,k}} + 1}s_{{\widetilde{a}_{h,k}} + 2} \cdots +   s_{{\widetilde{a}_{h,k}} + 1} \cdots s_{{\widetilde{a}_{h,k}} + a_{h+1,k}-1})
		\sum_{\sigma\in\mathcal{D}_{\nu_A}\cap \mathfrak{S}_\mu}\sigma \\
& 	 =
	( 1 + s_{\widetilde{a}_{h,k}} + s_{\widetilde{a}_{h,k}} s_{\widetilde{a}_{h,k} - 1} + \cdots +
	 s_{\widetilde{a}_{h,k}} \cdots s_{\widetilde{a}_{h,k} - a_{h,k} + 1})
		\sum_{\sigma\in\mathcal{D}_{\nu_{A^+_{h,k}}}\cap \mathfrak{S}_\mu}\sigma.
\end{align*}
As the expressions of the summands are all reduced,
then {\rm(1)} is correct.
Equation {\rm(2)}  could be proved by replacing {$A$} with {${A^-_{h,k}}$}.
\end{proof}

\begin{lem}\label{xTinverse}
For {$\alpha \in \CMN(n, r)$},
{$1\leq u < r$, and $ j \ge 0$},
assume
{$s_{u+1}, \cdots , s_{u+j} \in \fS_{\alpha}$} if $j \ge 1$.
\begin{itemize}
\item[(1)] If $s_u\in\fS_{\alpha}$, then
	$x_{\alpha} \TAIJ({u},  {u+j}) = x_{\alpha} \TDIJ({u+j},  {u})=\STEP{ j+2} x_{\alpha} $.
\item[(2)]
The following holds:
\begin{align*}
x_{\alpha} T_{u}^{-1} T_{u+1}^{-1} \cdots T_{u+j}^{-1}
&=
	{q}^{-1 - j} x_{\alpha}T_{u}   T_{u+1} \cdots T_{u + j}
	- {q}^{- j-1 }  ({q} - 1)   x_{\alpha} \TAIJ( {u}, { u + j  - 1})
.
\end{align*}
\end{itemize}
\end{lem}
\begin{proof} Assertion (1) is clear by \eqref{Tixlambda}, while assertion (2) can be proved by induction on $j$, noting \eqref{Tick-inv}(1).
\end{proof}
Following \cite[Section 4]{DW2}, we introduce the following elements in $\HCR$: for $1\leq i, j\leq r$,
\begin{equation}\label{cqij}
\aligned
	&c_{{q},i,j} = {q}^{j-i}c_i  + \cdots + {q} c_{j-1} + c_j\;\;\text{ for }i\leq j,\qquad c_{{q},i,j} =0\;\;\text{ for }i> j;\\
	&c'_{{q},i,j} = c_i + {q} c_{i+1} + \cdots + {q}^{j-i}c_j\;\;\text{ for }i\leq j,\qquad c'_{{q},i,j} =0\;\;\text{ for }i> j;.
	\endaligned
\end{equation}

\begin{lem}\label{xcT}
For {$\alpha \in \CMN(n, r)$},
{$1\leq u<r$, and $j \ge 0$}, assume
{$s_{u+1}, \cdots , s_{u+j} \in \fS_{\alpha}$} if {$j\ge 1$}.
\begin{itemize}
\item[(1)]
For {$j \ge 0$}, we have in $\HCR$
$$x_{\alpha}c_u T_{u} T_{u+1} \cdots T_{u+j}=
({q} - 1)  {q}^{j}  x_{\alpha}\sum_{k=0}^{j} {T^{-1}(u,k)} c_{u+k}+{q}^{j+1}   x_{\alpha}  {T^{-1}(u,j+1)} c_{u+j+1},$$
where $T^{-1}(u, k)=
\begin{cases}
1, &\text{ if }k=0;\\
T_{u}^{-1} T_{u+1}^{-1} T_{u+2}^{-1} \cdots T_{u+k-1}^{-1},&\text{ if }k \geq 1.
\end{cases}$

\item[(2)] Further, we have
\begin{align*}
&x_{\alpha}c_{u}  \TAIJ({u}, {u+j} )
 ={q}^{j+1} x_{\alpha} ( c_{u} + T_{u}^{-1} c_{u+1} +  T_{u}^{-1} T_{u+1}^{-1}   c_{u+2}
	+ \cdots
	+ T_{u}^{-1} T_{u+1}^{-1} T_{u+2}^{-1} \cdots T_{u+j}^{-1}  c_{u+j+1} ).
\end{align*}
Specially, if $s_u\in\fS_\alpha$ in the same time, then
$
x_{\alpha} c_{u} \TAIJ( {u}, {u+j})
= x_{\alpha} c_{q, u, u+j+1}.
$
\item[(3)] If {$s_{u},  s_{u-1} \cdots, s_{u-j} \in \fS_{\alpha}$},
then
$
x_{\alpha} c_{u+1} \TDIJ( {u}, {u-j})
= x_{\alpha} c_{q, u-j, u+1}.
$
\item[(4)] For $1\leq i\leq j\leq r$, $(c_{q,i,j})^2=-{\STEPX{j-i+1}{q^2}}$.
\end{itemize}
\end{lem}
\begin{proof}
We first prove (1).
Repeatedly applying \eqref{Tick-inv}(4) gives
\begin{align*}
x_{\alpha}c_u &T_{u} T_{u+1} \cdots T_{u+j} \\
&= x_{\alpha} ({q} T_{u}^{-1} c_{u+1}  + ({q} - 1) c_{u}) T_{u+1} \cdots T_{u+j}  \\
&=  {q}  x_{\alpha}  T_{u}^{-1} c_{u+1} T_{u+1} \cdots T_{u+j}   + ({q} - 1) {q}^{j}  x_{\alpha}  c_{u}   \\
&=  {q}  x_{\alpha}  T_{u}^{-1} ( {q} T_{u+1}^{-1} c_{u+2}
	+ ({q} - 1) c_{u+1} ) T_{u+2} \cdots T_{u+j}
	+ ({q} - 1) {q}^{j} x_{\alpha}  c_{u}  \\
&= {q}^2   x_{\alpha}  T_{u}^{-1} T_{u+1}^{-1} c_{u+2} T_{u+2} \cdots T_{u+j}
	+ ({q} - 1)  {q}^{j}  x_{\alpha}  (  c_{u}  +  T_{u}^{-1}c_{u+1} ) \\
& = \cdots \\
&=
		{q}^{j+1}   x_{\alpha}  T_{u}^{-1} T_{u+1}^{-1} T_{u+2}^{-1} \cdots T_{u+j}^{-1}  c_{u+j+1}\\
		& \qquad + ({q} - 1)  {q}^{j}  x_{\alpha}
			( c_{u} +  T_{u}^{-1}c_{u+1} + T_{u}^{-1} T_{u+1}^{-1} c_{u+2} + \cdots
			+  T_{u}^{-1} T_{u+1}^{-1} \cdots  T_{u+j-1}^{-1} c_{u+j}),
\end{align*}
proving (1). Now, by (1), we have
$$\aligned
x_{\alpha}c_{u}  \TAIJ({u}, {u+j} )&=x_\alpha c_u+\sum_{i=0}^jx_\alpha c_uT_u\cdots T_{u+i}\\
&=x_\alpha c_u+\sum_{i=0}^j\big(({q} - 1)  {q}^{i}  x_{\alpha}\sum_{k=0}^{i} 
	{T^{-1}(u,k) } c_{u+k}+{q}^{i+1}   x_{\alpha}  { T^{-1}(u,i+1) } c_{u+i+1}\big)\\
&=\sum_{k=0}^j\big(q^k+\sum_{i=k}^j(q-1)q^i\big)x_\alpha 
	 T^{-1}(u,k)  c_{u+k}+q^{j+1}x_\alpha  {T^{-1}(u,j+1)} c_{u+j+1}\\
&={q}^{j+1} x_{\alpha} ( c_{u} + T_{u}^{-1} c_{u+1} +  T_{u}^{-1} T_{u+1}^{-1}   c_{u+2}
	+ \cdots
	+ T_{u}^{-1} T_{u+1}^{-1} T_{u+2}^{-1} \cdots T_{u+j}^{-1}  c_{u+j+1} ),
\endaligned
$$
 since the telescoping sum $q^k+\sum_{i=k}^j(q-1)q^i=q^{j+1}$.
Hence (2) holds.

Thirdly, $s_{u-i}\in\fS_{\alpha}$ implies $x_{\alpha}T_{u-i}=qx_{\alpha}$ for $0\leq i\leq j$. 
Thus, by \eqref{Hecke-Cliff},
we obtain
\begin{align*}
x_{\alpha}c_{u+1}\TDIJ(u,u-j)&=x_{\alpha}c_{u+1}(1+T_u+T_uT_{u-1}+\cdots T_uT_{u-1}\cdots T_{u-j})\\
&=x_{\alpha}(c_{u+1}+T_uc_u+T_uT_{u-1}c_{u-1}+\cdots+T_uT_{u-1}\cdots T_{u-j}c_{u-j})\\
&=x_{\alpha}(c_{u+1}+qc_u+\cdots+q^{j+1}c_{u-j})=x_{\alpha} c_{q, u-j, u+1},
\end{align*}
proving (3).

Finally, (4) is straightforward by an induction.
\end{proof}

Let {$\ZG = \{0, 1\} \subset \ZZ$}.
We recall the following matrices introduced in \cite{DW2}:
\begin{align*}
	& \MNZN(n) = \{ A=(\SUP{A}) \where \SE{A} \in \MN(n), \SO{A} \in \MZ(n) \}, \\
	& \MNZ(n,r) = \{ A=(\SUP{A}) \in \MNZN(n) \where  \SE{A} + \SO{A} \in \MNR(n,r) \}.
\end{align*}
We use O to denote the zero matrix in {$\MN(n)$} or  {$\MZ(n)$}, and set {${O} = (\rm{O}|\rm{O}) \in \MNZN(n)$}.
For any {$A =(\SEE{a}_{i,j} | \SOE{a}_{i,j}) \in  \MNZN(n)$} with
 {$\SE{A}=(\SEE{a}_{i,j})$}, {$\SO{A}=(\SOE{a}_{i,j})$}, let
\begin{equation}\label{def_ahk_notations}
\begin{aligned}
&a_{i,j} := \SEE{a}_{i,j} + \SOE{a}_{i,j}, \quad
	 \ABSUM{A} :=(a_{i,j})\in M_n(\NN), \quad
	|A|:=|\ABSUM{A}|, \\
&	\AK(h,k)(A) =\AK(h,k)(\ABSUM{A}), \quad 
	\BK(h,k)(A) =\BK(h,k)(\ABSUM{A}), \quad 
	A^+_{h,k} :=\ABSUM{A}^+_{h,k},  \quad 
A^-_{h,k} :=\ABSUM{A}^-_{h,k}
\end{aligned}
\end{equation}
for all admissible $i,j,h,k$.
Applying the notations introduced in Lemma \ref{prop_pjshift}, we set
\begin{equation}\label{parity}
 \ro(A) := \ro(\ABSUM{A}),\quad
 \co(A): = \co(\ABSUM{A}),\quad
  \nu_{A}: = \nu_{\ABSUM{A}}, \quad
d_{A} := d_{\ABSUM{A}},\quad \parity{A} \equiv |\SO{A}| (\mathrm{mod} \ 2 ).
\end{equation}
Here the definition of the parity function $p$ will be justified in \eqref{eq cA} below.

Recall from \eqref{jmath} the shortest double coset representative $d_A$. 
This element has a reduced expression given in \eqref{d_A}. 
As a permutation, $d_A$ is given in \cite[Page 424]{Du} (or \cite[Exercise 8.2]{DDPW}). More precisely,
for $A=(a_{i,j})\in M(n,r)$ with $\lambda=\ro(A)$, the description of the pseudo-matrix associated with $A$ in loc. cit. can be easily interpreted in terms of the notations here as follow:
\begin{equation}\label{map d_A}
d_A(\widetilde{a}_{h-1,k}+ p )=\widetilde{\lambda}_{h-1} + \AK(h,k) + p ,
\end{equation}
for all $h,k\in[1,n] $ with $a_{h,k}>0$ and $p\in[1,a_{h,k}]$. 

In other words, $d_A$ is the permutation obtained by concatenating the permutations 
$$\left(\begin{matrix}& 
\tilde{a}_{i-1,j}+1 &\tilde{a}_{i-1,j}+2&\cdots&\tilde{a}_{i,j}\\
&\lambda_{i-1}+ \AK(i, j) +1
&\lambda_{i-1}+  \AK(i, j) +2
&\cdots&\lambda_{i-1}+  \AK(i, j) +a_{i,j}
\end{matrix}\right)$$
for all $a_{i,j}>0$ via the ordering from top to bottom of column 1, then top to bottom of column 2, and so on.

Thus, if $\SerA=\HCR|_{q=1}$ denotes the Sergeev algebra,
regarded as the ``semi-direct (or smash) product'' $\C_r\rtimes R\fS_r$,
 then $c_iw=wc_{w^{-1}(i)}$ in $\SerA$ and, 
 for $\AK(h,k) =\AK(h,k)(\ABSUM{A})$ and $\lambda=\ro(\widehat A)$, 
\begin{displaymath}
	c_{\widetilde{\lambda}_{h-1} + \AK(h,k) + p} {d_A} =  {d_A} c_{\widetilde{a}_{h-1,k}+ p}
\end{displaymath}
for all  $1\leq h,k\leq n$ with $a_{h,k}>0$ and $1\leq p\leq a_{h,k}$. 
However, if we replace {$d_A$} by {$T_{d_A}$},  this commutation formula may not hold in the Hecke-Clifford superalgebra $\HCR$.
So we introduce the following definition.
\begin{defn}\label{defn:SDP} 
For {$A =({a}_{i,j}) \in  \MN(n)$}, 
 or  {$A =(\SEE{a}_{i,j} | \SOE{a}_{i,j}) \in  \MNZN(n)$}  with {$a_{h,k}=\SEE{a}_{i,j}+\SOE{a}_{i,j}$}, 
let {$\lambda=\ro(A)$}, {$r=|A|$} and $\AK(h,k)=\AK(h,k)(A)$,
if $a_{h,k}>0$ and
\begin{displaymath}
	c_{\widetilde{\lambda}_{h-1}+\AK(h,k)+p} {T_{d_A}}
	 =  {T_{d_A}} c_{\widetilde{a}_{h-1,k} + p}\qquad(\mbox{in }\HCR)
\end{displaymath}
for each {$p \in [1,a_{h,k}]$}, then {$A$} is said to
satisfy the {\it semi-direct product} (SDP) {\it condition} at {$(h, k)$}.
If {$A$} satisfies the  SDP condition at {$(h, k)$} for every {$k\in [1,n]$} with {$a_{h,k} \ge 1$},
then {$A$} is said to satisfy the SDP condition on the $h$-th row.
\end{defn}

There are two ``extreme'' cases where the SDP condition is satisfied.
\begin{lem}\label{lem_dA1}
Let {$A =(\SEE{a}_{i,j} | \SOE{a}_{i,j})  \in \MNZ(n, r)$}.
If  {$d_A = 1$}, then, for any  {$h,k \in [1, n]$}  satisfying {${a}_{h, k}>0$},
$A$ satisfies the SDP condition at {$(h,k)$}.
\end{lem}
\begin{proof}
Suppose $d_A=1$. 
Then by \eqref{d_A} and \eqref{sigma-a} we have the following observation which will be helpful later in the proof: 
for any {$i \in [2, n]$}, {$j \in [1 , n-1]$} satisfying {$a_{i,j}> 0$},
Remark \ref{rem:dA=1} implies
 {$\sigma_{i-1,j}-\tilde{a}_{i-1,j}=\sum_{u \le i-1 , x \geq j+1 }a_{u,x} = 0$}, and hence {$a_{u,x} = 0$} for any {$u \in [1, i-1]$}, {$x \in [j+1, n]$}.
Conversely,
for any {$u \in [1, n-1]$}, {$x \in [2, n]$} satisfying {$a_{u,x} > 0$},
we have {$a_{i,j}= 0$} for all {$i \in [u+1, n]$}, {$j \in [1, x-1]$},
otherwise one can obtain {$w_{i,j} \ne 1$} and hence $d_A \ne 1$, which contradicts to the assumption $d_A=1$.

Set $\lambda=\ro(A)$.
Fix $h,k \in [1, n]$ with {${a}_{h, k}>0$} and let $p \in [0, {a}_{h, k} - 1]$. Then a direct calculation using \eqref{latilde}, \eqref{prsum} and \eqref{mtildehk} shows
\begin{align*}
 \Big(\tilde{\lambda}_{h-1}+\AK(h,k) + p + 1\Big) -   \Big( \tilde{a}_{h-1,k} + p + 1\Big)
=
	 \sum_{ 1 \le i \le h-1 \atop   k+1 \le j \le n}a_{ i ,  j }
		- \sum_{ h+1 \le i \le n \atop  1 \le j  \le k-1}a_{ i ,  j }.
\end{align*}
Therefore by Definition \ref{defn:SDP} it suffices to prove
\begin{equation}\label{relation_c_d_1}
\begin{aligned}
 \sum_{ 1 \le i \le h-1 \atop   k+1 \le j \le n}a_{ i ,  j }
		- \sum_{ h+1 \le i \le n \atop  1 \le j  \le k-1}a_{ i ,  j }
 = 0.
\end{aligned}
\end{equation}
To show this holds, we need to consider the following five cases:
\begin{enumerate}
\item
If {$h = k = 1$} or {$h = k = n$},
 we have
 {$	 \sum_{ 1 \le i \le h-1 \atop   k+1 \le j \le n}a_{ i ,  j }
 =  \sum_{ h+1 \le i \le n \atop  1 \le j  \le k-1}a_{ i ,  j }
 =0
 $}.

\item
If {$h,k \in [2, n-1]$}, then by the above observation
{$a_{h,k}> 0$} implies
{$a_{i,j} = 0$} for any {$i \in [1, h-1]$} and  {$j \in [k+1, n]$},
 {$a_{i,j}= 0$} for all {$i \in [h+1, n]$} and  {$j \in [1, k-1]$},
 hence  \eqref{relation_c_d_1}  holds.
\item
If {$h=1$}
and  {$k \in [2, n]$},
by \eqref{d_A} and \eqref{sigma-a} we have  {$	 \sum_{ 1 \le i \le h-1 \atop   k+1 \le j \le n}a_{ i ,  j }
 =0
 $}.
Meanwhile by the above observation
{$a_{1,k}> 0$} implies
 {$a_{i,j}= 0$} for all {$i \in [2, n]$}, {$j \in [1, k-1]$},
 hence
 {$\sum_{ h+1 \le i \le n \atop  1 \le j  \le k-1}a_{ i ,  j }
 =\sum_{   2 \le i \le n \atop  1 \le j  \le k-1}a_{ i ,  j }
 = 0$}.
Then \eqref{relation_c_d_1} is proved for {$h=1$} and {$k \in [1,n]$}.

 \item
If  {$h\in [2, n]$} and {$k=1$},
or
  {$h=n$} and  {$k\in [2, n-1]$},
by \eqref{d_A} and \eqref{sigma-a} we have {$ \sum_{ h+1 \le i \le n \atop  1 \le j  \le k-1}a_{ i ,  j } = 0$}.
On the other hand, by the above observation
 {$a_{h,k}> 0$} implies
{$a_{i,j} = 0$} for any {$i \in [1, h-1]$}, {$j \in [k+1, n]$}
and
 {$  \sum_{ 1 \le i \le h-1 \atop   k+1 \le j \le n}a_{ i ,  j } = 0$}.
Then we conclude  \eqref{relation_c_d_1}  holds in this case.

 \item
If  {$h\in [2, n-1]$} and {$k=n$}, by \eqref{d_A} and \eqref{sigma-a}
we have {$ \sum_{ 1 \le i \le h-1 \atop   k+1 \le j \le n}a_{ i ,  j }  = 0$}.
On the other hand, by the above observation
{$a_{h,k}> 0$} implies
 {$a_{i,j}= 0$} for all {$i \in [h+1, n]$}, {$j \in [1, n-1]$},
 hence   \eqref{relation_c_d_1}  holds.
\end{enumerate}
\end{proof}

\begin{exam}\label{exam_shift}
Let {$\lambda \in \CMN(n, r)$} for some {$r > 0$}.
If {$A=(a_{i,j})={\diag}(\lambda)$}, or
{${\diag}(\lambda) + \sum_{i=1}^{n-1} u E_{i, i+1}$},
or  {${\diag}(\lambda) + \sum_{i=1}^{n-1} u E_{i+1, i}$},
with all {$u \in \NN$},
then {$d_A = 1$},
and $A$ satisfies the SDP condition at {$(h,k)$}
for any {$h,k$}  satisfying {${a}_{h, k}>0$},
\end{exam}

\begin{lem}\label{shiftonN}
Every {$A =(\SEE{a}_{i,j} | \SOE{a}_{i,j})  \in \MNZ(n, r)$} satisfies the SDP condition on the $n$-th row.
\end{lem}
\begin{proof}
Suppose $1\leq k\leq n$ and $1\leq p<a_{n,k}$.
By \eqref{sigmaij}  we observe that $\widetilde{\lambda}_{n-1}+\AK(n,k)+p+1>\sigma_{i-1,j}+a_{i,j}$
for any $2\leq i\leq n$ and $1\leq j\leq k-1$,
or $2\leq i\leq n-1$ and $j=k$.
Therefore by \eqref{wij} we obtain
 that $c_{\widetilde{\lambda}_{n-1}+\AK(n,k)+p+1}$ commutes
with $T_{w_{ij}}$ for each pair of $i,j$ satisfying $2\leq i\leq n$ and $1\leq j\leq k-1$,
or $2\leq i\leq n-1$ and $j=k$. Meanwhile direct calculation show via \eqref{Hecke-Cliff} that the following holds
\begin{align*}
 c_{\widetilde{\lambda}_{n-1} +  \sum^{k-1}_{u=1}a_{n, u}+ p +1} T_{w_{n,k}}
=T_{w_{n,k}} c_{\widetilde{a}_{n-1,k} +  p +1}.
\end{align*}
Furthermore, observe that
$$
\widetilde{a}_{n-1,k} +  p +1<\widetilde{a}_{i-1,j}+1
$$
for all {$2\leq i\leq n$} and $k+1\leq j\leq n-1$,
which implies that $c_{\widetilde{a}_{n-1,k} +  p +1}$ commutes with $T_{w_{i,j}}$ for each pair of $i,j$ satisfying  {$2 \leq i\leq n$} and $k+1\leq j\leq n-1$.
Putting together the lemma is proved by \eqref{d_A}.
\end{proof}
\begin{rem}The SDP condition is the key to derive some multiplication formulas in Section 5. 
It would be interesting to know which matrices satisfying the SDP condition. 
More examples of such matrices are given in Lemma \ref{triang_aaa} and Corollary \ref{triang_aaa_cor}.
\end{rem}


\spaceintv
\section{Queer $q$-Schur superalebras and quantum queer supergroups}\label{sec_qqschur}
Given ${R}$-free supermodule\footnote{By an $R$-free supermodule, we mean an $R$-free module with a basis consisting of homogeneous elements.}  $V,W$,
let   $\Hom_{{R}}(V,W)$  be the set of  all ${R}$-linear  maps from $V$ to $W$.
We make $\Hom_{{R}}(V,W)$  into an $R$-free supermodule by declaring
${\Hom_{{R}}(V,W)}_{\ol{i}}$
to be  the set of  homogeneous map of parity $i\in \ZG$,
that is,
the linear maps $\theta:V\rightarrow W$ with $\theta(V_{\ol{j}})\subset W_{\overline{i}+\ol{j}}$ for {$j\in\ZG$}.
Note it is typical in a superalgebra to write the expressions
which only make sense for homogeneous elements,
and the expected meaning for arbitrary elements
 is obtained by extending linearly from the homogeneous case.

For any homogeneous element {$h \in \HCR$}, let {$\parity{h}$} be its parity, more precisely,  {$\parity{h} = 0$} if  {$h$} is even, or {$\parity{h} = 1$} if {$h$} is odd.
Following \cite{DW2},  the {\it queer $q$-Schur superalgebra} over {$R$} is defined to be
\begin{equation}\label{QqnrR}
	\QqnrR := \End_{\HCR}\Big(\bigoplus_{\lambda \in \CMN(n,r)} x_{\lambda} \HCR\Big),\qquad
	\Qqnr=\mathcal Q_{{v}^2}(n,r;\mathcal Z).
\end{equation}
Here we view $M=\bigoplus_{\lambda \in \CMN(n,r)} x_{\lambda} \HCR$ as a right $\HCR$-module and homomorphisms are $\HCR$-module homomorphisms, i.e., every $R$-linear $\phi:M\to M$ satisfies $\phi(mh)=\phi(m)h$ for all $m\in M$ and $h\in\HCR$.

We now describe a standard basis for $\QqnrR$.
For {$\lambda= (\lambda_1, \cdots , \lambda_m)\in {\ZZ}_+^m$} and {$\alpha \in \ZG^m $} with $m\geq 1$, recall \eqref{cqij} and the following elements introduced in \cite[Section 4]{DW2}
\begin{equation}\label{eq_c_a}
\begin{aligned}
	&c^{\alpha}_{\lambda} = {(c_{{q}, 1, \widetilde{\lambda}_1})}^{\alpha_1}
				{(c_{ {q}, \widetilde{\lambda}_1+1 , \widetilde{\lambda}_2} )}^{\alpha_2}
				\cdots
				{(c_{{q}, \widetilde{\lambda}_{m-1}+1,  \widetilde{\lambda}_m} )}^{\alpha_m}, \\
	&(c^{\alpha}_{\lambda})' = {(c'_{{q}, 1, \widetilde{\lambda}_1})}^{\alpha_1}
				{(c'_{{q}, \widetilde{\lambda}_1 + 1,  \widetilde{\lambda}_2} )}^{\alpha_2}
				\cdots
				{(c'_{{q}, \widetilde{\lambda}_{m-1} + 1,  \widetilde{\lambda}_m })}^{\alpha_m} .
\end{aligned}
\end{equation}
Here we use the convention $0^0=1$.
By \cite[Lemma 4.1 and Corollary 4.3]{DW2},
we have, for all $\alpha\leq\lambda$ (meaning $\alpha_i\leq\lambda_i $ for any $i$),
\begin{align}\label{xc}
	x_{\lambda} c_{\lambda}^{\alpha} = (c_{\lambda}^{\alpha})'  x_{\lambda}.
\end{align}
For $A =(\SEE{a}_{i,j} | \SOE{a}_{i,j}) \in\MNZ(n,r)$,
recall the notations  in \eqref{def_ahk_notations} and \eqref{parity},
and recall the partition $\nu=\nu_A=(a_{1,1},\ldots,a_{n,1},\ldots,a_{1,n},\ldots,a_{n,n})$ defined
in \eqref{nuA}. Let
\begin{equation}\label{eq cA}
	c_{A} =  c_{\nu}^{\alpha}, \qquad c'_{A} =  (c_{\nu}^{\alpha})', \qquad
	T_{A} =  x_{\ro(A)} T_{d_{A}} c_{A}
			\sum _{{\sigma} \in \D_{{\nu}_{A}} \cap \fS_{\co(A)}} T_{\sigma},
\end{equation}
where $\alpha=\nu_{A^{\bar1}}:=(\SOE{a}_{1,1},\ldots, \SOE{a}_{n,1},\ldots, \SOE{a}_{1,n},\ldots, \SOE{a}_{n,n})$.
Note $p(T_A)=p(A)$.

By \cite[Proposition 5.2]{DW2}, for any $\lambda,\mu\in\Lambda(n,r)$, the intersection $x_\lambda\HCR\cap \HCR x_\mu$ is a free $R$-modules with basis $\{T_A\mid A  \in\MNZ(n,r), \lambda=\ro(A), \mu=\co(A)\}$. Thus, we have the following.

\begin{prop}[{\cite[Theorem 5.3]{DW2}\label{DW-basis}}]
The queer $q$-Schur superalgebra $\QqnrR$ is a free $R$-module with basis $\{ \phi_{A} \where  A \in \MNZ(n,r)\}$, where {$\phi_{A} $} is defined by
\begin{equation*}
\phi_{A}: \bigoplus_{\lambda \in \CMN(n,r)} x_{\lambda} \HCR \longrightarrow \bigoplus_{\lambda \in \CMN(n,r)} x_{\lambda} \HCR,\;\;
 x_{\mu}h \longmapsto \delta_{\mu, \co(A)} T_{A}h,\;\forall  \mu \in \CMN(n,r), h \in \HCR,
\end{equation*}
 and $\phi_A$ has parity $p(A)$.
\end{prop}
We have, for $A,B\in \MNZ(n,r)$,
\begin{equation}\label{co=ro}
\phi_B\phi_A \ne  0 \implies \co(B) = \ro(A).
\end{equation}

For the convenience of later use, we now replace ``module homomorphisms" by ``supermodule homomorphisms'' in the definition of queer $q$-Schur superalgebras to introduce a twisted version of $\QqnrR$.

 Following \cite[Definition 1.1]{CW}),
 we define, for the right $\HCR$-modules $M,N$,
 $\Hom^s_{\HCR}(M,N)$ to be the $R$-supermodule generated by
 all homogeneous  $R$-linear maps $\Phi:M\to M$ satisfying
\begin{equation}\label{eq:twist}
 \Phi(mh)=(-1)^{\parity{\Phi}\parity{h}}\Phi(m)h
\end{equation}
for all homogeneous $m\in M$ and $h\in\HCR$. 
In particular, for $M=N=\bigoplus_{\lambda \in \CMN(n,r)} x_{\lambda} \HCR$, 
let $\End^s_{\HCR}(M):=\Hom^s_{\HCR}(M,M)$.
 One checks easily that, if
 $\Phi\in\End^s_{\HCR}(M)_{\bar i}$ and $\Psi\in\End^s_{\HCR}(M)_{\bar j}$, then $\Psi\Phi\in\End^s_{\HCR}(M)_{\bar i+\bar j}$. Thus, we obtain a new (associative) superalgebra
\begin{equation}\label{SQqnrR}
	\SQqnrR
	:= \End^s_{\HCR}\Big(\bigoplus_{\lambda \in \CMN(n,r)} x_{\lambda} \HCR\Big)
	=\bigoplus_{\lambda, \mu \in \CMN(n,r)} \Hom^s_{\HCR}(x_{\lambda}\HCR,x_{\mu}\HCR).
\end{equation}
This is called the {\it twisted queer $q$-Schur superalgebra}.

In the next two results, we will see that this twisted version is linearly isomorphic to $\QqnrR$ and if in addition $\sqrt{-1}\in R$, then $\SQqnrR$ is isomorphic to $\QqnrR$ as $R$-superalgebras.

\begin{prop}\label{prop_PhiAPhiB}
The twisted queer $q$-Schur superalgebra $\SQqnrR$ is $R$-free with basis $\{ \Phi_{A} \where  A \in \MNZ(n,r)\}$, where {$\Phi_{A} $} is defined by
\begin{equation}\label{PhiA}
 \Phi_{A} : \bigoplus_{\mu \in \CMN(n, r)} x_{\mu}\HCR  \longrightarrow \bigoplus_{\mu \in \CMN(n, r)} x_{\mu}\HCR,\;\;
		x_{\mu} h  \longmapsto   {(-1)}^{\parity{A} \cdot \parity{ h}} \delta_{\mu, \co(A)} T_{A} \cdot  h,
\end{equation}
Moreover,
for any  {$A, B \in \MNZ(n, r)$} with  {$\co(A) = \ro(B)$},
if $\phi_B \phi_A= \sum_{M} \gamma^M_{B,A} \phi_M$ ($\gamma^M_{B,A} \in R$), then
\begin{equation}\label{Phi_structure}
\Phi_B \Phi_A
 =  {(-1)}^{\parity{A} \cdot  \parity{ B } }  \sum_{M}    \gamma^M_{B,A} \Phi_M  .
\end{equation}
\end{prop}
\begin{proof} The freeness follows easily from the linear isomorphism
\begin{equation}\label{dagger}
(\;\;)^\dagger:\QqnrR\to\SQqnrR
\end{equation} as defined in \cite[(1.1)]{CW} or, more precisely, for a homogeneous $\phi\in\QqnrR$ of parity $\parity{\phi}$ and $M=\bigoplus_{\lambda \in \CMN(n,r)} x_{\lambda} \HCR$, $\phi^\dagger:M\to M$ is defined by \begin{align}\label{phi-dag}
\phi^{\dag}(m)=(-1)^{\parity{\phi}\cdot\parity{m}}\phi(m)
\end{align}
for homogeneous $m\in M$. It is straightforward to check  $\phi^{\dag}\in\SQqnrR $ and $\dagger\dagger=1$. The basis assertion is clear by setting  $\Phi_A=\phi_A^{\dag}$ for
all {$A \in \MNZ(n, r)$}.

Finally,
 if we write $T_A=x_{\ro(A)}h'$,
where $h'={T_{d_A}} c_{A}  \sum _{{\sigma} \in \D_{{\nu}_{_A}} \cap \fS_{\co(A)}} T_{\sigma}$ (see \eqref{eq cA}),
then $h'$ is homogeneous with parity $p(h')=p(A)$
and the structure constants $\gamma^M_{B,A}$
appearing in $\phi_B\phi_A = \sum_{M} \gamma^M_{B,A} \phi_M$
are determined by writing $T_Bh'\in x_{\ro(B)}\HCR\cap \HCR x_{\co(A)}$
as a linear combination of $T_M$'s. In other words,
\begin{equation}\label{eq_gB}
z(B,A):=\phi_B\phi_A(x_{\co(A)})=T_Bh'= \sum_{M \in \MNZ(n,r)} \gamma_{B,A}^M T_M,
\end{equation}
and for each {$M$}, we have {$\parity{M} = \parity{A} + \parity{B}$}.

Thus,
for any homogeneous element {$h \in \HCR$}
and  {$ \mu \in \CMN(n, r)$}, by \eqref{PhiA},
we have
\begin{align*}
\Phi_B \Phi_A (x_{\mu} h)
& = \delta_{\mu, \co(A)} {(-1)}^{\parity{A} \cdot \parity{ h }}  {(-1)}^{\parity{B} \cdot ( \parity{ A } + \parity{h})}   \phi_B (T_{A} h ) \\
& = \delta_{\mu, \co(A)} {(-1)}^{\parity{A} \cdot \parity{ h }}  {(-1)}^{\parity{B} \cdot ( \parity{ A } + \parity{h})}    \sum_{M} \gamma^M_{B,A} T_M \cdot h  \\
& = \delta_{\mu, \co(A)}  {(-1)}^{\parity{B} \cdot  \parity{ A } }   \sum_{M} {(-1)}^{\parity{M}\cdot \parity{ h }}   \gamma^M_{B,A} \phi_M(x_\mu) \cdot h \\
& = \delta_{\mu, \co(A)} {(-1)}^{\parity{A} \cdot  \parity{ B } }   \sum_{M}   \gamma^M_{B,A} \Phi_M (x_{\mu} h ),
\end{align*}
as desired.
\end{proof}
\begin{cor}\label{super_al_iso}
If {$\sqrt{-1} \in R$},
then the $R$-linear isomorphism $(\;\;)^\dagger$ in \eqref{dagger} induces  an
isomorphism of $R$-superalgebras
$$\fkf:  \QqnrR  \longrightarrow \SQqnrR,\;\;
{\phi}_{A}  \longmapsto  {(\sqrt{-1})}^{\parity{{A}}}  {\Phi}_{A},\;\;
\forall A\in\MNZ(n,r).$$
\end{cor}

\begin{proof} Clearly, $\fkf$ is an $R$-linear isomorphism. 
It remains to check {$ \fkf({\phi}_{A} {\phi}_{B}) =   \fkf ( {\phi}_{A} )  \fkf( {\phi}_{B}) $} 
for all $A,B\in\MNZ(n,r)$. By writing $\phi_B \phi_A
= \sum_{M} \gamma^M_{B,A} \phi_M$, as in Lemma \ref{prop_PhiAPhiB}, 
this can be easily proved by the following facts: \\
\qquad(a)\;\;${(\sqrt{-1})}^{\parity{A}}  {(\sqrt{-1})}^{\parity{B}}
	 {(-1)}^{ \parity{A} \parity{B} }
 = {(\sqrt{-1})}^{\parity{{\phi}_{A} {\phi}_{B}}}$;\\
 \qquad(b)\;\; If $\phi_B \phi_A
= \sum_{M} \gamma^M_{B,A} \phi_M$, 
as in Lemma \ref{prop_PhiAPhiB}, 
then {$\parity{{\phi}_{B}{\phi}_{A}}  = \parity{{\phi}_{M}} $} for each $M$.
\end{proof}
\begin{rem}We will use these twisted queer $q$-Schur superalgebras $\SQqnrR$ 
to construct the quantum queer supergroup $\Uvqn$ 
over $\Qv$ in Section \ref{sec_generators}. 
If the queer $q$-Schur superalgebras $\QqnrR$ were used, 
	then we would require the quantum queer supergroup $\Uvqn$ be defined over $\QQ[\sqrt{-1}]({v})$.
\end{rem}

We end this section with a display of generators and relations for queer $q$-Schur superalgebras via the quantum queer supergroup; see  \cite[Theorem 9.2]{DW1}.

In \cite[Section 4]{Ol}, Olshanski introduced a quantum deformation $\Uvqn$ of the  universal enveloping algebra of the queer Lie superalgebra $\qn$ using a modification of the Reshetikhin-Takhtajan-Faddeev method. The following definition of the quantum queer supergroup $\Uvqn$  is taken from 
\cite[Proposition 5.2]{DW1}.
Let  {$\{ \bs{\ep}_i \where  i \in [1,n]\}$} be the standard basis for ${\ZZ}^n$ and define the ``dot product'' $\bs\ep_i\centerdot \bs\ep_j=\delta_{i,j}$, linearly extended to ${\ZZ}^n$.

\begin{defn}\label{defqn}
The quantum queer supergroup $\Uvqn$ is the (Hopf) superalgebra
over $\Qv$  generated by
even generators  {${\genK}_{i}$}, {${\genK}_{i}^{-1}$},  {${\genE}_{j}$},  {${\genF}_{j}$},
and odd generators  {${\genK}_{\ol{i}}$},  {${\genE}_{\ol{j}}$}, {${\genF}_{\ol{j}}$},
for   {$ 1 \le i \le n$}, {$ 1 \le j \le n-1$}, subject to the following relations:
\begin{align*}
({\rm QQ1})\quad
&	{\genK}_{i} {\genK}_{i}^{-1} = {\genK}_{i}^{-1} {\genK}_{i} = 1,  \qquad
	{\genK}_{i} {\genK}_{j} = {\genK}_{j} {\genK}_{i} , \qquad
	{\genK}_{i} {\genK}_{\ol{j}} = {\genK}_{\ol{j}} {\genK}_{i}, \\
&	{\genK}_{\ol{i}} {\genK}_{\ol{j}} + {\genK}_{\ol{j}} {\genK}_{\ol{i}}
	= 2 {\delta}_{i,j} \frac{{\genK}_{i}^2 - {\genK}_{i}^{-2}}{{v}^2 - {v}^{-2}}; \\
({\rm QQ2})\quad
& 	{\genK}_{i} {\genE}_{j} = {v}^{\bs{\ep}_i\centerdot\alpha_j} {\genE}_{j} {\genK}_{i}, \qquad
	{\genK}_{i} {\genE}_{\ol{j}} = {v}^{\bs{\ep}_i\centerdot\alpha_j} {\genE}_{\ol{j}} {\genK}_{i}, \\
& 	{\genK}_{i} {\genF}_{j} = {v}^{-(\bs{\ep}_i\centerdot\alpha_j)} {\genF}_{j} {\genK}_{i}, \qquad
	{\genK}_{i} {\genF}_{\ol{j}} = {v}^{-(\bs{\ep}_i\centerdot \alpha_j)} {\genF}_{\ol{j}} {\genK}_{i}; \;
	(\mbox{here }\alpha_j=\bs\ep_j-\bs\ep_{j+1})\\
({\rm QQ3})\quad
& {\genK}_{\ol{i}} {\genE}_{i} - {v} {\genE}_{i} {\genK}_{\ol{i}} = {\genE}_{\ol{i}} {\genK}_{i}^{-1}, \qquad
	{v} {\genK}_{\ol{i}} {\genE}_{i-1} -  {\genE}_{i-1} {\genK}_{\ol{i}} = - {\genK}_{i}^{-1} {\genE}_{\ol{i-1}}, \\
& {\genK}_{\ol{i}} {\genF}_{i} - {v} {\genF}_{i} {\genK}_{\ol{i}} = - {\genF}_{\ol{i}} {\genK}_{i}, \qquad
	{v} {\genK}_{\ol{i}} {\genF}_{i-1} -  {\genF}_{i-1} {\genK}_{\ol{i}} = {\genK}_{i} {\genF}_{\ol{i-1}},\\
& {\genK}_{\ol{i}} {\genE}_{\ol{i}} + {v} {\genE}_{\ol{i}} {\genK}_{\ol{i}} = {\genE}_{i} {\genK}_{i}^{-1}, \qquad
	{v} {\genK}_{\ol{i}} {\genE}_{\ol{i-1}} +  {\genE}_{\ol{i-1}} {\genK}_{\ol{i}} =   {\genK}_{i}^{-1} {\genE}_{i-1}, \\
& {\genK}_{\ol{i}} {\genF}_{\ol{i}} + {v} {\genF}_{\ol{i}} {\genK}_{\ol{i}} =   {\genF}_{i} {\genK}_{i}, \qquad
	{v} {\genK}_{\ol{i}} {\genF}_{\ol{i-1}} +  {\genF}_{\ol{i-1}} {\genK}_{\ol{i}} = {\genK}_{i} {\genF}_{i-1}, \\
& 
 {\genK}_{\ol{i}} {\genE}_{j} - {\genE}_{j} {\genK}_{\ol{i}} =  {\genK}_{\ol{i}} {\genF}_{j} - {\genF}_{j} {\genK}_{\ol{i}}
 = {\genK}_{\ol{i}} {\genE}_{\ol{j}} + {\genE}_{\ol{j}} {\genK}_{\ol{i}} =  {\genK}_{\ol{i}} {\genF}_{\ol{j}} + {\genF}_{\ol{j}} {\genK}_{\ol{i}}
	= 0, \mbox{ for } j \ne i, i-1; \\
({\rm QQ4}) \quad
& {\genE}_{i} {\genF}_{j} - {\genF}_{j} {\genE}_{i}
	= \delta_{i,j}  \frac{{\genK}_{i} {\genK}_{i+1}^{-1} - {\genK}_{i}^{-1}{\genK}_{i+1}}{{v} - {v}^{-1}}, \\
&
{\genE}_{\ol{i}} {\genF}_{\ol{j}} + {\genF}_{\ol{j}} {\genE}_{\ol{i}}
	= \delta_{i,j}  ( \frac{{\genK}_{i} {\genK}_{i+1} - {\genK}_{i}^{-1} {\genK}_{i+1}^{-1}}{{v} - {v}^{-1}}
	 + ({v} - {v}^{-1}) {\genK}_{\ol{i}} {\genK}_{\ol{i+1}} )  ,\\
& {\genE}_{i} {\genF}_{\ol{j}} - {\genF}_{\ol{j}} {\genE}_{i}
	= \delta_{i,j}  ( {\genK}_{i+1}^{-1} {\genK}_{\ol{i}}  - {\genK}_{\ol{i+1}} {\genK}_{i}^{-1} ) , \qquad
 {\genE}_{\ol{i}} {\genF}_{j} - {\genF}_{j} {\genE}_{\ol{i}}
	= \delta_{i,j}  ( {\genK}_{i+1} {\genK}_{\ol{i}}  - {\genK}_{\ol{i+1}} {\genK}_{i} ) ;\\
({\rm QQ5}) \quad
&{\genE}_{\ol{i}}^2 = -\frac{ {v} - {v}^{-1} }{{v} + {v}^{-1}} {\genE}_{i}^2, \quad
	{\genF}_{\ol{i}}^2 = \frac{ {v} - {v}^{-1} }{{v} + {v}^{-1}} {\genF}_{i}^2, \\
&
{\genE}_{i} {\genE}_{\ol{j}} - {\genE}_{\ol{j}} {\genE}_{i}
	=  {\genF}_{i} {\genF}_{\ol{j}} - {\genF}_{\ol{j}} {\genF}_{i}
	= 0 ,  \quad \mbox{ for } |i - j| \ne 1,
\\
& 
{\genE}_{i} {\genE}_{j} - {\genE}_{j} {\genE}_{i} = {\genF}_{i} {\genF}_{j} - {\genF}_{j} {\genF}_{i}
= {\genE}_{\ol{i}}{\genE}_{\ol{j}}  + {\genE}_{\ol{j}}  {\genE}_{\ol{i}}= {\genF}_{\ol{i}} {\genF}_{\ol{j}}  + {\genF}_{\ol{j}} {\genF}_{\ol{i}}
	= 0 \quad \mbox{ for }\quad |i-j|  > 1, 	
\\
& {\genE}_{i} {\genE}_{i+1} - {v} {\genE}_{i+1} {\genE}_{i}
	= {\genE}_{\ol{i}} {\genE}_{\ol{i+1}} + {v} {\genE}_{\ol{i+1}} {\genE}_{\ol{i}}, \qquad
  {\genE}_{i} {\genE}_{\ol{i+1}} - {v} {\genE}_{\ol{i+1}} {\genE}_{i}
	= {\genE}_{\ol{i}} {\genE}_{i+1} - {v} {\genE}_{i+1} {\genE}_{\ol{i}}, \\
& {\genF}_{i} {\genF}_{i+1} - {v} {\genF}_{i+1} {\genF}_{i}
	= - ({\genF}_{\ol{i}} {\genF}_{\ol{i+1}} + {v} {\genF}_{\ol{i+1}} {\genF}_{\ol{i}}), \qquad
  {\genF}_{i} {\genF}_{\ol{i+1}}  - {v} {\genF}_{\ol{i+1}}  {\genF}_{i}
	=  {\genF}_{\ol{i}} {\genF}_{i+1} - {v} {\genF}_{i+1} {\genF}_{\ol{i}} ;\\
({\rm QQ6}) \quad
& {\genE}_{i}^2 {\genE}_{j} - ( {v} + {v}^{-1} ) {\genE}_{i} {\genE}_{j} {\genE}_{i} + {\genE}_{j}  {\genE}_{i}^2 = 0, \qquad
	{\genF}_{i}^2 {\genF}_{j} - ( {v} + {v}^{-1} ) {\genF}_{i} {\genF}_{j} {\genF}_{i} + {\genF}_{j}  {\genF}_{i}^2 = 0, \\
& {\genE}_{i}^2 {\genE}_{\ol{j}} - ( {v} + {v}^{-1} ) {\genE}_{i} {\genE}_{\ol{j}} {\genE}_{i} + {\genE}_{\ol{j}}  {\genE}_{i}^2 = 0, \qquad
	{\genF}_{i}^2 {\genF}_{\ol{j}} - ( {v} + {v}^{-1} ) {\genF}_{i} {\genF}_{\ol{j}} {\genF}_{i} + {\genF}_{\ol{j}}  {\genF}_{i}^2 = 0,  \\
& \qquad \mbox{ where } \quad |i-j| = 1.
\end{align*}
\end{defn}

\begin{rem}\label{induction_N}
Observe that the relations in (QQ3) imply
\begin{align*}
{\genE}_{\ol{j}}
&= - {v} {\genK}_{j+1} {\genK}_{\ol{j+1}} {\genE}_{j} + {v}^{-1} {\genE}_{j} {\genK}_{j+1}  {\genK}_{\ol{j+1}} , \\
{\genF}_{\ol{j}}
&= {v}  {\genK}_{j+1}^{-1} {\genK}_{\ol{j+1}} {\genF}_{j} - {v}^{-1} {\genF}_{j} {\genK}_{j+1}^{-1}  {\genK}_{\ol{j+1}},
\end{align*}
and the third relation  in  (QQ4) implies
\begin{align*}
{\genK}_{\ol{j}}
&=
	{\genE}_{j} {\genF}_{\ol{j}}  {\genK}_{j+1}  - {\genF}_{\ol{j}} {\genE}_{j} {\genK}_{j+1}  +  {\genK}_{j}^{-1} {\genK}_{\ol{j+1}} {\genK}_{j+1}\\
&= {\genE}_{j}  {\genK}_{\ol{j+1}} {\genF}_{j}
	 - {v}^{-1} {\genE}_{j}  {\genF}_{j}  {\genK}_{\ol{j+1}}
		- {v}  {\genK}_{\ol{j+1}} {\genF}_{j}  {\genE}_{j}
		+  {\genF}_{j}  {\genK}_{\ol{j+1}} {\genE}_{j}
		+  {\genK}_{j}^{-1} {\genK}_{\ol{j+1}} {\genK}_{j+1} .
\end{align*}
Hence by induction on {$j$} in descending order  from ({$n-1$}) to $1$,
one can obtain all other odd generators in $\Uvqn$. In other words, 
$\Uvqn$ can be generated by $\{ {\genE}_j, {\genF}_j, {\genK}_i^{\pm1}, {\genK}_{\ol{n}}\where  1\le j \le n-1, \  1 \le i \le n \}$. 
Moreover, as seen in \cite{GJKKK} (cf., \cite{DLZ}), 
one may use a subset of the relations (QQ1)--(QQ6) as the defining relations.
We will see more details later in Section \ref{sec_generators}.
\end{rem}

Following \cite{Ol},
there is  a Schur-Weyl type duality
between $\Uvqn$ and $\HCR$ in the case {$R=\QQ(v)$}
 on the tensor space $V(n|n)^{\otimes r}$ 
with $V(n|n)$ being the superspace of both even and odd dimension $n$. 
Meanwhile by \cite{DW1}, the tensor space $V(n|n)^{\otimes r}$ 
can be identified with $\oplus_{\mu\in\Lambda(n,r)}x_{\mu}\HCR$ as $\HCR$-supermdoules. 
Olshanski proved that there exists an action of $\Uvqn$ on $V(n|n)^{\otimes r}$ satisfying \eqref{eq:twist} 
(which can be shown by a direct calculation even though it wasn't written down clearly in \cite{Ol}) 
and hence there exists a surjective homomorphism $\eta_r:
\Uvqn\to\bsSQvnr$.
\footnote{It is worthwhile to clarify that the superalgebra defined in \cite[(9.4)]{DW1} is actually the twisted version $\bsSQvnr$ here even though the same term ``quantum queer Schur superalgebra'' was used therein.}
Then  by   \cite[Theorem 9.2]{DW1}, we have the following presentation for
$\bsSQvnr:=\SQqnrR$ where $R=\Qv$ and $q={v}^2$.

\begin{prop}\label{relations}
There is a superalgebra epimorphism $\bs{\eta}_r:\Uvqn\to\bsSQvnr$ such that  $\ker(\bs{\eta}_r)$ is generated by the elements
$${\genK}_1\cdots {\genK}_n-{v}^r,\;\; ({\genK}_i-1)({\genK}_i-{v})\cdots({\genK}_i-{v}^{r-1}), \;\;{\genK}_{\bar i} ({\genK}_i-1)({\genK}_i-{v})\cdots({\genK}_i-{v}^{r-1}),$$
for all $1\leq i\leq n$.
\end{prop}

\begin{rem} The superalgebra $\QqnrR$ introduced in \cite{DW2} by taking advantage of the tensor space 
is related to $\bsSQvnr$ in the way given in Corollary \ref{super_al_iso}, 
so it is reasonable to identify them 
when considering representation theory 
	over a field extension of the complex field as did in \cite[Proposition 6.4, Page 337]{DW2}.
\end{rem}

\spaceintv
\section{Multiplication formulas in {$\QqnrR$} and $\SQqnrR$: the even cases}\label{even case}

We compute, for any $A =  ({\SEE{a}_{i,j}} | {\SOE{a}_{i,j}})  \in\MNZ(n,r)$, the structure constants that appear in $\phi_X\phi_A$, where $X$ is one of the following six matrices:
\begin{equation}\label{even-odd}
\aligned
\text{\bf the even cases:}\;\;&(\lambda|\mathrm{O}),\;\; (\lambda + E_{h, h+1} -E_{h+1, h+1} |\mathrm{O}), \;\;(\lambda-E_{h,h}+E_{h+1,h}|\mathrm{O});\\
\text{\bf the\, odd\, cases:}\;\;& (\lambda-E_{h,h}|E_{h,h}),\;\;(\lambda-E_{h+1, h+1}| E_{h, h+1} ),\;\;(\lambda-E_{h,h}|E_{h+1,h}).
\endaligned
\end{equation}
 Here $1\leq j\leq n$ and $1\leq h\leq n-1$, and we regard any $\lambda \in \Lambda(n,r)$
 as a diagonal matrix $\lambda={\diag}(\lambda_1,\ldots,\lambda_n)$ in $\MN(n)$
 and write $\lambda+M:={\diag}(\lambda_1,\ldots,\lambda_n)+M$. In this section, we deal with the even case.

\begin{conv}\label{CONV}
The terms $\phi_M$ appearing in $\phi_X\phi_A$ have the form
$M=(\SE{A} \pm E|\SO{A} \pm E')$ for some {$E  \in \MN(n)$},   {$E'  \in M_n(\ZZ_2)$}. If $M\not\in \MNZN(n)$ (i.e., $M$ has a negative entry or $\SO{A} \pm E'$ has an entry $>1$), then we set $\phi_M$ to be 0. We make a similar convention for the elements $T_M$ in \eqref{eq cA} that defines $\phi_M$.
\end{conv}

The following notations will be used in the proofs of the multiplication formulas in this and next sections.

\begin{notn}[for the proofs in Sections \ref{even case} and \ref{odd case}]\label{rem_phi_short_0}
Let $A = ({\SEE{a}_{i,j}} | {\SOE{a}_{i,j}})  \in \MNZN(n)$. For fixed $1\leq h,k\leq n$, we decompose the elements $c_A, c'_A$ defined \eqref{eq cA} as a product of three segments:
\begin{align}\label{cA-decomp}
	&c_A =  {\cbefore} \cdot  {\cmiddle} \cdot {\cafter},\qquad c'_A =  ({\cbefore})' \cdot  ({\cmiddle})' \cdot ({\cafter})'
\end{align}
where ${\cmiddle} := c^{\SOE{a}_{h,k}}_{{q}, {\widetilde{a}}_{h-1,k}+1, {\widetilde{a}}_{h,k}}
		\cdot c^{\SOE{a}_{h+1,k}}_{{q}, {\widetilde{a}}_{h,k}+1, {\widetilde{a}}_{h+1,k}},$
and
{$ {\cbefore} $} (resp., {$ {\cafter} $}) are the segment of the product before (resp., after) $\cmiddle$. The decomposition for $c'_A$ is similar.

Recall the notations in  \eqref{def_ahk_notations} and \eqref{parity}, 
and let {$\AK(h+1,k)=\AK(h+1,k)(A)$},  $\BK(h,k)=\BK(h,k)(A)$. Then
 {$\ro(\AP)=\lambda + \bs{\ep}_{h} - \bs{\ep}_{h+1} =:\lambda_h^+$},  
 {$\co(\AM)=\lambda - \bs{\ep}_{h} + \bs{\ep}_{h+1} =:\lambda^-_h$}, and
 {$\co(\AP)  = \co(\AM) = \co(A)$}.
Let $\nu = \nu_{A} = (  \cdots ,a_{h,k},  a_{h+1,k}, \cdots )$ be as in \eqref{nuA} and let
\begin{equation}\label{nu+nu-}
\aligned
	&\nu^+ = \nu_{\AP} = (  \cdots,a_{h,k}+1,  a_{h+1,k}-1, \cdots ), \quad
	\delta^+ = ( \cdots, a_{h,k}, 1, a_{h+1,k}-1, \cdots ),\\
	&\nu^- = \nu_{\AM} = (  \cdots ,  a_{h,k}-1, a_{h+1,k}+1, \cdots ), \quad
	\delta^-= ( \cdots ,a_{h,k}-1, 1, a_{h+1,k} , \cdots ).
	\endaligned
\end{equation}
Here the omitted components are identical. Clearly, $\fS_{\delta^\pm}\subset \fS_{\nu^\pm}$.
If {${\D_{\pm}} $} denote the sets of the shortest representatives  of the left cosets of  {$\fS_{\delta^\pm}$} in {$\fS_{\nu^\pm}$}, then, putting {$u = \widetilde{a}_{h,k}$}, {$s = a_{h,k}$} and $t = a_{h+1,k}$, we have
\begin{equation}\label{subsubgroup}
\begin{aligned}
&	x_{\nu^+} = x_{\D_{+}} x_{\delta^+} = x_{\delta^+} \TDIJ({u},  {u-s+1}),\\
& x_{\nu^-} = x_{\D_{-}} x_{\delta^-} =x_{\delta^-} \TAIJ({u},  {u-s+1}),
\end{aligned}
\qquad x_{\nu} =  x_{\delta^+} \TAIJ( {u+1}, {u+t-1})
	=  x_{\delta^-} \TDIJ({u-1},  {u-s +1}).
\end{equation}
Let $'\!\nu^+ = \nu_{'\!\AP}$ and $'\!\nu^- = \nu_{'\!\AM}$ be defined as in \eqref{nuA} by using the transposes $'\!\AP$, $'\!\AM$ of $\AP$, $\AM$, respectively.
Then $\fS_{'\!\nu^\pm}\subseteq \fS_{\lambda_h^\pm}$ and we have (cf. \cite[Corollary 3.5]{DW2})
\begin{equation}\label{eq_ta+}
\aligned
	x_{{{\lambda}_{h}^{+}}} {T_{d_{\AP}}}
	&=  x_{{\lambda}_{h}^{+}\backslash{}'\!\nu^+}  {T_{d_{\AP}}}  x_{\nu^+}, \quad\text{where}\quad x_{{{\lambda}_{h}^{+}}\backslash{}'\! \nu^+}  := \sum_{u' \in {\D}^{-1}_{{}'\!\nu^+} \cap \fS_{{{\lambda}_{h}^{+}}}} T_{u'};\\
x_{{{\lambda}_{h}^{-}}} {\TDAM}
	&= x_{{{\lambda}_{h}^{-}}\backslash{}'\!\nu^-}  {\TDAM}  x_{\nu^-}, \quad\text{where}\quad
	x_{{{\lambda}_{h}^{-}}\backslash{}'\!\nu^-} := \sum_{u' \in {\D}^{-1}_{{}'\!\nu^-} \cap \fS_{{{\lambda}_{h}^{-}}}} T_{u'}.
\endaligned
\end{equation}

\end{notn}
We are now ready to derive families of multiplication formulas for $\phi_X\phi_A$ in the queer $q$-Schur superalgebra $\QqnrR$ (and, hence, in $\SQqnrR$ by Proposition \ref{prop_PhiAPhiB}). We first assume  $X=(\lambda|\mathrm{O})$,
 {$(\lambda + E_{h, h+1}-E_{h+1, h+1}|\mathrm{O})$} or {$(\lambda-E_{h,h} + E_{h+1,h} |\mathrm{O})$}. Recall the notation in \eqref{stepd}.

\begin{prop}\label{phiupper0}
Let {$A =  ({\SEE{a}_{i,j}} | {\SOE{a}_{i,j}})   \in \MNZ(n,r)$} with {$a_{i,j}=\SEE{a}_{i,j} + \SOE{a}_{i,j}$}, 
{$h \in [1,n-1]$}  and
	 {${\lambda} = \ro(A)$}. Then the following multiplication holds in $\QqnrR$:
\begin{align*}
(1)\;\;& \phi_{(\lambda |\mathrm{O})} \phi_A = \delta_{\lambda,\ro(A)}\phi_{A},\qquad
	 \phi_A \phi_{(\mu |\mathrm{O})}  =\delta_{\mu,\co(A)}\phi_{A}.\\
(2)\;\;& \phi_{(\lambda + E_{h, h+1}-E_{h+1, h+1} |\mathrm{O})}  \phi_A \\
 & = \sum_{k=1}^n \Big\{
	{q}^{\BK(h,k) + \SOE{a}_{h+1,k}}  \STEP{ \SEE{a}_{h,k} + 1}
		\phi_{(\SE{A} + E_{h,k}  - E_{h+1, k} | \SO{A} )}
	+
	{q}^{\BK(h,k)}  \phi_{(\SE{A}  | \SO{A} + E_{h,k} - E_{h+1, k}  )}  \\
	&\hspace{6cm} +
	{q}^{\BK(h,k)-1} \STEPPD{a_{h,k}+1}
	\phi_{(\SE{A} + 2E_{h,k} | \SO{A}  -E_{h,k} - E_{h+1,k} )}\Big
\}.\\
(3)\;\;& \phi_{( \lambda -E_{h, h} + E_{h+1, h} |\mathrm{O})}  \phi_A \\
&=\sum_{k=1}^n \Big\{
	 {q}^{\AK(h+1,k) }  \STEP{ \SEE{a}_{h+1, k} +1}
	\phi_{(\SE{A} - E_{h,k} + E_{h+1, k} | \SO{A} )}
	 +
	{q}^{\AK(h+1,k)  + a_{h, k} - 1} \phi_{(\SE{A}  |\SO{A} - E_{h,k} + E_{h+1,k})} \\
	& \hspace{4.5cm} -
	{q}^{\AK(h+1,k)  + a_{h, k} -2} \STEPPDR{a_{h+1, k} +1}
	\phi_{(\SE{A} + 2E_{h+1,k} | \SO{A}  - E_{h,k} - E_{h+1,k})}
\Big\}.
\end{align*}
\end{prop}
\begin{proof} Formula (1) is a special case of \eqref{co=ro}, noting that $\phi_{(\lambda|O)}$ is the identity map on $x_\lambda\HCR$.

We now prove (2).
Let {$X=X^+ = (\lambda + E_{h, h+1} -E_{h+1, h+1} |\mathrm{O})$}.
Then {$\ro(X) =\lambda+ \alpha_h = {\lambda}_{h}^{+}$},
{$d_X = 1$}, {$T_{d_X} = 1$},  {$c_X  = 1$},
and
{$
	 \D_{{\nu}_{X}} \cap \fS_{{\lambda}}   = \{ 1, s_{{\widetilde{\lambda}_{h}+1}},
		s_{{\widetilde{\lambda}_{h}+1}} s_{{\widetilde{\lambda}_{h}+2}},
		\cdots,
		s_{{\widetilde{\lambda}_{h}+1}} s_{{\widetilde{\lambda}_{h}+2}} \cdots s_{{\widetilde{\lambda}_{h+1}-1}}
	\}
$}.

Then,  for $\mu=\co(A)$, \eqref{eq_gB} with $B=X$ becomes
$$
z	:= x_{{{\lambda}_{h}^{+}}} \TAIJ( {\widetilde{\lambda}_{h}+1}, {\widetilde{\lambda}_{h+1}-1}) {T_{d_A}} c_{A}
		\sum _{{\sigma} \in \D_{{\nu}_{_A}} \cap \fS_{\mu}} T_{\sigma}
	 = 	\sum_{{1\leq k\leq n}\atop {a_{h+1,k}>0}} z_k,
$$
where, for each {$k \in [1, n]$} with $a_{h+1,k}>0$,
\begin{equation}\label{z_k}
\aligned
	z_k&=
		x_{{{\lambda}_{h}^{+}}}
		T_{\widetilde{\lambda}_{h}+1} T_{\widetilde{\lambda}_{h}+2} \cdots T_{\widetilde{\lambda}_{h}+\AK(h+1,k)}
		\TAIJ({\AK(h+1,k)+1},{\AK(h+1,k)+a_{h+1,k}-1}){T_{d_A}} c_{A}
		\sum _{{\sigma} \in \D_{{\nu}_{_A}} \cap \fS_{\mu}} T_{\sigma}\\
	     &=\sum_{j=\AK(h+1,k)}^{\AK(h+1,k+1)-1}x_{{{\lambda}_{h}^{+}}}
		T_{\widetilde{\lambda}_{h}+1} T_{\widetilde{\lambda}_{h}+2} \cdots  T_{\widetilde{\lambda}_{h}+j} \cdot
		{T_{d_A}} c_{A}
		\sum _{{\sigma} \in \D_{{\nu}_{_A}} \cap \fS_{\mu}} T_{\sigma}.
\endaligned
\end{equation}

Assume now {$t = a_{h+1,k}>0$}. For $\AK(h+1,k)\leq j\leq \AK(h+1,k+1)-1$, let {$p_j = j - \AK(h+1,k) $}.
Then, with {$u = \widetilde{a}_{h,k}$}, {$s = a_{h,k}$},  $t = a_{h+1,k}$, and $\nu=\nu_A$, \eqref{z_k} becomes
\begin{align*}
	z_k
&=
		\sum_{j=\AK(h+1,k)}^{\AK(h+1,k+1) - 1}
		x_{{{\lambda}_{h}^{+}}}
		T_{\widetilde{\lambda}_{h}} T_{\widetilde{\lambda}_{h}-1}\cdots T_{\widetilde{\lambda}_{h}-\BK(h,k)+1}
		{{\TDAP}}
		T_{\widetilde{a}_{h,k}+1} \cdots T_{\widetilde{a}_{h,k}+p_j}
		c_{A}
		\sum _{{\sigma} \in \D_{{\nu}} \cap \fS_{\mu}} T_{\sigma} \quad
		\mbox{(by Lemma  \ref{prop_pjshift})}
		\\
&=
	{q}^{\BK(h,k)} x_{{{\lambda}_{h}^{+}}}  {{\TDAP}}
	\TAIJ( {u+1}, {u+t-1})	c_A
	\sum _{{\sigma} \in \D_{{\nu}} \cap \fS_{\mu}} T_{\sigma}
	\quad ( \mbox{since } s_{\widetilde{\lambda}_{h}}, s_{\widetilde{\lambda}_{h}-1}, \cdots,
		s_{\widetilde{\lambda}_{h}-\BK(h,k)+1} \in {\fS}_{{{\lambda}_{h}^{+}}})
	\\
&=
	{q}^{\BK(h,k)} x_{{{\lambda}_{h}^{+}}\backslash{}'\!\nu^+} {{\TDAP}}
	x_{{\D_{+}}} x_{\delta^+} \TAIJ( {u+1}, {u+t-1}) c_A
	\sum _{{\sigma} \in \D_{{\nu}} \cap \fS_{\mu}} T_{\sigma}
	\quad
 {
	\mbox{(by \eqref{eq_ta+}  and \eqref{subsubgroup})}
 }
	\\
&=
	{q}^{\BK(h,k)} x_{{{\lambda}_{h}^{+}}\backslash{}'\!\nu^+} {{\TDAP}}
	x_{{\D_{+}}}  (c_A)' x_{\nu}
	\sum _{{\sigma} \in \D_{{\nu}} \cap \fS_{\mu}} T_{\sigma}
	\quad
 {
	\mbox{(by $x_{\delta^+} \TAIJ( {u+1}, {u+t-1})  = x_{\nu}$ in \eqref{subsubgroup} and \eqref{xc} )}
 }.
\end{align*}
By writing $x_{\nu}=x_{\delta^+} \TAIJ( {u+1}, {u+t-1})$ again and applying Lemma  \ref{Tnsum}(1), 
we have
 \begin{equation}\label{zkupper}
	\begin{aligned}
&z_k=
	{q}^{\BK(h,k)} x_{{{\lambda}_{h}^{+}}\backslash{}'\!\nu^+} {{\TDAP}}
	x_{{\D_{+}}}   (c_A)'  x_{\delta^+}  \TDIJ({u},  {u-s+1})
	\sum _{{\sigma} \in \D_{\nu^+} \cap \fS_{\mu}} T_{\sigma}
	\\
&=
	{q}^{\BK(h,k)} x_{{{\lambda}_{h}^{+}}\backslash{}'\!\nu^+} {{\TDAP}}
	x_{{\D_{+}}}   ({\cbefore})' ({\cmiddle})' ({\cafter})' x_{\delta^+}  \TDIJ({u},  {u-s+1})
	\sum _{{\sigma} \in \D_{\nu^+} \cap \fS_{\mu}} T_{\sigma} \\
&=
	{(-1)}^{(\SOE{a}_{h,k} + \SOE{a}_{h+1,k}) \YK } {q}^{\BK(h,k)} x_{{{\lambda}_{h}^{+}}\backslash{}'\!\nu^+} {{\TDAP}}
	x_{{\D_{+}}}  ({\cmiddle})'  x_{\delta^+}  \TDIJ({u},  {u-s+1}) {\cbefore} {\cafter}
	\sum _{{\sigma} \in \D_{\nu^+} \cap \fS_{\mu}} T_{\sigma} .
\end{aligned}
\end{equation}
Here, for the last equality, we used the following relations:
\begin{enumerate}[(a)]
\item
$({\cbefore})' ({\cmiddle})' ({\cafter})' ={(-1)}^{(\SOE{a}_{h,k} + \SOE{a}_{h+1,k}) \YK }({\cmiddle})' ({\cbefore})' ({\cafter})' $ (noting that $\YK$ is the partial sum at the $(h-1,k)$ position associated with $\nu_{A^{\bar 1}}$);
\item
$({\cbefore})' ({\cafter})' =(c_{\delta^+}^\alpha)'$ for some $\alpha\leq\delta^+$ with $\alpha_i=0,\;i\in[\widetilde a_{h-1,k}+1,\widetilde a_{h+1,k}]$,
and hence  $({\cbefore})' ({\cafter})' x_{\delta^+}=x_{\delta^+}({\cbefore})({\cafter}) $ by \eqref{xc}; 
\item
$({\cbefore})({\cafter})$ commutes with $\TDIJ({u},  {u-s+1})$ since $T_ic_j=c_jT_i$ for all $j\neq i,i+1$.
\end{enumerate}

To complete the computation of $z_k$, we need to consider the following four cases.

\noindent
 {\bf Case 1.} {$\SOE{a}_{h,k} = 0$}, {$\SOE{a}_{h+1,k} = 0$}.
In this case,  $(\cmiddle)' = 1$, $c_{A} ={\cbefore} {\cafter} =c_{A^+_{h,k}}$ and
\begin{align*}
z_k
&=
	{q}^{\BK(h,k)} x_{{{\lambda}_{h}^{+}}\backslash{}'\!\nu^+} {{\TDAP}}
	x_{\nu^+}  \TDIJ({u},  {u-s+1}) {\cbefore} {\cafter}
	\sum _{\sigma\in \D_{\nu^+ \cap \fS_\mu}} T_{\sigma}
	\qquad \mbox{ (by $x_{{\D_{+}}}  x_{\delta^+}  = x_{\nu^+} $) }
	\\
&=
	{q}^{\BK(h,k)} \STEP{s+1} x_{{{\lambda}_{h}^{+}}} {{\TDAP}}   {\cbefore} {\cafter}
	 \sum _{{\sigma} \in \D_{\nu^+ } \cap \fS_{\mu}}  T_{\sigma}
	\qquad  \mbox{(by Lemma \ref{xTinverse}(1))}
	\\
&=
	{q}^{\BK(h,k)} \STEP{ \SEE{a}_{h,k} +1} T_{(\SE{A} + E_{h,k} - E_{h+1, k} | \SO{A})} .
\end{align*}

\noindent
 {\bf Case 2.} {$\SOE{a}_{h,k} = 0$}, {$\SOE{a}_{h+1,k} = 1$}.
Then $(\cmiddle)' = c'_{{q}, u +1, u +t}$ and
\begin{align*}
z_k &=
	{(-1)}^{\YK } {q}^{\BK(h,k)} x_{{{\lambda}_{h}^{+}}\backslash{}' \nu^+} {{\TDAP}}\cdot
	x_{{\D_{+}}}  (c_{{q}, u +1, u +t})'  x_{\delta^+}  \TDIJ({u},  {u-s+1}) \cdot{\cbefore} {\cafter}
	\sum _{{\sigma} \in \D_{\nu^+} \cap \fS_{\mu}} T_{\sigma} .
\end{align*}
Direct calculation using \eqref{subsubgroup}, \eqref{xc}, and \eqref{cqij} shows
\begin{align*}
	x_{{\D_{+}}} (c_{{q}, u +1, u +t})'  &x_{\delta^+}  \TDIJ({u},  {u-s+1}) = x_{{\D_{+}}} ({q} (c_{{q}, u +2, u +t})' +  c_{u+1}) x_{\delta^+}  \TDIJ({u},  {u-s+1}) \\
	&= {q} x_{{\D_{+}}} x_{\delta^+} c_{{q}, u +2, u +t}  \TDIJ({u},  {u-s+1}) +  x_{{\D_{+}}} x_{\delta^+}c_{u+1} \TDIJ({u},  {u-s+1})  \\
	&=
		{q} \STEP{s +1}  x_{\nu^+} c_{{q}, u +2, u +t}  +  x_{\nu^+}c_{{q}, u-s+1, u +1} \qquad (\text{by Lemmas \ref{xTinverse}(1) and \ref{xcT}(3)}).
\end{align*}
Recall from \eqref{cqij} that $(c_{{q}, u +2, u +t})' =0=c_{{q}, u +2, u +t}$ if $a^{\bar 0}_{h+1,k}=0$ (i.e., $t=1$).
Hence, by noting ${\cbefore} c_{{q}, u +2, u +t}  {\cafter}=c_{\nu^+}^\alpha$ with $\alpha=\nu_{A^{\bar1}}$ and ${\cbefore} c_{{q}, u-s+1, u +1} {\cafter}=c_{\nu^+}^\beta$ with $\beta=\nu_{\SO{A} - E_{h+1, k}  + E_{h,k}}$, we have
\begin{align*}
z_k
&=
	{q}^{\BK(h,k)+1} \STEP{s+1}  x_{{{\lambda}_{h}^{+}}} {{\TDAP}}
	{\cbefore} c_{{q}, u +2, u +t}  {\cafter}
	\sum _{{\sigma} \in \D_{{\nu^+}} \cap \fS_{\mu}} T_{\sigma}  \\
	&\qquad + {q}^{\BK(h,k)} x_{{{\lambda}_{h}^{+}}} {{\TDAP}}
	{\cbefore} c_{{q}, u-s+1, u +1} {\cafter}
	\sum _{{\sigma} \in \D_{{\nu^+}} \cap \fS_{\mu}} T_{\sigma}  \\
&=
	{q}^{\BK(h,k)+1} \STEP{ \SEE{a}_{h,k} +1}  T_{(\SE{A} + E_{h,k} - E_{h+1, k} | \SO{A}) }
	+ {q}^{\BK(h,k)} T_{(\SE{A}  | \SO{A}  + E_{h,k} - E_{h+1, k} ) } .
\end{align*}
Here the sign in $z_k$ is cancelled after swapping $ c_{{q}, u +2, u +t}\cbefore$.
Note that, 
by Convention \ref{CONV}, $T_{(\SE{A} + E_{h,k} - E_{h+1, k} | \SO{A}) }=0$ if $a^{\bar 0}_{h+1,k}=0$.

\noindent
 {\bf Case 3.} {$\SOE{a}_{h,k} = 1$}, {$\SOE{a}_{h+1,k} = 0$}. Then we have {${\cmiddle} = c_{{q}, u-s+1, u}$} and
\begin{align*}
z_k &=
	{(-1)}^{\YK } {q}^{\BK(h,k)} x_{{{\lambda}_{h}^{+}}\backslash{}'\nu^+} {{\TDAP}}\cdot
	x_{{\D_{+}}}  (c_{{q}, u-s+1, u})'  x_{\delta^+}  \TDIJ({u},  {u-s+1}) \cdot{\cbefore} {\cafter}
	\sum _{{\sigma} \in \D_{{\nu^+}} \cap \fS_{\mu}} T_{\sigma}.
\end{align*}
Again using \eqref{subsubgroup}, \eqref{xc}, \eqref{cqij}, Lemma \ref{xTinverse}(1) and Lemma \ref{xcT}(3), one can deduce
\begin{align*}
 {q} x_{{\D_{+}}}  (c_{{q}, u-s+1, u})' & x_{\delta^+}  \TDIJ({u},  {u-s+1})
	=  x_{{\D_{+}}}  x_{\delta^+} \cdot {q} c_{{q}, u-s+1, u} \cdot  \TDIJ({u},  {u-s+1}) \\
	&=  x_{\nu^+}  c_{{q}, u-s+1, u +1} \TDIJ({u},  {u-s+1}) -x_{\nu^+}  c_{u+1} \TDIJ({u},  {u-s+1}) \\
	&=   (c_{{q}, u-s+1, u +1})'  x_{\nu^+} \TDIJ({u},  {u-s+1}) -x_{\nu^+}  c_{{q}, u-s+1, u +1} \\
	&= ( \STEP{s+1} -1 )x_{\nu^+} c_{{q}, u-s+1, u +1}.
\end{align*}
This leads to
\begin{align*}
z_k
&=
	{q}^{\BK(h,k)-1} ( \STEP{s+1} -1 ) x_{{{\lambda}_{h}^{+}}} {{\TDAP}}
	{\cbefore}c_{{q}, u-s+1, u +1} {\cafter}
	\sum _{{\sigma} \in \D_{{\nu^+}} \cap \fS_{\mu}} T_{\sigma} \\
&=
	{q}^{\BK(h,k) }   \STEP{ \SEE{a}_{h,k} +1}
	T_{(\SE{A}+ E_{h,k}  -E_{h+1, k}  |\SO{A})} .
\end{align*}

\noindent
 {\bf Case 4.} {$\SOE{a}_{h,k} = 1$}, {$\SOE{a}_{h+1,k} = 1$}.
Then {$(\cmiddle)' = c'_{{q}, u-s+1, u} c'_{{q}, u +1, u +t}$} and
\begin{align*}
z_k &=
	{q}^{\BK(h,k)} x_{{{\lambda}_{h}^{+}}\backslash{}'\nu^+} {{\TDAP}}\cdot
	x_{{\D_{+}}}  c'_{{q}, u-s+1, u} c'_{{q}, u +1, u +t}  x_{\delta^+}  \TDIJ({u},  {u-s+1})\cdot {\cbefore} {\cafter}
	\sum _{{\sigma} \in \D_{{\nu^+}} \cap \fS_{\mu}} T_{\sigma}.
\end{align*}
Using \eqref{subsubgroup}, \eqref{xc}, \eqref{cqij},
Lemma \ref{xTinverse}(1) and Lemma \ref{xcT}(3), we get
\begin{align*}
 {q} x_{{\D_{+}}} & c'_{{q}, u-s+1, u} c'_{{q}, u +1, u +t}  x_{\delta^+}  \TDIJ({u},  {u-s+1}) =-x_{{\D_{+}}}  c'_{{q}, u +1, u +t}(qc'_{{q}, u-s+1, u})   x_{\delta^+}  \TDIJ({u},  {u-s+1}) \\
&= - x_{{\D_{+}}} x_{\delta^+} ({q} (c_{{q}, u +2, u +t}) +  c_{u+1})  (c_{{q}, u-s+1, u +1} -  c_{u+1})  \TDIJ({u},  {u-s+1})  \\
&=	- {q}  x_{\nu^+}  c_{{q}, u +2, u +t}c_{{q}, u-s+1, u +1} \TDIJ({u},  {u-s+1})
	+ {q}  x_{\nu^+}  c_{{q}, u +2, u +t}  c_{u+1}  \TDIJ({u},  {u-s+1})  \\
	&\qquad -  x_{\nu^+} c_{u+1} c_{{q}, u-s+1, u +1} \TDIJ({u},  {u-s+1})
	 +  x_{\nu^+}c_{u+1} c_{u+1}  \TDIJ({u},  {u-s+1})  \\
&=	- {q}  (c_{{q}, u +2, u +t})' (c_{{q}, u-s+1, u +1})' x_{\nu^+}   \TDIJ({u},  {u-s+1})
	+ {q} (c_{{q}, u +2, u +t})'  x_{\nu^+} c_{u+1}  \TDIJ({u},  {u-s+1})  \\
	&\qquad + x_{\nu^+} (c_{{q}, u-s+1, u +1} c_{u +1}  + 2) \TDIJ({u},  {u-s+1})
	 - x_{\nu^+}  \TDIJ({u},  {u-s+1})  \\
&=	- {q}  \STEP{s+1} x_{\nu^+}  c_{{q}, u +2, u +t} c_{{q}, u-s+1, u +1}
	+ {q} x_{\nu^+} c_{{q}, u +2, u +t}  c_{{q}, u-s+1, u +1}   \\
	&\qquad + (c_{{q}, u-s+1, u +1} )' x_{\nu^+} c_{u +1}  \TDIJ({u},  {u-s+1})
	+ \STEP{s+1} x_{\nu^+}   \\
&=	{q}( \STEP{s+1} - 1 ) x_{\nu^+}  c_{{q}, u-s+1, u +1} c_{{q}, u +2, u +t} + (\STEP{s+1} - {\STEPP{s+1}} ) x_{\nu^+} \;\;(\text{by Lemma} \ref{xcT}(3)(4)),
\end{align*}
then, noting ${q}( \STEP{s+1} - 1 )=q^2\STEP{s}$, we have, in this case,
\begin{align*}
z_k
&=
	{q}^{\BK(h,k) + 1}  \STEP{s}
	x_{{{\lambda}_{h}^{+}}} {{\TDAP}}
	{\cbefore} c_{{q}, u-s+1, u +1} c_{{q}, u +2, u +t}  {\cafter}
	\sum _{{\sigma} \in \D_{{\nu^+}} \cap \fS_{\mu}} T_{\sigma} \\
	&\qquad + {q}^{\BK(h,k)-1} (\STEP{s+1} - {\STEPP{s+1}} )
	x_{{{\lambda}_{h}^{+}}} {{\TDAP}} {\cbefore} {\cafter}
	\sum _{{\sigma} \in \D_{{\nu^+}} \cap \fS_{\mu}} T_{\sigma} \\
&=
	{q}^{\BK(h,k) + 1}   \STEP{ \SEE{a}_{h,k}+1}
	T_{(\SE{A} + E_{h,k} - E_{h+1,k} | \SO{A})}
	  + {q}^{\BK(h,k)-1} \STEPPD{a_{h,k}+1}
	T_{( \SE{A} + 2E_{h,k} | \SO{A}  -E_{h,k} - E_{h+1,k} )} .
\end{align*}
We now combine the four cases above into a single expression.
Let
\begin{align*}
z'_k&=
	{q}^{\BK(h,k) + \SOE{a}_{h+1,k}}  \STEP{ \SEE{a}_{h,k} + 1}
		T_{(\SE{A}  + E_{h,k} - E_{h+1, k}| \SO{A} )}
	+
	{q}^{\BK(h,k)} T_{(\SE{A}  | \SO{A} + E_{h,k}- E_{h+1, k}   )}  \\
	&\hspace{5cm} +
	{q}^{\BK(h,k)-1} \STEPPD{a_{h,k}+1}
	T_{(\SE{A} + 2E_{h,k} | \SO{A}  -E_{h,k} - E_{h+1,k} )}.
\end{align*}
Then, by the notational Convention \ref{CONV}, $z'_k=0$ if $a_{h+1,k}=0$. A case-by-case argument shows that $z'_k=z_k$ if $a_{h+1,k}>0$. This proves (2).

Finally, the proof of (3) is symmetric. Observe that,
 for {$X=X^- = ( \lambda -E_{h, h} + E_{h+1, h}| \mathrm{O})$}, we have $\mathcal{D}_{\nu_X}\cap \mathfrak{S}_\lambda
= \{ 1,  s_{\widetilde{\lambda}_h - 1}, s_{\widetilde{\lambda}_h - 1}s_{\widetilde{\lambda}_h-2},
	\cdots,
	s_{\widetilde{\lambda}_h-1}s_{\widetilde{\lambda}_h-2}\cdots s_{\widetilde{\lambda}_{h-1}+1}\}.$
Via
replacing the ($+$) notations by the ($-$) notations in various places in the proof of (2),
we can obtain a formula for $z_k$ similar to \eqref{zkupper}.
Then an explicit computation for $z_k$ can also be divided into four cases and hence (3) follows easily. 
We leave the details to the reader.
\end{proof}

\spaceintv
\section{Multiplication formulas in {$\QqnrR$} and $\SQqnrR$: the odd cases}\label{odd case}

In this section, we shall derive the multiplication formulas for $\phi_X\phi_A$ with $X= (\lambda-E_{h,h}|E_{h,h})$,
 $(\lambda-E_{h+1, h+1}| E_{h, h+1} )$ or $(\lambda-E_{h,h}|E_{h+1,h})$ under certain conditions.
Recall Notation \ref{rem_phi_short_0} and the SDP condition  in Definition \ref{defn:SDP}.
\begin{prop}\label{phidiag1}
Let {$A =  ({\SEE{a}_{i,j}} | {\SOE{a}_{i,j}})   \in \MNZ(n,r)$}, {$h \in [1,n]$}, and
	 {${\lambda} = \ro(A)$}.
	Assume {$A$}  satisfies the SDP condition on the $h$-th row.
	 Then we have in $\QqnrR$
\begin{align*}
\phi_{(\lambda-E_{h, h}| E_{h, h})} \phi_A
	&=
\sum_{ k=1}^n
{(-1)}^{ {\SOE{\widetilde{a}}}_{h-1,k} } {q}^{\BK(h,k) }
\Big\{
	 \phi_{(\SE{A} -E_{h,k}|\SO{A}+ E_{h,k})}
	- {\STEPP{a_{h,k}}}
	\phi_{(\SE{A} +  E_{h,k}| \SO{A} - E_{h,k})}
\Big\}.
\end{align*}
\end{prop}
\begin{proof}
	Let  {$X=({\lambda-E_{h, h}| E_{h, h}})$}.
	Clearly  {$\ro(X) = \lambda$}, {$d_X = 1$},
	{$\D_{{\nu}_{_X}} \cap \fS_{\lambda}  = \{ 1\} $},
	and
	{$ c_X = c_{{q}, \widetilde{\lambda}_{h-1}+1, \widetilde{\lambda}_{h}} $}.
Then, for $B=X$ and $\mu=\co(A)$,   \eqref{eq_gB} becomes
\begin{align*}
z
:=\phi_X\phi_A(x_\mu)&=   x_{\lambda}   c_{{q}, \tilde{\lambda}_{h-1}+1, \tilde{\lambda}_{h}}
		{T_{d_A}} c_{A}
		\sum _{{\sigma} \in \D_{{\nu}_{_A}} \cap \fS_{\mu}} T_{\sigma} \\
&=   x_{\lambda}
		( {q}^{{\lambda}_{h} - 1} c_{\tilde{\lambda}_{h-1}+1}
		+ {q}^{{\lambda}_{h} - 2} c_{\tilde{\lambda}_{h-1}+2}+
		 \cdots+
		 {q} c_{\tilde{\lambda}_{h}- 1}
		+  c_{\tilde{\lambda}_{h}}
		)
		{T_{d_A}} c_{A}
		\sum _{{\sigma} \in \D_{{\nu}_{_A}} \cap \fS_{\mu}} T_{\sigma} \\
&=	 \sum_{{1\leq k\leq n}\atop {a_{h,k}>0}}
	 \sum_{j= {\AK(h,k) }+1 }^{ {\AK(h,k+1) }}
	 x_{\lambda}  {q}^{{\lambda_{h}}-j} c_{ \tilde{\lambda}_{h-1} +j} {T_{d_A}} c_{A}
	\sum _{{\sigma} \in \D_{{\nu}_{_A}} \cap \fS_{\mu}} T_{\sigma}.
\end{align*}
	For each
	$j \in [ {\AK(h,k) } +1 ,  { {\AK(h,k+1) }} ]$,
	let {$p_j = j- {\AK(h,k) } -1$},
	then {$0 \le p_j < {a_{h,k}} $} and
\begin{align*}
z
	&=  \sum_{{1\leq k\leq n}\atop {a_{h,k}>0}}  \sum_{j= {\AK(h,k) }+1 }^{ {\AK(h,k+1) }} x_{\lambda}
		{q}^{{\lambda_h}-j}
		c_{ \tilde{\lambda}_{h-1} +  {\AK(h,k) } + p_j + 1} {T_{d_A}} c_{A}
	\sum _{{\sigma} \in \D_{{\nu}_{_A}} \cap \fS_{\mu}} T_{\sigma}
	 \\
	&=  \sum_{{1\leq k\leq n}\atop {a_{h,k}>0}}  \sum_{j= {\AK(h,k) }+1 }^{ {\AK(h,k+1) }} x_{\lambda}  {q}^{{\lambda_h}-j}{T_{d_A}}  c_{\tilde{a}_{h-1,k} + p_j + 1}   c_{A}
	\sum _{{\sigma} \in \D_{{\nu}_{_A}} \cap \fS_{\mu}} T_{\sigma}
		\quad \mbox{(by the SDP condition)}
	\\
	&=  \sum_{{1\leq k\leq n}\atop {a_{h,k}>0}}  \sum_{j= {\AK(h,k) }+1 }^{ {\AK(h,k+1) }} x_{\lambda}
	{T_{d_A}}
	{q}^{{\lambda_h}-j} c_{\tilde{a}_{h-1,k} +j- {\AK(h,k) } }   c_{A}
	\sum _{{\sigma} \in \D_{{\nu}_{_A}} \cap \fS_{\mu}} T_{\sigma}
	\\
	&=  \sum_{{1\leq k\leq n}\atop {a_{h,k}>0}}  x_{\lambda}  {T_{d_A}}
	( {q}^{{\lambda_h} -  {\AK(h,k) } - 1} c_{\tilde{a}_{h-1,k} + 1 } + \cdots + {q}^{{\lambda_h} -  {\AK(h,k+1) }} c_{\tilde{a}_{h-1,k} + a_{h,k} })
	c_{A}
	\sum _{{\sigma} \in \D_{{\nu}_{_A}} \cap \fS_{\mu}} T_{\sigma} \\
	& = \sum_{ 1\leq k\leq n, a_{h,k}>0} z_k,
\end{align*}
where for each $1\leq k\leq n$ with $a_{h,k}>0$,
\begin{align*}
z_k
	&=   x_{\lambda}  {T_{d_A}}
	( {q}^{{\lambda_h} -  {\AK(h,k) } - 1} c_{\tilde{a}_{h-1,k} + 1 } + \cdots + {q}^{{\lambda_h} -  {\AK(h,k+1) }} c_{\tilde{a}_{h-1,k} + a_{h,k} })
	c_{A}
	\sum _{{\sigma} \in \D_{{\nu}_{_A}} \cap \fS_{\mu}} T_{\sigma}\\
&={q}^{{\lambda_h} -  {\AK(h,k+1) }}  x_{\lambda}  {T_{d_A}}
	c_{{q}, \tilde{a}_{h-1,k} + 1 , \tilde{a}_{h-1,k} + a_{h,k} }
	{\ckbefore} c^{\SOE{a}_{h,k}}_{{q}, \tilde{a}_{h-1,k} + 1 , \tilde{a}_{h-1,k} + a_{h,k} } {\ckafter}
	\sum _{{\sigma} \in \D_{{\nu}_{_A}} \cap \fS_{\mu}} T_{\sigma},
\end{align*}
due to the decomposition $c_A = {\ckbefore} c^{\SOE{a}_{h,k}}_{{q}, \tilde{a}_{h-1,k} + 1 , \tilde{a}_{h-1,k} + a_{h,k} } {\ckafter}$ similar to \eqref{cA-decomp}.

{\bf Case 1.}  {$\SOE{a}_{h,k} = 0$}.
In this case,  we have {$ c^{\SOE{a}_{h,k}}_{{q}, \tilde{a}_{h-1,k} + 1 , \tilde{a}_{h-1,k} + a_{h,k} } = 1$}  and  have
\begin{align*}
z_k
&=  {q}^{{\lambda_h} -  {\AK(h,k+1) }}  x_{\lambda}  {T_{d_A}}
	c_{{q}, \tilde{a}_{h-1,k} + 1 , \tilde{a}_{h-1,k} + a_{h,k} }
	{\ckbefore} {\ckafter}
	\sum _{{\sigma} \in \D_{{\nu}_{_A}} \cap \fS_{\mu}} T_{\sigma} \\
&=  {(-1)}^{\YK} {q}^{{\lambda_h} -  {\AK(h,k+1) }}  x_{\lambda}  {T_{d_A}}
	{\ckbefore}
	c_{{q}, \tilde{a}_{h-1,k} + 1 , \tilde{a}_{h-1,k} + a_{h,k} }
	{\ckafter}
	\sum _{{\sigma} \in \D_{{\nu}_{_A}} \cap \fS_{\mu}} T_{\sigma} \\
&=
 {(-1)}^{\YK} {q}^{{\lambda_h} -  {\AK(h,k+1) }} T_{( \SE{A} -E_{h,k}|\SO{A}+ E_{h,k} )} .
\end{align*}

{\bf Case 2.}  {$\SOE{a}_{h,k} = 1$}.
 In this case,
we have {$ c^{\SOE{a}_{h,k}}_{{q}, \tilde{a}_{h-1,k} + 1 , \tilde{a}_{h-1,k} + a_{h,k} } = c_{{q}, \tilde{a}_{h-1,k} + 1 , \tilde{a}_{h-1,k} + a_{h,k} } $}
and
$
{(c_{{q}, \widetilde{a}_{h-1,k} + 1 , \widetilde{a}_{h-1,k} + a_{h,k} })}^2 = -{\STEPP{a_{h,k}}},$ by Lemma \ref{xcT}(4).
 Hence,
\begin{align*}
z_k
&= {q}^{{\lambda_h} -  {\AK(h,k+1) }}  x_{\lambda}  {T_{d_A}}
	c_{{q}, \tilde{a}_{h-1,k} + 1 , \tilde{a}_{h-1,k} + a_{h,k} }
	{\ckbefore} c_{{q}, \tilde{a}_{h-1,k} + 1 , \tilde{a}_{h-1,k} + a_{h,k} } {\ckafter}
	\sum _{{\sigma} \in \D_{{\nu}_{_A}} \cap \fS_{\mu}} T_{\sigma} \\
&= {(-1)}^{\YK} {q}^{{\lambda_h} -  {\AK(h,k+1) }}  x_{\lambda}  {T_{d_A}}
	{(c_{{q}, \tilde{a}_{h-1,k} + 1 , \tilde{a}_{h-1,k} + a_{h,k} })^2}
	{\ckbefore} {\ckafter}
	\sum _{{\sigma} \in \D_{{\nu}_{_A}} \cap \fS_{\mu}} T_{\sigma}  \\
&= {(-1)}^{\YK+1} {q}^{{\lambda_h} -  {\AK(h,k+1) }}  {\STEPP{a_{h,k}}} x_{\lambda}  {T_{d_A}}
	{\ckbefore} {\ckafter}
	\sum _{{\sigma} \in \D_{{\nu}_{_A}} \cap \fS_{\mu}} T_{\sigma} \\
&=
 {(-1)}^{\YK+1} {q}^{{\lambda_h} -  {\AK(h,k+1) }}  {\STEPP{a_{h,k}}}   T_{ (\SE{A} + E_{h,k}|\SO{A}- E_{h,k} )}.
\end{align*}
Finally, by Convention \ref{CONV}, a case-by-case argument  proves  the formula.
\end{proof}

\begin{rem}\label{phidiagN}
By Lemma \ref{shiftonN}, if $h=n$, then the hypothesis that ``$A$ satisfies the SDP condition on the $n$-th row'' may be dropped.
\end{rem}

For the remaining odd cases,
the proof is analogous to the even cases in Section \ref{even case}.
\begin{prop}\label{phiupper1}
	Let {$A =  ({\SEE{a}_{i,j}} | {\SOE{a}_{i,j}})   \in \MNZ(n,r)$}, {$h \in [1,n-1]$}  and
	 {${\lambda} = \ro(A)$}.
	Assume, for each {$k \in [1,n]$} with {$a_{h+1,k} \ge 1$},
	{${A}^+_{h,k}$}  satisfies the SDP condition at $(h,k)$.
	Then
	we have in $\QqnrR$
\begin{align*}
	\phi_{(\lambda-E_{h+1, h+1}| E_{h, h+1})} \phi_A
	&= \sum_{k=1}^n \Big \{
	{(-1)}^{{\SOE{\widetilde{a}}}_{h-1,k}} {q}^{\BK(h,k) + \SOE{a}_{h+1,k} } \phi_{(\SE{A} - E_{h+1, k}| \SO{A}  + E_{h,k} )} \\
	& \qquad +
	{(-1)}^{{\SOE{\widetilde{a}}}_{h-1,k} + 1 - \SOE{a}_{h,k} }
	{q}^{\BK(h,k)  } \STEP{ \SEE{a}_{h,k} +1}
		\phi_{(\SE{A} + E_{h,k}| \SO{A} - E_{h+1, k} )}  \\
	& \qquad +
	{(-1)}^{{\SOE{\widetilde{a}}}_{h-1,k} } {q}^{\BK(h,k)-1 + \SOE{a}_{h+1,k}} \STEPPDR{a_{h,k}+1}
		\phi_{(\SE{A} + 2 E_{h,k}  -E_{h+1, k} |\SO{A} - E_{h,k})}
\Big\}.
\end{align*}
\end{prop}

\begin{proof}
The proof is quite similar to Proposition  \ref{phiupper0}{\rm (2)}.
More precisely,  let {$X= (\lambda-E_{h+1, h+1} | E_{h, h+1})$},
then we have
{$\ro(X) = {{{\lambda}_{h}^{+}}}  $},
 {$c_X = c_{\widetilde{\lambda}_{h}+1}$}.
Similar to Proposition \ref{phiupper0}(2),
with $B=X$ the equation \eqref{eq_gB} becomes
\begin{align*}
z = \sum_{{1\leq k\leq n}\atop {a_{h+1,k}>0}} z_k,
\qquad \text{ where }	z_k&=
		\sum_{j=\AK(h+1,k)}^{\AK(h+1,k+1) - 1}
		x_{{{\lambda}_{h}^{+}}}    c_{\widetilde{\lambda}_{h}+1}
		T_{\widetilde{\lambda}_{h}+1} T_{\widetilde{\lambda}_{h}+2} \cdots T_{\widetilde{\lambda}_{h}+j}
		{T_{d_A}} c_{A}
		\sum _{{\sigma} \in \D_{{\nu}_{_A}} \cap \fS_{\mu}} T_{\sigma}.
\end{align*}

For each {$k \in [1, n]$} with $a_{h+1,k}>0$,
using the notations in  the proof of Proposition \ref{phiupper0}(2),
Lemma \ref{prop_pjshift} implies
\begin{align*}
	z_k
&=
		\sum_{j=\AK(h+1,k)}^{\AK(h+1,k+1) - 1}
		x_{{{\lambda}_{h}^{+}}}  c_{\widetilde{\lambda}_{h}+1}
		(T_{\widetilde{\lambda}_{h}} T_{\widetilde{\lambda}_{h}-1}\cdots T_{\widetilde{\lambda}_{h}-\BK(h,k)+1}  {{\TDAP}}
		T_{\widetilde{a}_{h,k}+1} \cdots T_{\widetilde{a}_{h,k}+p_j})
		c_{A}
		\sum _{{\sigma} \in \D_{{\nu}_{_A}} \cap \fS_{\mu}} T_{\sigma} \\
	&=
		\sum_{j=\AK(h+1,k)}^{\AK(h+1,k+1) - 1}
		x_{{{\lambda}_{h}^{+}}}
		T_{\widetilde{\lambda}_{h}} T_{\widetilde{\lambda}_{h}-1}\cdots T_{\widetilde{\lambda}_{h}-\BK(h,k)+1}
		c_{\widetilde{\lambda}_{h}-\BK(h,k) + 1}
		{{\TDAP}}
		T_{\widetilde{a}_{h,k}+1} \cdots T_{\widetilde{a}_{h,k}+p_j}
		c_{A}
		\sum _{{\sigma} \in \D_{{\nu}_{_A}} \cap \fS_{\mu}} T_{\sigma} \\
	&=
		\sum_{j=\AK(h+1,k)}^{\AK(h+1,k+1) - 1}
		{q}^{\BK(h,k)} x_{{{\lambda}_{h}^{+}}}
		c_{\widetilde{\lambda}_{h}-\BK(h,k) + 1}
		{{\TDAP}}
		T_{\widetilde{a}_{h,k}+1} \cdots T_{\widetilde{a}_{h,k}+p_j}
		c_{A}
		\sum _{{\sigma} \in \D_{{\nu}_{_A}} \cap \fS_{\mu}} T_{\sigma} .
\end{align*}
By th assumption that  
	{${A}^+_{h,k}$}  satisfies the SDP condition at $(h,k)$, we obtain
\begin{align*}
z_k
&= {q}^{\BK(h,k)} x_{{{\lambda}_{h}^{+}}} {{\TDAP}} c_{u +1} \TAIJ( {u+1}, {u+t-1}) c_A
	\sum _{{\sigma} \in \D_{{\nu}_{_A}} \cap \fS_{\mu}} T_{\sigma} .
\end{align*}
Then,
similar to  the even case,
by applying \eqref{eq_ta+}, \eqref{subsubgroup},
Lemma  \ref{Tnsum},
and  \eqref{xc},
we obtain
\begin{align*}
z_k&
=
	{(-1)}^{(\SOE{a}_{h,k} + \SOE{a}_{h+1,k}) \YK } {q}^{\BK(h,k)} x_{{{\lambda}_{h}^{+}}\backslash{}'\!\nu^+} {{\TDAP}}
	x_{{\D_{+}}}  c_{u +1} ({\cmiddle})'  x_{\delta}   \TDIJ({u},  {u-s+1}) {\cbefore} {\cafter}
	\!\!  \sum _{{\sigma} \in \D_{\nu^+} \cap \fS_{\mu}} \!\! \!\! T_{\sigma} .
\end{align*}

Similar to the even situation, we need to consider four cases.

{\bf Case 1.}  {$\SOE{a}_{h,k} = 0$}, {$\SOE{a}_{h+1,k} = 0$}.
In this case we have   {${\cmiddle} = 1$} and
\begin{align*}
z_k &=
	{q}^{\BK(h,k)} x_{{{\lambda}_{h}^{+}}\backslash{}'\!\nu^+} {{\TDAP}}
	x_{{\D_{+}}}  c_{u +1} x_{\delta}  \TDIJ({u},  {u-s+1}) {\cbefore} {\cafter}
	\sum _{{\sigma} \in \D_{\nu^+} \cap \fS_{\mu}}  T_{\sigma} \\
&=
	{q}^{\BK(h,k)} x_{{{\lambda}_{h}^{+}}\backslash{}'\!\nu^+} {{\TDAP}}
	x_{\nu^+} c_{u +1} \TDIJ({u},  {u-s+1}) {\cbefore} {\cafter}
	 \sum _{{\sigma} \in \D_{\nu^+} \cap \fS_{\mu}}  T_{\sigma} \\
&=
	{(-1)}^{\YK} {q}^{\BK(h,k)} x_{{{\lambda}_{h}^{+}}\backslash{}'\!\nu^+} {{\TDAP}}
	x_{\nu^+}  {\cbefore}  c_{{q}, u-s+1, u +1}  {\cafter}
	\sum _{{\sigma} \in \D_{\nu^+} \cap \fS_{\mu}}  T_{\sigma} \\
&=
	{(-1)}^{\YK} {q}^{\BK(h,k)} T_{(\SE{A} - E_{h+1, k}| \SO{A}  + E_{h,k} )}.
\end{align*}

{\bf Case 2.}
 {$\SOE{a}_{h,k} = 0$}, {$\SOE{a}_{h+1,k} = 1$}.
Then {${\cmiddle} = c_{{q}, u +1, u +t}$} and
\begin{align*}
z_k &=
	{(-1)}^{\YK } {q}^{\BK(h,k)} x_{{{\lambda}_{h}^{+}}\backslash{}'\!\nu^+} {{\TDAP}}
	x_{{\D_{+}}}  c_{u +1} (c_{{q}, u +1, u +t})'  x_{\delta}  \TDIJ({u},  {u-s+1})  {\cbefore} {\cafter}
	 \sum _{{\sigma} \in \D_{\nu^+} \cap \fS_{\mu}}  T_{\sigma} .
\end{align*}
Meanwhile by \eqref{cqij}, \eqref{subsubgroup}, Lemma \ref{xcT}(2) and Lemma \ref{xcT}(3), we deduce that
\begin{align*}
	&	x_{{\D_{+}}} c_{u +1}  (c_{{q}, u +1, u +t})'  x_{\delta}   \TDIJ({u},  {u-s+1}) \\
	&= {q} x_{{\D_{+}}} c_{u +1} (c_{{q}, u +2, u +t})'x_{\delta} \TDIJ({u},  {u-s+1})  +  x_{{\D_{+}}} c_{u +1} c_{u+1} x_{\delta}  \TDIJ({u},  {u-s+1}) \\
	&= {q} x_{\nu^+}  c_{u +1} c_{{q}, u +2, u +t}  \TDIJ({u},  {u-s+1}) -  x_{\nu^+}  \TDIJ({u},  {u-s+1}) \\
	&= {q} x_{\nu^+}   c_{{q}, u-s+1, u +1} c_{{q}, u +2, u +t} -  \STEP{s+1}  x_{\nu^+},
\end{align*}
which  implies
\begin{align*}
z_k
&=
	{(-1)}^{\YK } {q}^{\BK(h,k)+1} x_{{{\lambda}_{h}^{+}}\backslash{}'\!\nu^+} {{\TDAP}}
	x_{\nu^+}   c_{{q}, u-s+1, u +1}  c_{{q}, u +2, u +t} \cdot {\cbefore} {\cafter}
	 \sum _{{\sigma} \in \D_{\nu^+} \cap \fS_{\mu}}  T_{\sigma} \\
&\qquad - {(-1)}^{\YK } {q}^{\BK(h,k)}  \STEP{s+1}  x_{{{\lambda}_{h}^{+}}\backslash{}'\!\nu^+} {{\TDAP}}
 	x_{\nu^+} {\cbefore} {\cafter}
	\sum _{{\sigma} \in \D_{\nu^+} \cap \fS_{\mu}}  T_{\sigma}  \\
&=
	{(-1)}^{\YK } {q}^{\BK(h,k)+1}
	T_{(\SE{A} - E_{h+1, k}| \SO{A}  + E_{h,k})} \\
	& \qquad +
	{(-1)}^{\YK +1 } {q}^{\BK(h,k)} \STEP{ \SEE{a}_{h,k} +1}
	T_{(\SE{A} + E_{h,k}| \SO{A} - E_{h+1, k} )} .
\end{align*}

{\bf Case 3.}
  {$\SOE{a}_{h,k} = 1$}, {$\SOE{a}_{h+1,k} = 0$}.
It follows {${\cmiddle} = c_{{q}, u-s+1, u}$} and
\begin{align*}
z_k &=
	{(-1)}^{\YK } {q}^{\BK(h,k)} x_{{{\lambda}_{h}^{+}}\backslash{}'\!\nu^+} {{\TDAP}}
	x_{{\D_{+}}}  c_{u +1}  (c_{{q}, u-s+1, u})'  x_{\delta}   \TDIJ({u},  {u-s+1}) {\cbefore} {\cafter}
	 \sum _{{\sigma} \in \D_{\nu^+} \cap \fS_{\mu}} T_{\sigma} .
\end{align*}
Using \eqref{cqij}, \eqref{subsubgroup}, Lemma \ref{xcT}(2) and Lemma \ref{xcT}(3), we have
\begin{align*}
&{q} x_{{\D_{+}}}  c_{u +1} (c_{{q}, u-s+1, u})'  x_{\delta}  \TDIJ({u},  {u-s+1})  \\
	&=  -x_{{\D_{+}}}   x_{\delta}  {q} c_{{q}, u-s+1, u} c_{u +1} \TDIJ({u},  {u-s+1}) \\
	&=  	- x_{\nu^+}  c_{{q}, u-s+1, u +1} c_{u +1} \TDIJ({u},  {u-s+1})
		+ x_{\nu^+}  c_{u+1} c_{u +1} \TDIJ({u},  {u-s+1}) \\
	&=  	- ( c_{{q}, u-s+1, u +1} )' x_{\nu^+}  c_{u +1} \TDIJ({u},  {u-s+1})
		- x_{\nu^+}  \TDIJ({u},  {u-s+1}) \\
	&=  	- x_{\nu^+}  c_{{q}, u-s+1, u +1}  c_{{q}, u-s+1, u +1}
		- \STEP{s+1} x_{\nu^+}  \\
	&=   ( {\STEPP{s+1}} - \STEP{s+1} ) x_{\nu^+},
\end{align*}
and hence
\begin{align*}
z_k
&=
	{(-1)}^{\YK } {q}^{\BK(h,k)-1}	( {\STEPP{a_{h,k}+1}}  - \STEP{a_{h,k}+1} )
	 T_{(\SE{A}+ 2 E_{h,k}  -E_{h+1, k}  |\SO{A} - E_{h,k})}.
\end{align*}

{\bf Case 4.}  {$\SOE{a}_{h,k} = 1$}, {$\SOE{a}_{h+1,k} = 1$}.
Then {${\cmiddle} = c_{{q}, u-s+1, u} c_{{q}, u +1, u +t}$} and
\begin{align*}
z_k &=
	{q}^{\BK(h,k)} x_{{{\lambda}_{h}^{+}}\backslash{}'\!\nu^+} {{\TDAP}}
	x_{{\D_{+}}}  c_{u +1} (c_{{q}, u-s+1, u} c_{{q}, u +1, u +t})'  x_{\delta}  \TDIJ({u},  {u-s+1}) {\cbefore} {\cafter}
	\sum _{{\sigma} \in \D_{\nu^+} \cap \fS_{\mu}}  T_{\sigma}.
\end{align*}
With  \eqref{cqij},
we have
\begin{align*}
&{q}  c_{u +1} (c_{{q}, u-s+1, u} c_{{q}, u +1, u +t})'  x_{\delta}  \\
&= -  c_{u +1} (c_{{q}, u +1, u +t})' x_{\delta} {q} c_{{q}, u-s+1, u}  \\
&=	-  c_{u +1} {q} (c_{{q}, u +2, u +t})' x_{\delta} c_{{q}, u-s+1, u +1}
	+   c_{u +1} {q} (c_{{q}, u +2, u +t})'  x_{\delta} c_{u+1}    \\
	&\qquad -  c_{u +1} c_{u+1} x_{\delta} c_{{q}, u-s+1, u +1}
	 +   c_{u +1} c_{u+1} x_{\delta} c_{u+1}    \\
&= 	{q}   x_{\delta} c_{{q}, u +2, u +t} c_{u +1} c_{{q}, u-s+1, u +1}
	+ {q}   x_{\delta} c_{{q}, u +2, u +t}
	+    x_{\delta} c_{{q}, u-s+1, u +1}
	 -    x_{\delta} c_{u+1}  .
\end{align*}
With the fact $c_{u +1} c_{{q}, u-s+1, u +1}= -c_{{q}, u-s+1, u +1} c_{u +1}  -2$
 and  \eqref{subsubgroup}, Lemma \ref{xcT}, 
 we have
\begin{align*}
&{q} x_{{\D_{+}}}  c_{u +1} (c_{{q}, u-s+1, u} c_{{q}, u +1, u +t})'  x_{\delta}  \TDIJ({u},  {u-s+1}) \\
&=	{q} x_{\nu^+} c_{{q}, u +2, u +t} (-c_{{q}, u-s+1, u +1} c_{u +1}  -2) \TDIJ({u},  {u-s+1})
	+ {q} x_{\nu^+} c_{{q}, u +2, u +t} \TDIJ({u},  {u-s+1})  \\
	&\qquad + x_{\nu^+} c_{{q}, u-s+1, u +1} \TDIJ({u},  {u-s+1})
	 - x_{\nu^+} c_{u+1}  \TDIJ({u},  {u-s+1})  \\
&=	 -{q} x_{\nu^+} c_{{q}, u +2, u +t} c_{{q}, u-s+1, u +1} c_{u +1}  \TDIJ({u},  {u-s+1})
	- {q} x_{\nu^+} c_{{q}, u +2, u +t}  \TDIJ({u},  {u-s+1}) \\
	&\qquad  + x_{\nu^+} c_{{q}, u-s+1, u +1} \TDIJ({u},  {u-s+1})
	- x_{\nu^+} c_{u+1}  \TDIJ({u},  {u-s+1})  \\
&=	-{q} (c_{{q}, u +2, u +t})' (c_{{q}, u-s+1, u +1})' x_{\nu^+}  c_{{q}, u-s+1, u +1}
	- {q} (c_{{q}, u +2, u +t} )' \STEP{s+1} x_{\nu^+}  \\
	&\qquad  + (c_{{q}, u-s+1, u +1})'\STEP{s+1}  x_{\nu^+}
	- x_{\nu^+} c_{{q}, u-s+1, u +1}   \\
&=	{q} ({\STEPP{s+1}} - \STEP{s+1}) x_{\nu^+} c_{{q}, u +2, u +t}
	+(\STEP{s+1} -1)  x_{\nu^+} c_{{q}, u-s+1, u +1}.
\end{align*}
Then
\begin{align*}
z_k
&=
	{(-1)}^{\YK} {q}^{\BK(h,k)} ({\STEPP{s+1}} - \STEP{s+1}) x_{{{\lambda}_{h}^{+}}\backslash{}'\!\nu^+} {{\TDAP}}
	 x_{\nu^+} {\cbefore} c_{{q}, u +2, u +t} {\cafter}
	\sum _{{\sigma} \in \D_{\nu^+} \cap \fS_{\mu}} T_{\sigma} \\
	& \qquad +
	{(-1)}^{\YK} {q}^{\BK(h,k)-1} (\STEP{s+1} -1) x_{{{\lambda}_{h}^{+}}\backslash{}'\!\nu^+} {{\TDAP}}
	 x_{\nu^+} {\cbefore} c_{{q}, u-s+1, u +1} {\cafter}
	\sum _{{\sigma} \in \D_{\nu^+} \cap \fS_{\mu}}  T_{\sigma} \\
&=
	{(-1)}^{\YK} {q}^{\BK(h,k)} ({\STEPP{a_{h,k} + 1}} - \STEP{ a_{h,k} +1})
	T_{( \SE{A}+2E_{h,k} - E_{h+1, k} | \SO{A}  - E_{h,k} ) } \\
	& \qquad +
	{(-1)}^{\YK} {q}^{\BK(h,k)} \STEP{ \SEE{a}_{h,k} + 1}
	T_{ ( \SE{A} + E_{h,k} | \SO{A}  - E_{h+1,k} ) } .
\end{align*}
Then,
Convention \ref{CONV} and a case-by-case argument  proves the result.
\end{proof}

\begin{prop}\label{philower1}
Let {$A =  ({\SEE{a}_{i,j}} | {\SOE{a}_{i,j}})   \in \MNZ(n,r)$}, {$h \in [1,n-1]$}  and
{${\lambda} = \ro(A)$}.
Suppose   {$A$}  satisfies the SDP condition on the $h$-th row,
then we have
\begin{align*}
\phi_{(\lambda-E_{h, h}| E_{h+1, h})} \phi_A
 &=\sum_{k=1}^n \Big\{
	{(-1)}^{{\SOE{\widetilde{a}}}_{h-1,k} + \SOE{a}_{h,k}} {q}^{\AK(h+1,k) } 	\phi_{(\SE{A} - E_{h,k}| \SO{A}+ E_{h+1, k})} \\
	& \qquad +
	{(-1)}^{{\SOE{\widetilde{a}}}_{h-1,k} + \SOE{a}_{h,k}+1}
	{q}^{\AK(h+1,k) -1}  \STEPPDR{a_{h+1, k} +1}
	\phi_{(\SE{A} - E_{h,k} + 2E_{h+1, k} | \SO{A}  -E_{h+1, k})} \\
	& \qquad +
	{(-1)}^{{\SOE{\widetilde{a}}}_{h-1,k}+1} {q}^{\AK(h+1,k)  + a_{h,k}-1}
	\STEP{ \SEE{a}_{h+1,k}+1}
	\phi_{(\SE{A} + E_{h+1,k} | \SO{A}  - E_{h,k} )}
\Big\}.
\end{align*}
\end{prop}
\begin{proof}
Similar to  the even situation,
the proof is symmetric to Proposition \ref{phiupper1}. As the SDP condition looks different from the one in Proposition \ref{phiupper1}, we shall give an explanation on computing $z_k$. The equation \eqref{eq_gB} with $B=(\lambda-E_{h, h}| E_{h+1, h})$ leads to
\begin{align*}
z = 	 \sum_{{1\leq k\leq n}\atop {a_{h,k}>0}} z_k,
\qquad
\mbox{ where }
z_k
&=
		\sum_{j=\BK(h,k)}^{\BK(h,k-1) - 1}
		x_{{{\lambda}_{h}^{-}}}
		 T_{\widetilde{\lambda}_{h}-1} \cdots T_{\widetilde{\lambda}_{h}-j } c_{\widetilde{\lambda}_{h}-j} {T_{d_A}} c_{A}
		\sum _{{\sigma} \in \D_{{\nu}_{_A}} \cap \fS_{\mu}} T_{\sigma}.
\end{align*}
For pair $k$ and $j$, denote {$q_j = j - \BK(h,k)$},
the assumption of the SDP condition on $(h, k)$ implies
\begin{align*}
	c_{\widetilde{\lambda}_{h}-j} {T_{d_A}} = {T_{d_A}} c_{\widetilde{a}_{h,k} - {q}_j}
\end{align*}
and then
\begin{align*}
z_k
&=
	\sum_{j=\BK(h,k)}^{\BK(h,k-1) - 1}
	x_{{{\lambda}_{h}^{-}}}
	 T_{\widetilde{\lambda}_{h}-1} \cdots T_{\widetilde{\lambda}_{h}-\BK(h,k) - {q}_j } {T_{d_A}}  c_{\widetilde{a}_{h,k} - {q}_j} c_{A}
	\sum _{{\sigma} \in \D_{{\nu}_{_A}} \cap \fS_{\mu}} T_{\sigma} .
\end{align*}
The rest of the proof is symmetric to Proposition \ref{phiupper1}
and we omit the details.
\end{proof}

\medskip

\noindent
{\bf Some special cases.}
In the rest of this section, we shall establish several special multiplication formulas.
Recall the notation in Remark \ref{rem_phi_short_0}.
\begin{prop}\label{phiupper2}
	For any given {$\mu \in \CMN(n,r-1)$}, {$h \in [1,n-1]$}, the following formulas hold
\begin{align*}
{\rm {(1)}} \quad
\phi_{(\mu | E_{h, h+1})}
\phi_{({\mu}+  E_{h+1, h}|O)}
&= 			\phi_{({\mu}-E_{h+1, h+1}   + E_{h+1, h}| E_{h, h+1})}
	+	 {q}^{{\mu}_{h+1}}  \phi_{( {\mu}   | E_{h,h})} \\
	& \qquad
	- 	({q}-1) \STEP{{\mu}_{h}+1}	\phi_{( {\mu}-E_{h+1, h+1}   + E_{h, h}  | E_{h+1,h+1})}, \\
{\rm {(2)}} \,\qquad
\phi_{({\mu}| E_{h, h+1})}
\phi_{({\mu} | E_{h+1, h})}
&=
-	\STEP{{\mu}_{h} + 1} {q}^{ {\mu}_{h+1} }
		\phi_{( {\mu}  +  E_{h,h}|\mathrm{O})}  - 	\phi_{({\mu}-E_{h+1, h+1}   | E_{h, h+1} + E_{h+1, h} )}\\
	& \qquad
	+ 	({q}-1) 	\phi_{( {\mu}-E_{h+1, h+1}    | E_{h,h} + E_{h+1,h+1})} .
\end{align*}
\end{prop}
\begin{proof}
Let $A =  ({\SEE{a}_{i,j}} | {\SOE{a}_{i,j}})  =(\mu+E_{h+1,h}|O)$
or $A =  ({\SEE{a}_{i,j}} | {\SOE{a}_{i,j}})  =(\mu|E_{h+1,h})$.
 Then we have {$d_A = 1$},
 {$c_{A} = 1$} or
  {$c_{A} = c_{\widetilde{\mu}_{h} + 1}$}, and
\begin{align*}
\sum _{{\sigma} \in \D_{{\nu}_{_A}} \cap \fS_{\mu}} T_{\sigma}
= \TDIJ({\widetilde{\mu}_{h}},  {\widetilde{\mu}_{h} -  ({\mu}_{h} - 1)}).
\end{align*}
Let  {$X=({\mu} | E_{h, h+1})$},
 {$\xi = \ro(X)$}.
It is clear  {$c_X = c_{\widetilde{\mu}_{h}+1}$}, and  with $B=X$ the equation \eqref{eq_gB} becomes
\begin{align*}
	z=\sum_{k=1}^n z_k,
	\quad
	\mbox{with }
	 z_k&=
		\sum_{j=\AK(h+1,k)}^{\AK(h+1,k+1) - 1}
		x_{\xi}  c_{\widetilde{\mu}_{h}+1}
		T_{\widetilde{\mu}_{h}+1} T_{\widetilde{\mu}_{h}+2} \cdots T_{\widetilde{\mu}_{h}+j}
		{T_{d_A}} c_{A}
		\sum _{{\sigma} \in \D_{{\nu}_{_A}} \cap \fS_{\mu}} T_{\sigma} .
\end{align*}
Set {$C = {\mu}-E_{h+1, h+1} + E_{h, h+1} + E_{h+1, h}$}.
Direct calculation shows {$d_C = s_{\widetilde{\mu}_{h} + 1}$},
and
\begin{align*}
\sum _{{\sigma} \in \D_{{\nu}_{_C}} \cap \fS_{\mu}} T_{\sigma}
&=\TAIJ({\widetilde{\mu}_{h}+2} ,  {\widetilde{\mu}_{h}+ {\mu}_{h+1}}  	)
	 \TDIJ( {\widetilde{\mu}_{h}},  {\widetilde{\mu}_{h} -  ({\mu}_{h} - 1)} )
=	\TAIJ( {\widetilde{\mu}_{h}+2} ,  {\widetilde{\mu}_{h}+ {\mu}_{h+1}}  )
\sum _{{\sigma} \in \D_{{\nu}_{_A}} \cap \fS_{\mu}} T_{\sigma} .
\end{align*}
When  $k\neq h, h+1$, we have {$a_{h+1,k} = 0$} and $z_k=0$, hence
\begin{align*}
z=z_h+z_{h+1}.
\end{align*}
 Meanwhile
{$a_{h+1,h} = 1$} and hence
\begin{displaymath}
z_{h}
=		x_{\xi}  c_{\widetilde{\mu}_{h}+1}
		c_{A}
		\TDIJ( {\widetilde{\mu}_{h}} , {\widetilde{\mu}_{h} -  ({\mu}_{h} - 1)})
=\left\{
\begin{aligned}
&T_{(\mu|E_{h,h})}, \text{ if } A=(\mu+E_{h+1,h}|\mathrm{O}),\\
& -	\STEP{{\mu}_{h} + 1} T_{( {\mu}   +  E_{h,h}| \mathrm{O})},\text{ if }A=(\mu|E_{h+1,h}).
\end{aligned}
\right.
\end{displaymath}
When {$k=h+1$}, {$\AK(h+1,k) = 1, \AK(h+1,k+1) = \AK(h+1,k) + a_{h+1, h+1} = {\mu}_{h+1}+1$}.
Then
\begin{align*}
z_{h+1}
&=
		x_{\xi}  (T_{\widetilde{\mu}_{h}+1}  c_{\widetilde{\mu}_{h}+2} + ({q}-1)(c_{\widetilde{\mu}_{h}+1} -  c_{\widetilde{\mu}_{h}+2}))
		\TAIJ( {\widetilde{\mu}_{h}+2} , {\widetilde{\mu}_{h}+ {\mu}_{h+1}}  )
		 c_{A}
		\sum _{{\sigma} \in \D_{{\nu}_{_A}} \cap \fS_{\mu}} T_{\sigma} \\
&=
		x_{\xi}  T_{\widetilde{\mu}_{h}+1}  c_{\widetilde{\mu}_{h}+2}
		\TAIJ({\widetilde{\mu}_{h}+2} ,  {\widetilde{\mu}_{h}+ {\mu}_{h+1}}  )
		 c_{A}
		\sum _{{\sigma} \in \D_{{\nu}_{_A}} \cap \fS_{\mu}} T_{\sigma} \\
&\qquad
	+
		x_{\xi}  ({q}-1) c_{\widetilde{\mu}_{h}+1}
		\TAIJ({\widetilde{\mu}_{h}+2} ,  {\widetilde{\mu}_{h}+ {\mu}_{h+1}}  )
		 c_{A}
		\sum _{{\sigma} \in \D_{{\nu}_{_A}} \cap \fS_{\mu}} T_{\sigma} \\
&\qquad
		-
		x_{\xi} ({q}-1) c_{\widetilde{\mu}_{h}+2}
		\TAIJ({\widetilde{\mu}_{h}+2} , {\widetilde{\mu}_{h}+ {\mu}_{h+1}}  	)
		 c_{A}
		\sum _{{\sigma} \in \D_{{\nu}_{_A}} \cap \fS_{\mu}} T_{\sigma} .
\end{align*}
Observe that when $A=(\mu+E_{h+1,h}|\mathrm{O})$, we have {$c_{A} = 1$} and hence
\begin{align*}
z_{h+1}
&=
		T_{({\mu}-E_{h+1, h+1}   + E_{h+1, h}| E_{h, h+1})}
	+
		  ({q}-1) \STEP{ {\mu}_{h+1} } T_{ ({\mu}  | E_{h,h})}
		 \\
&\qquad
		-
		({q}-1) \STEP{{\mu}_{h}+1}
		T_{( {\mu}-E_{h+1, h+1}   + E_{h, h}   | E_{h+1,h+1})}		.
\end{align*}
When $A=(\mu|E_{h+1,h})$, we have {$c_{A} = c_{\widetilde{\mu}_{h}+1}$} and hence
\begin{align*}
z_{h+1}
&=
		-T_{( {\mu}-E_{h+1, h+1}   | E_{h, h+1} + E_{h+1, h} )}
		+
		({q}-1) 	T_{( {\mu}-E_{h+1, h+1}  | E_{h,h} + E_{h+1,h+1})}
		 \\
&\qquad
	 -  ({q}-1)  \STEP{ {\mu}_{h+1} }  \STEP{ {\mu}_{h} + 1 }
	 	T_{( {\mu}   + E_{h, h}  | O)}	 .
\end{align*}
Putting {$z_h$} and   {$z_{h+1}$} together, the proposition is proved.
\end{proof}

\begin{prop}\label{philower2}
	For any given {$h \in [1,n-1]$}, {$\mu \in \CMN(n,r-1)$},
	the following formulas hold
\begin{align*}
{\rm {(1)}} \quad
&\phi_{({\mu}| E_{h+1, h})}
\phi_{({\mu}+  E_{h, h+1}|O)}
=  \phi_{({\mu} - E_{h,h} + E_{h, h+1}| E_{h+1, h})}
+\phi_{( {\mu} | E_{h+1, h+1})}, \\
{\rm (2)} \quad
& \phi_{({\mu}| E_{h+1, h})}
\phi_{({\mu} | E_{h, h+1})}
= \phi_{({\mu} - E_{h,h} |   E_{h, h+1} +  E_{h+1, h})}
	- \STEP{ {\mu}_{h+1} + 1 } \phi_{( {\mu} + E_{h+1, h+1} | O)}.
\end{align*}
\end{prop}
\begin{proof}
Let $A=({\mu}+  E_{h, h+1}|\mathrm{O})$ or $A=({\mu}| E_{h, h+1})$
and {$X=({\mu} | E_{h+1, h})$}, {$\xi = \ro(X)$}.
Then 	 {$d_A = 1$} and {$c_X = c_{\widetilde{\mu}_{h} + 1}$}, and

\begin{align*}
z
&=
		x_{\xi}  c_{X}
		 \TDIJ( {\widetilde{\mu}_{h}} , {\widetilde{\mu}_{h-1}+1})
		{T_{d_A}} c_{A}
		\sum _{{\sigma} \in \D_{{\nu}_{_A}} \cap \fS_{\mu}} T_{\sigma} \\
&=
		x_{\xi}  c_{\widetilde{\mu}_{h}+1}
		(  \TDIJ({\widetilde{\mu}_{h}}, {\widetilde{\mu}_{h-1}+1}) - 1)
		c_A  \TAIJ( {\widetilde{\mu}_h + 1}, {\widetilde{\mu}_h +  \mu_{h+1}})
	+ x_{\xi}  c_{\widetilde{\mu}_h + 1} c_{A}
		\TAIJ( {\widetilde{\mu}_h + 1}, {\widetilde{\mu}_h +  \mu_{h+1}}) \\
&=
		x_{\xi}  c_{\widetilde{\mu}_{h}+1} 
		T_{\widetilde{\mu}_{h}}
		\TDIJ( {\widetilde{\mu}_h - 1} , {\widetilde{\mu}_{h-1}+1})
		c_A \TAIJ( {\widetilde{\mu}_h + 1}, {\widetilde{\mu}_h +  \mu_{h+1}})
	+ x_{\xi}  c_{\widetilde{\mu}_h + 1} c_{A}
		\TAIJ( {\widetilde{\mu}_h + 1}, {\widetilde{\mu}_h +  \mu_{h+1}}) \\
&=
	x_{\xi}
		T_{\widetilde{\mu}_{h}}  c_{\widetilde{\mu}_{h}} c_A
		\TDIJ( {\widetilde{\mu}_h - 1},  {\widetilde{\mu}_{h-1}+1})
		\TAIJ({\widetilde{\mu}_h + 1}, {\widetilde{\mu}_h +  \mu_{h+1}})
	+ x_{\xi}  c_{\widetilde{\mu}_h + 1} c_{A}
		\TAIJ( {\widetilde{\mu}_h + 1}, {\widetilde{\mu}_h +  \mu_{h+1}}) .
\end{align*}

When {$A = ({\mu}+  E_{h, h+1}|\mathrm{O}) $}, we have {$c_A = 1$}.
By Lemma \ref{xcT}{\rm (3)}, we have {$x_{\xi}  c_{\widetilde{\mu}_h + 1}
	\TAIJ( {\widetilde{\mu}_h + 1}, {\widetilde{\mu}_h +  \mu_{h+1}})
= c_{{q},  \widetilde{\mu}_h + 1, \widetilde{\mu}_{h+1} + 1}
$}
and
\begin{align*}
z
&=
		x_{\xi}
		T_{\widetilde{\mu}_{h}}  c_{\widetilde{\mu}_{h}}
		\TDIJ( {\widetilde{\mu}_h - 1} , {\widetilde{\mu}_{h-1}+1})
		\TAIJ({\widetilde{\mu}_h + 1}, {\widetilde{\mu}_h +  \mu_{h+1}})
	+  	x_{\xi}  c_{\widetilde{\mu}_h + 1}
	\TAIJ( {\widetilde{\mu}_h + 1}, {\widetilde{\mu}_h +  \mu_{h+1}}) \\
&=
		x_{\xi}
		T_{\widetilde{\mu}_{h}}  c_{\widetilde{\mu}_{h}}
		\TDIJ( {\widetilde{\mu}_h - 1} , {\widetilde{\mu}_{h-1}+1})
		\TAIJ({\widetilde{\mu}_h + 1}, {\widetilde{\mu}_h +  \mu_{h+1}})
	+ 	x_{\xi}  c_{{q},  \widetilde{\mu}_h + 1, \widetilde{\mu}_{h+1}+1 }  \\
&= T_{({\mu} - E_{h,h} + E_{h, h+1}| E_{h+1, h})}
+ 	T_{({\mu} | E_{h+1, h+1})}.
\end{align*}

If {$ A= ({\mu} | E_{h, h+1}) $}, we have {$c_A = c_{\widetilde{\mu}_h + 1}$}.
Let {$C = ({\mu} - E_{h,h} |   E_{h, h+1} +  E_{h+1, h})$},
direct calculation shows {$d_C = s_{\widetilde{\mu}_{h}}$},
and
\begin{align*}
\sum _{{\sigma} \in \D_{{\nu}_{_C}} \cap \fS_{\mu}} T_{\sigma}
&= \TDIJ( {\widetilde{\mu}_h - 1}  , {\widetilde{\mu}_{h-1} + 1} )
\sum _{{\sigma} \in \D_{{\nu}_{_A}} \cap \fS_{\mu}} T_{\sigma} .
\end{align*}
This means
\begin{align*}
z
&=
		x_{\xi}
		T_{\widetilde{\mu}_{h}}  c_{\widetilde{\mu}_{h}}  c_{\widetilde{\mu}_h + 1}
		\TDIJ( {\widetilde{\mu}_h - 1}, {\widetilde{\mu}_{h-1}+1})
		\TAIJ({\widetilde{\mu}_h + 1} , {\widetilde{\mu}_h +  \mu_{h+1}})
+ 	x_{\xi}  c_{\widetilde{\mu}_h + 1} c_{\widetilde{\mu}_h + 1}
		\TAIJ( {\widetilde{\mu}_h + 1}, {\widetilde{\mu}_h +  \mu_{h+1}}) \\
&=
		x_{\xi}
		T_{\widetilde{\mu}_{h}}  c_{\widetilde{\mu}_{h}}  c_{\widetilde{\mu}_h + 1}
		\TDIJ( {\widetilde{\mu}_h - 1}, {\widetilde{\mu}_{h-1}+1})
		\TAIJ({\widetilde{\mu}_h + 1} , {\widetilde{\mu}_h +  \mu_{h+1}})
 			- x_{\xi}
		\TAIJ( {\widetilde{\mu}_h + 1}, {\widetilde{\mu}_h +  \mu_{h+1}}) \\
&=
		x_{\xi}
		T_{d_C}  c_{\widetilde{\mu}_{h}}  c_{\widetilde{\mu}_h + 1}
		\sum _{{\sigma} \in \D_{{\nu}_{_C}} \cap \fS_{\mu}} T_{\sigma}
		- \STEP{ {\mu}_{h+1} + 1 } x_{\xi} \\
&= T_{({\mu} - E_{h,h} |   E_{h, h+1} +  E_{h+1, h})}
 - \STEP{ {\mu}_{h+1} + 1 } T_{( {\mu} + E_{h+1, h+1} | O)}.
\end{align*}
Then the proposition is proved.
\end{proof}
\begin{rems}\label{raw}
(1) Using Proposition \ref{prop_PhiAPhiB}, one may write down the multiplication formulas for $\Phi_X\Phi_A$, where $X$ is one of the six matrices given in \eqref{even-odd}, 
(together with the required SDP conditions)
 in the twisted queer $q$-Schur superalgebra $\SQvnrR$.
 
 (2) The multiplication formulas in Propositions \ref{phiupper0}, \ref{phidiag1}, \ref{phiupper1} and \ref{philower1} 
are the ``raw'' formulas which are the counterpart of \cite[Lemma 3.2]{BLM} or \cite[Lemma 3.1]{DG}. 
 By normalising the $\phi$-basis, 
a normalised version of these formulas using symmetric Gaussian polynomials in \cite[Lemma 3.4 and \S4.6]{BLM}
 or \cite[Proposition 4.4 and 4.5]{DG} are used in the construction of the BLM type realizations. 
 However, an inspection on the coefficients of these multiplication formulas in Sections 4 and 5 shows that,
for the candidate basis $\{[A]\mid A \in \MNZ(n,r)\}$, where 
{$[A]:={v}^{-l(w^+_{\widehat A}) + l(w_{0,\ro(\widehat A)})}\Phi_A$}, 
 normalized versions of these formulas are unlikely to be derived,
 \footnote{Here $w^+_{\widehat A}$ is the longest elements of the double coset 
	associated with $\widehat A$ (see \cite[Lemma 13.10]{DDPW} for a geometric interpretation of $-w^+_{\widehat A}+w_{0,ro(\widehat A)}$.}
  since the coefficients involve not just the entries of $\widehat A$.
Fortunately,  
we are able to present a new BLM type realization for $\Uvqn$ 
by simply using these raw multiplication formulas. 
This requires certain adjustments on generators; 
see footnotes \ref{raw1} and \ref{raw2} below.
\end{rems}

\spaceintv
\section{Preliminaries for building the defining relations in $\bsSQvnr$}\label{sec_spanningsets}
In this section, we shall introduce a family of elements in $\SQvnrR$
 and derive multiplication formulas among them using the basic formulas established in Sections 4 and 5. 
These formulas serve as an important step to build the defining relations (QQ1)--(QQ6) for $\Uvqn$ in $\bsSQvnr$. 
This eventually allows us to find a new construction for the quantum queer supergroup $\Uvqn$.

Recall that   {${v}$}, {${q}$} are indeterminates.
For {$R = \Qv$},   and  {${q} = {v}^2 $}, let
\begin{align*}
\bsSQvnr = {{\widetilde{\mathcal{Q}}_{\lcase{q}}(\lcase{n},\lcase{r}; {\Qv})}}.
\end{align*}
Recall the notations defined in \eqref{def_ahk_notations}.
Set
{$ \MNZNS(n)=\{ A=(\SEE{a}_{i,j}|\SOE{a}_{i,j}) \in \MNZN(n)   \where  \SEE{a}_{i,i} = 0 \mbox{, for all } i\}$}.
For any   {$\bs{j} \in {\ZZ}^{n}$} and {$\lambda \in \CMN(n,r) $}, set {$\snorm{\bs{j}} = \sum_{k} j_{k}$} and
 {$ {\lambda} \cdot {\bs{j}} = {\bs{j}}\cdot{\lambda} = \sum_{k=1}^n {{\lambda}_k} {{j}_k}$}.

Similar to \cite{BLM},
for  given {$r, n \in \NN^+$}
and arbitrary  {$A=(\SUP{A}) \in \MNZNS(n)$}, {$\bs{j} \in {\ZZ}^{n} $},
we define the following elements\footnote{\label{raw1}Unlike the case for $\mathfrak{gl}_n$ or $\mathfrak{gl}_{m|n}$, the basis used in the definition below is not normalized. This eventually requires to modify the definition for the generators in \eqref{raw generator}.} in {$\SQvnrR$}:
\begin{equation}\label{def_ajr}
\begin{aligned}
	\AJRS(A, \bs{j}, r) =
\left\{
\begin{aligned}
	& \sum_{\substack{\lambda \in \CMN(n, r-\snorm{A})} }
	 {v}^{\lambda \cdot \bs{j}} {\Phi}_{( \SE{A} + \lambda | \SO{A} )},
		\ &\mbox{if } \snorm{A} \le r; \\
	&0, &\mbox{otherwise}.
\end{aligned}
\right.
\end{aligned}
\end{equation}

By \eqref{def_ajr}, the following holds for {$B=( { \SUP{B}} ) \in \MNZNS(n)$},
{$\bs{j} \in {\ZZ}^{n} $},
{$h \in [1, n]$}:
\begin{align}\label{eq simple diag}
\sum_{\substack{ \lambda \in \CMN(n, r - \snorm{B} + 1) \\
		\lambda_h \ge 1
} }
		{v}^{\lambda \cdot {\bs{j}}}
			{\Phi}_{(\SE{B} + \lambda -E_{h,h}|\SO{B})}
	={v}^{j_h}\ABJRS( \SE{B}, \SO{B}, \bs{j},  r ).
\end{align}


\begin{lem}\label{formstepodd}
Assume {$B=( { \SUP{B}} ) \in \MNZNS(n)$},
{$\bs{j} \in {\ZZ}^{n} $},
{$h \in [1, n]$},
the following holds in {${\SQvnrR}$}:
\begin{align*}
{\rm(1)}\qquad  &\sum_{\substack{  \lambda \in \CMN(n,r-\snorm{B}-2) } }
	{v}^{\lambda \cdot {\bs{j}}}
	(\STEPP{{\lambda}_h + 2 } - {\STEP{{\lambda}_h + 2 }} )
	{\Phi}_{(\SE{B} + \lambda  + 2E_{h,h} |  \SO{B}  )} \\
&=
	\frac{1}{({q}^2 - 1){v}^{2 j_{h}}} \{
	\ABJRS( \SE{B},  \SO{B} , \bs{j} + 4 \bs{\ep}_{h},  r )
	 -
	{({q} + 1)}
	\ABJRS( \SE{B},  \SO{B} , \bs{j} + 2 \bs{\ep}_{h},  r )
	+
	{q}
	\ABJRS( \SE{B},  \SO{B} , \bs{j},  r ) \} .\\
{\rm(2)}\qquad 	&	 \sum_{\substack{  \lambda \in \CMN(n,r-\snorm{B}-1) } }
	{v}^{\lambda \cdot {\bs{j}}}
	\STEP{ {\lambda}_{h}  + 1}
	{\Phi}_{(\SE{B} + \lambda    + E_{h,h} | \SO{B} )}\\
&= \frac{1}{({q} - 1) {v}^{j_{h}}}
	( \ABJRS(\SE{B}    ,  \SO{B} , {\bs{j}} + 2 \bs{\ep}_{h}, r )
	-
	\ABJRS( \SE{B}    ,  \SO{B} , {\bs{j}} , r ) ). \\
{\rm(3)}\qquad 	&	\sum_{\substack{ \lambda \in \CMN(n,r-\snorm{B} - 1) } }
		 {v}^{\lambda \cdot {\bs{j}}}
		{\STEPP{ {\lambda}_h + 1} }
		{\Phi}_{(\SE{B} + \lambda+  E_{h,h}| \SO{B} )}\\
 &=
		\frac{1}{({q}^2 - 1) {v}^{j_h}}
		\{ \ABJRS( \SE{B}, \SO{B} , \bs{j} +4 \bs{\ep}_{h},  r )
		- \ABJRS( \SE{B}, \SO{B} , \bs{j},  r ) \}  .
\end{align*}
\end{lem}
\begin{proof}
Equation {\rm(1)} can be proved by direct calculation:
\begin{align*}
&\sum_{\substack{  \lambda \in \CMN(n,r-\snorm{B}-2) } }
	{v}^{\lambda \cdot {\bs{j}}}
	(\STEPP{{\lambda}_h + 2 } - {\STEP{{\lambda}_h + 2 }} )
	{\Phi}_{(\SE{B} + \lambda  + 2E_{h,h} |  \SO{B}  )} \\
&=
	\sum_{\substack{  \mu \in \CMN(n,r-\snorm{B})  \\ {\mu}_{h} \ge 2}}
	{v}^{(\mu - 2\bs{\ep}_{h}) \cdot {\bs{j}}}
	( \frac{{q}^{2 {\mu}_{h} } - 1}{{q}^2 - 1} -   \frac{{q}^{ {\mu}_{h}} - 1}{{q} - 1})
	{\Phi}_{(\SE{B} + \mu |  \SO{B}   )} \\
&=
	\sum_{\substack{  \mu \in \CMN(n,r-\snorm{B})  } }
	{v}^{\mu  \cdot {\bs{j}}}
	 \cdot \frac{{q}^{2 {\mu}_{h} }  - {q}^{ {\mu}_{h} +1} + {q} - {q}^{ {\mu}_{h}} }{({q}^2 - 1){v}^{2 j_{h}}}
	\cdot {\Phi}_{(\SE{B} + \mu |  \SO{B}   )} \\
&=
	\frac{1}{({q}^2 - 1){v}^{2 j_{h}}} \{
	\ABJRS( \SE{B},  \SO{B} , \bs{j} + 4 \bs{\ep}_{h},  r )
	  -
	{({q} + 1)}
	\ABJRS( \SE{B},  \SO{B} , \bs{j} + 2 \bs{\ep}_{h},  r )
	  +
	{q}
	\ABJRS( \SE{B},  \SO{B} , \bs{j},  r ) \},
\end{align*}
where for the first equality
we set $ \mu = \lambda  + 2 \bs{\ep}_{h}$
and the second equality is due to
the fact ${q}^{2 {\mu}_{h} }  - {q}^{ {\mu}_{h} +1} + {q} - {q}^{ {\mu}_{h}}  = 0$ either $\mu_h=0$ or $1$.
	Equations (2) and {\rm(3)} can be proved similarly and we leave details to the reader.
\end{proof}

We now derive the multiplication formulas for these ``long elements'' $A(\bs{j},r)$. We first look at the even case.

\begin{prop}\label{mulformzero}
Let 
{$h \in [1,n]$} and $\bs{j},\bs{k}\in {\ZZ}^n$. For {$O:=(\mathrm  O|\mathrm{O}),A \in  \MNZNS(n)$},
the following holds in {${\SQvnrR}$} for all  {$r\geq\snorm{A}$}:
\begin{itemize}
    \item[{\rm(1)}]$ \AJRS({O}, \bs{k}, r) \cdot \AJRS(A, \bs{j}, r)
		 = {v}^{\sum_{h=1}^{n}   \sum_{u=1}^n a_{h,u} k_h}   \AJRS(A, \bs{k} + \bs{j}, r);$
   \item[ {\rm(2)}] $\AJRS(A, \bs{j}, r)  \cdot  \AJRS({O}, \bs{k}, r)
	={v}^{\sum_{h=1}^{n} \sum_{u=1}^n a_{u,h} k_h  }   \AJRS(A, \bs{k} + \bs{j}, r).$
\end{itemize}
In particular, $\AJRS({O}, \bs{0}, r) \cdot  \AJRS(A, \bs{j}, r)
		=  \AJRS(A, \bs{j}, r) \cdot \AJRS({O}, \bs{0}, r)
		=  \AJRS(A, \bs{j}, r). $
\end{prop}
\begin{proof}
%
A direct calculation using \eqref{def_ajr} and Proposition \ref{phiupper0}(1) gives rise to
\begin{align*}
	\quad  \AJRS({O}, \bs{k}, r) \cdot \AJRS(A, \bs{j}, r)
	& =\sum_{\substack{ \lambda \in \CMN(n, r-\snorm{A}) } }
		{v}^{{(\ro(A + \lambda))} \cdot {\bs{k}} }
		{v}^{\lambda \cdot {\bs{j}} } {\Phi}_{( \SE{A} + \lambda | \SO{A})}
	= {v}^{\sum_{h=1}^{n} \sum_{u=1}^n a_{h,u} k_h  }   \AJRS(A, \bs{k} + \bs{j}, r).
\end{align*}
This proves {\rm(1)}. Similarly, Equation {\rm{(2)}} holds.
\end{proof}
As a consequence of  Proposition \ref{mulformzero}, we have the following corollary.
\begin{cor}\label{mulformzerocor}
	Let 
	{$h \in [1,n]$}.
	For any $A \in \MNZNS(n)$ and $\bs{j}\in {\ZZ}^n$, the following multiplication formulas hold in $\SQvnrR$ for all  {$r\geq\snorm{A} $}
\begin{itemize}
\item[{\rm(1)}] $ \AJRS({O}, \pm \bs{\ep}_h, r) \cdot \AJRS(A, \bs{j}, r)
	= {v}^{ \pm  \sum_{u=1}^n a_{h,u} }   \AJRS(A, \bs{j} \pm\bs{\ep}_h, r), $
\item[{\rm(2)}] $ \AJRS(A, \bs{j}, r) \cdot  \AJRS({O}, \pm \bs{\ep}_h, r)
	= {v}^{ \pm  \sum_{u=1}^n a_{u, h} }   \AJRS(A, \bs{j} \pm\bs{\ep}_h, r).$
\end{itemize}
\end{cor}

\begin{conv}\label{rem_ajr_short_0}
Like Cnnvention \ref{CONV}, we use the convention {$\AJRS(  {A},  \bs{ j }, r ) = 0$} in {$\SQvnrR$} below for any {$\bs{j} \in \ZZ^n$}
if {$A \notin \MNZNS(n)$} (or {$\snorm{A}> r$}).
\end{conv}

Recall the notation {${\STEPX{m}{y, z}} ={\STEPX{m}{y}} - {\STEPX{m}{z}}$} in \eqref{stepd}.
\begin{prop}\label{mulformeven}
	Let 
	{$h \in [1,n-1]$}.
	For any {$A \in \MNZNS(n)$},
	the following multiplication formulas hold in {${\SQvnrR}$} for all  {$r\geq\snorm{A}$}:
\begin{align*}
{\rm(1)} \quad & \ABJRS( E_{h, h+1}, \mathrm{O}, \bs{ 0 },  r ) \cdot \ABJRS( \SE{A}, \SO{A}, \bs{j},  r ) \\
&=
	\sum_{k<h}
	{v}^{ 2 {\BK(h,k)}  +2  \SOE{a}_{h+1,k}}   \VSTEP{ \SEE{a}_{h,k} + 1}
	\ABJRS( \SE{A}- E_{h+1, k} + E_{h,k}, \SO{A}, \bs{j} + 2 \bs{\ep}_{h},  r ) \\
	&\qquad +
	 \frac{{v}^{ 2 {\BK(h,h)}  + 2 \SOE{a}_{h+1,h} - j_h }}{{v}^2 - 1} \Big\{
	\ABJRS( \SE{A}  - E_{h+1, h}, \SO{A}, \bs{j} + 2 \bs{\ep}_{h},  r )
	- \ABJRS( \SE{A}  - E_{h+1, h}, \SO{A}, \bs{j},  r )
	\Big\}\\
	& \qquad +
	{v}^{ 2 {\BK(h,h+1)} + 2 \SOE{a}_{h+1,h+1} + {j}_{h+1}}  \VSTEP{ \SEE{a}_{h,h+1} + 1}
	\ABJRS( \SE{A} + E_{h,h+1}, \SO{A}, \bs{j},  r )  \\
	& \qquad +
	\sum_{k>h+1}
	{v}^{ 2 {\BK(h,k) }  + 2 \SOE{a}_{h+1,k}}  \VSTEP{ \SEE{a}_{h,k} + 1}
	\ABJRS( \SE{A} - E_{h+1, k} + E_{h,k}, \SO{A}, \bs{j},  r )  \\
	& \qquad + \sum_{k<h}
	{v}^{2 {\BK(h,k)} }
	\ABJRS( \SE{A}, \SO{A} - E_{h+1, k}  + E_{h,k}, \bs{j} + 2 \bs{\ep}_{h},  r )  \\
	& \qquad +
	\sum_{k \ge h}
	{v}^{ 2 {\BK(h,k)} }
	\ABJRS( \SE{A}, \SO{A} - E_{h+1, k}  + E_{h,k}, \bs{j},  r ) \\
	&\qquad + \sum_{k<h}
	{v}^{ 2 {\BK(h,k)} - 2} \VSTEPPD{{a}_{h,k}+1}
	\ABJRS( \SE{A} + 2E_{h,k}, \SO{A}  -E_{h,k} - E_{h+1,k}, \bs{j} + 2 \bs{\ep}_{h},  r ) \\
	& \qquad +
	\frac{{v}^{ 2 {\BK(h,h)} - 2 j_h - 2}}{{v}^4 - 1}
	 \Big\{
	\ABJRS( \SE{A}, \SO{A}  -E_{h,h} - E_{h+1,h}, \bs{j} + 4 \bs{\ep}_{h},  r )  -  ({v}^2 + 1)\\
	&\qquad \qquad\cdot
	\ABJRS( \SE{A}, \SO{A}  -E_{h,h} - E_{h+1,h}, \bs{j} + 2 \bs{\ep}_{h},  r ) + {v}^2 \ABJRS( \SE{A}, \SO{A}  -E_{h,h} - E_{h+1,h}, \bs{j},  r )
	\Big\} \\
	& \qquad +
	\sum_{k>h}
	{v}^{ 2 {\BK(h,k)} - 2} \VSTEPPD{{a}_{h,k}+1}
	\ABJRS( \SE{A} + 2E_{h,k}, \SO{A}  -E_{h,k} - E_{h+1,k}, \bs{j},  r ) ; \\
{\rm(2)} \quad & \ABJRS( E_{h+1, h}, \mathrm{O}, \bs{ 0 },  r ) \cdot \ABJRS( \SE{A}, \SO{A}, \bs{j},  r ) \\
&=
	\sum_{k < h}
	{v}^{2 \AK(h+1,k)}
	\VSTEP{ \SEE{a}_{h+1, k} +1}
	\ABJRS( \SE{A} - E_{h,k} + E_{h+1, k}, \SO{A}, \bs{j},  r ) \\
	&\qquad +
	{v}^{2 \AK(h+1,h)+ j_h}
	\VSTEP{ \SEE{a}_{h+1, h} +1}
	\ABJRS( \SE{A} + E_{h+1, h}, \SO{A}, \bs{j},  r ) \\
		&\qquad +
	\frac{{v}^{2 \AK(h+1,h+1) - j_{h+1}}}{ {v}^2 - 1 }\Big\{
	\ABJRS( \SE{A} - E_{h,h+1}, \SO{A}, \bs{j} + 2 \bs{\ep}_{h+1},  r )
	 -
	\ABJRS( \SE{A} - E_{h,h+1}, \SO{A}, \bs{j},  r ) \Big \} \\
	&\qquad +
	\sum_{k > h+1}
	{v}^{ 2 \AK(h+1,k)}
	\VSTEP{ \SEE{a}_{h+1, k} +1}
	\ABJRS( \SE{A} - E_{h,k} + E_{h+1, k}, \SO{A},  \bs{j} + 2 \bs{\ep}_{h+1}, r ) \\
	& \qquad + \sum_{k<h}
	{v}^{2 \AK(h+1,k) + 2 {a}_{h, k} - 2}
	\ABJRS(  \SE{A}, \SO{A} - E_{h,k} + E_{h+1,k} , \bs{j},  r ) \\
	& \qquad +
	{v}^{2 \AK(h+1,h)  }
	\ABJRS(\SE{A}, \SO{A} - E_{h,h} + E_{h+1,h},  {\bs{j}} + 2 \bs{\ep}_{h}, r) \\
	& \qquad +
	{v}^{ 2 \AK(h+1,h+1) + 2 {a}_{h, h+1} - 2}
	\ABJRS(\SE{A}, \SO{A} - E_{h,h+1} + E_{h+1,h+1}, \bs{j}, r) \\
	& \qquad +
	\sum_{k>h+1}
	{v}^{ 2 \AK(h+1,k) + 2 {a}_{h, k} - 2}
	\ABJRS(\SE{A}, \SO{A} - E_{h,k} + E_{h+1,k}, {\bs{j}} + 2 \bs{\ep}_{h+1},  r ) \\
	&\qquad - \sum_{k < h }
	{v}^{2 \AK(h+1,k) + 2 {a}_{h, k} - 4}
	\VSTEPPDR{  {a}_{h+1, k} +1}
	\ABJRS( \SE{A} + 2E_{h+1,k}, \SO{A}  - E_{h,k} - E_{h+1,k}, \bs{j}, r ) \\
	& \qquad -
	{v}^{2 \AK(h+1,h) - 2}
	\VSTEPPDR{  {a}_{h+1, h} +1}
	\ABJRS( \SE{A}  + 2E_{h+1,h}, \SO{A}  - E_{h,h} - E_{h+1,h}, \bs{j} + 2 \bs{\ep}_{h},  r  ) \\
	& \qquad -
	\frac{{v}^{ 2 \AK(h+1,h+1) + 2 {a}_{h, h+1} - 2 j_{h+1}- 4}}{ {v}^4 - 1 }
	\Big \{
	  \ABJRS( \SE{A}, \SO{A}  - E_{h,h+1} - E_{h+1,h+1}, \bs{j} +  4 \bs{\ep}_{h+1},  r ) \\
	& \qquad  \qquad \qquad -
	({v}^2 + 1) \ABJRS( \SE{A}, \SO{A}  - E_{h,h+1} - E_{h+1,h+1}, \bs{j}+ 2 \bs{\ep}_{h+1},  r ) \\
	& \qquad \qquad \qquad +
	{v}^2 \ABJRS( \SE{A}, \SO{A}  - E_{h,h+1} - E_{h+1,h+1}, \bs{j},  r )
	 \Big\} \\
	& \qquad -
	\sum_{k>h+1}
	{v}^{ 2 \AK(h+1,k) + 2 {a}_{h, k} - 4}
	\VSTEPPDR{  {a}_{h+1, k} +1}
	\ABJRS( \SE{A} + 2E_{h+1,k}, \SO{A}  - E_{h,k} - E_{h+1,k}, \bs{j} + 2 \bs{\ep}_{h+1}, r ).
\end{align*}
\end{prop}

\begin{proof}
Observe that for any {$\mu \in \CMN(n, r-1)$},
{${(-1)}^{\parity{\mu+ E_{h,h+1}|O} \cdot \parity{A}}  = {(-1)}^{\parity{\mu+ E_{h+1,h}|O} \cdot \parity{A}}  = 1$}.
We only prove (1);  the proof of
{\rm(2)} is similar.

By  Proposition \ref{prop_PhiAPhiB},  \eqref{def_ajr} and Proposition \ref{phiupper0}(2), we have
\begin{align*}
& \ABJRS( E_{h, h+1}, \mathrm{O}, \bs{ 0 },  r ) \cdot \ABJRS( \SE{A}, \SO{A}, \bs{j},  r ) \\
& =
	\sum_{\substack{
		\lambda \in \CMN(n,r-\snorm{A}) \\
	} }
	{v}^{\lambda \cdot {\bs{j}}}
	{\Phi}_{(  \ro(A+\lambda)  - E_{h+1, h+1} + E_{h, h+1} |O )} {\Phi}_{( \SE{A} + \lambda | \SO{A} )}= {\fcY}_1 + {\fcY}_2 + {\fcY}_3,
\end{align*}
where
\begin{align*}
{\fcY}_1 &=
	\sum_{\substack{
		\lambda \in \CMN(n,r-\snorm{A}) \\
	} }
	{v}^{\lambda \cdot {\bs{j}}}
	\sum_{k=1}^n
	{q}^{ {\sum^{n}_{u=k+1}{(A+\lambda)}_{h, u}}  + \SOE{a}_{h+1,k}}  \STEP{ \SEE{(A+\lambda)}_{h,k} + 1}
		{\Phi}_{(\SE{A} + \lambda  - E_{h+1, k} + E_{h,k} | \SO{A} )}  , \\
{\fcY}_2 &=
	\sum_{\substack{
		\lambda \in \CMN(n,r-\snorm{A}) \\
	} }
	{v}^{\lambda \cdot {\bs{j}}}
	\sum_{k=1}^n
	{q}^{ {\sum^{n}_{u=k+1}{(A+\lambda)}_{h, u}} } {\Phi}_{(\SE{A} + \lambda   | \SO{A} - E_{h+1, k}  + E_{h,k} )} , \\
{\fcY}_3 &=
	\sum_{\substack{
		\lambda \in \CMN(n,r-\snorm{A}) \\
	} }
	{v}^{\lambda \cdot {\bs{j}}}
	\sum_{k=1}^n
	{q}^{ {\sum^{n}_{u=k+1}{(A+\lambda)}_{h, u}} -1} \\
	& \qquad \qquad \cdot \STEPPD{{(A+\lambda)}_{h,k}+1}
	{\Phi}_{(\SE{A} + \lambda  + 2E_{h,k} | \SO{A}  -E_{h,k} - E_{h+1,k} )}.
\end{align*}
Notice that
\begin{align*}
{\sum^{n}_{u=k+1}{(A+\lambda)}_{h, u}}
=
\left\{
\begin{aligned}
& {\BK(h,k)}  + \lambda_h  , & \mbox{ if } k< h, \\
& {\BK(h,k)}  , & \mbox{ if } k \ge h.
\end{aligned}
\right.
\end{align*}
Then we have
\begin{align*}
{\fcY}_1
& =
	\sum_{k<h}
	{q}^{ {\BK(h,k)} + \SOE{a}_{h+1,k}} \STEP{ \SEE{a}_{h,k} + 1}
	\sum_{\substack{  \lambda \in \CMN(n,r-\snorm{A}) } }
	{v}^{\lambda \cdot {\bs{j}}}
	{q}^{ {\lambda}_h }
	{\Phi}_{(\SE{A} + \lambda  - E_{h+1, k} + E_{h,k} | \SO{A} )}  \\
	& \qquad +
	{q}^{ {\BK(h,h)}   + \SOE{a}_{h+1,h}}
	\sum_{\substack{  \lambda \in \CMN(n,r-\snorm{A}) } }
	{v}^{\lambda \cdot {\bs{j}}}
	\STEP{ {\lambda}_{h} + 1}
	{\Phi}_{(\SE{A} + \lambda  - E_{h+1, h} + E_{h,h} | \SO{A} )}  \\
	& \qquad +
	{q}^{ {\BK(h,h+1)}  + \SOE{a}_{h+1,h+1}}  \STEP{ \SEE{a}_{h,h+1} + 1}
	\sum_{\substack{  \lambda \in \CMN(n,r-\snorm{A}) } }
	{v}^{\lambda \cdot {\bs{j}}}
	{\Phi}_{(\SE{A} + \lambda  - E_{h+1, h+1} + E_{h,h+1} | \SO{A} )}  \\
	& \qquad +
	\sum_{k>h+1}
	{q}^{ {\BK(h,k)}  + \SOE{a}_{h+1,k}}  \STEP{ \SEE{a}_{h,k} + 1}
	\sum_{\substack{  \lambda \in \CMN(n,r-\snorm{A}) } }
	{v}^{\lambda \cdot {\bs{j}}}
	{\Phi}_{(\SE{A} + \lambda  - E_{h+1, k} + E_{h,k} | \SO{A} )} \\
& =
	\sum_{k<h}
	{q}^{ {\BK(h,k)}  + \SOE{a}_{h+1,k}}  \STEP{ \SEE{a}_{h,k} + 1}
	\ABJRS( \SE{A}- E_{h+1, k} + E_{h,k}, \SO{A}, \bs{j} + 2 \bs{\ep}_{h},  r ) \\
	& \qquad +
	\frac{{q}^{ {\BK(h,h)} + \SOE{a}_{h+1,h}} }{({q} - 1) {v}^{j_h}}
 		( \ABJRS(\SE{A}  - E_{h+1, h} ,  \SO{A} , {\bs{j}} + 2 \bs{\ep}_{h}, r )
 		-  \ABJRS( \SE{A}  - E_{h+1, h} ,  \SO{A} , {\bs{j}} , r ) )  \\
	& \qquad +
	{q}^{ {\BK(h,h+1)}   + \SOE{a}_{h+1,h+1}}   {v}^{{j}_{h+1}}  \STEP{ \SEE{a}_{h,h+1} + 1}
	\ABJRS( \SE{A}  + E_{h,h+1}, \SO{A}, \bs{j},  r )  \\
	& \qquad +
	\sum_{k>h+1}
	{q}^{ {\BK(h,k)}  + \SOE{a}_{h+1,k}}  \STEP{ \SEE{a}_{h,k} + 1}
	\ABJRS( \SE{A} - E_{h+1, k} + E_{h,k}, \SO{A}, \bs{j},  r ) ,
\end{align*}
where the last equality is due to \eqref{def_ajr}, equation \eqref{eq simple diag} and Lemma \ref{formstepodd}{\rm(2)}.
Similarly, by \eqref{def_ajr}, we have
\begin{align*}
{\fcY}_2
& =
	\sum_{\substack{k<h\\  \lambda \in \CMN(n,r-\snorm{A}) } }
	{v}^{\lambda \cdot {\bs{j}}}
	{q}^{{\lambda}_h + {\BK(h,k)} } {\Phi}_{(\SE{A} + \lambda   | \SO{A} - E_{h+1, k}  + E_{h,k} )}  +
	\sum_{\substack{k \ge h\\  \lambda \in \CMN(n,r-\snorm{A}) } }
	{v}^{\lambda \cdot {\bs{j}}}
	{q}^{ {\BK(h,k)} } {\Phi}_{(\SE{A} + \lambda   | \SO{A} - E_{h+1, k}  + E_{h,k} )}   \\
& =
	\sum_{k<h}
	{q}^{{\BK(h,k)} }
	\ABJRS( \SE{A}, \SO{A} - E_{h+1, k}  + E_{h,k}, \bs{j} + 2 \bs{\ep}_{h},  r )  +
	\sum_{k \ge h}
	{q}^{ {\BK(h,k)} }
	\ABJRS( \SE{A}, \SO{A} - E_{h+1, k}  + E_{h,k}, \bs{j},  r ).
\end{align*}
Finally, by \eqref{def_ajr} and Lemma \ref{formstepodd}(1), one can obtain
\begin{align*}
{\fcY}_3
& =
	\sum_{k<h}
	{q}^{{\BK(h,k)} -1} \STEPPD{{a}_{h,k}+1}
	\ABJRS( \SE{A} + 2E_{h,k}, \SO{A}  -E_{h,k} - E_{h+1,k}, \bs{j} + 2 \bs{\ep}_{h},  r ) \\
	& \qquad +
	\frac{{q}^{ {\BK(h,h)} -1}}{({q}^2 - 1) {v}^{2 j_h}} \cdot  \{ \ 
	\ABJRS( \SE{A}, \SO{A}  -E_{h,h} - E_{h+1,h}, \bs{j} + 4 \bs{\ep}_{h},  r ) \\
	&\qquad \qquad -
	{({q} + 1)}
	\ABJRS( \SE{A}, \SO{A}  -E_{h,h} - E_{h+1,h}, \bs{j} + 2 \bs{\ep}_{h},  r ) \\
	&\qquad \qquad  +
	{q}
	\ABJRS( \SE{A}, \SO{A}  -E_{h,h} - E_{h+1,h}, \bs{j},  r ) \  \}  \\
	& \qquad +
	\sum_{k>h}
	{q}^{ {\BK(h,k)} -1}  \STEPPD{{a}_{h,k}+1}
	\ABJRS( \SE{A} + 2E_{h,k}, \SO{A}  -E_{h,k} - E_{h+1,k}, \bs{j},  r ).
\end{align*}
Putting the computations for {${\fcY}_1$},  {${\fcY}_2$},  {${\fcY}_3$} together and replacing {${q}$} with {${v}^2$},
we proved part {\rm(1)} of the proposition.
%
\end{proof}


For any  {$A \in \MNZN(n)$} and   {$\lambda \in \CMN(n, r-\snorm{A})$},
set  {$(A+\lambda)  = (\SE{A} + \lambda | \SO{A}) \in \MNZN(n)$}.

\begin{prop}\label{mulformdiag}
Let 
{$h \in [1,n]$} and {$A \in \MNZNS(n)$}.
Assume for any  {$\lambda \in \CMN(n, r-\snorm{A})$},
	 {$(A+\lambda)$}  satisfies the SDP condition on the $h$-th row.
Then the following multiplication formula holds in {${\SQvnrR}$} for all  {$r\geq\snorm{A} $}:

\begin{align*}
&\ABJRS( \mathrm{O}, E_{h,h}, \bs{ 0 },  r ) \cdot \ABJRS( \SE{A}, \SO{A}, \bs{j},  r )  \\
&	 =
	 \sum_{ k < h
	}
		{(-1)}^{ {{\SOE{\widetilde{a}}}_{h-1,k}} + \parity{A} }
		{v}^{ 2 \BK(h,k)}
		\ABJRS( \SE{A} -E_{h,k}, \SO{A}+ E_{h,k}, \bs{j} + 2 \bs{\ep}_{h},  r ) \\
		& \qquad +
		{(-1)}^{ {{\SOE{\widetilde{a}}}_{h-1,h}} + \parity{A}}
		{v}^{2 \BK(h,h)+ j_h}
		\ABJRS( \SE{A}, \SO{A}+ E_{h,h}, \bs{j},  r ) \\
		& \qquad +
		\sum_{ k > h  
		}
		{(-1)}^{ {{\SOE{\widetilde{a}}}_{h-1,k}} + \parity{A}}
		{v}^{2 \BK(h,k)}
		\ABJRS( \SE{A}  -E_{h,k}, \SO{A}+ E_{h,k}, \bs{j},  r ) \\
		& \qquad + \sum_{ k<h
		}
		{(-1)}^{ {{\SOE{\widetilde{a}}}_{h-1,k}} + \parity{A} +1 }
		{v}^{ 2 \BK(h,k)}
		{\VSTEPP{{a}_{h,k}}}
		\ABJRS( \SE{A} +  E_{h,k}, \SO{A} - E_{h,k}, \bs{j} + 2 \bs{\ep}_{h},  r )  \\
		&\qquad+
		{(-1)}^{ {{\SOE{\widetilde{a}}}_{h-1,h}} + \parity{A} +1 }
		\frac{{v}^{2 \BK(h,h)- j_h} }{ {v}^4 - 1 }
		\Big\{  \ABJRS( \SE{A}, \SO{A} - E_{h,h}, \bs{j} +4 \bs{\ep}_{h},  r )
		 - \ABJRS( \SE{A}, \SO{A} - E_{h,h}, \bs{j},  r )
		 \Big \}\\
		& \qquad +
		\sum_{ k>h 
		}
		{(-1)}^{ {{\SOE{\widetilde{a}}}_{h-1,k}} + \parity{A}+1 }
		{v}^{2 \BK(h,k)}
		{\VSTEPP{{a}_{h,k}}}
		\ABJRS( \SE{A} +  E_{h,k}, \SO{A} - E_{h,k}, \bs{j},  r ) .
\end{align*}
\end{prop}
\begin{proof}
Observe that for any {$\mu \in \CMN(n, r-1)$},
{${(-1)}^{\parity{\mu|E_{h,h}} \cdot \parity{A}} = {(-1)}^{\parity{A}}$}.
By  \eqref{def_ajr},  Proposition \ref{prop_PhiAPhiB} and Proposition \ref{phidiag1}, one can obtain
\begin{align*}
	&\ABJRS( \mathrm{O}, E_{h,h}, \bs{ 0 },  r ) \cdot \ABJRS( \SE{A}, \SO{A}, \bs{j},  r ) \\
	& = \sum_{\substack{
			\mu \in \CMN(n, r-1) \\
			\lambda \in \CMN(n, r-\snorm{A}) \\
			\co(\mu + E_{h,h}) = \ro(A+\lambda)
			}
		}
		{v}^{{\mu} \cdot {\bs{0}}}
		{v}^{\lambda \cdot {\bs{j}}}
		\cdot
		{\Phi}_{({\mu}|E_{h,h})} {\Phi}_{( \SE{A} + \lambda | \SO{A} )} \\
	& = \sum_{\substack{ \lambda \in \CMN(n,r-\snorm{A}) } }
		{v}^{ \lambda \cdot {\bs{j}} }
		{(-1)}^{\parity{A}}
		\sum_{ k=1}^n \Big\{
			{(-1)}^{ {{\SOE{\widetilde{a}}}_{h-1,k}} } {q}^{{{\ro(A+\lambda)}_h} - \sum_{u=1}^{k} {(A+\lambda)}_{h,u} }
			{\Phi}_{(\SE{A} + \lambda -E_{h,k}|\SO{A}+ E_{h,k})}  \\
		&\qquad \qquad +
		{(-1)}^{ {{\SOE{\widetilde{a}}}_{h-1,k}} +1} {q}^{{{\ro(A+\lambda)}_h} - \sum_{u=1}^{k} {(A+\lambda)}_{h,u}}  {\STEPP{{(A+\lambda)}_{h,k}}}
		{\Phi}_{(\SE{A} + \lambda+  E_{h,k}| \SO{A} - E_{h,k})}
		\Big\} \quad  \\
&= {(-1)}^{\parity{A}}( {\fcY}_1 + {\fcY}_2) ,
\end{align*}
where
\begin{align}
{\fcY}_1
& =
		\sum_{\substack{ \lambda \in \CMN(n,r-\snorm{A}) } }
		\sum_{ k=1}^n
		 {v}^{ \lambda \cdot {\bs{j}} }
			{(-1)}^{ {{\SOE{\widetilde{a}}}_{h-1,k}} } {q}^{{{\ro(A+\lambda)}_h} - \sum_{u=1}^{k} {(A+\lambda)}_{h,u} }
			{\Phi}_{(\SE{A} + \lambda -E_{h,k}|\SO{A}+ E_{h,k})} \notag \\
	&=
	\sum_{ k < h}
		{(-1)}^{ {{\SOE{\widetilde{a}}}_{h-1,k}} }
		{q}^{\BK(h,k)}
		\sum_{\substack{ \lambda \in \CMN(n,r-\snorm{A}) } }
		 {v}^{ \lambda \cdot {(\bs{j} +2 \bs{\ep}_{h})}}
			{\Phi}_{(\SE{A} + \lambda -E_{h,k}|\SO{A}+ E_{h,k})}\label{eq:compute Y1}  \\
		& \qquad +
		{(-1)}^{ {{\SOE{\widetilde{a}}}_{h-1,h}} }
		{q}^{\BK(h,h)}
		\sum_{\substack{ \lambda \in \CMN(n,r-\snorm{A}) } }
		{v}^{ \lambda \cdot {\bs{j}} }
			{\Phi}_{(\SE{A} + \lambda -E_{h,h}|\SO{A}+ E_{h,h})}\notag  \\
		& \qquad +
		\sum_{ k > h}
		{(-1)}^{ {{\SOE{\widetilde{a}}}_{h-1,k}} }
		{q}^{\BK(h,k)}
		\sum_{\substack{ \lambda \in \CMN(n,r-\snorm{A}) } }
		 {v}^{ \lambda \cdot {\bs{j}} }
			{\Phi}_{(\SE{A} + \lambda -E_{h,k}|\SO{A}+ E_{h,k})} \notag \\
	&=
	\sum_{ k < h}
		{(-1)}^{ {{\SOE{\widetilde{a}}}_{h-1,k}} }
		{q}^{\BK(h,k)}
		\ABJRS( \SE{A} -E_{h,k}, \SO{A}+ E_{h,k}, \bs{j} + 2 \bs{\ep}_{h},  r )\notag \\
		& \qquad +
		{(-1)}^{ {{\SOE{\widetilde{a}}}_{h-1,h}} }
		{q}^{\sum_{u=h+1}^{n} {a}_{h,u}} {v}^{j_h}
		\ABJRS( \SE{A}, \SO{A}+ E_{h,h}, \bs{j},  r )
		\quad \mbox{( by \eqref{eq simple diag} ) }
		\notag \\
		& \qquad +
		\sum_{ k > h}
		{(-1)}^{ {{\SOE{\widetilde{a}}}_{h-1,k}} }
		{q}^{\BK(h,k)}
		\ABJRS( \SE{A}  -E_{h,k}, \SO{A}+ E_{h,k}, \bs{j},  r ) ,\notag
\end{align}
and
\begin{align*}
{\fcY}_2 & =
		\sum_{\substack{ \lambda \in \CMN(n,r-\snorm{A}) } }
		\sum_{ k=1}^n
		{v}^{ \lambda \cdot {\bs{j}} }
		{(-1)}^{ {{\SOE{\widetilde{a}}}_{h-1,k}} +1} {q}^{{{\ro(A+\lambda)}_h} - \sum_{u=1}^{k} {(A+\lambda)}_{h,u}}  {\STEPP{{(A+\lambda)}_{h,k}}}
		{\Phi}_{(\SE{A} + \lambda+  E_{h,k}| \SO{A} - E_{h,k})} \\
	&=
		\sum_{ k<h }
		{(-1)}^{ {{\SOE{\widetilde{a}}}_{h-1,k}} +1}
		{q}^{\BK(h,k)}
		{\STEPP{{a}_{h,k}}}
		\ABJRS( \SE{A} +  E_{h,k}, \SO{A} - E_{h,k}, \bs{j} + 2 \bs{\ep}_{h},  r )  \\
		&\qquad \qquad +
		{(-1)}^{ {{\SOE{\widetilde{a}}}_{h-1,h}} +1}
		{q}^{\BK(h,h)}
		\sum_{\substack{ \lambda \in \CMN(n,r-\snorm{A}) } }
		 {v}^{ \lambda \cdot {\bs{j}} }
		{\STEPP{\SOE{a}_{h,h} + {\lambda}_h}}
		{\Phi}_{(\SE{A} + \lambda+  E_{h,h}| \SO{A} - E_{h,h})} \\
		&\qquad \qquad +
		\sum_{ k>h}
		{(-1)}^{ {{\SOE{\widetilde{a}}}_{h-1,k}} +1}
		{q}^{\BK(h,k)}
		{\STEPP{{a}_{h,k}}}
		\ABJRS( \SE{A} +  E_{h,k}, \SO{A} - E_{h,k}, \bs{j},  r ) \\
	&=
		\sum_{ k<h }
		{(-1)}^{ {{\SOE{\widetilde{a}}}_{h-1,k}} +1}
		{q}^{\BK(h,k)}
		{\STEPP{{a}_{h,k}}}
		\ABJRS( \SE{A} +  E_{h,k}, \SO{A} - E_{h,k}, \bs{j} + 2 \bs{\ep}_{h},  r )  \\
		&\qquad \qquad 		+
		{(-1)}^{ {{\SOE{\widetilde{a}}}_{h-1,h}} +1}
		\frac{{q}^{\BK(h,h)}}{({q}^2 - 1) {v}^{ j_h }} \{
		\ABJRS( \SE{A}, \SO{A} - E_{h,h}, \bs{ j} +4 \bs{\ep}_{h},  r )
		-
		\ABJRS( \SE{A}, \SO{A} - E_{h,h}, \bs{j},  r ) \} \\
		&\qquad \qquad +
		\sum_{ k>h}
		{(-1)}^{ {{\SOE{\widetilde{a}}}_{h-1,k}} +1}
		{q}^{\BK(h,k)}
		{\STEPP{{a}_{h,k}}}
		\ABJRS( \SE{A} +  E_{h,k}, \SO{A} - E_{h,k}, \bs{j},  r ),
\end{align*}
where the second equality for computing ${\fcY}_2$ is due to the fact that
 {${\Phi}_{(\SE{A} + \lambda+  E_{h,h}| \SO{A} - E_{h,h})} $} appears only in the case $\SO{a}_{h,h} = 1$ and Lemma \ref{formstepodd} {\rm(3)}.
Putting {${\fcY}_1$} and {${\fcY}_2$} together, and replacing {$q$} by {${v}^2$}, the proposition is proved.
\end{proof}
By Lemma \ref{shiftonN} and Remark \ref{phidiagN}, we may drop the SDP condition for $h=n$.
\begin{cor}\label{mulformdiagN}
	For any {$A \in \MNZNS(n)$},
	the  multiplication formula in Proposition \ref{mulformdiag} holds for {$h=n$} in {${\SQvnrR}$} for all {$r\geq\snorm{A}$}.
\end{cor}

Considering  the SDP condition
and \eqref{def_ajr},
and applying  Proposition \ref{phiupper1},
we have the following Proposition \ref{mulformodd1} and Proposition \ref{mulformodd2}. Their detailed proof are given in Appendix A.
\begin{prop}\label{mulformodd1}
	Let 
	{$h \in [1,n-1]$} and  {$A \in \MNZNS(n)$}.
Assume for each $1\leq k\leq n$ and $\lambda \in \CMN(n, r-\snorm{A})$ with {${(A+\lambda)}_{h+1,k} \ge 1$},
{$({A+\lambda})^+_{h,k}$} satisfies SDP condition at $(h,k)$,
then the following multiplication formula  holds in {${\SQvnrR}$} for all  {$r\geq \snorm{A}$}:
\begin{align*}
& \ABJRS(\mathrm{O}, E_{h, h+1}, \bs{ 0 },  r  ) \cdot \ABJRS( \SE{A}, \SO{A}, \bs{j},  r ) \\
&=
	\sum_{k<h}
	{(-1)}^{{ {\SOE{\widetilde{a}}}_{h-1,k}} + \parity{A}}
	{v}^{ 2{\BK(h,k)}  + 2 \SOE{a}_{h+1,k} }
	\ABJRS( \SE{A} - E_{h+1, k}, \SO{A}  + E_{h,k}, \bs{j} + 2 \bs{\ep}_{h},  r ) \\
	& \qquad +
	{(-1)}^{{ {\SOE{\widetilde{a}}}_{h-1,h}} + \parity{A}}
	{v}^{ 2 {\BK(h,h)}  + 2 \SOE{a}_{h+1,h} }
	\ABJRS( \SE{A} - E_{h+1, h}, \SO{A}  + E_{h,h}, \bs{j}, r )\\
	& \qquad +
	{(-1)}^{{ {\SOE{\widetilde{a}}}_{h-1,h+1}} + \parity{A}}
	{v}^{ 2 {\BK(h,h+1)}  + 2 \SOE{a}_{h+1,h+1} +  j_{h+1} }
	\ABJRS( \SE{A}, \SO{A}  + E_{h,h+1}, \bs{j},  r ) \\
	& \qquad +
	\sum_{ \substack{ k > h+1}  }
	{(-1)}^{{ {\SOE{\widetilde{a}}}_{h-1,k}} + \parity{A}}
	{v}^{ 2 {\BK(h,k)}  + 2 \SOE{a}_{h+1,k} }
	\ABJRS( \SE{A} - E_{h+1, k}, \SO{A}  + E_{h,k}, \bs{j}, r )\\
	&\qquad +
	\sum_{k < h}
	{(-1)}^{{ {\SOE{\widetilde{a}}}_{h-1,k}} + 1 - \SOE{a}_{h,k} + \parity{A}}
	{v}^{ 2 {\BK(h,k)} }
	\VSTEP{ \SEE{a}_{h,k} +1}
	\ABJRS( \SE{A}+ E_{h,k}, \SO{A} - E_{h+1, k}, \bs{j} + 2 \bs{\ep}_{h},  r )  \\
	& \qquad +
	{(-1)}^{{ {\SOE{\widetilde{a}}}_{h-1,h}} + 1 - \SOE{a}_{h,h} + \parity{A} }
	\frac{{v}^{ 2 {\BK(h,h)} - j_h} } { {v}^2 - 1 }
		( \ABJRS( \SE{A}, \SO{A}- E_{h+1, h} , {\bs{j}} + 2 \bs{\ep}_{h}, r )
		-  \ABJRS(  \SE{A}, \SO{A}- E_{h+1, h} , {\bs{j}} , r ) ) \\
	& \qquad +
	\sum_{k > h}
	{(-1)}^{{ {\SOE{\widetilde{a}}}_{h-1,k}} + 1 - \SOE{a}_{h,k} + \parity{A}}
	{v}^{ 2 {\BK(h,k)}  }
	\VSTEP{ \SEE{a}_{h,k} +1}
	\ABJRS( \SE{A}  + E_{h,k}, \SO{A} - E_{h+1, k}, \bs{j},  r )  \\
	& \qquad +
	\sum_{k < h}
	{(-1)}^{{ {\SOE{\widetilde{a}}}_{h-1,k}} + \parity{A}}
	{v}^{ 2 {\BK(h,k)}  + 2 \SOE{a}_{h+1,k} - 2 }
	{\VSTEPPDR{ {a}_{h,k} + 1}} \\
	& \qquad \qquad \cdot
	\ABJRS( \SE{A} -E_{h+1, k} + 2 E_{h,k}, \SO{A} - E_{h,k}, {\bs{j}} + 2 \bs{\ep}_{h} ,  r )  \\
	& \qquad +
	{(-1)}^{{ {\SOE{\widetilde{a}}}_{h-1,h}} + \parity{A} }
	\frac{{v}^{ 2 {\BK(h,h)} + 2  \SOE{a}_{h+1,h} - 2 j_h - 2}}{ {v}^4 - 1 }
	\Big\{
	\ABJRS( \SE{A}  -E_{h+1, h}, \SO{A}  -E_{h,h}, \bs{j} + 4 \bs{\ep}_{h},  r ) \\
	&\qquad \qquad -
	({v}^2 +1) \ABJRS( \SE{A} -E_{h+1, h}, \SO{A}  -E_{h,h}, \bs{j} + 2 \bs{\ep}_{h},  r ) \\
	&\qquad \qquad +
	{v}^2 \ABJRS( \SE{A} -E_{h+1, h}, \SO{A}  -E_{h,h}, \bs{j},  r )  \Big \} \\
	& \qquad +
	{(-1)}^{{ {\SOE{\widetilde{a}}}_{h-1,h+1}} + \parity{A}}
	{v}^{ 2  {\BK(h,h+1)}  + 2 \SOE{a}_{h+1,h+1} + j_{h+1} - 2}
	\VSTEPPDR{ {a}_{h,h+1} + 1} \\
	& \qquad \qquad \cdot
	\ABJRS( \SE{A}  + 2 E_{h,h+1}, \SO{A} - E_{h,h+1}, \bs{j},  r  )  \\
	& \qquad +
	\sum_{k>h+1}
	{(-1)}^{{ {\SOE{\widetilde{a}}}_{h-1,k}} + \parity{A}}
	{v}^{ 2 {\BK(h,k)} + 2\SOE{a}_{h+1,k} - 2}
	\VSTEPPDR{ {a}_{h,k} + 1} \\
	& \qquad \qquad \cdot
	\ABJRS( \SE{A}  -E_{h+1, k} + 2 E_{h,k}, \SO{A} - E_{h,k}, \bs{j},  r ),
\end{align*}
\end{prop}

\begin{prop}\label{mulformodd2}
	Let 
	{$h \in [1,n-1]$} and {$A \in \MNZNS(n)$}.
Assume for any   $\lambda \in \CMN(n, r-\snorm{A})$,
{$({A+\lambda})$}  satisfies the SDP condition on the $h$-th row,
then the following multiplication formulas hold in {${\SQvnrR}$} for all  {$r\geq \snorm{A} $}.:
\begin{align*}
& \ABJRS(\mathrm{O}, E_{h+1, h}, \bs{ 0 },  r  ) \cdot \ABJRS( \SE{A}, \SO{A}, \bs{j},  r ) \\
&=
 	\sum_{k<h}
	{(-1)}^{{\SOE{\widetilde{a}}}_{h-1,k} + \SOE{a}_{h,k} + \parity{A}}
	{v}^{ 2 \AK(h+1,k)}
	\ABJRS( \SE{A} - E_{h,k}, \SO{A}+ E_{h+1, k}, \bs{j},  r ) \\
	&\qquad +
	{(-1)}^{{\SOE{\widetilde{a}}}_{h-1,h} + \SOE{a}_{h,h} + \parity{A}}
	{v}^{ 2 \AK(h+1,h) +  j_h }
	\ABJRS( \SE{A} , \SO{A}+ E_{h+1, h}, \bs{j},  r ) \\
	&\qquad +
	{(-1)}^{{\SOE{\widetilde{a}}}_{h-1,h+1} + \SOE{a}_{h,h+1} + \parity{A}}
	{v}^{ 2 \AK(h+1,h+1)}
	\ABJRS( \SE{A} - E_{h,h+1}, \SO{A}+ E_{h+1, h+1}, \bs{j},  r ) \\
	&\qquad +
 	\sum_{k>h+1}
	{(-1)}^{{\SOE{\widetilde{a}}}_{h-1,k} + \SOE{a}_{h,k} + \parity{A}}
	{v}^{  2 \AK(h+1,k)}
	\ABJRS( \SE{A} - E_{h,k}, \SO{A}+ E_{h+1, k}, \bs{j} + 2 \bs{\ep}_{h+1},  r ) \\
	&\qquad +
 	\sum_{k<h}
	{(-1)}^{{\SOE{\widetilde{a}}}_{h-1,k} + \SOE{a}_{h,k} + \parity{A} + 1}
	{v}^{ 2 \AK(h+1,k) - 2}
	\VSTEPPDR{ {a}_{h+1,k} +1} \\
	& \qquad \qquad  \qquad  \cdot
	\ABJRS( \SE{A}  - E_{h,k} + 2E_{h+1, k}, \SO{A}  -E_{h+1, k}, \bs{j},  r ) \\
	& \qquad +
	{(-1)}^{{\SOE{\widetilde{a}}}_{h-1,h} + \SOE{a}_{h,h}+ \parity{A} + 1}
	{v}^{ 2 \AK(h+1,h) + j_h - 2}
	\VSTEPPDR{ {a}_{h+1,h} +1} \\
	& \qquad \qquad  \qquad  \cdot
	\ABJRS( \SE{A} + 2E_{h+1, h}, \SO{A}  -E_{h+1, h}, \bs{j},  r ) \\
	& \qquad +
	{(-1)}^{{\SOE{\widetilde{a}}}_{h-1,h+1} + \SOE{a}_{h,h+1} + \parity{A} + 1}
	\frac{{v}^{ 2 \AK(h+1,h+1) - 2 j_{h+1} - 2 }}{ {v}^4 - 1 } \Big \{
	\ABJRS( \SE{A}  - E_{h,h+1}, \SO{A}  -E_{h+1, h+1}, \bs{j} + 4 \bs{\ep}_{h+1},  r ) \\
	&\qquad \qquad -
	 ({v}^2 + 1) \ABJRS( \SE{A}  - E_{h,h+1}, \SO{A}  -E_{h+1, h+1}, \bs{j} + 2 \bs{\ep}_{h+1},  r ) \\
	&\qquad \qquad +
	 {v}^2 \ABJRS( \SE{A}  - E_{h,h+1}, \SO{A}  -E_{h+1, h+1}, \bs{j},  r )
	 \Big \} \\
	& \qquad +
 	\sum_{k>h+1}
	{(-1)}^{{\SOE{\widetilde{a}}}_{h-1,k} + \SOE{a}_{h,k} + \parity{A} + 1}
	{v}^{ 2 \AK(h+1,k) - 2}
	\VSTEPPDR{ {a}_{h+1,k} +1} \\
	& \qquad \qquad  \qquad  \cdot
	\ABJRS( \SE{A}  - E_{h,k} + 2E_{h+1, k}, \SO{A}  -E_{h+1, k}, \bs{j}+ 2 \bs{\ep}_{h+1},  r ) \\
	& \qquad + \sum_{k < h}
 	{(-1)}^{{\SOE{\widetilde{a}}}_{h-1,k} + \parity{A} + 1}
	{v}^{ 2 \AK(h+1,k) + 2 {a}_{h,k} - 2}
	\VSTEP{ \SEE{a}_{h+1,k}+1}
	\ABJRS( \SE{A} +  E_{h+1,k}, \SO{A}  - E_{h,k}, \bs{j},  r ) \\
	& \qquad +
 	{(-1)}^{{\SOE{\widetilde{a}}}_{h-1,h} + \parity{A} + 1}
	{v}^{ 2 \AK(h+1,h) }
	\VSTEP{ \SEE{a}_{h+1,h}+1}
	\ABJRS( \SE{A} +  E_{h+1,h}, \SO{A}  - E_{h,h}, \bs{j} + 2 \bs{\ep}_{h},  r ) \\
	& \qquad +
 	{(-1)}^{{\SOE{\widetilde{a}}}_{h-1,h+1} + \parity{A} + 1}
		\frac{{v}^{ 2 \AK(h+1,h+1) + 2 {a}_{h,h+1} - j_{h+1} - 2} }{  {v}^2 - 1  }\Big\{
	\ABJRS( \SE{A}, \SO{A}  - E_{h,h+1}, \bs{j} + 2 \bs{\ep}_{h+1},  r ) \\
	& \qquad \qquad  -
	\ABJRS( \SE{A}, \SO{A}  - E_{h,h+1}, \bs{j},  r )
	 \Big \} \\
	& \qquad +
	\sum_{k>h+1}
 	{(-1)}^{{\SOE{\widetilde{a}}}_{h-1,k} + \parity{A} + 1}
	{v}^{ 2 \AK(h+1,k) + 2{a}_{h,k} - 2}
	\VSTEP{ \SEE{a}_{h+1,k}+1}
	\ABJRS( \SE{A}  + E_{h+1,k}, \SO{A}  - E_{h,k}, \bs{j} + 2 \bs{\ep}_{h+1},  r ).
\end{align*}
\end{prop}

\spaceintv
\section{Superalgebra homomorphisms from $\Uvqn$ to $\bsSQvnr$}\label{defining relations}
Recall that   {${v}$}, {${q}$} are indeterminate and  {${v}^2 = {q} $}.
By Proposition \ref{mulformzero}, {$\AJRS({O}, \bs{0}, r)$} is the identity element in {${\SQvnrR}$}.
Define the superalgebra
$$
{\SQvnR}  = \prod_{r \ge 0 } {\SQvnrR}.
$$
The elements $(f_r)_{r\geq0}$ in $\SQvnR$ 
are written as formal sum $\sum_{r\geq0}f_r$
 which is homogenous of degree $\bar i$ if every $f_r$ is homogenous of degree $\bar i$.
In particular, for any  $A \in \MNZNS(n), \bs{j} \in {\ZZ}^{n}$, 
let $\AJS(A, \bs{j}) := \sum_{r \ge 0 } \AJRS(A, \bs{j}, r) \in {\SQvnR}$ and form the vector space
\begin{equation}\label{Anv}
    \USnv = \tspan_{\Qv} \{\AJS(A, \bs{j})\  \where  A \in \MNZNS(n), \bs{j} \in {\ZZ}^{n} \} \subset {\SQvnR}.
\end{equation}

\begin{rem}\label{rem:prod alg}
By definition of the product in the algebra ${\SQvnR} $, since the coefficients of the multiplication formulas in {${\SQvnrR}$} given in Proposition \ref{mulformzero}, \ref{mulformeven}, \ref{mulformdiag}, \ref{mulformodd1},  and  \ref{mulformodd2}
are independent of $r$ for all $r\geq |A|$ (and $\AJRS(A, \bs{j}, r)=0$ if $r<|A|$),
 it follows that removing the $r$ in every involved {$\AJRS(A, \bs{j}, r)$} yields
multiplication formulas in {${\SQvnR}  $}.
\end{rem}

We introduce the following elements\footnote{\label{raw2}These elements play the role as generators for $\Uvqn$; see Theorems \ref{qqschur_reltion} and \ref{map_iso} below. \eqref{mult-GA} shows that they differ by a factor from those elements used in \cite[Lemma 5.4(b)]{BLM} or \cite[Theorem 8.1]{DG}. See Remark \ref{raw}(2).} in $\USnv$:
\begin{equation}\label{raw generator}
\aligned
&G_{i}^{\pm 1}
	= \ABJS(\mathrm{O}, \mathrm{O}, \pm \bs{\ep}_{i}), \qquad
X_{j} = \ABJS(E_{j, j+1}, \mathrm{O}, -\bs{\ep}_{j}) , \qquad
Y_{j} = \ABJS(E_{j+1, j}, \mathrm{O}, -\bs{\ep}_{j+1}) , \\
& G_{\ol{i}} 	=  \ABJS(\mathrm{O}, E_{i,i}, - \bs{\ep}_{i}), \qquad
 X_{\ol{j}} =    \ABJS( \mathrm{O}, E_{j,j+1}, - \bs{\ep}_{j} ),  \qquad
 Y_{\ol{j}} 	=   \ABJS(\mathrm{O}, E_{j+1,j},   - \bs{\ep}_{j+1} ).
 \endaligned
\end{equation}
with {$1 \le i \le n, 1\le j \le n-1$}.
Observe that by Corollary \ref{mulformzerocor} and Remark \ref{rem:prod alg},
for $1\leq i\leq n$ and {$\bs{j}\in \ZZ^n$},
 we have
\begin{align}\label{mult-GA}
G_i^{\pm} \AJS(A, \bs{j})
	=v^{\pm\sum_{u=1}^na_{i,u}} \AJS(A, \bs{j} \pm \bs{\ep}_i),
	\quad
\AJS(A, \bs{j}) G_i^{\pm} =v^{\pm\sum_{u=1}^na_{u,i}} \AJS(A, \bs{j} \pm \bs{\ep}_i).
\end{align}
Moreover for $1\leq j\leq n-1$, we have
\begin{equation}\label{XYbar-comp}
\begin{aligned}
X_j&=vG_j^{-1}\ABJS(E_{j, j+1},O ,\bs{0} )=\ABJS(E_{j, j+1},\mathrm{O}, \bs{0} )G_j^{-1}, \\
 X_{\ol{j}}&=vG_{j}^{-1}\ABJS(\mathrm{O},E_{j, j+1} ,\bs{0} )=\ABJS(\mathrm{O},E_{j, j+1}, \bs{0} )G_j^{-1}, \\
Y_j&=vG_{j+1}^{-1}\ABJS(E_{j+1, j},\mathrm{O}, \bs{0} )=\ABJS(E_{j+1, j},\mathrm{O}, \bs{0} )G_{j+1}^{-1}, \\
 Y_{\ol{j}}&=vG_{j+1}^{-1}\ABJS(\mathrm{O},E_{j+1, j}, \bs{0} )=\ABJS(\mathrm{O},E_{j+1, j}, \bs{0} )G_{j+1}^{-1}.
\end{aligned}
\end{equation}

Because some of the matrices appearing in the computation of the multiplication
{$ X_{\ol{h}} Y_{\ol{h}} $}, {$ X_{\ol{h}} Y_{h}$}
{$ Y_{\ol{h}} X_{\ol{h}} $}, {$ Y_{\ol{h}} X_{h}$}
do not satisfy the SDP condition, the formulas in Section \ref{sec_spanningsets} are not applicable. We shall compute them separately in Lemma \ref{mulform_ef}.
Simliarly, since some of the matrices involved in computing the multiplications
{$ X_{\ol{h}} X_{\ol{h+1}} $}, {$ X_{\ol{h+1}} X_{\ol{h}}$},
{$ Y_{\ol{h}} Y_{\ol{h+1}} $}, {$ Y_{\ol{h+1}} Y_{\ol{h}}$}
do not satisfy the SDP condition,  
we shall compute them  in Lemma \ref{relation_eei}.
\begin{lem}\label{mulform_ef}
For any integer {$h$} with {$ 1 \le h \le n-1$}, the following equations hold in {${\SQvnR} $}:
\begin{align*}
{\rm (1)} \qquad
X_{\ol{h}} Y_{\ol{h}} + Y_{\ol{h}} X_{\ol{h}}
&= \frac{G_{h} G_{h+1} - G_{h}^{-1} G_{h+1}^{-1}}{{v} - {v}^{-1}}
	 + ({v} - {v}^{-1}) G_{\ol{h}} G_{\ol{h+1}} ,\\
{\rm (2)} \qquad
X_{h} Y_{\ol{h}}    - Y_{\ol{h}} X_{h}
&=   G_{\ol{h}}  G_{h+1}^{-1}   -  G_{h}^{-1} G_{\ol{h+1}}, \\
{\rm (3)} \qquad
X_{\ol{h}} Y_{h} - Y_{h} X_{\ol{h}}
&=  G_{h+1} G_{\ol{h}}  - G_{\ol{h+1}} G_{h} .
\end{align*}
\end{lem}
\begin{proof}

{\rm(1)}  \
By Proposition \ref{prop_PhiAPhiB}, \ref{philower2} {\rm (2)}, \eqref{def_ajr} and Lemma \ref{formstepodd} {\rm(1)},
we have
\begin{align*}
&\ABJRS(\mathrm{O}, E_{h+1, h}, \bs{0}, r) \cdot \ABJRS(\mathrm{O}, E_{h, h+1}, \bs{0}, r) \\
& =
		\sum_{\substack{
			\mu \in \CMN(n,r)  \\
			\lambda \in  \CMN(n,r) \\
			\co(\mu | E_{h+1, h}  ) = \ro(\lambda | E_{h, h+1})
		} }
	{v}^{\mu \cdot {\bs{0}}} {\Phi}_{(\mu | E_{h+1, h} )}
		\cdot
		{v}^{\lambda \cdot {\bs{0}}} {\Phi}_{(\lambda | E_{h, h+1})} \\
& =
		\sum_{		\mu \in  \CMN(n,r-1)  } (
	-{\Phi}_{({\mu} - E_{h,h} |   E_{h, h+1} +  E_{h+1, h})}
	+ \VSTEP{ {\mu}_{h+1} + 1 } {\Phi}_{( {\mu} + E_{h+1, h+1} | O)} ) \\
&= -\sum_{	 \substack{ \mu \in  \CMN(n,r-1) \\ \mu_h > 0} }
	{\Phi}_{({\mu} - E_{h,h} |   E_{h, h+1} +  E_{h+1, h})}
+\sum_{ \mu \in  \CMN(n,r-1)  }
		\VSTEP{ {\mu}_{h+1} + 1 } {\Phi}_{( {\mu} + E_{h+1, h+1} | O)}\\
&=	-\ABJRS(\mathrm{O}, E_{h, h+1} + E_{h+1, h}, \bs{0}, r )+\frac{ 1 }{{v^2} - 1}   \AJRS(\bs{O},2 \bs{\ep}_{h+1}, r)
	- \frac{ 1 }{{v^2} - 1}    \AJRS(\bs{O}, \bs{0}, r)
\end{align*}
Similarly,  by Proposition \ref{phiupper2} {\rm(2)},
  \eqref{def_ajr} and Lemma \ref{formstepodd} {\rm(1)},
we have
\begin{align*}
&\ABJRS(\mathrm{O}, E_{h, h+1}, \bs{0}, r) \cdot \ABJRS(\mathrm{O}, E_{h+1, h}, \bs{0}, r) \\
& =
		\sum_{\substack{
			\mu \in \CMN(n,r-1)  \\
			\lambda \in  \CMN(n,r-1) \\
			\co(\mu | E_{h, h+1}  ) = \ro(\lambda | E_{h+1, h})
		} }
	{v}^{\mu \cdot {\bs{0}}} {\Phi}_{(\mu | E_{h, h+1} )}
		\cdot
		{v}^{\lambda \cdot {\bs{0}}} {\Phi}_{(\lambda | E_{h+1, h})} \\
&= \sum_{		\mu \in  \CMN(n,r-1)  }
	 	\VSTEP{{\mu}_{h} + 1} {q}^{ {\mu}_{h+1}  }
		{\Phi}_{( {\mu} +  E_{h,h}| \mathrm{O})}+\sum_{		\mu \in  \CMN(n,r)  }
	{\Phi}_{({\lambda} - 2 E_{h+1, h+1}   | E_{h, h+1} + E_{h+1, h} )} \\
&-\sum_{		\mu \in  \CMN(n,r-1)  }
	({q}-1) 	{\Phi}_{( {\lambda}- 2 E_{h+1, h+1}    | E_{h,h} + E_{h+1,h+1})}\\
&= \frac{ 1}{ {v^2} - 1 }
		\AJRS({O}, 2 \bs{\ep}_{h} +  2 \bs{\ep}_{h+1}, r)
		- \frac{ 1}{ {v^2} - 1 }
		\AJRS({O},  2 \bs{\ep}_{h+1}, r)+\ABJRS( \mathrm{O} ,  E_{h, h+1} + E_{h+1, h}, \bs{0}, r)\\
&-({v^2}-1) 	\ABJRS( \mathrm{O} , E_{h,h} + E_{h+1,h+1} , \bs{0}, r)
\end{align*}
Putting together, one can obtain by \eqref{XYbar-comp}
\begin{align*}
X_{\ol{h}}  Y_{\ol{h}} + Y_{\ol{h}} X_{\ol{h}}
&=v G_{h}^{-1}\ABJS(\mathrm{O},E_{h,h+1},\bs{0})\ABJS(\mathrm{O},E_{h+1,h},\bs{0})G_{h+1}^{-1}+vG_{h+1}^{-1}\ABJS(\mathrm{O},E_{h+1,h},\bs{0})\ABJS(\mathrm{O},E_{h,h+1},\bs{0}))G_{h}^{-1}\\
&=
	\frac{ 1}{ {v} - {v}^{-1} } (
	 \AJS({O}, \bs{\ep}_{h} +   \bs{\ep}_{h+1})
	-  \AJS({O},  -\bs{\ep}_{h}  -\bs{\ep}_{h+1}) )\\
&	- {v}^{-1} ({q}-1) 	\ABJS( \mathrm{O}, E_{h, h} + E_{h+1, h+1} ,  -\bs{\ep}_{h}  -\bs{\ep}_{h+1})
	\\
&=
	\frac{  G_{h} G_{h+1} - G_{h}^{-1} G_{h+1}^{-1} }{ {v} - {v}^{-1} }
	+  ({v} - {v}^{-1} ) G_{\ol{h}} G_{\ol{h+1} } .
\end{align*}

{\rm(2)}
By Proposition \ref{philower2} {\rm (1)},
we have
\begin{align*}
&\ABJRS(\mathrm{O}, E_{h+1, h}, \bs{0}, r) \cdot \ABJRS(E_{h, h+1}, \mathrm{O}, \bs{0}, r) \\
& =
		\sum_{\substack{
			\mu \in \CMN(n,r-1)  \\
			\lambda \in  \CMN(n,r-1) \\
			\co(\mu | E_{h+1, h}) = \ro(\lambda + E_{h, h+1}| O)
		} }
	{v}^{\mu \cdot {\bs{0}}} {\Phi}_{(\mu | E_{h+1, h} )}
		\cdot
		{v}^{\lambda \cdot {\bs{0}}} {\Phi}_{(\lambda + E_{h, h+1}| O)} \\
&= 	\sum_{		\mu \in  \CMN(n,r-1)  \atop \mu_{h} > 0}
	 {\Phi}_{({\mu} -  E_{h,h} + E_{h, h+1}| E_{h+1, h})}
+ 	\sum_{		\mu \in  \CMN(n,r-1)  }
		{\Phi}_{( {\mu} | E_{h+1, h+1})}  \\
& = 	\ABJRS(E_{h, h+1}, E_{h+1, h}, \bs{0}, r)
 + 	\ABJRS(\mathrm{O}, E_{h+1, h+1}, \bs{0}, r).
\end{align*}
Meanwhile by Proposition \ref{mulformeven}, we have
\begin{align*}
	\ABJRS(E_{h, h+1}, \mathrm{O}, \bs{0}, r) \ABJRS(\mathrm{O}, E_{h+1, h}, \bs{0},r  )
&=
\ABJRS(  E_{h, h+1}, E_{h+1, h}, \bs{ 0 }, r )
	 +
	\ABJRS( \mathrm{O},  E_{h, h}, \bs{ 0 }, r ).
\end{align*}
As a consequence, by \eqref{XYbar-comp}  one can obtain
\begin{align*}
X_{h} Y_{\ol{h}}
&=
	 {v}^{-1} \ABJS(  E_{h, h+1}, E_{h+1, h}, -\bs{\ep}_{h} -\bs{\ep}_{h+1}  )
	+    \ABJS( \mathrm{O},  E_{h, h},  -\bs{\ep}_{h} -\bs{\ep}_{h+1}  ),\\
Y_{\ol{h}} X_{h}
&=
	  {v}^{-1} \ABJS( E_{h, h+1} ,   E_{h+1, h}, - \bs{\ep}_{h} - \bs{\ep}_{h+1} )
	 +  \ABJS( \mathrm{O}, E_{h+1, h+1}, - \bs{\ep}_{h} - \bs{\ep}_{h+1} ) , \\
X_{h} Y_{\ol{h}}    - Y_{\ol{h}} X_{h}
&=
	    \ABJS( \mathrm{O},  E_{h, h},  -\bs{\ep}_{h} -\bs{\ep}_{h+1}  )
	-  \ABJS( \mathrm{O}, E_{h+1, h+1}, - \bs{\ep}_{h} - \bs{\ep}_{h+1} )
 \\
&= G_{\ol{h}}  G_{h+1}^{-1}   -  G_{h}^{-1} G_{\ol{h+1}}.
\end{align*}
Similarly, by Proposition \ref{phiupper2} {\rm {(1)}}, \eqref{def_ajr},
Proposition \ref{mulformeven},  Lemma \ref{formstepodd} {\rm(2)} as well
as  \eqref{XYbar-comp}, one can show the third part {\rm (3)} and we leave the details to the reader.

\end{proof}

\begin{lem}\label{relation_eei}
For any integer {$h$} with {$ 1 \le h \le n-1$}, the following equations hold in {${\SQvnR} $}:
\begin{align*}
{\rm (1)} \qquad
& X_{h} X_{h+1} - {v} X_{h+1} X_{h}
	= X_{\ol{h}} X_{\ol{h+1}} + {v} X_{\ol{h+1}} X_{\ol{h}}, \\
{\rm (2)} \qquad
& Y_{h} Y_{h+1} - {v} Y_{h+1} Y_{h}
	= - (Y_{\ol{h}} Y_{\ol{h+1}} + {v} Y_{\ol{h+1}} Y_{\ol{h}}).
\end{align*}
\end{lem}
\begin{proof}

 {\rm (1)}
 Referring to \eqref{XYbar-comp},
it suffices to prove the equation
\begin{align*}
	X_{h} X_{h+1}(x_{\lambda}g)- {v} X_{h+1} X_{h}(x_{\lambda}g)
	= X_{\ol{h}} X_{\ol{h+1}}(x_{\lambda}g) + {v} X_{\ol{h+1}} X_{\ol{h}}(x_{\lambda}g)
\end{align*}
holds for any {$\lambda \in \CMN(n, r)$},  and homogeneous {$ g \in \HCR$}.
As
\begin{align*}
 X_{\ol{h}} X_{\ol{h+1}} ( x_{\lambda} g)
&= \ABJRS( \mathrm{O}, E_{h, h+1}, - \bs{\ep}_{h},  r  ) \cdot \ABJRS( \mathrm{O}, E_{h+1, h+2},- \bs{\ep}_{h+1},  r  )  (x_{\lambda} g ) \\
& =
		  \sum_{\substack{
			\alpha \in \CMN(n, r-1),\ 
			\beta \in \CMN(n,r - 1)  \\
			\co( \alpha | E_{h, h+1}  ) = \ro( \beta |  E_{h+1, h+2}  )
		} }
		{v}^{ -{\alpha}_{h} - {\beta}_{h+1} }
		{\Phi}_{( \alpha | E_{h, h+1} )}
		\cdot
		{\Phi}_{( \beta |  E_{h+1, h+2} )}	(x_{\lambda} g )\\
& = {(-1)}^{\parity{g}}
		 \sum_{\substack{
			\alpha \in \CMN(n, r-1),  \ 
			\beta \in \CMN(n,r - 1)  \\
			\alpha + \bs{\ep}_{h+1}  = \beta + \bs{\ep}_{h+1}  \\
			\beta + \bs{\ep}_{h+2}  = \lambda
		} }
		{v}^{ -{\alpha}_{h} -{\beta}_{h+1} }
		{\Phi}_{( \alpha | E_{h, h+1} )}
		( T_{( \beta |  E_{h+1, h+2}) } \cdot g ) \\
&=(-1)^{\parity{g}}  v^{-\lambda_h-\lambda_{h+1}}
		 \Phi_{(\lambda-\epsilon_{h+1}|E_{h,h+1})}
		(T_{(\lambda-\epsilon_{h+1}|E_{h+1,h+2})}\cdot g) \\
&={(-1)}^{\parity{g}}
		{v}^{ -{\alpha}_{h} -{\lambda}_{h+1} }
		{\Phi}_{( \lambda-\epsilon_{h+2} | E_{h, h+1} )}
		( x_{\ro(A)}
		T_{d_A} c_A
		\sum _{{\sigma} \in \D_{{\nu}_{A}} \cap \fS_{\co(A)}} T_{\sigma}
		\cdot g )\\
&=	- {v}^{ -{\lambda}_{h} -{\lambda}_{h+1} }
		x_{\ro(B)}
		T_{d_B} c_B
		\sum _{{\sigma} \in \D_{{\nu}_{B}} \cap \fS_{\co(B)}} T_{\sigma}
		T_{d_A} c_A
		\sum _{{\sigma} \in \D_{{\nu}_{A}} \cap \fS_{\co(A)}} T_{\sigma}
		\cdot g,
\end{align*}
where {$A = ( \lambda -  \bs{\ep}_{h+2}  |  E_{h+1, h+2})$},  {$B =  ( \lambda - \bs{\ep}_{h+2} | E_{h, h+1}) $} and the last equality is due to
 {${(-1)}^{\parity{g}} \cdot {(-1)}^{\parity{B} \cdot (\parity{g} + \parity{A})} = -1$}.
It is easy to verify
{$d_B = d_A = 1$},
{$c_B = c_{\widetilde{\lambda}_{h} + 1}$} ,
{$c_A = c_{\widetilde{\lambda}_{h+1} + 1}$},
{$\ro(B) = \ro( \lambda - \bs{\ep}_{h+2} | E_{h, h+1}) =  \lambda - \bs{\ep}_{h+2} + \bs{\ep}_{h}$},
and
\begin{align*}
\sum _{{\sigma} \in \D_{{\nu}_{A}} \cap \fS_{\co(A)}} T_{\sigma}
= \TAIJ( {\widetilde{\lambda}_{h+1} +1} , {\widetilde{\lambda}_{h+1} + \lambda_{h+2} - 1} ), \quad
\sum _{{\sigma} \in \D_{{\nu}_{B}} \cap \fS_{\co(B)}} T_{\sigma}
=
\TAIJ({\widetilde{\lambda}_{h}+1} , {\widetilde{\lambda}_{h} + \lambda_{h+1} }).
\end{align*}
Hence
\begin{align*}
X_{\ol{h}} X_{\ol{h+1}} ( x_{\lambda} g)
& =
		-{v}^{ -{\lambda}_{h} -{\lambda}_{h+1} }
		x_{(\lambda - \bs{\ep}_{h+2} + \bs{\ep}_{h})}
		c_{\widetilde{\lambda}_{h} + 1}
		\TAIJ({\widetilde{\lambda}_{h}+1} , {\widetilde{\lambda}_{h} + \lambda_{h+1} })
		c_{\widetilde{\lambda}_{h+1} + 1}
		 \TAIJ( {\widetilde{\lambda}_{h+1} +1} , {\widetilde{\lambda}_{h+1} + \lambda_{h+2} - 1} )
		\cdot g .
\end{align*}
Similarly, set
{$J = ( \lambda + \bs{\ep}_{h} - \bs{\ep}_{h+1} - \bs{\ep}_{h+2} | E_{h+1, h+2} )$},
{$ K = ( \lambda -\bs{\ep}_{h+1} | E_{h, h+1} ) $},
then we have
\begin{align*}
\sum _{{\sigma} \in \D_{{\nu}_{J}} \cap \fS_{\co(J)}} T_{\sigma}
 =\TAIJ( {\widetilde{\lambda}_{h+1} +1}, {\widetilde{\lambda}_{h+1} + \lambda_{h+2} - 1} ) \ , \qquad
\sum _{{\sigma} \in \D_{{\nu}_{K}} \cap \fS_{\co(K)}} T_{\sigma}
= \TAIJ( {\widetilde{\lambda}_{h}+1} , {\widetilde{\lambda}_{h} + \lambda_{h+1} - 1}) \ ,
\end{align*}
and direct calculation shows
\begin{align*}
X_{\ol{h+1}} X_{\ol{h}}  ( x_{\lambda} g)
&=  \ABJRS( \mathrm{O}, E_{h+1, h+2},- \bs{\ep}_{h+1},  r  )
\cdot
 \ABJRS( \mathrm{O}, E_{h, h+1}, - \bs{\ep}_{h},  r  )  (x_{\lambda} g ) \\
& =
		-{v}^{ -{\lambda}_{h} -{\lambda}_{h+1} + 1}
		x_{(\lambda - \bs{\ep}_{h+2} + \bs{\ep}_{h})}
		c_{\widetilde{\lambda}_{h+1} + 1}
		\sum _{{\sigma} \in \D_{{\nu}_{J}} \cap \fS_{\co(J)}} T_{\sigma}
		c_{\widetilde{\lambda}_{h} + 1}
		\sum _{{\sigma} \in \D_{{\nu}_{K}} \cap \fS_{\co(K)}} T_{\sigma}
		\cdot g   \\
& =
		-{v}^{ -{\lambda}_{h} -{\lambda}_{h+1} + 1}
		x_{(\lambda - \bs{\ep}_{h+2} + \bs{\ep}_{h})}
		c_{\widetilde{\lambda}_{h+1} + 1}
		\TAIJ( {\widetilde{\lambda}_{h+1} +1}, {\widetilde{\lambda}_{h+1} + \lambda_{h+2} - 1} )
		c_{\widetilde{\lambda}_{h} + 1}
		\TAIJ( {\widetilde{\lambda}_{h}+1} , {\widetilde{\lambda}_{h} + \lambda_{h+1} - 1})
		\cdot g .
\end{align*}
Using the proof of Proposition \ref{phiupper0}(2), a straightforward calculation implies
\begin{align*}
X_{h} X_{h+1}  ( x_{\lambda} g)
&= \ABJRS( E_{h, h+1}, \mathrm{O}, - \bs{\ep}_{h},  r  ) \cdot \ABJRS( E_{h+1, h+2}, \mathrm{O},  - \bs{\ep}_{h+1},  r  )  (x_{\lambda} g ) \\
& =
		{v}^{ -{\lambda}_{h} -{\lambda}_{h+1} }
		x_{(\lambda - \bs{\ep}_{h+2} + \bs{\ep}_{h})}
		\TAIJ({\widetilde{\lambda}_{h}+1} , {\widetilde{\lambda}_{h} + \lambda_{h+1} })
		 \TAIJ( {\widetilde{\lambda}_{h+1} +1} , {\widetilde{\lambda}_{h+1} + \lambda_{h+2} - 1} )
		\cdot g , \\
X_{h+1} X_{h}  ( x_{\lambda} g)
&= \ABJRS( E_{h+1, h+2}, \mathrm{O},  - \bs{\ep}_{h+1},  r  ) \ABJRS( E_{h, h+1}, \mathrm{O},  - \bs{\ep}_{h},  r  )  (x_{\lambda} g ) \\
& =
		{v}^{ -{\lambda}_{h} -{\lambda}_{h+1} + 1}
		x_{(\lambda - \bs{\ep}_{h+2} + \bs{\ep}_{h})}
		\sum _{{\sigma} \in \D_{{\nu}_{J}} \cap \fS_{\co(J)}} T_{\sigma}
		\sum _{{\sigma} \in \D_{{\nu}_{K}} \cap \fS_{\co(K)}} T_{\sigma}
		\cdot g \\
& =
		{v}^{ -{\lambda}_{h} -{\lambda}_{h+1} + 1}
		x_{(\lambda - \bs{\ep}_{h+2} + \bs{\ep}_{h})}
		\TAIJ( {\widetilde{\lambda}_{h+1} +1}, {\widetilde{\lambda}_{h+1} + \lambda_{h+2} - 1} )
		\TAIJ( {\widetilde{\lambda}_{h}+1} , {\widetilde{\lambda}_{h} + \lambda_{h+1} - 1})
		\cdot g .
\end{align*}
Putting together, in order to prove
{$  X_{\ol{h}} X_{\ol{h+1}} + {v} X_{\ol{h+1}} X_{\ol{h}} = X_{h} X_{h+1} - {v} X_{h+1} X_{h} $},
it is enough to prove
{$
  {\fcY}_1 + {\fcY}_2 = {\fcY}_3 - {\fcY}_4
  $},
where
\begin{align*}
{\fcY}_1
&=
		-x_{(\lambda - \bs{\ep}_{h+2} + \bs{\ep}_{h})}
		c_{\widetilde{\lambda}_{h} + 1}
		 \TAIJ( {\widetilde{\lambda}_{h}+1} , {\widetilde{\lambda}_{h} + \lambda_{h+1} })
		c_{\widetilde{\lambda}_{h+1} + 1}
		 \TAIJ( {\widetilde{\lambda}_{h+1} +1}, {\widetilde{\lambda}_{h+1} + \lambda_{h+2} - 1}) , \\
{\fcY}_2
&=
		-{v}^{2}
		x_{(\lambda - \bs{\ep}_{h+2} + \bs{\ep}_{h})}
		c_{\widetilde{\lambda}_{h+1} + 1}
		 \TAIJ( {\widetilde{\lambda}_{h+1} +1}, {\widetilde{\lambda}_{h+1} + \lambda_{h+2} - 1})
		c_{\widetilde{\lambda}_{h} + 1}
		 \TAIJ( {\widetilde{\lambda}_{h}+1}, {\widetilde{\lambda}_{h} + \lambda_{h+1} - 1}	), \\
{\fcY}_3
&=
		x_{(\lambda - \bs{\ep}_{h+2} + \bs{\ep}_{h})}
		 \TAIJ(  {\widetilde{\lambda}_{h}+1} , {\widetilde{\lambda}_{h} + \lambda_{h+1} }	)
		 \TAIJ( {\widetilde{\lambda}_{h+1} +1} , {\widetilde{\lambda}_{h+1} + \lambda_{h+2} - 1}) , \\
{\fcY}_4
&=
		{v}^{2}
		x_{(\lambda - \bs{\ep}_{h+2} + \bs{\ep}_{h})}
		 \TAIJ( {\widetilde{\lambda}_{h+1} +1} , {\widetilde{\lambda}_{h+1} + \lambda_{h+2} - 1} )
		 \TAIJ( {\widetilde{\lambda}_{h}+1} , {\widetilde{\lambda}_{h} + \lambda_{h+1} - 1}	).
\end{align*}

Let {$\xi = (\lambda - \bs{\ep}_{h+2} + \bs{\ep}_{h})$},
denote {$ u = \widetilde{\lambda}_{h}$}, {$ w = \widetilde{\lambda}_{h+1}$},  {$ p = \widetilde{\lambda}_{h+2}$} for short,
then by Lemma \ref{xcT}(2) one can obtain
\begin{align*}
{\fcY}_1
&=
		- x_{\xi}  c_{u + 1} T_{ u+1, w}  c_{w + 1} T_{(w +1, p - 1)} \\
&=
	-{q}^{ w - u }x_{\xi}
	( c_{u+1} + T_{u+1}^{-1} c_{u+2} +  T_{u+1}^{-1} T_{u+2}^{-1}   c_{u+3}
		+ \cdots
		+ T_{u+1}^{-1}
		\cdots T_{w}^{-1}  c_{w +1}
	)  c_{w + 1} T_{(w +1, p - 1)}, \\
{\fcY}_2
&=
	-{v}^{2}
	x_{\xi}
	c_{w + 1}T_{(w +1, p - 1)}
		c_{u + 1}T_{(u+1, w-1)} \\
&=
	{q}^{ w - u }	x_{\xi}
	( c_{u+1} + T_{u+1}^{-1} c_{u+2} +  T_{u+1}^{-1} T_{u+2}^{-1}   c_{u+3}
	+ \cdots
	+ T_{u+1}^{-1}
	\cdots T_{w-1}^{-1}  c_{w} )
	c_{w + 1}T_{(w +1, p - 1)}.
\end{align*}
As a consequence,
\begin{align*}
{\fcY}_1 + {\fcY}_2
&= -	{q}^{ w - u }x_{\xi} ( c_{u+1} + T_{u+1}^{-1} c_{u+2} +  T_{u+1}^{-1} T_{u+2}^{-1}   c_{u+3}
	+ \cdots
	+ T_{u+1}^{-1}
	\cdots T_{w}^{-1}  c_{w +1}  )  c_{w + 1} T_{(w +1, p - 1)} \\
	& \qquad
	+ {q}^{ w - u }	x_{\xi}
	( c_{u+1} + T_{u+1}^{-1} c_{u+2} +  T_{u+1}^{-1} T_{u+2}^{-1}   c_{u+3}
	+ \cdots
	+ T_{u+1}^{-1}
	\cdots T_{w-1}^{-1}  c_{w} )
	c_{w + 1}T_{(w +1, p - 1)} \\
&=	 {q}^{ w - u }x_{\xi} T_{u+1}^{-1} T_{u+2}^{-1} T_{u+3}^{-1} \cdots T_{w}^{-1}   T_{(w +1, p - 1)} \\
&=	 {q}^{ w - u }
		(
			{q}^{  u-w  } x_{\xi}  T_{ u+1 }   T_{ u+2} \cdots T_{  w }
		- {q}^{ u- w  }  ({q} - 1)  x_{\xi}  T_{( u+1 ,   w-1 )}  )
		T_{(w +1, p - 1)}
		 \ \mbox{( by Lemma \ref{xTinverse}  )}
		\\
&=
	 {q}^{ w - u }   {q}^{ u - w} x_{\xi}T_{ u+1 }   T_{ u+2} \cdots T_{ w }  T_{(w +1, p - 1)}
	- {q}^{ w - u }   {q}^{ u -w }  ({q} - 1)   x_{\xi}  T_{( u+1 ,   w-1 )}   T_{(w +1, p - 1)}
	\\
&=
	x_{\xi}    T_{ u+1 }   T_{ u+2} \cdots T_{ w }  T_{(w +1, p - 1)}
	-  ({q} - 1)  x_{\xi}   T_{(u+1 ,   w - 1)}  T_{(w +1, p - 1)}.
\end{align*}
On the other hand, we have
\begin{align*}
{\fcY}_3
&=	x_{\xi} 	T_{(u+1, w)}	T_{(w +1, p - 1)}, \\
{\fcY}_4
&=		{v}^{2}	x_{\xi}  T_{(w +1, p - 1)} T_{(u+1, w-1)}
=		{q}	x_{\xi}  T_{(u+1, w-1)} T_{(w +1, p - 1)} ,
\end{align*}
and
\begin{align*}
{\fcY}_3 - {\fcY}_4
&=
		 x_{\xi}  T_{u+1} T_{u+2} \cdots T_{w}  T_{(w +1, p - 1)}
		- ( {q} - 1)x_{\xi}  T_{(u+1, w-1)} T_{(w +1, p - 1)}.
\end{align*}
Hence
{$  {\fcY}_1 + {\fcY}_2 = {\fcY}_3 - {\fcY}_4$}
and equation {\rm(1)} is proved.

{\rm(2)}
Similarly,
for any {$\lambda \in \CMN(n, r)$}  and homogeneous {$ g \in \HCR$},
one can prove that
\begin{align*}
  Y_{\ol{h+1}} Y_{\ol{h}} (x_{\lambda} g)
  &=-{v}^{ -{\lambda}_{h+1} -{\lambda}_{h+2}}
		x_{(\lambda - \bs{\ep}_{h} + \bs{\ep}_{h+2})}
		c_{\widetilde{\lambda}_{h+1} }
		 \TDIJ( {\widetilde{\lambda}_{h+1}  - 1}, {\widetilde{\lambda}_{h+1} - \lambda_{h+1}  } )
		c_{\widetilde{\lambda}_{h} }
		 \TDIJ( {\widetilde{\lambda}_{h}-1}, {\widetilde{\lambda}_{h} -  \lambda_{h} +1 })
		\cdot g  , \\
 Y_{h} Y_{h+1} (x_{\lambda} g)
 & =
		{v}^{ -{\lambda}_{h+1} -{\lambda}_{h+2} +1}
		x_{(\lambda - \bs{\ep}_{h} + \bs{\ep}_{h+2})}
		 \TDIJ( {\widetilde{\lambda}_{h}-1} , {\widetilde{\lambda}_{h} -  \lambda_{h} +1 })
		 \TDIJ( {\widetilde{\lambda}_{h+1}  - 1}, {\widetilde{\lambda}_{h+1} - \lambda_{h+1} + 1  } )
		\cdot g , \\
Y_{h+1} Y_{h}  (x_{\lambda} g)
& =
		{v}^{ -{\lambda}_{h+1} -{\lambda}_{h+2} }
		x_{(\lambda - \bs{\ep}_{h} + \bs{\ep}_{h+2})}
		 \TDIJ( {\widetilde{\lambda}_{h+1}  - 1}, {\widetilde{\lambda}_{h+1} - \lambda_{h+1}  } )
		 \TDIJ( {\widetilde{\lambda}_{h}-1}, {\widetilde{\lambda}_{h} -  \lambda_{h} +1 })
		\cdot g
\end{align*}
for any $\lambda\in\Lambda(n,r)$ and homogeneous element $g\in \HCR$.
Direct calculation shows
\begin{align*}
-( Y_{\ol{h}} Y_{\ol{h+1}} + {v} Y_{\ol{h+1}} Y_{\ol{h}})(x_{\lambda}  g)
&=- x_{(\lambda - \bs{\ep}_{h} + \bs{\ep}_{h+2})}
		T_{\widetilde{\lambda}_{h+1}-1} \cdots T_{\widetilde{\lambda}_{h+1} - \lambda_{h+1} }
		 \TDIJ( {\widetilde{\lambda}_{h}-1} , {\widetilde{\lambda}_{h} -  \lambda_{h} +1 }	)\cdot g\\
(Y_{h} Y_{h+1} - {v} Y_{h+1} Y_{h} )(x_{\lambda} g)
&=-
		x_{(\lambda - \bs{\ep}_{h} + \bs{\ep}_{h+2})}
		T_{\widetilde{\lambda}_{h+1}-1} \cdots T_{\widetilde{\lambda}_{h+1} - \lambda_{h+1}  }
		 \TDIJ( {\widetilde{\lambda}_{h}-1} , {\widetilde{\lambda}_{h} -  \lambda_{h} +1 }	)\cdot g
\end{align*}
and then {\rm(2)} is proved.
\end{proof}

\begin{thm}\label{qqschur_reltion}
There is  a superalgebra  homomorphism $\bs{\xi}_n: \Uvqn \to {\SQvnR} $ defined by
$${\genE}_{j} \mapsto X_{j}, \;
{\genE}_{\ol{j}} \mapsto X_{\ol{j}}, \;
{\genF}_{j} \mapsto Y_{j}, \;
{\genF}_{\ol{j}} \mapsto Y_{\ol{j}},  \; {\genK}_{i}^{\pm 1} \mapsto G_{i}^{\pm 1}, \;
{\genK}_{\ol{i}} \mapsto G_{\ol{i}},$$
for all $1 \le i \le n, 1 \le j \le n-1$.
\end{thm}
\begin{proof}
We verify all the relations (QQ1)--(QQ6) in Definition \ref{defqn}.

{\bf Relation (QQ1).} \
By Example \ref{exam_shift}, any diagonal matrix statisfies SDP condition at $(h,k)$ for $1\leq h,k\leq n$, and hence by Proposition \ref{mulformdiag} we have
for $1\leq i<j\leq n$ and {$1 \le k \le n$}
\begin{align*}
&\ABJS( \mathrm{O}, E_{i,i}, \bs{ 0 } )  \ABJS( \mathrm{O}, E_{j,j}, \bs{ 0 } )
	 = -\ABJS( \mathrm{O}, E_{j,j}+ E_{i,i}, \bs{ 0 } ), \\
&\ABJS( \mathrm{O}, E_{j,j}, \bs{ 0 } )  \ABJS( \mathrm{O}, E_{i,i}, \bs{ 0 } )
 =  \ABJS( \mathrm{O}, E_{j,j}+ E_{i,i}, \bs{ 0 } ),\\
&
\ABJS( \mathrm{O}, E_{k,k}, \bs{ 0 } )  \ABJS( \mathrm{O}, E_{k,k}, \bs{ 0 } )
=\frac{1}{({v}^4 - 1) }
		\{ \ABJS( \mathrm{O}, \mathrm{O}, 4 \bs{\ep}_{k} )  - \ABJS( \mathrm{O}, \mathrm{O}, \bs{ 0 } )
		  \}.
\end{align*}
Therefore, we obtain
\begin{align*}
G_{\ol{i}} G_{\ol{j}} + G_{\ol{j}} G_{\ol{i}}=0,\qquad
G_{\ol{k}}^2= \frac{  G_k^{2} - G_k^{-2}  }{{v}^2 - {v}^{-2}  }.
\end{align*}
The other equations in (QQ1) can be proved by direct calculation.

{\bf Relation (QQ2).} 
The equations can be proved directly by \eqref{mult-GA}.

{\bf Relation (QQ3).}
By Example \ref{exam_shift},
for any {$\lambda \in \CMN(n, r)$} for any {$r \ge 1$},
the matrices
{$(E_{i, i} + \lambda)$},
{$(E_{i, i+1} + \lambda)$},
{$(E_{i+1, i} + \lambda)$}
satisfy the SDP condition at {$(i, i)$}, {$(i, i+1)$}. This means the formulas in
Proposition \ref{mulformdiag}-\ref{mulformodd2} works here.
For the first relation,  we have
\begin{align*}
G_{\ol{i}} X_{i}
&=     {v} G_{i}^{-1}  \{\  \ABJS( \mathrm{O}, E_{i,i}, \bs{ 0 }) \ABJS(E_{i, i+1}, \mathrm{O},  \bs{0}) \ \} G_{i}^{-1} \\
&=     {v} G_{i}^{-1} \  \{ \  {v}^2 \ABJS( E_{i, i+1}, E_{i,i}, \bs{ 0 } )
		+ \ABJS( \mathrm{O},  E_{i,i+1}, \bs{ 0 } ) \  \} \   G_{i}^{-1}
	\quad	\mbox{ (by Proposition \ref{mulformdiag})}
		\\
&=
     \ABJS( E_{i, i+1}, E_{i,i}, -2 \bs{\ep}_{i} )
		+    \ABJS( \mathrm{O},  E_{i,i+1}, -2 \bs{\ep}_{i} ),  \\
X_{i} G_{\ol{i}}
&=    {v}^2  G_{i}^{-1} \ABJS(E_{i, i+1}, \mathrm{O}, \bs{0})
		\cdot \ABJS( \mathrm{O}, E_{i,i}, \bs{ 0 })  G_{i}^{-1} \\
&=   {v}^2  G_{i}^{-1} \ABJS( E_{i,i+1},E_{i,i}, \bs{ 0 } ) 	G_{i}^{-1}
\quad	\mbox{ (by Proposition \ref{mulformeven} {\rm(1)})}
\\
&=    {v}^{-1} \ABJS( E_{i,i+1},E_{i,i}, -2 \bs{\ep}_{i} ) ,\\
X_{\ol{i}} G_{i}^{-1}
	&=    \ABJS( \mathrm{O}, E_{i,i+1},  - 2\bs{\ep}_{i}  )
	\quad	\mbox{ (by Corollary \ref{mulformzerocor})} \\
&=G_{\ol{i}} X_{i}  - {v} X_{i} G_{\ol{i}}  .
\end{align*}
For the third one,
\begin{align*}
G_{\ol{i}} Y_{i}
&=    {v}   G_{i}^{-1}  \ABJS( \mathrm{O}, E_{i,i}, \bs{ 0 }) \ABJS(E_{i+1, i}, \mathrm{O}, \bs{0}) G_{i+1}^{-1}   \\
&=    {v} G_{i}^{-1}  \{ \ 
		\ABJS( \mathrm{O}, E_{i,i}, \bs{ 0 }) \ABJS(E_{i+1, i}, \mathrm{O}, \bs{0})
	\ \}G_{i+1}^{-1}
	\quad	\mbox{ (by Proposition \ref{mulformdiag})}
	\\
&=   \ABJS( E_{i+1, i},  E_{i,i},- \bs{\ep}_{i} - \bs{\ep}_{i+1}  ), \\
Y_{i} G_{\ol{i}}
&=      {v}^2  G_{i+1}^{-1}   \ABJS(E_{i+1, i}, \mathrm{O}, \bs{0})
		\cdot \ABJS( \mathrm{O}, E_{i,i}, \bs{ 0 }) G_{i}^{-1}   \\
&=     {v}^2 G_{i+1}^{-1}\   \{
	\ABJS( E_{i+1, i}, E_{i,i}, \bs{ 0 } )
	 + \ABJS(\mathrm{O}, E_{i+1,i},  2 \bs{\ep}_{i} )
  \} \  G_{i}^{-1}
\  	\mbox{ (by Proposition \ref{mulformeven} {\rm(2)})}
\\
&=
	   {v}^{-1} \ABJS( E_{i+1, i}, E_{i, i}, -\bs{\ep}_{i+1} -\bs{\ep}_{i} )
	+    \ABJS(\mathrm{O}, E_{i+1,i},   \bs{\ep}_{i} - \bs{\ep}_{i+1} ) , \\
Y_{\ol{i}} G_{i}
	& =    {v} \ABJS(\mathrm{O}, E_{i+1,i},    \bs{\ep}_{i} - \bs{\ep}_{i+1} )
		\quad	\mbox{ (by Corollary \ref{mulformzerocor})} \\
&=-G_{\ol{i}} Y_{i}  + {v} Y_{i} G_{\ol{i}} .
\end{align*}

For the second relation,
\begin{align*}
G_{\ol{i}}  X_{i-1}
&=   {v} G_{i}^{-1} \ABJS(\mathrm{O}, E_{i,i}, \bs{0} )  \ABJS(E_{i-1,i}, \mathrm{O},   \bs{0} )  G_{i-1}^{-1} \\
&=   {v} G_{i}^{-1} \ABJS( E_{i-1,i},  E_{i,i}, \bs{ 0 } ) G_{i-1}^{-1}
\quad	\mbox{ (by Proposition \ref{mulformdiag})}
\\
&=   \ABJS( E_{i-1,i},  E_{i,i},  -\bs{\ep}_{i-1} - \bs{\ep}_{i} ) , \\
X_{i-1}  G_{\ol{i}}
&=   {v}^2 G_{i-1}^{-1} \{ \ABJS(E_{i-1,i}, \mathrm{O},  \bs{0} ) \ABJS(\mathrm{O}, E_{i,i},  \bs{0} ) \} G_{i}^{-1} \\
&=   {v}^2 G_{i-1}^{-1} \{
{v}^2	\ABJS( E_{i-1,i}, E_{i,i}, \bs{ 0 } )
	 + 	\ABJS( \mathrm{O}, E_{i-1,i}, \bs{ 0 } )
 \} G_{i}^{-1}
 \ 	\mbox{ (by Proposition \ref{mulformeven} {\rm(1)})}
 \\
&=   {v} 	\ABJS( E_{i-1,i}, E_{i,i}, -\bs{\ep}_{i-1} - \bs{\ep}_{i} )
		+  	\ABJS( \mathrm{O}, E_{i-1,i}, -\bs{\ep}_{i-1} - \bs{\ep}_{i} )  ,\\
-G_{i}^{-1} X_{\ol{i-1}}
&=
 - \ABJS( \mathrm{O}, E_{i-1,i}, - \bs{\ep}_{i-1} - \bs{\ep}_{i} )
	\quad	\mbox{ (by Corollary \ref{mulformzerocor})} \\
&= {v} G_{\ol{i}}  X_{i-1} - X_{i-1}  G_{\ol{i}}.
\end{align*}
The  fourth relation is analogous to the second one,
\begin{align*}
G_{\ol{i}}  Y_{i-1}
&=   {v} G_{i}^{-1} \ABJS(\mathrm{O}, E_{i,i}, \bs{0} )  \ABJS(E_{i,i-1}, \mathrm{O},   \bs{0} )  G_{i}^{-1} \\
&=   {v} G_{i}^{-1} \{
		\ABJS( E_{i,i-1}, E_{i,i}, \bs{ 0 } ) +
		\ABJS( \mathrm{O}, E_{i,i-1}, 2 \bs{\ep}_{i+1} )
\} G_{i}^{-1}
\quad	\mbox{ (by Proposition \ref{mulformdiag})}
\\
&=
		  {v}^{-2} \ABJS( E_{i,i-1}, + E_{i,i}, -2 \bs{\ep}_{i+1} )
		+   \ABJS( \mathrm{O}, E_{i,i-1},  \bs{ 0 } ),  \\
 Y_{i-1}  G_{\ol{i}}
&=   {v}^2 G_{i}^{-1} \{ \ABJS(E_{i,i-1}, \mathrm{O},  \bs{0} ) \ABJS(\mathrm{O}, E_{i,i},  \bs{0} ) \} G_{i}^{-1} \\
&=   {v}^2 G_{i}^{-1} \ABJS(E_{i,i-1},   E_{i,i}, \bs{ 0 } )  G_{i}^{-1}
\quad	\mbox{ (by Proposition \ref{mulformeven} {\rm(2)})}
\\
&=   {v}^{-1} \ABJS( E_{i,i-1},  E_{i,i},  - 2 \bs{\ep}_{i} ) , \\
G_{i} Y_{\ol{i-1}}
&=  {v}  \ABJS(\mathrm{O}, E_{i,i-1}, \bs{0} ),
	\quad	\mbox{ (by \eqref{mult-GA} )}
	\\
&={v} G_{\ol{i}}  Y_{i-1} - Y_{i-1}  G_{\ol{i}}  .
\end{align*}

The other relations in (QQ3) are analogous to that of the first four equations,
we omit the proof.

{\bf Relation (QQ4).}
The first  relation  can be verified directly using Proposition \ref{mulformeven}.
We omit the details.
The other three relations follow from Lemma \ref{mulform_ef}.

{\bf Relation (QQ5).}
By Example \ref{exam_shift},
for any {$\lambda \in \CMN(n, r)$},
{${(2E_{i, i+1} + \lambda)}^{+}_{i,i+1}$}
satisfies the SDP condition at {$(i, i+1)$},
and {$(2E_{i+1, i} + \lambda)$} satisfies the SDP condition on the $i$-th row.
So the formulas in Proposition \ref{mulformodd1}, \ref{mulformodd2} are applicable,
then one can obtain
\begin{align*}
X_{\ol{i}} ^2
&=  {v}G_{i}^{-1}  \ABJS( \mathrm{O}, E_{i,i+1}, \bs{0} ) \cdot \ABJS( \mathrm{O}, E_{i,i+1}, \bs{0} )  G_{i}^{-1}  \\
&= -{v}^{-1} ({v}^2 - 1) \ABJS( 2 E_{i,i+1}, \mathrm{O}, - 2 \bs{\ep}_{i}  );\\
X_{i} ^2
&= {v}G_{i}^{-1}  \{ \ABJS(E_{i,i+1},  \mathrm{O}, \bs{0} ) \cdot \ABJS(E_{i,i+1},  \mathrm{O}, \bs{0} ) \}  G_{i}^{-1} \\
&= {v}^{-1}( 1 + {v}^2 ) 	\ABJS( 2 E_{i,i+1}, \mathrm{O}, - 2 \bs{\ep}_{i} ) ; \\
- \frac{{v} - {v}^{-1}}{{v} + {v}^{-1}} X_{i}^2
&=- ({v}^2 - 1) \cdot
	{v}^{-1} 	\ABJS( 2 E_{i,i+1}, \mathrm{O}, - 2 \bs{\ep}_{i} )  = X_{\ol{i}} ^2.
\end{align*}
Similarly, we can show
\begin{align*}
  \frac{{v} - {v}^{-1}}{{v} + {v}^{-1}} Y_{i}^2
&=  ({v}^2 - 1)
	{v}^{-1}  \ABJS( 2E_{i+1, i}, \mathrm{O},  - 2 \bs{\ep}_{i+1} ) =Y_{\ol{i}}^2 .
\end{align*}

For the third relation, if {$j = i$},
\begin{align*}
X_{i} X_{\ol{i}}
&=   {v}  G_{i}^{-1}\ABJS( E_{i,i+1}, E_{i, i+1}, \bs{ 0 } )  G_{i}^{-1}, \\
X_{\ol{i}} X_{i}
&=   {v}  G_{i}^{-1} \ABJS( E_{i, i+1}, E_{i,i+1}, \bs{ 0} )  G_{i}^{-1},
\end{align*}
and hence $X_{i} X_{\ol{i}} - X_{\ol{i}} X_{i}= 0.$

When {$j > i+1 $} or {$j < i-1 $},
\begin{align*}
X_{i} X_{\ol{j}}
&=    {v}  G_{i}^{-1} \ABJS(E_{i, i+1}, \mathrm{O}, \bs{0})  \ABJS(\mathrm{O}, E_{j,j+1}, \bs{0} ) G_{j}^{-1} \\
&=   \ABJS( E_{i,i+1}, E_{j,j+1},  - \bs{\ep}_{i} - \bs{\ep}_{j} ) , \\
X_{\ol{j}} X_{i}
&=   {v}  G_{j}^{-1} \ABJS(\mathrm{O}, E_{j,j+1}, \bs{0} )   \ABJS(E_{i, i+1}, \mathrm{O}, \bs{0})   G_{i}^{-1} \\
&=  \ABJS( E_{i,i+1}, E_{j,j+1},  - \bs{\ep}_{i} - \bs{\ep}_{j} ) ,
\end{align*}
and hence $X_{i} X_{\ol{j}} - X_{\ol{j}} X_{i}= 0.$
The proofs for the fifth and seventh relations are given in Lemma \ref{relation_eei},
and the proof of  the rest  relations  are omitted.


{\bf Relation (QQ6).}
The first two relations are similar to the nonsuper case and can be directly proved using Proposition \ref{mulformeven}. We omit the details.
For the third relation, we only give the proof in the case {$j=i-1$} ,
while the case $j=i+1$ can be proved similarly. That is, we shall prove
\begin{align}\label{serre-X}
& X_{i}^2 X_{\ol{i-1}} - ( {v} + {v}^{-1} ) X_{i} X_{\ol{i-1}} X_{i} + X_{\ol{i-1}}  X_{i}^2 = 0,
\end{align}
By replacing {$i$} with {$i+1$}, it suffices to prove
\begin{align*}
& X_{i+1}^2 X_{\ol{i}} - ( {v} + {v}^{-1} ) X_{i+1} X_{\ol{i}} X_{i+1} + X_{\ol{i}}  X_{i+1}^2 = 0,
\end{align*}
Direct calculation shows
\begin{align*}
& X_{i+1}^2 X_{\ol{i}} - ( {v} + {v}^{-1} ) X_{i+1} X_{\ol{i}} X_{i+1} + X_{\ol{i}}  X_{i+1}^2 \\
&=	- {v}^{-1}X_{i+1} (X_{\ol{i}} X_{i+1} -  {v}X_{i+1} X_{\ol{i}} )
	+ (X_{\ol{i}}  X_{i+1}  - {v} X_{i+1} X_{\ol{i}} )X_{i+1}.
\end{align*}
As we have proved (QQ5), the following holds
\begin{align*}
  X_{i} X_{\ol{i+1}} - {v} X_{\ol{i+1}} X_{i}
	= X_{\ol{i}} X_{i+1} - {v} X_{i+1} X_{\ol{i}}.
\end{align*}
Putting together we have
\begin{align*}
& X_{i+1}^2 X_{\ol{i}} - ( {v} + {v}^{-1} ) X_{i+1} X_{\ol{i}} X_{i+1} + X_{\ol{i}}  X_{i+1}^2 \\
&=	 - {v}^{-1}X_{i+1}  X_{i} X_{\ol{i+1}} +  X_{i+1}  X_{\ol{i+1}} X_{i}
	+ X_{i} X_{\ol{i+1}}X_{i+1} - {v} X_{\ol{i+1}} X_{i} X_{i+1} \\
&=	 -  {v}^2 G_{i+1}^{-2}  G_{i}^{-1}  \ABJS(E_{i+1, i+2}, \mathrm{O}, \bs{0})  \cdot
	\ABJS(E_{i, i+1}, \mathrm{O}, \bs{0}) \cdot
		 \ABJS(\mathrm{O}, E_{i+1, i+2}, \bs{0}) \\
	& \qquad  +  {v}^4 G_{i+1}^{-2}  G_{i}^{-1}  \ABJS(E_{i+1, i+2}, \mathrm{O}, \bs{0})  \cdot
				\ABJS(\mathrm{O}, E_{i+1, i+2}, \bs{0}) \cdot
				\ABJS(E_{i, i+1}, \mathrm{O}, \bs{0})   \\
	& \qquad + {v}^2 G_{i}^{-1} G_{i+1}^{-2} \ABJS(E_{i, i+1}, \mathrm{O}, \bs{0}) \cdot
		 \ABJS(\mathrm{O}, E_{i+1, i+2},  \bs{0})  \cdot
		\ABJS(E_{i+1, i+2}, \mathrm{O}, \bs{0})  \\
	& \qquad - {v}^4   G_{i+1}^{-2} G_{i}^{-1} \ABJS(\mathrm{O}, E_{i+1, i+2}, \bs{0})  \cdot
		\ABJS(E_{i, i+1}, \mathrm{O}, \bs{0})  \cdot
		\ABJS(E_{i+1, i+2}, \mathrm{O}, \bs{0})  .
\end{align*}
Applying Proposition \ref{mulformeven}{\rm(1)} and  Proposition \ref{mulformodd1}, we have
\begin{align*}
&\ABJS(E_{i+1, i+2}, \mathrm{O}, \bs{0})  \cdot
	\ABJS(E_{i, i+1}, \mathrm{O}, \bs{0}) \cdot
		 \ABJS(\mathrm{O}, E_{i+1, i+2}, \bs{0}) \\
&= \ABJS( E_{i,i+1}+ E_{i+1,i+2}, E_{i+1, i+2}, \bs{0} ) + \ABJS( E_{i+1,i+2}, E_{i,i+2}, \bs{0} ) , \\
&\ABJS(E_{i+1, i+2}, \mathrm{O}, \bs{0})  \cdot
				\ABJS(\mathrm{O}, E_{i+1, i+2}, \bs{0}) \cdot
				\ABJS(E_{i, i+1}, \mathrm{O}, \bs{0})  \\
&=\ABJS( E_{i, i+1} + E_{i+1,i+2}, E_{i+1,i+2}, \bs{0} ) , \\
&  \ABJS(E_{i, i+1}, \mathrm{O}, \bs{0}) \cdot
		 \ABJS(\mathrm{O}, E_{i+1, i+2}, \bs{0}) \cdot
		\ABJS(E_{i+1, i+2}, \mathrm{O}, \bs{0})  \\
&=	\ABJS( E_{i+1, i+2} + E_{i,i+1}, E_{i+1, i+2}, \bs{0} )
	+ {v}^{ 2 } \ABJS( E_{i,i+2}, E_{i+1, i+2}, \bs{0} )
	+ \ABJS( E_{i+1, i+2}, E_{i, i+2}, \bs{0} ) , \\
&\ABJS(\mathrm{O}, E_{i+1, i+2}, \bs{0})  \cdot
		\ABJS(E_{i, i+1}, \mathrm{O}, \bs{0})  \cdot
		\ABJS(E_{i+1, i+2}, \mathrm{O}, \bs{0}) \\
&=\ABJS(E_{i+1, i+2} + E_{i,i+1}, E_{i+1,i+2}, \bs{0} )
	+ \ABJS(  E_{i,i+2}, E_{i+1,i+2}, \bs{0} ).
\end{align*}
This leads to
\begin{align*}
0 =
& -  \ABJS(E_{i+1, i+2}, \mathrm{O}, \bs{0})  \cdot
	\ABJS(E_{i, i+1}, \mathrm{O}, \bs{0}) \cdot
		 \ABJS(\mathrm{O}, E_{i+1, i+2}, \bs{0}) \\
	& \qquad  +  {v}^2 \ABJS(E_{i+1, i+2}, \mathrm{O}, \bs{0})  \cdot
				\ABJS(\mathrm{O}, E_{i+1, i+2}, \bs{0}) \cdot
				\ABJS(E_{i, i+1}, \mathrm{O}, \bs{0})   \\
	& \qquad +  \ABJS(E_{i, i+1}, \mathrm{O}, \bs{0}) \cdot
		 \ABJS(\mathrm{O}, E_{i+1, i+2}, \mathrm{O}, \bs{0}) \cdot
		\ABJS(E_{i+1, i+2}, \mathrm{O}, \bs{0})  \\
	& \qquad - {v}^2 \ABJS(\mathrm{O}, E_{i+1, i+2}, \bs{0})  \cdot
		\ABJS(E_{i, i+1}, \mathrm{O}, \bs{0})  \cdot
		\ABJS(E_{i+1, i+2}, \mathrm{O}, \bs{0}) .
\end{align*}
Hence the equation \eqref{serre-X} is proved.

For the last relation,
because for any {$\lambda \in \CMN(n, r)$} with {$r>0$},
we have {$d_A = 1$} when {$A=\lambda + E_{i+1, i}$} or {$\lambda + 2E_{i+1, i}$}.
By Lemma \ref{lem_dA1},
Proposition \ref{mulformodd2}  can be used for the conditions {$j=i \pm 1$}.
Then the proof is analogous to the second relation.
\end{proof}

In the last two sections, we determine the image of the superalgebra homomorphism in Theorem \ref{qqschur_reltion} and prove that $\bs{\xi}_n$ is injective. 

We end this section with an application of the relations established above. 
Let 
\begin{equation}\label{setG}
 \fsG = \{ G_{i}, G_{i}^{-1}, G_{\ol{i}},
	X_{j}, X_{\ol{j}},
	Y_{j}, Y_{\ol{j}}
	\where  1 \le i \le n,\  1 \le j \le n-1
	\}
\end{equation}.

\begin{cor}\label{common_form}
For any {$A \in \MNZNS(n)$},
and any $Z  \in\fsG $, 
there exist $g_{B,\bs{j}}(Z,A) \in \Qv $, 
for some $B \in \MNZNS(n)$ and $\bs{j}\in {\ZZ}^n$, 
such that in $\SQvnR$ 
\begin{align*}
Z  \cdot \AJS(A, \bs{0})
&=
	\sum_{B, \bs{j}} g_{B, \bs{j}}(Z,A )  \AJS(B, \bs{j} ).
\end{align*}
\end{cor}
\begin{proof} If $Z$ is even (i.e., $Z\in\{G_{i}, G_{i}^{-1},X_{j},Y_{j}\}$), the assertion follows from Propositions \ref{mulformzero} and \ref{mulformeven}. 
If $Z$ is odd, prove the existence of the formula by induction. 
First, by Proposition \ref{mulformzero} and Corollary \ref{mulformdiagN},
the assertion is true for $Z=G_{\ol{n}}$. 
Similar to Remark \ref{induction_N},
from the relations in (QQ3) and (QQ4),
we  have
\begin{equation}\label{induction_XYG}
\begin{aligned}
{\rm(a)}&\;\;X_{\ol{i}}
= - {v} G_{i+1} G_{\ol{i+1}} X_{i} + {v}^{-1} X_{i} G_{i+1}  G_{\ol{i+1}} , \\
{\rm(b)}&\;\;Y_{\ol{i}}
= {v}  G_{i+1}^{-1} G_{\ol{i+1}} Y_{i} - {v}^{-1} Y_{i} G_{i+1}^{-1}  G_{\ol{i+1}} ,\\
{\rm(c)}&\;\;G_{\ol{i}}=
	X_{i} Y_{\ol{i}}  G_{i+1}  - Y_{\ol{i}} X_{i} G_{i+1}  +  G_{i}^{-1} G_{\ol{i+1}} G_{i+1}.
\end{aligned}
\end{equation}
Thus, the assertion is true for $Z=X_{\ol{n-1}},Y_{\ol{n-1}}$ by \eqref{induction_XYG}(a),(b), 
and thus for $Z=G_{\ol{n-1}}$ by \eqref{induction_XYG}(c).
 Now, a downward induction from $n-1$ to 1 prove the assertion for all odd $Z$.
\end{proof}

\begin{rem}\label{rem:induct-mult}By Propositions \ref{mulformzero} and \ref{mulformeven} and Corollary \ref{mulformdiagN},
there are explicit formulas for the coefficients $g_{B, \bs{j}}(Z,A )$ for all $A$ and $Z\in\{G_{i}, G_{i}^{-1},X_{j},Y_{j},G_{\ol n}\}$. For the remaining generators, the coefficients are only inductively defined. However, with the new realization given in Theorem \ref{map_iso}, the explicit multiplication formulas for $Z\in\{G_{i}, G_{i}^{-1},X_{j},Y_{j},G_{\ol n}\}$ are sufficient to define a new presentation for $\Uvqn$; see Remark \ref{induction_N}.
\end{rem}

\spaceintv
\section{Two triangularly related bases for $\USnv$}\label{sec_generators}
We are now ready to present a new realization for $\Uvqn$ 
via the twisted  quantum queer Schur superalgebra ${\SQvnrR}$ for $r\geq 0$ in the next two sections. 
By Theorem \ref{qqschur_reltion} and Corollary \ref{common_form}, 
it suffices to prove that $\bs{\xi}_n$ is an isomorphism.
 In this section, we determine the image of $\bs{\xi}_n$.

The following result can be proved by a method 
similar  to the  proof of \cite[Proposition 4.1{\rm(2)}]{DF2}. 
Recall the $\Qv$-space $\USnv$ defined in \eqref{Anv}.
\begin{prop}\label{unv_basis}
The set {$\LS =\{  \AJS(A, \bs{j})  \where   A \in \MNZNS(n), \bs{j} \in {\ZZ}^{n} \} $} is a {$\Qv$}-basis of {$\USnv$}.
\end{prop}
For any {$A \in \MNZNS(n)$},
let
\begin{align*}
& \co^{-}(A)_{\min} =
\left\{
\begin{aligned}
& n , \qquad \mbox{ if } a_{i,j} = 0 \mbox{ for all } i >j; \\
& \min \{ j \in \NN \where \sum_{i=j+1}^{n}a_{i, j} > 0\} \in [1, n-1] ,
	\qquad \mbox{ otherwise}.
\end{aligned}
\right.
\end{align*}
For {$\alpha, \beta \in \NN^t$},
denote {$\alpha < \beta $} if there exists {$k \le t$} such that
{$\alpha_k <  \beta_k $}
and
{$ \alpha_i = \beta_i $} for all {$i < k$}.
For {$A=(\SEE{a}_{i,j}| \SOE{a}_{i,j}), B=(\SEE{b}_{i,j}| \SOE{b}_{i,j}) \in \MNZN(n)$}, and {$1 \le k \le n-1$},  we say {$ B <_{k} A $} if one of the following conditions hold:
\begin{enumerate}
\item {$\sum_{i=k+1}^{n} a_{i,k}  > \sum_{i=k+1}^{n} b_{i,k}$},   or
\item {$\sum_{i=k+1}^{n} a_{i,k}  = \sum_{i=k+1}^{n} b_{i,k} $}
	 and  {$( \SOE{a}_{k, k}, \SOE{a}_{k+1, k}, \cdots \SOE{a}_{n, k} )  > ( \SOE{b}_{k, k}, \SOE{b}_{k+1, k}, \cdots \SOE{b}_{n, k} ) $}, or
\item {$\sum_{i=k+1}^{n} a_{i,k}  = \sum_{i=k+1}^{n} b_{i,k} $},
	{$ (\SOE{a}_{k, k}, \SOE{a}_{k+1, k}, \cdots \SOE{a}_{n, k} )  = (\SOE{b}_{k, k}, \SOE{b}_{k+1, k}, \cdots \SOE{b}_{n, k} )$},  \\
	 and {$(a_{k, k}, a_{k+1, k}, \cdots a_{n, k} )  < (b_{k, k}, b_{k+1, k}, \cdots b_{n, k} )$}.
\end{enumerate}

We say {$ A =_{k} B $} if
{$\SEE{a}_{i, k} = \SEE{b}_{i, k}$},
{$\SOE{a}_{i, k} = \SOE{b}_{i, k}$}
for all {$i \in [k, n]$}.
We say {$A =^+ B$} if {${a}_{i, j} = {b}_{i, j}$}
for all {$i  < j$},
and denote  {$A {\ne}^+ B$} if there exists a pair {$(i, j)$} such that {$i < j$} and   {${a}_{i, j} \ne {b}_{i, j}$}.

\begin{defn}\label{defn:partialorder}
For {$A, B \in \MNZN(n)$}, we say {$ B \prec A $} if one of the following conditions holds:
\begin{enumerate}
\item {$\sum_{i \le s, j\ge t } b_{i,j} \le \sum_{i \le s, j\ge t } a_{i,j}$}
	 for all {$s<t $}, and {$A {\ne}^+ B$};  or
\item {$A =^+ B$}, {$\co^{-}(A)_{\min} < \co^{-}(B)_{\min}$};  or
\item {$A =^+ B$}, {$\co^{-}(A)_{\min} = \co^{-}(B)_{\min}=k <n$}
	 and  {$B <_{k} A$}; or
\item {$A =^+ B$}, {$\co^{-}(A)_{\min} = \co^{-}(B)_{\min}=k <n$},
	 there exists  {$ t \in [k+1,n-1] $}  such that  {$B <_{t} A $}
	 and for all {$ j < t$} , {$A =_{j} B$}.
\end{enumerate}
\end{defn}

Notice for any {$\bs{j} \in {\ZZ}^{n}$},
\begin{align}
 \AJS(A, \bs{j}) = g \cdot (\prod_{h=1}^{n}G_{h}^{{j}_h}  ) \cdot \AJS(A, \bs{0}) ,
 	\qquad \mbox{ where } g \in \Qv, \label{Aj_0}
\end{align}
it means we can get {$\AJS(A, \bs{j})$}  from  {$\AJS(A, \bs{0})$}, or conversely.

\begin{lem}\label{triang_aaa}
For any {$A \in \MN(n)$} and {$k \in [1,n]$},
if {$a_{i,j}=0$} for any {$i>j$},
or {$i<j$} but {$j>k$},
then $A$ satisfies the SDP condition at  {$(k-1,k)$} and  {$(k,k)$}.
\end{lem}
\begin{proof}
Denote {$\lambda = \ro(A)$}.
To prove the SDP condition at {$(k-1,k)$},
we need to prove for any {$p$} satisfying {$0 \le p \le a_{k-1,k}-1$},
\begin{displaymath}
	c_{\widetilde{\lambda}_{k-2} + a_{k-1, k-1} + p + 1} {T_{d_A}} =  {T_{d_A}} c_{\widetilde{a}_{k-2,k} + p + 1}.
\end{displaymath}

By assumption, $A$ can be written as
\begin{equation}\label{eq:special matrix}
A =
	\begin{pmatrix}
		a_{1,1} &a_{1,2}	&\cdots &a_{1,k-1} 	&a_{1,k}	&0 	&\cdots &0  \\
		0 	&a_{2,2}	&\cdots &a_{2,k-1}	&a_{2,k}	&0 	&\cdots &0  \\
		\vdots 	&\vdots		&\cdots &\vdots 	&\vdots		&\vdots  	&\cdots&\vdots     \\
		0	&0	 	&\cdots &a_{k-2,k-1} 	&a_{k-2,k}	&0  	&\cdots&0  \\
		0	&0	 	&\cdots &a_{k-1,k-1} 	&a_{k-1,k}	&0  	&\cdots&0  \\
		0	&0	 	&\cdots &0		&a_{k,k} 	&0 	&\cdots&0  \\
		0	&0	 	&\cdots &0		&0		&a_{k+1,k+1} &\cdots &0  \\
		\vdots 	&\vdots   	&\cdots &\cdots 	&\vdots		&\vdots	&\cdots&\vdots     \\
		0	 &0	 	&\cdots &0		&0	 	&0 	&\cdots &a_{n,n}  \\
	\end{pmatrix} .
\end{equation}
By \eqref{d_A} and \eqref{wij},
direct calculation shows
{$\sigma_{i-1,k} = \widetilde{a}_{i-1,k}$} when  {$2 \le i \le k$},
which means {$w_{i,k} = 1$}  when {$2 \le i \le k$}.
As a consequence, we have
\begin{align*}
d_A = w_{2,2} \cdot (w_{2,3} w_{3,3} )
	\cdots (w_{2,k-1} w_{3,k-1} \cdots w_{k-1,k-1} ).
\end{align*}
Recall
\begin{align*}
	w_{i,j} =
		&(s_{  \widetilde{\lambda}_{i-1} + \sum_{u > i-1, p < j}  a_{u,p} } s_{ \widetilde{\lambda}_{i-1} + \sum_{u > i-1, p < j}  a_{u,p}  - 1} \cdots s_{\widetilde{a}_{i-1,j} +1}) \\
		&(s_{  \widetilde{\lambda}_{i-1} + \sum_{u > i-1, p < j}  a_{u,p}  + 1} s_{ \widetilde{\lambda}_{i-1} + \sum_{u > i-1, p < j}  a_{u,p}  } \cdots s_{\widetilde{a}_{i-1,j} +2}) \\
		&\cdots \\
		&(s_{ \widetilde{\lambda}_{i-1} + \sum_{u > i-1, p < j}  a_{u,p}  +a_{i,j}-1} s_{ \widetilde{\lambda}_{i-1} + \sum_{u > h-1, p < j}  a_{u,p}  + a_{i,j}-2 } \cdots s_{\widetilde{a}_{i,j}}).
\end{align*}
When {$1 \le i \le k-2$} and {$i\leq j\leq k-1$},
 we have
\begin{align*}
{\widetilde{\lambda}_{k-2} + a_{k-1, k-1} + p + 1}  >  \widetilde{\lambda}_{i-1} + \sum_{u > i-1, p <j}  a_{u,p}  +a_{i,j},
\end{align*}
which means in this case,
\begin{displaymath}
	c_{\widetilde{\lambda}_{k-2} + a_{k-1, k-1} + p + 1}  {T_{w_{i,j}}}
	=  {T_{w_{i,j}}} c_{\widetilde{\lambda}_{k-2} + a_{k-1, k-1} + p + 1}.
\end{displaymath}
For  the case $i=j=k-1$,
 observe that {$ \sum_{u > k-2, p < k-1}  a_{u,p} = 0$} and moreover
we have
\begin{align*}
w_{k-1,k-1}=
		&(s_{  \widetilde{\lambda}_{k-2}  } s_{ \widetilde{\lambda}_{k-2}  - 1} \cdots s_{\widetilde{a}_{k-2,k-1} +1}) \\
		&(s_{  \widetilde{\lambda}_{k-2}  + 1} s_{ \widetilde{\lambda}_{k-2}   } \cdots s_{\widetilde{a}_{k-2,k-1} +2}) \\
		&\cdots \\
		&(s_{ \widetilde{\lambda}_{k-2}  +a_{k-1,k-1}-1} s_{ \widetilde{\lambda}_{k-2} + a_{k-1,k-1}-2 } \cdots s_{\widetilde{a}_{k-1,k-1}}).
\end{align*}
As $
	({\widetilde{\lambda}_{k-2} + a_{k-1, k-1} + p + 1} ) - ( \widetilde{\lambda}_{k-2}  +a_{k-1,k-1}-1)
	=  p + 2
	\ge 2,
$
we obtain
\begin{displaymath}
	c_{\widetilde{\lambda}_{k-2} + a_{k-1, k-1} + p + 1}  {T_{w_{k-1, k-1}}}
	=  {T_{w_{k-1, k-1}}} c_{\widetilde{\lambda}_{k-2} + a_{k-1, k-1} + p + 1}  ,
\end{displaymath}
and hence
\begin{displaymath}
	c_{\widetilde{\lambda}_{k-2} + a_{k-1, k-1} + p + 1}  {T_{d_A}}
	=  {T_{d_A}} c_{\widetilde{\lambda}_{k-2} + a_{k-1, k-1} + p + 1}
	=  {T_{d_A}} c_{\widetilde{a}_{k-2, k} + p + 1} .
\end{displaymath}

To prove the SDP condition at {$(k,k)$},
we need to prove for any {$p$} satisfying {$0 \le p \le a_{k,k}-1$},
\begin{displaymath}
	c_{\widetilde{\lambda}_{k-1} + p + 1} {T_{d_A}} =  {T_{d_A}} c_{\widetilde{a}_{k-1,k} + p + 1}.
\end{displaymath}
Similar to the case for {$(k-1, k)$},
 we have
{$
{\widetilde{\lambda}_{k-2} + a_{k-1, k-1} + p + 1}  >  \widetilde{\lambda}_{i-1} + \sum_{u > i-1, p <j}  a_{u,p}  +a_{i,j}
$} for all {$i,j$} satisfying {$1 \le i \le j  \le k-1$},
then
\begin{align*}
	c_{\widetilde{\lambda}_{k-1} + p + 1}  {T_{d_A}}
	=  {T_{d_A}} c_{\widetilde{\lambda}_{k-1} + p + 1}
	=  {T_{d_A}} c_{\widetilde{a}_{k-1, k} + p + 1} .
\end{align*}
\end{proof}

Recall that $\ABSUM{A}= \SE{A} + \SO{A}$.
\begin{cor}\label{triang_aaa_cor}
Let  {$A \in \MNZNS(n)$}
and {$\lambda \in \CMN(n, r)$} for any {$r > \snorm{A}$}.
Assume there is a {$k \in [2, n]$} such that
{$a_{i,j}=0$} when {$i > j$}, or  {$j > k$}.
Then
\begin{enumerate}
\item
{${(A + \lambda)}_{k-1,k}^+$} satisfies the SDP condition at {$(k-1,k)$},
\item
{$ {(A + \lambda)}$} satisfies the SDP condition at  {$(k,k)$}.
\end{enumerate}
\end{cor}
\begin{proof}
Clearly {${\ABSUM{A + \lambda}}_{k-1,k}^+ $},  {${\ABSUM{A + \lambda}} \in \MN(n)$}
are of the form \eqref{eq:special matrix} in Lemma \ref{triang_aaa}. Hence the corollary holds.
\end{proof}

Following \cite{GLL},
for  {$u \le k$},
we apply the convention
for the orders of products:
\begin{equation}\label{eq_product}
\begin{aligned}
 		\prod_{u \le h \le k}{\ft}_h:={\ft}_{u} {\ft}_{u+1} \cdots {\ft}_{k} , \quad
		\prod_{k \ge  h \ge u} {\ft}_h := {\ft}_{k} {\ft}_{k-1} \cdots {\ft}_{u}.
\end{aligned}
\end{equation}
Then we set
\begin{equation}\label{def_udl}
\begin{aligned}
&
U^{A}_{i,j} 
={( X_{i}     X_{i+1}  \cdots   X_{j-2} X_{\ol{j-1}})}^{\SOE{a}_{i,j}}
	X_{i}^{\SEE{a}_{i,j}}    X_{i+1}^{\SEE{a}_{i,j}} \cdots   X_{j-2}^{\SEE{a}_{i,j}} X_{{j-1}}^{\SEE{a}_{i,j}} ,
	 \qquad &&\mbox{for all } 1 \le i < j \le n, \\
&D^{A}_{i,i} =  G_{\ol{i}}^{\SOE{a}_{i,i}} ,
	\qquad && \mbox{for all }   1 \le i   \le n, \\
&L^{A}_{i,j} =
	 {(X_{i} X_{i+1}  \cdots   X_{n-1}  G_{\ol{n}} Y_{n-1}  \cdots Y_{i+1}   Y_{i})}^{\SOE{a}_{i,j} }
	 \cdot  Y_{i-1}^{\sum_{k=i}^n  {a}_{k,j}}  ,
	 \qquad  && \mbox{for all } 1 \le j < i \le n.
\end{aligned}
\end{equation}
We also define
\begin{equation}\label{MA}
 {\frm}^A  :={\prod_{{1} \le j \le {n-1}}} \Big({\prod_{{n} \ge i \ge {j+1}}} L^{A}_{i,j}\Big)
	\cdot
	{\prod_{{n} \ge j \ge {1}}}  \Big[  D^{A}_{j,j}\big ( {\prod_{{j-1} \ge i \ge {1}}} U^{A}_{i,j}\big) \Big ].
\end{equation}
Notice that the expression
{$  X_{i}^{k}    X_{i+1}^{k} \cdots   X_{j-2}^{k} $}
is set to be $1$ when {$j=i+1$},
and  the expression  {$  (X_{i} X_{i+1}  \cdots   X_{n-1}  G_{\ol{n}} Y_{n-1}  \cdots  Y_{i+1}  Y_{i} )$}
is set to be {$  G_{\ol{n}}  $}   when {$i = n$}.

\begin{rem}
Because of the SDP condition,
the multiplication formulas involving {$Y_{\ol{j}}$} do  not work for arbitrary {$\AJS(A, \bs{j})$},
and this means we could not handle the lower triangular part by introducing elements parallel to {$U^{A}_{i,j}$}.
By \eqref{induction_XYG},
we have the following equations:
\begin{align*}
&Y_{\ol{n-1}} = {v} G_{n}^{-1} G_{\ol{n}} Y_{n-1} + \cdots , \\
&Y_{\ol{i-1}} = f_{i} \cdot {(X_{i}  \cdots   X_{n-1}  G_{\ol{n}} Y_{n-1}  \cdots   Y_{i})} \cdot Y_{i-1}+ \cdots , 
\end{align*}
where for each {$i \in [2, n]$},  {$f_{i}$} is a monomial in {${v}^{\pm 1}$}, {$G_{i}^{\pm 1}$},  {$\cdots$}, {$G_{n}^{\pm 1}$}. The largest term
with respect to the partial order `{$\prec $}' defined before on the right hand side appears {$(X_{i}  \cdots   X_{n-1}  G_{\ol{n}} Y_{n-1}  \cdots   Y_{i})  Y_{i - 1}$}. This motives us to introduce the elements {$L^{A}_{i,j} $} whose odd part plays the same role as {$Y_{\ol{i-1}}$}.
\end{rem}


Analogous to  \cite[Lemma 6.4]{GLL},
by applying Proposition \ref{mulformeven} and Remark \ref{rem:prod alg} directly,
we have the following:
\begin{lem}\label{triang_prec0_q} 
	Fix {$h  \in [1, n-1]$},  {$a \in \NN^+$}, {$k>h+1$}. 
	Let {$M=(\SEE{m}_{i,j} | \SOE{m}_{i,j}) , P=(\SEE{p}_{i,j} | \SOE{p}_{i,j})  \in \MNZNS(n)$}, {$  P  \prec {M}$},
and we assume {${m}_{i,j} = 0$} when {$i > j$} or {$k < j \le n$}.
Then   in {$\USnv$}
\begin{enumerate}
\item 
	if {$\SOE{m}_{h,k} = \SOE{m}_{h+1,k} = 0$},  {$\SEE{m}_{h+1,k} \ge a$}, 
	we have 
\begin{align*}
X_{h}^{a}  \cdot  \AJS(M, \bs{0})
   =  g_A \AJS( A , \bs{j}_A )  
	  + \sum_{\substack{
	 		B \in \MNZNS(n) \\
	 		B  \prec A ,\ 
	 		{\bs{b}} \in {\ZZ}^n
	} } g_{B, \bs{b} }\AJS(B, {\bs{b}}), 
\end{align*}
where {$	A= ( \SE{M} + a E_{h,k} - a E_{h+1, k}| \SO{M})$},
nonzero {$g_A  \in \Qv$},
finitely many  {$ g_{B, \bs{b} } \in \Qv$},  {$\bs{j}_A  \in \ZZ^n$};
\item 
	if {$a=1$}, {$\SOE{m}_{h,k} =0$}, {$   \SOE{m}_{h+1,k} = 1$},  {$\SEE{m}_{h+1,k}  = 0$}, 
	we have
\begin{align*}
X_{h}  \cdot  \AJS(M, \bs{0})
   =  g_A \AJS( A ,\bs{j}_A )  
	  + \sum_{\substack{
	 		B \in \MNZNS(n) \\
	 		B  \prec A ,\ 
	 		{\bs{b}} \in {\ZZ}^n
	} } g_{B, \bs{b} }\AJS(B, {\bs{b}}), 
\end{align*}
where  {$A= ( \SE{M} | \SO{M} +  E_{h,k} -  E_{h+1, k}) $},
nonzero {$g_A  \in \Qv$},
finitely many  {$ g_{B, \bs{b} } \in \Qv$},  {$\bs{j}_A  \in \ZZ^n$};
\item
if {$m_{h+1, k} = \SEE{m}_{h+1, k} +  \SOE{m}_{h+1, k} \ge a$}, 
we have 
\begin{align*}
 &  X_{h}^{a}   \cdot   \AJS(P, \bs{0})
	 =
	 \sum_{\substack{
	 		B \in \MNZNS(n) \\
	 		\ABSUM{B}  \prec A' ,\ 
	 		{\bs{b}} \in {\ZZ}^n
	} } g'_{B, \bs{b} }\AJS(B, {\bs{b}}), 
\end{align*}
where  {$A'= ( \SE{M} + \SO{M} +  {a} E_{h,k} - {a} E_{h+1, k}) \in \MN(n)$},
finitely many {$g_{B', \bs{b} } \in \Qv$}.
\end{enumerate}
\end{lem}


Analogous to  \cite[Corollary 6.6]{GLL},
by applying Lemma \ref{triang_prec0_q},
we have the following:
\begin{cor}\label{triang_prec_all_q}
	Let {$M=(\SEE{m}_{i,j} | \SOE{m}_{i,j})   \in \MNZNS(n)$}, 
	 {$k \in [2, n]$}, {$u \in [1, k-1]$}.
	Assume  {${m}_{i,j} = 0$} when {$i > j$} or {$k < j \le n$},
and  {$ {m}_{i,k} = 0$} for {$i  \ge u$}.
Then for 
{$\SEE{a} \in \NN $},  {$\SOE{a} \in \ZG$}, 
 we have in {$\USnv$}

\begin{align*}
& 
{(  X_{u}     X_{u+1}  \cdots   X_{k-2} X_{\ol{k-1}} )}^{\SOE{a}} \cdot 
\prod_{u \le i \le k-1} X_{i}^{\SEE{a}} 
 \cdot  \AJS(M, \bs{0})
   =   g_A \AJS( A ,\bs{j}_A)  
	  + \sum_{\substack{
	 		B \in \MNZNS(n) \\
	 		B  \prec A ,\ 
	 		{\bs{b}} \in {\NN}^n
	} } g_{B, \bs{b} }\AJS(B, {\bs{b}}),
\end{align*}
where {$A= ( \SE{M} + \SEE{a} E_{u,k} | \SO{M} + \SOE{a} E_{u,k})$}, 
{$\bs{j}_A \in \ZZ^n$}, 
nonzero {$ g_A \in \Qv$},
and finitely many  {$  g_{B, \bs{b} } \in \Qv$}.
\end{cor}
\begin{proof}
The matrices appearing in the computation of
the left side
 are of the form \eqref{eq:special matrix}, 
hence by Lemma \ref{triang_aaa} and Corollary \ref{triang_aaa_cor}, 
the formula  in Proposition \ref{mulformodd1} is applicable. 
Then the proof is analogous to \cite[Corollary 6.6]{GLL} 
by applying Lemma \ref{triang_prec0_q} 
and we omit the details.
For convenience, for any {$i \in [u, k-1]$}, 
we set  {$ M_{i}^0  = ( \SE{M} + \SEE{a} E_{i,k} | \SO{M})$}.
Applying Lemma \ref{triang_prec0_q} and noticing there is only one leading term,
we have 
\begin{align*}
&  X_{k-1}^{\SEE{a}} 
 \cdot  \AJS(M, \bs{0})
   =  g_{ M_{k-1}^0 }  \AJS(  M_{k-1}^0 , \bs{j}_{M_{k-1}^0} ) 
	  + \sum_{\substack{
	 		C \in \MNZNS(n) \\
	 		C  \prec  M_{k-1}^0  ,\ 
	 		{\bs{c}} \in {\NN}^n
	} } g^{ M_{k-1}^0 }_{C, \bs{c} }\AJS(C, {\bs{c}}),
\end{align*}
where {$g^{ k-1,0 }_{C, \bs{c} } \in \Qv$},  {$ \bs{j}_{M_{k-1}^0}  \in \ZZ^n$}.
By induction, we have
\begin{equation}\label{eq_part_triang_o1} 
\begin{aligned}
& 
\prod_{u \le i \le k-1} X_{i}^{\SEE{a}} 
 \cdot  \AJS(M, \bs{0})
   =  g_{ M_{u}^0 }   \AJS( { M_{u}^0 }  ,\bs{j}_{M_{u}^0 })  
	  + \sum_{\substack{
	 		D \in \MNZNS(n) \\
	 		D  \prec { M_{u}^0 }  ,\ 
	 		{\bs{d}} \in {\NN}^n
	} } g^{ u, 0 }_{D, \bs{d} }\AJS(D, {\bs{d}}).
\end{aligned}
\end{equation}
If {$\SO{a} = 1$},
will will multiply {$X_{\ol{k-1}}$}  on the summands of the right hand side of \eqref{eq_part_triang_o1}. 
By Lemma \ref{triang_aaa} and Corollary \ref{triang_aaa_cor}, 
the formula in Proposition \ref{mulformodd1} is applicable. 
Then for {$ \AJS( { M_{u}^0 }  ,\bs{j}_{M_{u}^0}) $} and each {$\AJS(D, {\bs{d}})$} in \eqref{eq_part_triang_o1},
  we have 
\begin{align*}
& 
X_{\ol{k-1}}  \cdot  \AJS( { M_{u}^0 }  ,\bs{j}_{M_{u}^0}) 
   =    g_{ M_{k-1}^1 }  \AJS( { M_{k-1}^1 }  ,\bs{j}_{M_{k-1}^1})  
	  + \sum_{\substack{
	 		P \in \MNZNS(n) \\
	 		P  \prec { M_{k-1}^1  }  ,\ 
	 		{\bs{p}} \in {\NN}^n
	} } g^{ k-1, 1 }_{P, \bs{p} }\AJS(P, {\bs{p}}),\\
& 
X_{\ol{k-1}}  \cdot \AJS(D, {\bs{d}})
   =  
    \sum_{\substack{
	 		R \in \MNZNS(n) \\
	 		R  \prec { M_{k-1}^1  }  ,\ 
	 		{\bs{r}} \in {\NN}^n
	} } g^{ k-1, 1 }_{R, \bs{r} }\AJS(R, {\bs{r}}),\\
\end{align*}
and 
\begin{align*}
& 
X_{\ol{k-1}} \prod_{u \le i \le k-1} X_{i}^{\SEE{a}} 
 \cdot  \AJS(M, \bs{0})
   =   g_{ M_{k-1}^1 }  \AJS( { M_{k-1}^1 }  ,\bs{j}_{M_{k-1}^1})  
	  + \sum_{\substack{
	 		B \in \MNZNS(n) \\
	 		B  \prec { M_{k-1}^1  }  ,\ 
	 		{\bs{b}} \in {\NN}^n
	} } g^{ k-1, 1 }_{B, \bs{b} }\AJS(B, {\bs{b}}),
\end{align*}
where {$M_{k-1}^1 = ( \SE{M} + \SEE{a} E_{u,k} | \SO{M} + \SOE{a} E_{k-1,k})$},
and {$ g^{ k-1, 1 }_{B, \bs{b} } \in \Qv$}.
Then  Lemma \ref{triang_prec0_q} shows
\begin{align*}
& 
X_{k-2} X_{\ol{k-1}} \prod_{u \le i \le k-1} X_{i}^{\SEE{a}} 
 \cdot  \AJS(M, \bs{0})
   =  g_{ M_{k-2}^1 }   \AJS( { M_{k-2}^1 }  ,\bs{j}_{M_{k-2}^1})  
	  + \sum_{\substack{
	 		B \in \MNZNS(n) \\
	 		B  \prec { M_{k-2}^1  }  ,\ 
	 		{\bs{b}} \in {\NN}^n
	} } g^{ k-2, 1 }_{B, \bs{b} }\AJS(B, {\bs{b}}),
\end{align*}
where {$M_{k-2}^1 = ( \SE{M} + \SEE{a} E_{u,k} | \SO{M} + \SOE{a} E_{k-2,k})$}.
Consequently,  applying Lemma  \ref{triang_prec0_q}  inductively, 
the result is then proved.
\end{proof}

\begin{lem}\label{triangular_lower}
Fix {$1\le u \le n-1$},
assume   {$\bs{j} \in \ZZ^n$},
and {$M=( {\SEE{m}_{i,j}} | {\SOE{m}_{i,j}} ) \in \MNZNS(n)$}
satisfies {$ \co^{-}(M)_{\min} \ge u+1$}.
Let {$( \SEE{a}_{u+1, u}, \cdots, \SEE{a}_{n, u}) \in \NN^{n-u}$},
 {$( \SOE{a}_{u+1, u}, \cdots, \SOE{a}_{n, u}) \in \ZG^{n-u}$},
and {$M' = ( \SE{M} + \sum_{k=u+1}^n \SEE{a}_{k,u} E_{k, u} |  \SO{M} + \sum_{k=u+1}^n \SOE{a}_{k,u} E_{k, u}) $}.
Then we have
\begin{align*}
	\prod_{{n} \ge i \ge {u+1}} L_{M'}^{i,u}
	\cdot
	\AJS(M, \bs{j})
	= g_{M',\bs{j}'} \AJS(M', \bs{j}')
		+ \sum_{\substack{
			B \in \MNZNS(n) ,
			\bs{j}'' \in \ZZ^{n} \\
			B \prec M' }
			} g_{B, \bs{j}''}  \AJS(B, \bs{j}''),
\end{align*}
where {$\bs{j}', \bs{j}'' \in \ZZ^n$},  nonzero {$ g_{M',\bs{j}'} \in \Qv$},
and finitely many  {$g_{B, \bs{j}''} \in \Qv$}.
\end{lem}
\begin{proof}
We shall compute $L_{M'}^{i,u} \cdots  L_{M'}^{u+1, u}    \AJS(M, \bs{j}) $ for $u+1\leq i\leq n$ by induction on {$i$}.
Set {$a_u = \sum_{k=u+1}^{n} a_{k, u} $}.
Applying Proposition \ref{mulformeven} {\rm(2)}, we have
\begin{align*}
 Y_{u} \AJS(M, \bs{j})
& =     g_1'  \ABJS( \SE{M} + E_{u+1, u} , \SO{M} , \bs{j}')
		+ {\delta}_{1, \SOE{m}_{u, u}} \AJS(  (\SE{M} | \SO{M} -E_{u,u} +E_{u+1, u} ), \bs{j}')  \\
	& \qquad
		+ \sum_{\substack{
			M'' \in \MNZNS(n), \\
			\bs{j}'' \in \ZZ^{n}
			}
		} g_{M'', \bs{j}''}   \AJS(M'', \bs{j}'') ,
\end{align*}
where $ g_1' , g_{M'', \bs{j}''}\in \Qv$ and  each {$M''$} satisfies
{$  \co^{-}(M'')_{\min} \ge u+1$}.
By Definition \ref{defn:partialorder}
and comparing the {$u$}-th colomn  of the lower triangular parts, we have
\begin{align*}
&M'' \prec ( \SE{M} + E_{u+1, u} | \SO{M}), \\
& (\SE{M} | \SO{M} -E_{u,u} +E_{u+1, u} )  \prec ( \SE{M} + E_{u+1, u} | \SO{M}).
\end{align*}
Then applying Proposition \ref{mulformeven} {\rm(2)} repeatedly, we have
\begin{equation}\label{tmp_part_1}
\begin{aligned}
 Y_{u}^{a_u} \AJS(M, \bs{j})
& =     g'  \AJS(  (\SE{M} + a_u  E_{u+1, u} | \SO{M}  ), \bs{j}')
	& \qquad
		+ \sum_{\substack{
			B \in \MNZNS(n) ,
			\bs{j}'' \in \ZZ^{n} \\
			B \prec  (\SE{M}  + a_u  E_{u+1, u} | \SO{M} )}
			} g^M_{B, \bs{j}''}   \AJS(B, \bs{j}'') ,
\end{aligned}
\end{equation}
for some {$g' ,g^M_{B, \bs{j}''}\in \Qv$}.

If {$\SOE{a}_{u+1,u} = 1$},
we need to consider  the
multiplication of  {$X_{u+1}  \cdots   X_{n-1}  G_{\ol{n}} Y_{n-1}  \cdots   Y_{u+1} $}
with {$Y_{u}^{a_u} \AJS(M, \bs{j})$}
and this means we need to compute
the multiplication of  {$X_{u+1}  \cdots   X_{n-1}  G_{\ol{n}} Y_{n-1}  \cdots   Y_{u+1} $}
with each summands on the right hand side of  \eqref{tmp_part_1}.
Firstly, for each {$ \AJS(B, \bs{j}'')$} with {$B=(\SEE{b}_{i,j}|\SOE{b}_{i,j})$} appearing in the summands \eqref{tmp_part_1},
let
\begin{align}
X_{u+1}  \cdots   X_{n-1}  G_{\ol{n}} Y_{n-1}  \cdots   Y_{u+1}  \AJS(B, \bs{j}'')
= \sum_{\substack{
			C \in \MNZNS(n) ,
			\bs{j}''' \in \ZZ^{n} }
			} g^{B,\bs{j}''}_{C, \bs{j}'''}   \AJS(C, \bs{j}''') . \label{triangular_tmp_C}
\end{align}
Denote {$b_u = \sum_{i=u+1}^{n} b_{i, u} $},
{$c_u = \sum_{i=u+1}^{n} c_{i, u} $}.
By Proposition \ref{mulformeven}, \ref{mulformdiag},
we have the following three cases:
\begin{enumerate}
\item
{$b_{i,j}=0$} when {$i>j$} and {$j<u+1$}.
It is clear {$ \co^{-}(B)_{\min} \ge u+1$}, {$ \co^{-}(C)_{\min} \ge u+1$};
\item
{$b_{i,j}=0$} when {$i>j$} and {$j<u$},
{$b_{i,u}=0$} when {$i>u+1$},
and {$b_{u+1,u} < a_u$}.
In this case,
 {$c_u =  b_u < a_{u} $},
{$(\SOE{m}_{u,u}, 1, 0, ...)  \ge ( \SOE{c}_{u,u}, \SOE{c}_{u+1,u}, \cdots, \SOE{c}_{n,u} )$};
\item
{$b_{i,j}=0$} when {$i>j$} and {$j<u$},
{$b_{i,u}=0$} when {$i>u+1$},
and {$b_{u+1,u} = a_u$}.
This happens only when  {$\SOE{m}_{u, u} = 1$}
and  {$B = \AJS(  (\SE{M} + (a_u - 1) E_{u+1, u} | \SO{M} -E_{u,u} +E_{u+1, u} ), \bs{j}') $}.
In this case,
 {$c_u = b_u = a_{u} $},
{$\SOE{c}_{u,u} = 0$}.
Therefore,
{$(\SOE{m}_{u,u}, 1, 0, ...)  > ( \SOE{c}_{u,u}, \SOE{c}_{u+1,u}, \cdots, \SOE{c}_{n,u} )$}.
\end{enumerate}
This means that in each case, every {$ \AJS(C, \bs{j}''')$}  appearing in \eqref{triangular_tmp_C}
satisfies
\begin{align*}
C \prec  (\SE{M}_{u+1} + (a_u - 1 ) E_{u+1, u}  | \SO{M}+ E_{u+1, u} ).
\end{align*}
and hence the following holds
\begin{align}
X_{u+1}  \cdots   X_{n-1}  G_{\ol{n}} Y_{n-1}  \cdots   Y_{u+1}  \AJS(B, \bs{j}'')
= \sum_{\substack{
			C \in \MNZNS(n) ,
			\bs{j}''' \in \ZZ^{n},\\
 C \prec  (\SE{M}_{u+1} + (a_u - 1 ) E_{u+1, u}  | \SO{M}+ E_{u+1, u} )}
			} g^{B,\bs{j}''}_{C, \bs{j}'''}   \AJS(C, \bs{j}''') , \label{triangular_tmp_C-final}
\end{align}
for each $B(\bs{j}'')$ in \eqref{tmp_part_1}.

Secondly, for the leading term {$ \AJS(  (\SE{M} + a_u  E_{u+1, u} | \SO{M} ), \bs{j}')$} in \eqref{tmp_part_1},
by Proposition \ref{mulformeven} (2) we can write
\begin{equation}\label{eq:tmp leading-1}
\begin{aligned}
Y_{n-1}  \cdots   Y_{u+1}
 \AJS(  (\SE{M} + a_u  E_{u+1, u} | \SO{M} ), \bs{j}')
& =
g^M_{\bs{j}''}    \AJS(  (\SE{M}_{u+1} + (a_u - 1 ) E_{u+1, u} + E_{n, u} | \SO{M} ), \bs{j}'') \\
& \qquad \qquad +
\sum_{\substack{
			R \in \MNZNS(n) ,
			\bs{j}''' \in \ZZ^{n} }
			} g^M_{R, \bs{j}'''}   \AJS(R, \bs{j}''') ,
\end{aligned}
\end{equation}
each  {$ \AJS(R, \bs{j}''')$}  appears in the equation  with {$R=(\SEE{r}_{i,j}|\SOE{r}_{i,j})$} satisfies
{$ \co^{-}(R)_{\min} = u$},
{$\SOE{r}_{i,u}=0$} for any {$i \ge u+1$},
and
\begin{enumerate}
\item
{$r_{u+1, u} = a_u$},  {$r_{k, u} = 0$} for all {$k \in [u+2, n]$}, or
\item
  {$r_{u+1, u} = a_u - 1$} and {$r_{t, u} = 1$} for some  {$t \in [u+2 , n-1]$}.
\end{enumerate}
So for each $\AJS( R, \bs{j}''')$ appearing in \eqref{eq:tmp leading-1} we have
\begin{align*}
X_{u+1}  \cdots   X_{n-1}  G_{\ol{n}} \AJS( R, \bs{j}')
= \sum_{\substack{
			D \in \MNZNS(n) ,
			\bs{j}'''' \in \ZZ^{n} }
			} g^R_{D, \bs{j}''''}   \AJS(D, \bs{j}'''') ,
\end{align*}
where each  {$ \AJS(D, \bs{j}'''')$}  with  {$D=(\SEE{d}_{i,j}|\SOE{d}_{i,j})$} satisfies
{$ \co^{-}(D)_{\min} = u$} and
 {$\sum_{i=u+1}^{n} \SOE{d}_{i, u} = 0$},
 and it is clear
\begin{align*}
D \prec  (\SE{M}_{u+1} + (a_u - 1 ) E_{u+1, u}  | \SO{M}+ E_{u+1, u} ).
\end{align*}
Meanwhile,
Proposition \ref{mulformdiagN}, \ref{mulformeven} {\rm(1)} imply
\begin{align*}
& X_{u+1}  \cdots   X_{n-1}  G_{\ol{n}}
\cdot
\AJS(  (\SE{M} + (a_u - 1 ) E_{u+1, u} + E_{n, u} | \SO{M}), \bs{j}'') \\
&=
g^{M}_{\bs{j}'''}  \AJS(  (\SE{M}_{u+1} + (a_u - 1 ) E_{u+1, u}  | \SO{M}+ E_{u+1, u} ), \bs{j}''')  +
\sum_{\substack{
			P \in \MNZNS(n) ,
			\bs{j}'''' \in \ZZ^{n} }
			} g^M_{P, \bs{j}'''}   \AJS(P, \bs{j}'''') ,
\end{align*}
where each  {$ \AJS(P, \bs{j}'''')$}   with  {$P=(\SEE{p}_{i,j}|\SOE{p}_{i,j})$} satisfies
{$ \co^{-}(P)_{\min} = u$} and
{$\SOE{p}_{u,u} = \SOE{m}_{u,u}$},
{$ \SOE{p}_{u+1,u}  = 0$},
{$\sum_{k=u+2}^{n} \SOE{p}_{k,u} = 1$}.
Then we have
{$(\SOE{m}_{u,u}, 1, 0, \cdots, 0) >  (\SOE{p}_{u,u},  \SOE{p}_{u+1,u}, \cdots, \SOE{p}_{n,u} ) $}.
Therefore for the leading term {$ \AJS(  (\SE{M} + a_u  E_{u+1, u} | \SO{M} ), \bs{j}')$} in \eqref{tmp_part_1} we obtain
\begin{align}\label{eq:tmp leading-2}
&X_{u+1}  \cdots   X_{n-1}  G_{\ol{n}} Y_{n-1}  \cdots   Y_{u+1}\AJS(  (\SE{M} + a_u  E_{u+1, u} | \SO{M} ), \bs{j}')\\
&=
g^{M}_{\bs{j}'''}  
\AJS(  (\SE{M} + (a_u -1) E_{u+1, u} | \SO{M} + E_{u+1, u}), \bs{j}''')
		+ \sum_{\substack{
			B \in \MNZNS(n) ,
			\bs{j}'' \in \ZZ^{n} \\
			B \prec  (\SE{M} + ( a_{u} -1) E_{u+1, u} | \SO{M} + E_{u+1, u}) }
			} g'_{B, \bs{j}''}   \AJS(B, \bs{j}'')\notag
\end{align}
with $g'_{B, \bs{j}''}\in \Qv$.
Then by \eqref{tmp_part_1}, \eqref{triangular_tmp_C-final}, \eqref{eq:tmp leading-2} and Definition \ref{defn:partialorder}, we obtain
\begin{equation}\label{tmp_part_2}
\begin{aligned}
&X_{u+1}  \cdots   X_{n-1}  G_{\ol{n}} Y_{n-1}  \cdots   Y_{u+1}
 Y_{u}^{ a_{u} }   \AJS(M, \bs{j})\\
& =     g'  \AJS(  (\SE{M} + (a_u -1) E_{u+1, u} | \SO{M} + E_{u+1, u}), \bs{j}')
		+ \sum_{\substack{
			B \in \MNZNS(n) ,
			\bs{j}'' \in \ZZ^{n} \\
			B \prec  (\SE{M} + ( a_{u} -1) E_{u+1, u} | \SO{M} + E_{u+1, u}) }
			} g_{B, \bs{j}''}   \AJS(B, \bs{j}'') .
\end{aligned}
\end{equation}
Summarizing equations  \eqref{tmp_part_1} and  \eqref{tmp_part_2},
we have
\begin{align*}
L_{M'}^{u+1, u}    \AJS(M, \bs{j})
 =     g'_{u+1, u}  \AJS( M^{u+1}_{u} , \bs{j}')
		+ \sum_{\substack{
			B \in \MNZNS(n) , \\
			\bs{j}'' \in \ZZ^{n} , \
			B \prec M^{u+1}_{u}  }
			} g_{B, \bs{j}''}   \AJS(B, \bs{j}'') ,
\end{align*}
where {$M^{u+1}_{u} =  (\SE{M} + (  \sum_{k=u+1}^n a_{k,u} - \SOE{a}_{u+1, u}  ) E_{u+1, u} | \SO{M} + \SOE{a}_{u+1, u}  E_{u+1, u})$}.

By induction on {$i$},
with the same argument above,
we can prove for {$i \ge u+1$},
\begin{align*}
&L_{M'}^{i,u} \cdots  L_{M'}^{u+1, u}    \AJS(M, \bs{j})
  =     g'_u  \AJS(	M^{i}_{u}, \bs{j}')
		+ \sum_{\substack{
			B \in \MNZNS(n) ,
			\bs{j}'' \in \ZZ^{n} \\
			B \prec
			M^{i}_{u}
				}
			} g_{B, \bs{j}''}   \AJS(B, \bs{j}'') ,
\end{align*}
where {$M^{i}_{u} = (\SE{M} + \sum_{t=u+1}^{i-1} \SEE{a}_{t, u} E_{t, u}
			+ ( \sum_{t=i}^{n} {a}_{t, u} ) E_{i, u} - \SOE{a}_{i,u} E_{i,u}
			| \SO{M} +  \sum_{t=u+1}^{i-1} \SOE{a}_{t, u} E_{t, u}
				+ \SOE{a}_{i,u} E_{i,u}
)$}.
In particular when {$i=n$}, we have
\begin{align*}
L_{M'}^{n,u} \cdots  L_{M'}^{u+1, u}    \AJS(M, \bs{j})
& =     g  \AJS(  M', \bs{j}')
		+ \sum_{\substack{
			B \in \MNZNS(n) ,\\
			\bs{j}'' \in \ZZ^{n}, \
			B \prec	  M' }
			} g_{B, \bs{j}''}   \AJS(B, \bs{j}'') .
\end{align*}
\end{proof}
For any   {$B=( {\SEE{b}_{i,j}} | {\SOE{b}_{i,j}} ) \in \MNZNS(n)$},
for convenience, we introduce
{${B}_{k}^{+} = ({\SE{(B^+)}_{k}} | {\SO{(B^+)}_{k}} )$},
with {$ {\SE{(B^+)}_{k}}  = (\SEE{m}_{i,j})$}, {${\SO{(B^+)}_{k}}  = (\SOE{m}_{i,j})$},
satisfying {$\SEE{m}_{i,j} =  \SOE{m}_{i,j} = 0$} when {$i > j$} or {$j > k$},
and {$\SEE{m}_{i,j} = \SEE{b}_{i,j}$} , {$\SOE{m}_{i,j} = \SOE{b}_{i,j}$} otherwise.
Meanwhile, set  {${B}_{k} = (\SE{B}_{k} | \SO{B}_{k} )$},
with {$\SE{B}_{k} = (\SEE{m}_{i,j})$}, {$\SO{B}_{k} = (\SOE{m}_{i,j})$},
where {$\SEE{m}_{i,j} =  \SOE{m}_{i,j} = 0$} when {$i > j$} and {$j < k$},
and {$\SEE{m}_{i,j} = \SEE{b}_{i,j}$}, {$\SOE{m}_{i,j} = \SOE{b}_{i,j}$} otherwise.
Notice that {$B_n = B_n^+$}.

By applying
Corollary \ref{mulformzero},
Proposition \ref{mulformdiag},
Proposition \ref{mulformodd1},
Corollary \ref{triang_aaa_cor},  \ref{triang_prec_all_q},
Lemma \ref{triangular_lower},
we are ready to  show the important  triangular relation between the elements {${\frm}^A$} defined in \eqref{MA} and the elements $A(\bs{j})$.
\begin{thm}\label{triangular_relation_q}
For any {$A \in \MNZNS(n)$},  we have in $\SQvnR$
\begin{align*}
 {\frm}^A 
	= g_A  \AJS(A, \bs{j}_A) +
		\sum_{B \prec A , \  \bs{j}' \in \ZZ^{n} }
		g_{B, \bs{j}'} \AJS(B, \bs{j}'),
\end{align*}
where {$\bs{j}_A, \bs{j}' \in \ZZ^n$},
   nonzero  {$g_A \in \Qv$} 
and finitely many  {$g_{B, \bs{j}'} \in \Qv$}.
In particular, the $\Qv$-spave  {$\USnv$} is a subalgebra generated by the set $\fsG$ in \eqref{setG}.
\end{thm}
\begin{proof}
We shall prove the result in several steps.
Example \ref{example_order_q}  explains the generating order of the entries.

For the upper triagular and the diagonal parts,
the proof is similar to \cite[Theorem 6.8]{GLL} and we will give a sketchy explanation.
We first compute, for $k>j$,
\begin{align*}
[ ( {\prod_{{k-1} \ge i \ge {1}}} U^{A}_{i,k})  ]{\prod_{{k-1} \ge j \ge {1}}}  [  D^{A}_{j,j} ( {\prod_{{j-1} \ge i \ge {1}}} U^{A}_{i,j})  ]
\end{align*}
inductively for $1\leq k\leq n$
and by Corollary \ref{triang_prec_all_q} the matrices involved
are of the form \eqref{eq:special matrix}.
Then by Corollary \ref{triang_aaa_cor}
we can use the formulas in Proposition \ref{mulformdiag} to compute
\begin{align*}
 [ D^{k,k}_A( {\prod_{{k-1} \ge i \ge {1}}} U^{A}_{i,k})  ]{\prod_{{k-1} \ge j \ge {1}}}  [  D^{A}_{j,j} ( {\prod_{{j-1} \ge i \ge {1}}} U^{A}_{i,j})  ]={\prod_{{k} \ge j \ge {1}}}  [  D^{A}_{j,j} ( {\prod_{{j-1} \ge i \ge {1}}} U^{A}_{i,j})  ]
\end{align*}
and  obtain
\begin{align*}
{\prod_{{k} \ge j \ge {1}}}  [  D^{A}_{j,j} ( {\prod_{{j-1} \ge i \ge {1}}} U^{A}_{i,j})  ]
&= g_{A^+_k}
\AJS(A^+_k, \bs{j}_{A^+_k}) 
+
		\sum_{C \prec A^+_k  , \  \bs{j}' \in \ZZ^{n} }
		g_{C, \bs{j}'} \AJS(C, \bs{j}').
\end{align*}
When  {$k=n$},
observe that  {$A_n = A_n^+$},
then we have
\begin{equation}\label{triagular_upper}
\begin{aligned}
{\prod_{{n} \ge j \ge {1}}}  [  D^{A}_{j,j} ( {\prod_{{j-1} \ge i \ge {1}}} U^{A}_{i,j})  ]
&=
g_{A_n } \AJS(A_n, \bs{j}_{A_n}) 
+
		\sum_{C \prec A_n ,   \  \bs{j}' \in \ZZ^{n} }
		g_{C, \bs{j}'} \AJS(C, \bs{j}') ,
\end{aligned}
\end{equation}
where each  {$\AJS(C, \bs{j}')$} with {$C=(\SEE{c}_{i,j}| \SOE{c}_{i,j})$}
satisfies {$c_{i,j} = 0$} when {$i>j$}.

More effort is needed to handle the lower triangular part. Clearly
\begin{align*}
	{\prod_{{1} \le j \le {n-1}}} {\prod_{{n} \ge i \ge {j+1}}} L^{A}_{i,j}
=  (L^{A}_{n,1} 	L^{A}_{n-1,1} \cdots L^{A}_{2,1}) \cdot
	(L^{A}_{n,2} 	L^{A}_{n-1,2} \cdots L^{A}_{3,2}) \cdots
	(L^{A}_{n,n-2} 	L^{A}_{n-1,n-2})  \cdot
	L^{A}_{n,n-1} .
\end{align*}
We will prove the following equation by induction on {$k$}:
\begin{equation}\label{triangular_lower_k}
\begin{aligned}
	{\prod_{{k} \le j \le {n-1}}} {\prod_{{n} \ge i \ge {j+1}}} L^{A}_{i,j}
	\cdot
	{\prod_{{n} \ge j \ge {1}}}  [  D^{A}_{j,j} ( {\prod_{{j-1} \ge i \ge {1}}} U^{A}_{i,j})  ]
	= g_{A_{k}}  \AJS(A_{k}, \bs{j}_{A_{k}})
		+ \sum_{\substack{
			B \in \MNZNS(n) ,
			\bs{j}' \in \ZZ^{n} \\
			B \prec A_{k}}
			} g_{B, \bs{j}'}  \AJS(B, \bs{j}'),
\end{aligned}
\end{equation}
where each  {$\AJS(B, \bs{j}')$}
satisfies {$ \co^{-}(B)_{\min} \ge k$}.

For   {$ \AJS(A_{n}, \bs{j}_{A_n}) $} and each {$  \AJS(C, \bs{j}')$} appearing in \eqref{triagular_upper},
applying Lemma \ref{triangular_lower},
we obtain
\begin{align}
 L^{A}_{n,n-1}
	\cdot
	\AJS(A_n, \bs{j}_{A_{n}})
&= g_1 \AJS(A_{n-1}, \bs{j}_{A_{n-1}})
		+ \sum_{\substack{
			H \in \MNZNS(n) ,
			\bs{j}'' \in \ZZ^{n} \\
			H \prec A_{n-1} }
			} g_{H, \bs{j}''}  \AJS(H, \bs{j}''),  \label{tmp_h} \\
 L^{A}_{n,n-1}
	\cdot
	\AJS(C, \bs{j}')
&= g_2 \AJS(C_{n-1}, \bs{j}''')
		+ \sum_{\substack{
			P \in \MNZNS(n) ,
			\bs{j}'' \in \ZZ^{n} \\
			P \prec C_{n-1} }
			} g_{P, \bs{j}''''}  \AJS(P, \bs{j}''''),  \label{tmp_j}
\end{align}
where
{$C_{n-1} = ( \SE{C} +  \SEE{a}_{n,n-1} E_{n, n-1} |  \SO{C} +  \SOE{a}_{n, n-1} E_{n, n-1}) $}.
This means
{$C_{n-1} \prec A_{n-1}$},
{$H \prec A_{n-1}$},
{$J \prec A_{n-1}$}.
Hence \eqref{triangular_lower_k} is proved  for the case {$k=n-1$}.

Now assume \eqref{triangular_lower_k} holds for  {$k$}. Then
applying Lemma \ref{triangular_lower}
on {$ \AJS(A_{k},  \bs{j}_{A_{k}}) $} and {$  \AJS(B, \bs{j}')$} in \eqref{triangular_lower_k} one can directly
proves that \eqref{triangular_lower_k} also holds for {$k-1$}, which means \eqref{triangular_lower_k} holds for all {$k \in [1, n-1]$}.

 So when {$k=1$},  noticing that {$A_{1} = A$}, we have
\begin{align*}
{\prod_{{1} \le j \le {n-1}}} {\prod_{{n} \ge i \ge {j+1}}} L^{A}_{i,j}
\cdot
{\prod_{{n} \ge j \ge {1}}}  [  D^{A}_{j,j} ( {\prod_{{j-1} \ge i \ge {1}}} U^{A}_{i,j})  ]
&	= g_A  \AJS(A, \bs{j}_A)
		+ \sum_{\substack{
			B \in \MNZNS(n) ,
			\bs{j}' \in \ZZ^{n} \\
			B \prec A
			} } g_{B, \bs{j}'}  \AJS(B, \bs{j}').
\end{align*}
and hence the equation is proved.

Then for any {$\bs{j} \in \ZZ^n$},
equation \eqref{Aj_0} implies {$  \AJS(A, \bs{j}) $}
can be generated by {$  \AJS(A, \bs{j}_A) $}
and  {$\{ G_{i}^{\pm 1} \where i \in [1, n]\}$},
hence  {$\AJS(A, \bs{j})$} can be generated by the set {$\fsG$}.
\end{proof}

\begin{exam}\label{example_order_q}
In order  to  illustrate the proof of Theorem \ref{triangular_relation_q},
we use an example to show   the leading terms for important steps.
Let {$A = (\SEE{a}_{i,j} | \SOE{a}_{i,j}) \in \MNZNS(3)$},
and recall the notations {$A_k^+$}
and {$A_{k}$} defined before Theorem \ref{triangular_relation_q}.
The ``lower terms''  in the expressions
means a {$\Qv$}-linear combination of finitely many lower elements
with respect to the partial order `{$\prec $}' defined before.
The detail for producing  {$A$} is quite different from the one in \cite[Example 6.9]{GLL}.
\begin{align*}
& D^{A}_{1,1}
	 = g_{A_1^+} \AJS(A_1^+,\bs{j}_{A_1^+}),  \\
& U^{A}_{1,2} D^{A}_{1,1}
	 =  g_{A_{1,2}^+} \AJS(A_{1,2}^+,\bs{j}_{A_{1,2}^+})
	 + \mbox{lower terms},  \\
& D^{A}_{2,2} U^{A}_{1,2} D^{A}_{1,1}
	 = g_{A_{2}^+}  \AJS(A_{2}^+,\bs{j}_{A_{2}^+}) 
	 + \mbox{lower terms},  \\
& U^{A}_{1,3}  D^{A}_{2,2} U^{A}_{1,2} D^{A}_{1,1}
	 = g_{A_{1,3}^+} \AJS(A_{1,3}^+,  \bs{j}_{A_{1,3}^+} )
	 + \mbox{lower terms},  \\
& U^{A}_{2,3}  U^{A}_{1,3}  D^{A}_{2,2} U^{A}_{1,2} D^{A}_{1,1}
	 = g_{A_{2,3}^+}  \AJS(A_{2,3}^+,\bs{j}_{A_{2,3}^+})
	 + \mbox{lower terms},  \\
& D^{A}_{3,3}  U^{A}_{2,3}  U^{A}_{1,3}  D^{A}_{2,2} U^{A}_{1,2} D^{A}_{1,1}
	 = g_{A_{3}^+}  \AJS(A_{3}^+,\bs{j}_{A_{3}^+})
	 + \mbox{lower terms},  \\
& L^{A}_{3,2}  D^{A}_{3,3}  U^{A}_{2,3}  U^{A}_{1,3}  D^{A}_{2,2} U^{A}_{1,2} D^{A}_{1,1}
	 = g_{A_{2}}  \AJS(A_{2},\bs{j}_{A_{2}})
	 + \mbox{lower terms},  \\
& L^{A}_{2,1} L^{A}_{3,2}  D^{A}_{3,3}  U^{A}_{2,3}  U^{A}_{1,3}  D^{A}_{2,2} U^{A}_{1,2} D^{A}_{1,1}
	 = g_{A_{2,1}^-}  \AJS(A_{2,1}^-,\bs{j}_{A_{2,1}^{-}})
	 + \mbox{lower terms},  \\
& L^{A}_{3,1}   L^{A}_{2,1} L^{A}_{3,2}  D^{A}_{3,3}  U^{A}_{2,3}  U^{A}_{1,3}  D^{A}_{2,2} U^{A}_{1,2} D^{A}_{1,1}
	 = g_{A}  \AJS(A,\bs{j}_{A})
	 + \mbox{lower terms}.
\end{align*}
Here the matrices of the leading terms  are listed as followed:
\begin{align*}
&
A_{1}^+ =
\left (
\begin{matrix}
         0       &   0   & 0\\
         0 &      0 & 0 \\
         0 & 0 & 0
\end{matrix}
\right.
\left |
\begin{matrix}
         \SOE{a}_{1 1}  & 0 & 0 \\
         0 & 0 & 0 \\
         0 & 0 & 0
\end{matrix}
\right) ,
\qquad
A_{1,2}^+ =
\left (
\begin{matrix}
         0       &   \SEE{a}_{1 2}   & 0\\
         0 &      0 & 0 \\
         0 & 0 & 0
\end{matrix}
\right.
\left |
\begin{matrix}
         \SOE{a}_{1 1}  &   \SOE{a}_{1 2} & 0 \\
         0 & 0 & 0 \\
         0 & 0 & 0
\end{matrix}
\right) ,\\
&
A_{2}^+ =
\left (
\begin{matrix}
         0       &   \SEE{a}_{1 2}   & 0\\
         0 &      0 & 0 \\
         0 & 0 & 0
\end{matrix}
\right.
\left |
\begin{matrix}
         \SOE{a}_{1 1}  &   \SOE{a}_{1 2} & 0 \\
         0 & \SOE{a}_{2 2} & 0 \\
         0 & 0 & 0
\end{matrix}
\right),
\qquad
A_{1,3}^+ =
\left (
\begin{matrix}
         0       &   \SEE{a}_{1 2}   & \SEE{a}_{1 3}\\
         0 &      0 & 0 \\
         0 & 0 & 0
\end{matrix}
\right.
\left |
\begin{matrix}
         \SOE{a}_{1 1}  &   \SOE{a}_{1 2} &  \SOE{a}_{1 3} \\
         0 & \SOE{a}_{2 2} & 0 \\
         0 & 0 & 0
\end{matrix}
\right) ,\\
&
A_{2,3}^+ =
\left (
\begin{matrix}
         0       &   \SEE{a}_{1 2}   & \SEE{a}_{1 3}\\
         0 &      0 & \SEE{a}_{2 3} \\
         0 & 0 & 0
\end{matrix}
\right.
\left |
\begin{matrix}
         \SOE{a}_{1 1}  &   \SOE{a}_{1 2} &  \SOE{a}_{1 3} \\
         0 & \SOE{a}_{2 2} & \SOE{a}_{2 3} \\
         0 & 0 & 0
\end{matrix}
\right) ,
\qquad
A_{3}^+ =
\left (
\begin{matrix}
         0       &   \SEE{a}_{1 2}   & \SEE{a}_{1 3}\\
         0 &      0 & \SEE{a}_{2 3} \\
         0 & 0 & 0
\end{matrix}
\right.
\left |
\begin{matrix}
         \SOE{a}_{1 1}  &   \SOE{a}_{1 2} &  \SOE{a}_{1 3} \\
         0 & \SOE{a}_{2 2} & \SOE{a}_{2 3} \\
         0 & 0 & \SOE{a}_{3 3}
\end{matrix}
\right) ,
\\
&
A_{2} =
\left (
\begin{matrix}
         0       &   \SEE{a}_{1 2}   & \SEE{a}_{1 3}\\
         0 &      0 & \SEE{a}_{2 3} \\
         0 & \SEE{a}_{3 2}  & 0
\end{matrix}
\right.
\left |
\begin{matrix}
         \SOE{a}_{1 1}  &   \SOE{a}_{1 2} &  \SOE{a}_{1 3} \\
         0 & \SOE{a}_{2 2} & \SOE{a}_{2 3} \\
         0 & \SOE{a}_{3 2}  & \SOE{a}_{3 3}
\end{matrix}
\right) ,
\qquad
A_{2,1}^- =
\left (
\begin{matrix}
         0       &   \SEE{a}_{1 2}   & \SEE{a}_{1 3}\\
         \SEE{a}_{2 1}+  {a}_{3 1} &      0 & \SEE{a}_{2 3} \\
         0 & \SEE{a}_{3 2}  & 0
\end{matrix}
\right.
\left |
\begin{matrix}
         \SOE{a}_{1 1}  &   \SOE{a}_{1 2} &  \SOE{a}_{1 3} \\
         \SOE{a}_{2,1} & \SOE{a}_{2 2} & \SOE{a}_{2 3} \\
         0 & \SOE{a}_{3 2}  & \SOE{a}_{3 3}
\end{matrix}
\right) ,
\\
&
A =
\left (
\begin{matrix}
         0       &   \SEE{a}_{1 2}   & \SEE{a}_{1 3}\\
         \SEE{a}_{2 1}&      0 & \SEE{a}_{2 3} \\
         \SEE{a}_{3 1} & \SEE{a}_{3 2}  & 0
\end{matrix}
\right.
\left |
\begin{matrix}
         \SOE{a}_{1 1}  &   \SOE{a}_{1 2} &  \SOE{a}_{1 3} \\
         \SOE{a}_{2,1} & \SOE{a}_{2 2} & \SOE{a}_{2 3} \\
         \SOE{a}_{3 1} & \SOE{a}_{3 2}  & \SOE{a}_{3 3}
\end{matrix}
\right) .
\end{align*}
\end{exam}

For any {$A \in \MNZNS(n)$} and  {$ \bs{j}=(j_1, \cdots j_n) \in {\ZZ}^{n} $},
denote
{$ {\frm}^{A, \bs{j}}= {\frm}^{A} \cdot \prod_{1 \le i \le n} {G_i}^{{j}_{i}} $}.
With Proposition \ref{unv_basis} and Theorem \ref{triangular_relation_q},
we construct the monomial basis for {$\USnv$} as  the following corollary immediately.
\begin{cor}\label{monomial_basis}
The set {$\fcM = \{  {\frm}^{A, \bs{j}}
	\where   A \in \MNZNS(n), \bs{j} \in {\ZZ}^{n} \} $} is a {$\Qv$}-basis of {$\USnv$}.
\end{cor}

\spaceintv
\section{A new realization of {$\Uvqn$}}
In this section we will prove the homomorphism {$ \bs{\xi}_n$} is actually an isomorphism of superalgebras.
We will construct a  PBW-type basis following  \cite{DW1},
and then prove the image of this basis under the homomorphism {$ \bs{\xi}_n$} is still a basis of {$\USnv$},
which means {$ \bs{\xi}_n$} is  in fact  an isomorphism.

Following \cite[Theorem 6.2]{Ol} and \cite[Remark 5.6, Lemma 5.7]{DW1},
for {$1 \le i \le n-1$}, set
\footnote{Here  {${\RZ}_{i,j}$}  is in fact  the quantum root vector which is denoted by {${{X}}_{i,j}$} in \cite{DW1}.}
\begin{align*}
{\RZ}_{i, i+1} = {\genE}_{i}, \qquad
{\RZ}_{i+1, i} = {\genF}_{i}, \qquad
{\ol{\RZ}}_{i, i+1} = {\genE}_{\ol{i}}, \qquad
{\ol{\RZ}}_{i+1, i} = {\genF}_{\ol{i}};
\end{align*}
for {$|j-i| > 1$}, set
\begin{align*}
{\RZ}_{i, j} =
\left\{
\begin{aligned}
&{\RZ}_{i, k} {\RZ}_{k, j} - {v}  {\RZ}_{k, j} {\RZ}_{i, k},   &\mbox{ if } i<j, \\
&{\RZ}_{i, k} {\RZ}_{k, j} - {v}^{-1}  {\RZ}_{k, j} {\RZ}_{i, k},   &\mbox{ if } i>j,
\end{aligned}
\right.
\quad
{\ol{\RZ}}_{i, j} =
\left\{
\begin{aligned}
&{\RZ}_{i, k} {\ol{\RZ}}_{k, j} - {v}  {\ol{\RZ}}_{k, j} {\RZ}_{i, k},  &\mbox{ if } i<j, \\
&{\RZ}_{i, k} {\ol{\RZ}}_{k, j} - {v}^{-1}  {\ol{\RZ}}_{k, j} {\RZ}_{i, k},  &\mbox{ if } i>j,
\end{aligned}
\right.
\end{align*}
where {$k$} is any number strictly between {$i$} and {$j$},
and  {${\RZ}_{i, j}$} and {${\ol{\RZ}}_{i, j}$} does not depend on the choice of $k$.

For any {$A \in \MNZNS(n)$} and  {$ \bs{j} \in \ZZ^n$},
denote {$\bs{ {\genK}}^{\bs{j}} = \prod_{k=1}^n {\genK}_{k}^{j_k}$},
and
\begin{align*}
 {\frb}^{A, \bs{j}}
= \bs{ {\genK}}^{\bs{j}}\cdot  ({\prod_{{1} \le j \le {n-1}}} \ 
	 {\prod_{{n} \ge i \ge {j+1}}}  {\RZ}_{i,j}^{\SEE{a}_{i,j}} {\ol{\RZ}}_{i,j}^{\SOE{a}_{i,j}}	)
	\cdot {\prod_{{n} \ge j \ge {1}}}  [   {\genK}_{\ol{j}}^{\SOE{a}_{j,j}}
			( {\prod_{{j-1} \ge i \ge {1}}} {\RZ}_{i,j}^{\SEE{a}_{i,j}} {\ol{\RZ}}_{i,j}^{\SOE{a}_{i,j}})  ] .
\end{align*}

Similar to \cite[Proposition 5.8]{DW1}, we have the following
\begin{prop}\label{new_pbw_basis}
The set
{$ \fcB = \{ {\frb}^{A, \bs{j}}
	\where
	A \in \MNZNS(n), \ \bs{j} \in \ZZ^n
	\}
$}
is a {$\Qv$}-basis of {$\Uvqn$}.
\end{prop}
Recall the algebra homomorphism {$ \bs{\xi}_n$},
denote {${\RP}_{i,j} = \bs{\xi}_n({\RZ}_{i,j})$}, {${\ol{\RP}}_{i,j} = \bs{\xi}_n({\ol{\RZ}}_{i,j})$}.
For any {$A \in \MNZNS(n)$}, {$ \bs{j} \in \ZZ^n$},
denote {$\bs{G}^{\bs{j}} = \xn(\bs{ {\genK}}^{\bs{j}}) = \prod_{k=1}^n G_{k}^{j_k}$}
and  {${\frp}^{A, \bs{j}} = \xn( {\frb}^{A, \bs{j}} )$}.
It is clear
\begin{align*}
 {\frp}^{A, \bs{j}}
= \xn( {\frb}^{A, \bs{j}} )
=\bs{G}^{\bs{j}} \cdot ({\prod_{{1} \le j \le {n-1}}}  \ 
	{\prod_{{n} \ge i \ge {j+1}}} {\RP}_{i,j}^{\SEE{a}_{i,j}} {\ol{\RP}}_{i,j}^{\SOE{a}_{i,j}}	)
	\cdot {\prod_{{n} \ge j \ge {1}}}  [  G_{\ol{j}}^{\SOE{a}_{j,j}}
			( {\prod_{{j-1} \ge i \ge {1}}} {\RP}_{i,j}^{\SEE{a}_{i,j}} {\ol{\RP}}_{i,j}^{\SOE{a}_{i,j}})  ] .
\end{align*}
\begin{prop}\label{map_basis}
The set
{$
\fcP = \{ {\frp}^{A, \bs{j}}
	\where
	A \in \MNZNS(n), \ \bs{j} \in \ZZ^n
	\}
$}
is a {$\Qv$}-basis of {$\USnv$}.
\end{prop}
\begin{proof}
We shall apply the similar calculation used in the proof of Theorem \ref{triangular_relation_q} to compute ${\frp}^{A, \bs{j}}$.

For any {$i,j$}, define  {$U(i, j, x)$} to be
a set of {$x \cdot (|j-i|+1)$} numbers,
while each {$k$} is between {$i$} and {$j$}, and  appears $x$ times in it.
Let  {$S(i, j, x)$} be the set of all possible sequences of the members of  {$U(i, j, x)$}.
For each {$\mu \in S(i, j, x)$},
set {$X_{\mu} = X_{\mu_1} \cdots X_{\mu_{t}}$},
where {$t=x \cdot (|j-i| + 1)$} is the length of {$\mu$}.
For the odd element,
for any   {$\mu \in S(i, j, 1)$},
define {$X_{\ol{\mu}}$} to be the one obtained from {$X_{\mu}$} by replacing {$X_{k}$} with {$X_{\ol{k}}$}, where {$k= \mathrm{max} U(i, j, 1)$} is the maximum number in {$U(i, j, 1)$}.

For the pair {$i,j \in [1, n]$} with {$i < j$},
set {$\mu_{(0)}$} be the sequence such that
{$X_{\mu_{(0)}} =  X_{i}^{\SEE{a}_{i,j}}   X_{i+1}^{\SEE{a}_{i,j}}   \cdots  X_{j-1}^{\SEE{a}_{i,j}}  $},
and set {$\mu_{(1)}$} be the sequence such that
{$X_{\ol{\mu_{(1)}}} =  X_{i} X_{i+1} \cdots  X_{\ol{j-1}}$}.
Direct calculation shows
\begin{align*}
{\RP}_{i,j}^{\SEE{a}_{i,j}}
&=  X_{i}^{\SEE{a}_{i,j}}   X_{i+1}^{\SEE{a}_{i,j}}   \cdots  X_{j-1}^{\SEE{a}_{i,j}}
	+ \sum_{\mu \in S(i, j-1, \SEE{a}_{i,j}) \atop \mu \ne \mu_{(0)}}  f_{\mu} X_{\mu}  ,\\
{\ol{\RP}}_{i,j}
&=  X_{i} X_{i+1} \cdots X_{j-2} X_{\ol{j-1}}
	+ \sum_{\mu \in S(i, j-1, 1) \atop \mu \ne \mu_{(1)}}  f_{\ol{\mu}} X_{\ol{\mu} },
\end{align*}
where each {$f_{\mu}$} or {$f_{\ol{\mu}}$} is $0$  or {${v}^{k}$} for some {$k \in \NN$}.

For any {$M=(\SEE{m}_{h,k} | \SOE{m}_{h,k}) \in \MNZNS(n)$},
{$\bs{p} \in \ZZ^n$},
assume {$m_{h,k}=0$} for all {$h>k$}, or {$k > j$}, or  {$h > i$} and {$k = j$}.
Then for {$\mu \in S(i, j-1, \SEE{a}_{i,j}) $},
Proposition \ref{mulformeven}{\rm(1)} implies
\begin{align*}
X_{\mu}  \AJS(M, \bs{p})
&= f_0 \ABJS(\SE{M} + \SEE{a}_{i,j} E_{i,j}, \SO{M}, \bs{p}'_{M})
	+ \sum_{\substack{B \in \MNZNS(n) , \
			\bs{p}' \in \ZZ^n \\
			B \prec (\SE{M} + \SEE{a}_{i,j} E_{i,j}| \SO{M})} }
			 f_{B, \bs{p}'} \AJS(B, \bs{p}') ,
\end{align*}
where 
{$f_0,  f_{B, \bs{p}'} \in \Qv$}, 
and {$f_0 \ne 0$}  when {$\mu = \mu_{(0)}$}.

As the matrices involved in computing $X_{\ol{\mu}} \AJS(M, \bs{p})$ satisfies the SDP condition, the formulas in Proposition \ref{mulformodd1} are applicable and this means for {$\mu \in S(i, j-1, 1)$},
\begin{align*}
X_{\ol{\mu}} \AJS(M, \bs{p})
=  f_{{1}} \ABJS(\SE{M} ,  \SO{M} + E_{i,j}, \bs{p}''_{M})
	+ \sum_{\substack{B \in \MNZNS(n) , \
			\bs{p}' \in \ZZ^n \\
			B \prec  (\SE{M} | \SO{M} +  E_{i,j})}}
			 f_{B, \bs{p}'} \AJS(B, \bs{p}') ,
\end{align*}
{$f_1,  f_{B, \bs{p}'} \in \Qv$}, 
and {$f_1 \ne 0$}  when {$\mu = \mu_{(1)}$}.
As a consequence, we have
\begin{align*}
{\RP}_{i,j}^{\SEE{a}_{i,j}}   {\ol{\RP}}_{i,j}^{\SOE{a}_{i,j}}  \AJS(M, \bs{p})
&= f_0 \ABJS(\SE{M} + \SEE{a}_{i,j} E_{i,j}, \SO{M} + \SOE{a}_{i,j} E_{i,j}, \bs{p}'''_{M})
	+ \sum_{\substack{B \in \MNZNS(n),  \
			\bs{p}' \in \ZZ^n \\
			B \prec (\SE{M} + \SEE{a}_{i,j} E_{i,j}| \SO{M} + \SOE{a}_{i,j} E_{i,j})}}
			 f_{B, \bs{p}'} \AJS(B, \bs{p}'),
\end{align*}
with {$f_0, f_{B, \bs{p}'}  \in \Qv$} and {$ f_0 \ne 0$}.
Fix $j$,
by induction on $i$ from {$1$} to {$j-1$},
we have
\begin{align*}
({\prod_{{j-1} \ge i \ge {1}}} {\RP}_{i,j}^{\SEE{a}_{i,j}} {\ol{\RP}}_{i,j}^{\SOE{a}_{i,j}})
 \AJS(A^+_{j-1}, \bs{p}_{A^+_{j-1}})
&= f'_j \ABJS(\SE{(A^+_{j})} , \SO{(A^+_{j})} -  \SOE{a}_{j,j} E_{j,j}, \bs{p}'_{A^+_{j-1}})  \\
	&  \qquad
	+ \sum_{\substack{B \in \MNZNS(n),  \
			\bs{p}' \in \ZZ^n \\
			B \prec (  \SE{(A^+_{j})} | \SO{(A^+_{j})} -  \SOE{a}_{j,j} E_{j,j}  )}}
			 f_{B, \bs{p}'} \AJS(B, \bs{p}').
\end{align*}
Applying Proposition \ref{mulformdiag}, we have
\begin{align*}
 G_{\ol{j}}^{\SOE{a}_{j,j}}
({\prod_{{j-1} \ge i \ge {1}}} {\RP}_{i,j}^{\SEE{a}_{i,j}} {\ol{\RP}}_{i,j}^{\SOE{a}_{i,j}})   \AJS(A^+_{j-1},  \bs{p}_{A^+_{j-1}})
&= f^{+}_j  \AJS(A^+_{j}, \bs{p}_{A^+_{j}})
	+ \sum_{\substack{B \in \MNZNS(n)  \\
			\bs{p}'' \in \ZZ^n,  \
			B \prec  A^+_{j} }}
			 f_{B, \bs{p}'} \AJS(B, \bs{p}').
\end{align*}
Then by induction on $k$, we have
\begin{align*}
{\prod_{{1} \le j \le {k}}} G_{\ol{j}}^{\SOE{a}_{j,j}}
	( {\prod_{{j-1} \ge i \ge {1}}} {\RP}_{i,j}^{\SEE{a}_{i,j}} {\ol{\RP}}_{i,j}^{\SOE{a}_{i,j}})
&= f^{+}_k   \AJS(A^+_{k}, \bs{p}_{A^+_{k}})
	+ \sum_{\substack{B \in \MNZNS(n)  \\
			\bs{p}' \in \ZZ^n,  \
			 B \prec  A^+_{k}}}
			 f_{B, \bs{p}'} \AJS(B, \bs{p}')
\end{align*}
and
\begin{align*}
{\prod_{{1} \le j \le {n}}} G_{\ol{j}}^{\SOE{a}_{j,j}}
	( {\prod_{{j-1} \ge i \ge {1}}} {\RP}_{i,j}^{\SEE{a}_{i,j}} {\ol{\RP}}_{i,j}^{\SOE{a}_{i,j}})
&= f^{+}_n  \AJS(A_n^+, \bs{p}_{A_n^+})
	+ \sum_{\substack{B \in \MNZNS(n)  \\
			\bs{p}' \in \ZZ^n,  \
			B \prec  A_n^+} }
			 f_{B, \bs{p}'} \AJS(B, \bs{p}').
\end{align*}

For {$i,j \in [1, n]$} with {$i > j$},
set {$\mu_{(0)}$} be the sequence such that
{$Y_{\mu_{(0)}} =  Y_{i-1}^{\SEE{a}_{i,j}}   Y_{i-2}^{\SEE{a}_{i,j}}   \cdots  Y_{j}^{\SEE{a}_{i,j}}   $},
and set {$\mu_{(1)}$} be the sequence such that
{$Y_{\ol{\mu_{(1)}}} =  Y_{\ol{i-1}}  Y_{i-2}   \cdots  Y_{j} $}.
Direct calculation shows
\begin{equation}\label{p_induct_1}
\begin{aligned}
{\RP}_{i,j}^{\SEE{a}_{i,j}}
&=  Y_{i-1}^{\SEE{a}_{i,j}}   Y_{i-2}^{\SEE{a}_{i,j}}   \cdots  Y_{j}^{\SEE{a}_{i,j}}
	+ \sum_{\mu \in S(i-1, j, \SEE{a}_{i,j}),  \  \mu \ne \mu_{(0)}}  f_{\mu} Y_{\mu}  ,\\
{\ol{\RP}}_{i,j}
&= Y_{\ol{i-1}}  Y_{i-2}   \cdots  Y_{j}
	+ \sum_{\mu \in S(i-1, j, 1),  \  \mu \ne \mu_{(1)}}  f_{\ol{\mu}} Y_{\ol{\mu} },
\end{aligned}
\end{equation}
where each {$f_{\mu}$} or {$f_{\ol{\mu}}$} is $0$  or {${v}^{k}$} for some {$k \in \NN$}.
Referring to equations in \eqref{induction_XYG},
we have
\begin{align}
Y_{\ol{i}}
&= {v}  G_{i+1}^{-1} G_{\ol{i+1}} Y_{i} - {v}^{-1} Y_{i} G_{i+1}^{-1}  G_{\ol{i+1}} , \label{p_induct_2}\\
G_{\ol{i}}
&=X_{i} \cdots X_{n-1}  G_{\ol{n}} Y_{n-1}  \cdots  Y_{i}
	+ \sum_{R} f_R R, \label{p_induct_3}
\end{align}
where {$f_0, f_R \in \Qv$},
each $R$ is a product generated by {$X_{i} , \cdots X_{n-1} $},
{$Y_{i}, \cdots , Y_{n-1}$}, {$G_{i}^{\pm}, \cdots , G_{n}^{\pm}$},
{$G_{\ol{n}}$},
in which each factor appears only $0$ or $1$ times,
and each $R$ is different from the leading term.

As a consequence, fix any {$j \in [1, n-1]$},
the equations in (QQ6) and \eqref{p_induct_1}, \eqref{p_induct_2} imply
\begin{equation}\label{p_induct_4}
\begin{aligned}
{\prod_{{n} \ge i \ge {j+1}}} {\RP}_{i,j}^{\SEE{a}_{i,j}} {\ol{\RP}}_{i,j}^{\SOE{a}_{i,j}}
&=f_0 {\prod_{{n} \ge i \ge {j+1}}}
		G_{\ol{i}}^{\SOE{a}_{i,j} }
		Y_{i-1}^{\sum_{k=i}^n  {a}_{k,j}} 	+ \sum_{Q} f_Q Q,
\end{aligned}
\end{equation}
where {$f_0, f_Q \in \Qv$},
each $Q$ is a product generated by
all the factors in the leading term,
but in different orders from the leading term.

For the pair $i$ and $j$ with {$i>j$},
 {$M=(\SEE{m}_{h,k} | \SOE{m}_{h,k}) \in \MNZNS(n)$} and {$\bs{p} \in \ZZ^n$},
assume {$ \co^{-}(M)_{\min} = j$}, {$\sum_{h=i}^n \SOE{a}_{h, j} = 0$}.
Applying Proposition \ref{mulformeven}{\rm(2)},
when {$i=j+1$} or {$\SEE{m}_{i-1, j} > {\sum_{k=i}^n  {a}_{k,j}} $},
we have,
\begin{equation}\label{p_induct_5}
\begin{aligned}
Y_{i-1}^{\sum_{k=i}^n  {a}_{k,j}}  \AJS(M, \bs{p})
= f^{-}_{M'} \AJS(M' , \bs{p}'_{M'})
	+ \sum_{\substack{B \in \MNZNS(n)  \\
			\bs{p}'' \in \ZZ^n , \
			B \prec M'}}
			 f_{B, \bs{p}''} \AJS(B, \bs{p}''),
\end{aligned}
\end{equation}
where $M' = (\SE{M} - (1-\delta_{i, j+1}) {\sum_{k=i}^n  {a}_{k,j}}  E_{i-1,j} + {\sum_{k=i}^n  {a}_{k,j}}  E_{i,j}| \SO{M})$,
{$f^{-}_{M'},  f_{B, \bs{p}''}  \in \Qv$}.
According to the proof of Lemma \ref{triangular_lower}, we have
\begin{equation}\label{pbw_gi_1}
\begin{aligned}
X_{i} \cdots X_{n-1}  G_{\ol{n}} Y_{n-1}  \cdots  Y_{i} \AJS(M, \bs{p})
&= f^{-}_{M''} \AJS(M'', \bs{p}'_{M''})
	+ \sum_{\substack{B \in \MNZNS(n)  \\
			\bs{p}'' \in \ZZ^n , \
			B \prec M''}}
			 f_{B, \bs{p}''} \AJS(B, \bs{p}''),
\end{aligned}
\end{equation}
where {$M''= (\SE{M} - E_{i,j}| \SO{M} + E_{i,j})$}, {$\bs{p}' \in \ZZ^n$}, 
{$f^{-}_{M''},  f_{B, \bs{p}''}  \in \Qv$}  and {$ f_{M''}  \ne 0$}.
Similar to the discussions in  Lemma \ref{triangular_lower},
for each $R$ in \eqref{p_induct_3}, we also have
\begin{equation}\label{pbw_gi_2}
\begin{aligned}
R \cdot \AJS(M, \bs{p})
&=\sum_{\substack{B \in \MNZNS(n)  , \
			\bs{p}'' \in \ZZ^n \\
			B \prec (\SE{M} - E_{i,j}| \SO{M} + E_{i,j})}}
			 f_{B, \bs{p}''} \AJS(B, \bs{p}'').
\end{aligned}
\end{equation}
Then by \eqref{pbw_gi_1} and  \eqref{pbw_gi_2}, we have
\begin{equation}\label{p_induct_6}
\begin{aligned}
G_{\ol{i}}   \AJS(M, \bs{p})
&= f^{-}_{M'''} \ABJS(\SE{M} - E_{i,j}, \SO{M} + E_{i,j}, \bs{p}'_{M'''})
	+ \sum_{\substack{B \in \MNZNS(n),   \
			\bs{p}'' \in \ZZ^n \\
			B \prec (\SE{M} - E_{i,j}| \SO{M} + E_{i,j})}}
			 f_{B, \bs{p}''} \AJS(B, \bs{p}''),
\end{aligned}
\end{equation}
where {$\bs{p}' \in \ZZ^n$}, {$f^{-}_{M'''},  f_{B, \bs{p}''}  \in \Qv$}, and {$ f \ne 0$}.

By \eqref{p_induct_4}, \eqref{p_induct_5}, \eqref{p_induct_6}
together with the same discussion
for each $Q$ in \eqref{p_induct_4},
we obtain
\begin{align*}
{\prod_{{n} \ge i \ge {j+1}}} {\RP}_{i,j}^{\SEE{a}_{i,j}} {\ol{\RP}}_{i,j}^{\SOE{a}_{i,j}}	  \AJS(A_{j+1}, \bs{p})
&= f' \AJS(A_{j} , \bs{p}'_{A_{j}})
	+ \sum_{\substack{B \in \MNZNS(n)  \\
			\bs{p}'' \in \ZZ^n , \
			B \prec A_{j}}}
			 f_{B, \bs{p}''} \AJS(B, \bs{p}''),
\end{align*}
where {$\bs{p}' \in \ZZ^n$}, {$f,  f_{B, \bs{p}''}  \in \Qv$}.
As a consequence,
\begin{align*}
& ({\prod_{{1} \le j \le {n-1}}} \quad  {\prod_{{n} \ge i \ge {j+1}}} {\RP}_{i,j}^{\SEE{a}_{i,j}} {\ol{\RP}}_{i,j}^{\SOE{a}_{i,j}}	)
\cdot {\prod_{{n} \ge j \ge {1}}}  [  G_{\ol{j}}^{\SOE{a}_{j,j}}
	( {\prod_{{j-1} \ge i \ge {1}}} {\RP}_{i,j}^{\SEE{a}_{i,j}} {\ol{\RP}}_{i,j}^{\SOE{a}_{i,j}})  ] \\
&
=  f'_A \AJS(A , \bs{p}_A)
	+ \sum_{\substack{B \in \MNZNS(n)  \\
			\bs{p}' \in \ZZ^n,  \
			B \prec A}}
			 f_{B, \bs{p}'} \AJS(B, \bs{p}'),
\end{align*}
where {$\bs{p}_A \in \ZZ^n$}, {$f'_A,  f_{B, \bs{p}'}  \in \Qv$}, and {$ f'_A \ne 0$}.
Applying Proposition \ref{mulformzero}, we have
\begin{align*}
& \bs{G}^{\bs{j}} \cdot  ({\prod_{{1} \le j \le {n-1}}} \quad  {\prod_{{n} \ge i \ge {j+1}}} {\RP}_{i,j}^{\SEE{a}_{i,j}} {\ol{\RP}}_{i,j}^{\SOE{a}_{i,j}}	)
\cdot {\prod_{{n} \ge j \ge {1}}}  [  G_{\ol{j}}^{\SOE{a}_{j,j}}
	( {\prod_{{j-1} \ge i \ge {1}}} {\RP}_{i,j}^{\SEE{a}_{i,j}} {\ol{\RP}}_{i,j}^{\SOE{a}_{i,j}})  ] \\
&  =
		 f_{A, \bs{j}} \AJS(A ,  \bs{j} + \bs{p}_A)
	+ \sum_{\substack{B \in \MNZNS(n)  \\
			\bs{p}' \in \ZZ^n, \
			B \prec A}}
			 f_{B, \bs{p}'} \AJS(B, \bs{p}'),
\end{align*}
where 
{$ f_{A, \bs{j}} ,   f_{B, \bs{p}'}  \in \Qv$}, and {$ f_{A, \bs{j}} \ne 0$}.

By Proposition \ref{unv_basis},
the set {$\LS =\{  \AJS(A, \bs{p}_A + \bs{j})  \where   A \in \MNZNS(n), \bs{j} \in {\ZZ}^{n} \} $} is a basis of {$\USnv$}.
Hence the result is proved.
\end{proof}

\begin{thm}\label{map_iso}
The vector space {$\USnv$}  is a superalgebra generated by
	the elements
	{$G_{i}^{\pm 1}$}, 
	{$G_{\ol{i}}$},
	{$X_{j}$}, {$X_{\ol{j}}$},
	{$Y_{j}$}, {$Y_{\ol{j}}$},
	({$ 1 \le i \le n$}, {$ 1 \le j \le n-1$}).
Furthermore, 
the image of the {$\Qv$}-superalgebra homomorphism {$\bs{\xi}_n$} given in Theorem \ref{qqschur_reltion} is $\USnv$ 
and $\bs{\xi}_n$ induces a superalgebra isomorphism $\bs{\xi}_n:\Uvqn\overset\sim\to\USnv$.
\end{thm}
\begin{proof}
Theorem \ref{qqschur_reltion} shows  {$\bs{\xi}_n$}  is an algebra homomorphism
into {${\SQvnR} $}.
Theorem \ref{triangular_relation_q} and Proposition \ref{common_form}
show  {$\USnv$} is a superalgebra generated by the elements
{$G_{i}$}, {$G_{i}^{-1}$},{$G_{\ol{i}}$},
	{$X_{j}$}, {$X_{\ol{j}}$},
	{$Y_{j}$}, {$Y_{\ol{j}}$},
which means the image of {$\bs{\xi}_n$}  is {$\USnv$},
and {$\bs{\xi}_n$}  is an epimorphism onto {$\USnv$}.
Proposition \ref{new_pbw_basis} and \ref{map_basis}
show  {$\bs{\xi}_n$} is injective and hence is an isomorphism.
\end{proof}

As a consequence, there is a superalgebra epimorphism
\begin{align}
    \bs{\eta}_r: \qquad
& \USnv \to  {\SQvnrR},  \qquad
  \AJS(A, \bs{j}) \mapsto   \AJRS(A, \bs{j}, r).
  \label{superalgebra_epi}
\end{align}

Because of the isomorphism {$ \bs{\xi}_n$},
we can identify
{$ G_{i}^{\pm 1} $},  {$  G_{\ol{i}} $},
{$ X_{j} $}, {$  X_{\ol{j}} $},
{$ Y_{j} $}, {$  Y_{\ol{j}} $}
with
{$ {\genK}_{i}^{\pm 1} $},   {$  {\genK}_{\ol{i}} $},
{$ {\genE}_{j} $}, {$  {\genE}_{\ol{j}} $},
	{$ {\genF}_{j} $}, {$  {\genF}_{\ol{j}} $} resp.,
for all {$1 \le i \le n, 1 \le j \le n-1$}.
Then one can directly write down the multiplication formulas for the generators {$ G_{i}^{\pm 1} $},  {$  G_{\ol{i}} $},
{$ X_{j} $}, {$  X_{\ol{j}} $},
{$ Y_{j} $}, {$  Y_{\ol{j}} $} with the basis {$\{\AJS(A, \bs{j})  \where   A \in \MNZNS(n), \bs{j} \in {\ZZ}^{n} \}$}.

\vspace{0.5cm}
{\it Acknowledgement}: 
\iffalse
The work was  partially supported by the UNSW Science FRG 
and the Natural Science Foundation of China (\#12071129, \#11871404, \#12122101, \#12071026).
\else
The first author acknowledges the support from UNSW Science FRG.
The second author  acknowledges the support from NSFC-12071129. 
The third author would like to thank the support from Professor Yanan Lin in Xiamen University,
and  NSFC-11871404.
The fourth author is supported by NSFC-12122101 and NSFC-12071026. 
\fi
\appendix
\section{Proofs of Propositions \ref{mulformeven}(2), \ref{mulformodd1}, \ref{mulformodd2}}

\noindent
\begin{proof}[\bf Proof of Proposition \ref{mulformeven}(2)]
By  Proposition \ref{prop_PhiAPhiB},  \eqref{def_ajr} and Proposition \ref{phiupper0}(3), we can write
\begin{align*}
  \ABJRS( E_{h+1, h}, \mathrm{O}, \bs{ 0 },  r ) \cdot \ABJRS( \SE{A}, \SO{A}, \bs{j},  r )
& =	\sum_{\substack{ \\ \lambda \in  \CMN(n,r-\snorm{A}) }
	} {v}^{\lambda \cdot {\bs{j}}} {\Phi}_{(E_{h+1, h} + \ro(A)+\lambda-\bs{\ep}_h | \mathrm{O} )}{\Phi}_{( \SE{A} + \lambda | \SO{A} )} \\
&= {\fcY}_1 + {\fcY}_2 - {\fcY}_3,
\end{align*}
where
\begin{align*}
{\fcY}_1 &=
	\sum_{\substack{  \lambda \in \CMN(n,r-\snorm{A}) } }
	{v}^{\lambda \cdot {\bs{j}}}
	\sum_{k=1}^n
	{q}^{\sum^{k-1}_{j=1}{(A + \lambda )}_{h+1, j}}
	\STEP{ \SEE{(A + \lambda )}_{h+1, k} +1}
	{\Phi}_{(\SE{A} + \lambda  - E_{h,k} + E_{h+1, k} | \SO{A} )}, \\
{\fcY}_2
&=
	\sum_{\substack{  \lambda \in \CMN(n,r-\snorm{A}) } }
	{v}^{\lambda \cdot {\bs{j}}}
	\sum_{k=1}^n
	{q}^{\sum^{k-1}_{j=1}{(A + \lambda )}_{h+1, j} + {(A + \lambda )}_{h, k} - 1}
	{\Phi}_{(\SE{A} + \lambda |\SO{A} - E_{h,k} + E_{h+1,k})}, \\
{\fcY}_3
&=
	\sum_{\substack{  \lambda \in \CMN(n,r-\snorm{A}) } }
	{v}^{\lambda \cdot {\bs{j}}}
	\sum_{k=1}^n
	{q}^{\sum^{k-1}_{j=1}{(A + \lambda )}_{h+1, j} + {(A + \lambda )}_{h, k} -2}  \\
	&\qquad \qquad \cdot
	{\STEPPD{  {(A + \lambda )}_{h+1, k} +1}}
	{\Phi}_{(\SE{A} + \lambda  + 2E_{h+1,k} | \SO{A}  - E_{h,k} - E_{h+1,k})}.
\end{align*}

To compute $\fcY_1$, we consider the  four cases $k<h, k=h, k=h+1$ and $k>h+1$ similar to the approach using in \eqref{eq:compute Y1} and then by Lemma \ref{formstepodd}{\rm(2)} we have
\begin{align*}
{\fcY}_1
&=
	\sum_{k < h}
	{q}^{\AK(h+1,k)}
	\STEP{ \SEE{a}_{h+1, k} +1}
	\ABJRS( \SE{A} - E_{h,k} + E_{h+1, k}, \SO{A}, \bs{j},  r ) \\
	&\qquad +
	{q}^{\AK(h+1,h)} {v}^{j_h}
	\STEP{ \SEE{a}_{h+1, h} +1}
	\ABJRS( \SE{A} + E_{h+1, h}, \SO{A}, \bs{j},  r ) \\
	&\qquad +
	\frac{{q}^{\AK(h+1,h+1)} }{ ({q} - 1){v}^{j_{h+1}} }\{ \quad
	\ABJRS( \SE{A} - E_{h,h+1}, \SO{A}, \bs{j} + 2 \bs{\ep}_{h+1},  r )
	 -
	\ABJRS( \SE{A} - E_{h,h+1}, \SO{A}, \bs{j},  r ) \quad \}
	\\
	&\qquad +
	\sum_{k > h+1}
	{q}^{ \AK(h+1,k)}
	\STEP{ \SEE{a}_{h+1, k} +1}
	\ABJRS( \SE{A} - E_{h,k} + E_{h+1, k}, \SO{A}, \bs{j} + 2 \bs{\ep}_{h+1}, r ),
\end{align*}
Similarly, we have
\begin{align*}
{\fcY}_2
&=
	\sum_{k<h}
	{q}^{\AK(h+1,k) + {a}_{h, k} - 1}
	\ABJRS(  \SE{A}, \SO{A} - E_{h,k} + E_{h+1,k} , \bs{j},  r ) \\
	& \qquad +
	{q}^{\AK(h+1,h)  }
	\ABJRS(\SE{A}, \SO{A} - E_{h,h} + E_{h+1,h},  {\bs{j}} + 2 \bs{\ep}_{h}, r) \\
	& \qquad +
	{q}^{\AK(h+1,h+1) + {a}_{h, h+1} - 1}
	\ABJRS(\SE{A}, \SO{A} - E_{h,h+1} + E_{h+1,h+1}, \bs{j}, r) \\
	& \qquad +
	\sum_{k>h+1}
	{q}^{\AK(h+1,k) + {a}_{h, k} - 1}
	\ABJRS(\SE{A}, \SO{A} - E_{h,k} + E_{h+1,k}, {\bs{j}} + 2 \bs{\ep}_{h+1},  r ) ,
\end{align*}
where in the computation the fact  ${\Phi}_{(\SE{A} + \lambda |\SO{A} - E_{h,h} + E_{h+1,h})} \mbox{ appears when }  \SOE{a}_{h, h} = 1$ is used.
Finally,
\begin{align*}
{\fcY}_3
&=
	\sum_{k < h }
	{q}^{\AK(h+1,k) + {a}_{h, k} -2}
	\STEPPDR{  {a}_{h+1, k} +1}
	\ABJRS( \SE{A} + 2E_{h+1,k}, \SO{A}  - E_{h,k} - E_{h+1,k}, \bs{j}, r ) \\
	& \qquad +
	{q}^{\AK(h+1,h) -1}
	\STEPPDR{  {a}_{h+1, h} +1}
	\ABJRS( \SE{A}  + 2E_{h+1,h}, \SO{A}  - E_{h,h} - E_{h+1,h}, \bs{j} + 2 \bs{\ep}_{h},  r  ) \\
	& \qquad +
	{q}^{\AK(h+1,h+1) + {a}_{h, h+1} -2}
	\frac{1}{({q}^2 - 1){v}^{2 j_{h+1} }}
	\{ \quad
	\ABJRS( \SE{A}, \SO{A}  - E_{h,h+1} - E_{h+1,h+1}, \bs{j} +  4 \bs{\ep}_{h+1},  r ) \\
	&\qquad \qquad -
	({q} + 1) \ABJRS( \SE{A}, \SO{A}  - E_{h,h+1} - E_{h+1,h+1}, \bs{j} +  2 \bs{\ep}_{h+1},  r ) \\
	&\qquad \qquad +
	{q} \ABJRS( \SE{A}, \SO{A}  - E_{h,h+1} - E_{h+1,h+1}, \bs{j},  r )
	 \quad \}
	 \qquad 
	 \\
	& \qquad +
	\sum_{k>h+1}
	{q}^{ \AK(h+1,k) + {a}_{h, k} -2}
	\STEPPDR{  {a}_{h+1, k} +1}
	\ABJRS( \SE{A} + 2E_{h+1,k}, \SO{A}  - E_{h,k} - E_{h+1,k}, \bs{j} + 2 \bs{\ep}_{h+1}, r ) ,
\end{align*}
where we use the fact that ${\Phi}_{(A^0+\lambda+2E_{h+1,h}|A^1-E_{h,h}-E_{h+1,h})}$ appears only when $a_{h,h}^1=1=a_{h+1,h}^1$
and ${\Phi}_{(A^0+\lambda+2E_{h+1,h+1}|A^1-E_{h,h+1}-E_{h+1,h+1})}$ appears only when $a_{h+1,h+1}^1=1=a_{h,h+1}^1$. Hence {\rm(2)} is proved.
\end{proof}

\begin{proof}[\bf Proof of Proposition \ref{mulformodd1}]
Observe that for any {$\mu \in \CMN(n, r-1)$},
{${(-1)}^{\parity{\mu|E_{h,h+1}} \cdot \parity{A}}
= {(-1)}^{\parity{A}}$}.

By the assumption that for each $1\leq k\leq n$ and $\lambda \in \CMN(n, r-\snorm{A})$ with {${(A+\lambda)}_{h+1,k} \ge 1$},
{$({A+\lambda})^+_{h,k}$} satisfies SDP condition at $(h,k)$,
we have via  Proposition \ref{prop_PhiAPhiB},  \eqref{def_ajr} and Proposition \ref{phiupper1}
\begin{align*}
  \ABJRS( \mathrm{O}, E_{h, h+1}, \bs{ 0 },  r ) \cdot \ABJRS( \SE{A}, \SO{A}, \bs{j},  r )
& =
	\sum_{\substack{
		\lambda \in \CMN(n,r-\snorm{A}) \\
		\mu \in \CMN(n, r-1) \\
		\co(\mu + E_{h,h+1}) = \ro(A+\lambda)
	} }
	{v}^{\mu \cdot {\bs{0}}} {v}^{\lambda \cdot {\bs{j}}}
	{\Phi}_{(\mu| E_{h, h+1})} {\Phi}_{( \SE{A} + \lambda | \SO{A} )} \\
&=  {(-1)}^{\parity{A}}({\fcY}_1 + {\fcY}_2 + {\fcY}_3),
\end{align*}
where
\begin{align*}
{\fcY}_1&=
	\sum_{\substack{  \lambda \in \CMN(n,r-\snorm{A}) } }
	{v}^{\mu \cdot {\bs{0}}} {v}^{\lambda \cdot {\bs{j}}}
	\sum_{k=1}^n
	{(-1)}^{{ {\SOE{\widetilde{a}}}_{h-1,k}}} {q}^{ {\sum^{n}_{u=k+1}{(A+\lambda)}_{h, u}}  + \SOE{a}_{h+1,k} }
	{\Phi}_{(\SE{A} + \lambda  - E_{h+1, k}| \SO{A}  + E_{h,k} )} \\
{\fcY}_2&=\sum_{\substack{  \lambda \in \CMN(n,r-\snorm{A}) } }
	{v}^{\mu \cdot {\bs{0}}} {v}^{\lambda \cdot {\bs{j}}}
	\sum_{k=1}^n
	{(-1)}^{{ {\SOE{\widetilde{a}}}_{h-1,k}} + 1 - \SOE{a}_{h,k} } {q}^{ {\sum^{n}_{u=k+1}{(A+\lambda)}_{h, u}}  }  \STEP{\SEE{(A+\lambda)}_{h,k} +1}
		{\Phi}_{(\SE{A} + \lambda  + E_{h,k}| \SO{A} - E_{h+1, k} )}\\
{\fcY}_3&=
	\sum_{\substack{  \lambda \in \CMN(n,r-\snorm{A}) } }
	{v}^{\mu \cdot {\bs{0}}} {v}^{\lambda \cdot {\bs{j}}}
	\sum_{k=1}^n
	{(-1)}^{{ {\SOE{\widetilde{a}}}_{h-1,k}} }
	{q}^{ {\sum^{n}_{u=k+1}{(A+\lambda)}_{h, u}} -1 + \SOE{a}_{h+1,k}}
	\STEPPDR{ {(A+\lambda)}_{h,k} + 1} \\
	& \qquad \qquad \cdot  {\Phi}_{(\SE{A} + \lambda  -E_{h+1, k} + 2 E_{h,k} |\SO{A} - E_{h,k})}.
\end{align*}
Direct calculation shows
\begin{align*}
{\fcY}_1
&=
	\sum_{k<h}
	{(-1)}^{{ {\SOE{\widetilde{a}}}_{h-1,k}}}
	{q}^{ {\BK(h,k)}  + \SOE{a}_{h+1,k} }
	\ABJRS( \SE{A} - E_{h+1, k}, \SO{A}  + E_{h,k}, \bs{j} + 2 \bs{\ep}_{h},  r ) \\
	& \qquad +
	{(-1)}^{{ {\SOE{\widetilde{a}}}_{h-1,h}}}
	{q}^{ {\BK(h,h)}  + \SOE{a}_{h+1,h} }
	\ABJRS( \SE{A} - E_{h+1, h}, \SO{A}  + E_{h,h}, \bs{j}, r )\\
	& \qquad +
	{(-1)}^{{ {\SOE{\widetilde{a}}}_{h-1,h+1}}}
	{q}^{ {\BK(h,h+1)}  + \SOE{a}_{h+1,h+1} } {v}^{j_{h+1} }
	\ABJRS( \SE{A}, \SO{A}  + E_{h,h+1}, \bs{j},  r ) \\
	& \qquad +
	\sum_{ \substack{ k > h+1}  }
	{(-1)}^{{ {\SOE{\widetilde{a}}}_{h-1,k}}}
	{q}^{ {\BK(h,k)}  + \SOE{a}_{h+1,k} }
	\ABJRS( \SE{A} - E_{h+1, k}, \SO{A}  + E_{h,k}, \bs{j}, r ),
\end{align*}
Again similar to \eqref{eq:compute Y1}, by Lemma \ref{formstepodd}{\rm(2)} one can obtain
\begin{align*}
{\fcY}_2
&=
	\sum_{k < h}
	{(-1)}^{{ {\SOE{\widetilde{a}}}_{h-1,k}} + 1 - \SOE{a}_{h,k} }
	{q}^{ {\BK(h,k)} }
	\STEP{ \SEE{a}_{h,k} +1}
	\ABJRS( \SE{A}+ E_{h,k}, \SO{A} - E_{h+1, k}, \bs{j} + 2 \bs{\ep}_{h},  r )  \\
	& \qquad +
	\frac{{(-1)}^{{ {\SOE{\widetilde{a}}}_{h-1,h}} + 1 - \SOE{a}_{h,h} } {q}^{ {\BK(h,h)} } }{({q} - 1) {v}^{j_h}}
		\{ \quad \ABJRS( \SE{A}, \SO{A}- E_{h+1, h} , {\bs{j}} + 2 \bs{\ep}_{h}, r )
		-  \ABJRS(  \SE{A}, \SO{A}- E_{h+1, h} , {\bs{j}} , r ) \quad \} \\
	& \qquad +
	\sum_{k > h}
	{(-1)}^{{ {\SOE{\widetilde{a}}}_{h-1,k}} + 1 - \SOE{a}_{h,k} }
	{q}^{ {\BK(h,k)}}
	\STEP{ \SEE{a}_{h,k} +1}
	\ABJRS( \SE{A}  + E_{h,k}, \SO{A} - E_{h+1, k}, \bs{j},  r ).
\end{align*}
Finally direct calculation using Lemma \ref{formstepodd}(1) shows
\begin{align*}
{\fcY}_3
&=
	\sum_{k < h}
	{(-1)}^{{ {\SOE{\widetilde{a}}}_{h-1,k}} }
	\STEPPDR{ {a}_{h,k} + 1}
	{q}^{ {\BK(h,k)} -1 + \SOE{a}_{h+1,k}}
	\ABJRS( \SE{A} -E_{h+1, k} + 2 E_{h,k}, \SO{A} - E_{h,k}, {\bs{j}} + 2 \bs{\ep}_{h} ,  r )  \\
	& \qquad +
	{(-1)}^{{ {\SOE{\widetilde{a}}}_{h-1,h}} }
	\frac{{q}^{ {\BK(h,h)} -1 + \SOE{a}_{h+1,h}}}{({q}^2 - 1) {v}^{2 j_h}} \{ \quad
	\ABJRS( \SE{A}  -E_{h+1, h}, \SO{A}  -E_{h,h}, \bs{j} + 4 \bs{\ep}_{h},  r ) \\
	&\qquad \qquad -
	{({q} + 1)}
	\ABJRS( \SE{A} -E_{h+1, h}, \SO{A}  -E_{h,h}, \bs{j} + 2 \bs{\ep}_{h},  r ) \\
	&\qquad \qquad +
	{v}^{2}
	\ABJRS( \SE{A} -E_{h+1, h}, \SO{A}  -E_{h,h}, \bs{j},  r ) \quad \}
	\\
	& \qquad +
	{(-1)}^{{ {\SOE{\widetilde{a}}}_{h-1,h+1}} }
	{q}^{ {\BK(h,h+1)} -1 + \SOE{a}_{h+1,h+1}}{v}^{j_{h+1}}
	\STEPPDR{ {a}_{h,h+1} + 1}
	\ABJRS( \SE{A}  + 2 E_{h,h+1}, \SO{A} - E_{h,h+1}, \bs{j},  r  )  \\
	& \qquad +
	\sum_{k>h+1}
	{(-1)}^{{ {\SOE{\widetilde{a}}}_{h-1,k}} }
	{q}^{ {\BK(h,k)} -1 + \SOE{a}_{h+1,k}}
	\STEPPDR{ {a}_{h,k} + 1}
	\ABJRS( \SE{A}  -E_{h+1, k} + 2 E_{h,k}, \SO{A} - E_{h,k}, \bs{j},  r ) ,
\end{align*}
where the second equality is due to fact that ${\Phi}_{(A^0+\lambda-E_{h+1,h+1}+2E_{h,h}|A^1-E_{h,h})}$ appears only when $a_{h,h}^1=1$.
Putting together and replacing $q$ by $v^2$, we have proved the proposition.
\end{proof}

\begin{proof}[\bf Proof of Proposition \ref{mulformodd2}]

Observe that for any {$\mu \in \CMN(n, r-1)$},
{${(-1)}^{\parity{\mu|E_{h+1,h}} \cdot \parity{A}}
= {(-1)}^{\parity{A}}$}.

By the assumption 
for any  $\lambda \in \CMN(n, r-\snorm{A})$,
{$({A+\lambda})$}  satisfies the SDP condition on the $h$-th row,
we compute by  Proposition \ref{prop_PhiAPhiB},  Proposition \ref{philower1} and \eqref{def_ajr}
\begin{align*}
  \ABJRS( \mathrm{O}, E_{h+1, h}, \bs{ 0 },  r ) \cdot \ABJRS( \SE{A}, \SO{A}, \bs{j},  r )
& =\sum_{\substack{  \\ \mu \in \CMN(n,r-1)  }}
	{v}^{\mu \cdot {\bs{0}}} {\Phi}_{(\mu |E_{h+1, h} )}
	\cdot
	\sum_{\substack{ \\ \lambda \in  \CMN(n,r-\snorm{A}) }
	} {v}^{\lambda \cdot {\bs{j}}} {\Phi}_{(\SE{A} + \lambda  | \SO{A})} \\
&=  {(-1)}^{\parity{A}}({\fcY}_1 + {\fcY}_2 + {\fcY}_3),
\end{align*}
where
\begin{align*}
{\fcY}_1&=
	\sum_{\substack{  \lambda \in \CMN(n,r-\snorm{A}) } }
	{v}^{\lambda \cdot {\bs{j}}}
 	\sum_{k=1}^n
	{(-1)}^{{\SOE{\widetilde{a}}}_{h-1,k} + \SOE{a}_{h,k}} {q}^{ \sum^{k-1}_{j=1}{(A + \lambda )}_{h+1, j}}
	{\Phi}_{(\SE{A} + \lambda - E_{h,k}| \SO{A}+ E_{h+1, k})} \\
{\fcY}_2&=
	\sum_{\substack{  \lambda \in \CMN(n,r-\snorm{A}) } }
	{v}^{\lambda \cdot {\bs{j}}}
 	\sum_{k=1}^n
	{(-1)}^{{\SOE{\widetilde{a}}}_{h-1,k} + \SOE{a}_{h,k}+1}
	{q}^{ \sum^{k-1}_{j=1}{(A + \lambda )}_{h+1, j} -1}
	\STEPPDR{ {(A + \lambda )}_{h+1,k} +1}\\
	&\qquad \qquad \cdot   {\Phi}_{(\SE{A} + \lambda  - E_{h,k} + 2E_{h+1, k} | \SO{A}  -E_{h+1, k})} \\
{\fcY}_3&=\sum_{\substack{  \lambda \in \CMN(n,r-\snorm{A}) } }
	{v}^{\lambda \cdot {\bs{j}}}
 	\sum_{k=1}^n
	{(-1)}^{{\SOE{\widetilde{a}}}_{h-1,k}+1} {q}^{ \sum^{k-1}_{j=1}{(A + \lambda )}_{h+1, j} + {(A + \lambda )}_{h,k}-1}
	\STEP{ \SEE{(A + \lambda )}_{h+1,k}+1} \\
	& \qquad \qquad  \qquad  \cdot
	{\Phi}_{(\SE{A} + \lambda  + E_{h+1,k} | \SO{A}  - E_{h,k} )} \\
\end{align*}
Direct calculation shows
\begin{align*}
{\fcY}_1 &=
 	\sum_{k<h}
	{(-1)}^{{\SOE{\widetilde{a}}}_{h-1,k} + \SOE{a}_{h,k}}
	{q}^{ \AK(h+1,k)}
	\ABJRS( \SE{A} - E_{h,k}, \SO{A}+ E_{h+1, k}, \bs{j},  r ) \\
	&\qquad +
	{(-1)}^{{\SOE{\widetilde{a}}}_{h-1,h} + \SOE{a}_{h,h}}
	{q}^{ \AK(h+1,h)} {v}^{ j_h }
	\ABJRS( \SE{A} , \SO{A}+ E_{h+1, h}, \bs{j},  r ) \\
	&\qquad +
	{(-1)}^{{\SOE{\widetilde{a}}}_{h-1,h+1} + \SOE{a}_{h,h+1}}
	{q}^{ \AK(h+1,h+1)}
	\ABJRS( \SE{A} - E_{h,h+1}, \SO{A}+ E_{h+1, h+1}, \bs{j},  r ) \\
	&\qquad +
 	\sum_{k>h+1}
	{(-1)}^{{\SOE{\widetilde{a}}}_{h-1,k} + \SOE{a}_{h,k}}
	{q}^{  \AK(h+1,k)}
	\ABJRS( \SE{A} - E_{h,k}, \SO{A}+ E_{h+1, k}, \bs{j} + 2 \bs{\ep}_{h+1},  r ).
\end{align*}
Meanwhile by Lemma \ref{formstepodd}(1) and the fact that ${\Phi}_{(A^0+\lambda-E_{h,h}+2E_{h+1,h+1}|A^1-E_{h+1,h+1})}$ appears only when $a^1_{h+1,h+1}=1$
we obtain
\begin{align*}
{\fcY}_2
&=
 	\sum_{k<h}
	{(-1)}^{{\SOE{\widetilde{a}}}_{h-1,k} + \SOE{a}_{h,k}+1}
	{q}^{ \AK(h+1,k)-1}
	( {\STEPP{ {a}_{h+1,k} +1}} -  \STEP{ {a}_{h+1,k}+1}) \\
	& \qquad \qquad \cdot
	\ABJRS( \SE{A}  - E_{h,k} + 2E_{h+1, k}, \SO{A}  -E_{h+1, k}, \bs{j},  r ) \\
	& \qquad +
	{(-1)}^{{\SOE{\widetilde{a}}}_{h-1,h} + \SOE{a}_{h,h}+1}
	{q}^{ \AK(h+1,h) -1 }{v}^{j_h}
	\STEPPDR{ {a}_{h+1,h} +1}
	\ABJRS( \SE{A} + 2E_{h+1, h}, \SO{A}  -E_{h+1, h}, \bs{j},  r ) \\
	& \qquad +
	{(-1)}^{{\SOE{\widetilde{a}}}_{h-1,h+1} + \SOE{a}_{h,h+1}+1}
	{q}^{ \AK(h+1,h+1)-1}
	\frac{1}{({q}^2 - 1){v}^{2 j_{h+1}}}  \{ \quad
	\ABJRS( \SE{A}  - E_{h,h+1}, \SO{A}  -E_{h+1, h+1}, \bs{j} + 4 \bs{\ep}_{h+1},  r ) \\
	&\qquad \qquad -
	 ({q} + 1) \ABJRS( \SE{A}  - E_{h,h+1}, \SO{A}  -E_{h+1, h+1}, \bs{j} + 2 \bs{\ep}_{h+1},  r ) \\
	&\qquad \qquad +
	 {q} \ABJRS( \SE{A}  - E_{h,h+1}, \SO{A}  -E_{h+1, h+1}, \bs{j},  r )
	 \quad \}
	 \\
	& \qquad +
 	\sum_{k>h+1}
	{(-1)}^{{\SOE{\widetilde{a}}}_{h-1,k} + \SOE{a}_{h,k}+1}
	{q}^{ \AK(h+1,k)-1}
	\STEPPDR{ {a}_{h+1,k} +1} \\
	& \qquad \qquad \cdot \ABJRS( \SE{A}  - E_{h,k} + 2E_{h+1, k}, \SO{A}  -E_{h+1, k}, \bs{j}+ 2 \bs{\ep}_{h+1},  r ) .
\end{align*}
Finally straightforward computation leads to
\begin{align*}
{\fcY}_3
&=
	\sum_{k < h}
 	{(-1)}^{{\SOE{\widetilde{a}}}_{h-1,k}+1} {q}^{ \AK(h+1,k) + {a}_{h,k}-1}
	\STEP{ \SEE{a}_{h+1,k}+1}
	\ABJRS( \SE{A} +  E_{h+1,k}, \SO{A}  - E_{h,k}, \bs{j},  r ) \\
	& \qquad +
 	{(-1)}^{{\SOE{\widetilde{a}}}_{h-1,h}+1}
	{q}^{ \AK(h+1,h) }
	\STEP{ \SEE{a}_{h+1,h}+1}
	\ABJRS( \SE{A} +  E_{h+1,h}, \SO{A}  - E_{h,h}, \bs{j} + 2 \bs{\ep}_{h},  r ) \\
	& \qquad +
 	{(-1)}^{{\SOE{\widetilde{a}}}_{h-1,h+1}+1}
		\frac{{q}^{ \AK(h+1,h+1) + {a}_{h,h+1}-1} }{ ({q} - 1){v}^{j_{h+1}} }\{ \quad
	\ABJRS( \SE{A}, \SO{A}  - E_{h,h+1}, \bs{j} + 2 \bs{\ep}_{h+1},  r )  \\
	& \qquad \qquad \qquad -
	\ABJRS( \SE{A}, \SO{A}  - E_{h,h+1}, \bs{j},  r )
	 \quad \}
	 \quad \mbox{(by Lemma \ref{formstepodd} {\rm(1)})}
	 \\
	& \qquad +
	\sum_{k>h+1}
 	{(-1)}^{{\SOE{\widetilde{a}}}_{h-1,k}+1}
	{q}^{ \AK(h+1,k) + {a}_{h,k}-1}
	\STEP{ \SEE{a}_{h+1,k}+1}
	\ABJRS( \SE{A}  + E_{h+1,k}, \SO{A}  - E_{h,k}, \bs{j} + 2 \bs{\ep}_{h+1},  r ).
\end{align*}
Then the proposition is proved.
\end{proof}


\end{document}